\documentclass{mcom-l}

\usepackage{graphicx}
\usepackage[margin=1 in]{geometry}
\usepackage{amssymb,amsfonts}
\usepackage[all,arc]{xy}
\usepackage{enumerate}
\usepackage[shortlabels]{enumitem}
\usepackage{mathrsfs}
\usepackage{bbm}
\usepackage{color}
\usepackage[cmyk,usenames,dvipsnames]{xcolor}
\usepackage{verbatim}
\usepackage{cite}
\usepackage{hyperref} 
\usepackage{cancel}
\usepackage[export]{adjustbox}
\usepackage{dsfont}

\newcommand{\indi}{\mathds{1}}

\newcommand{\rr}{\ensuremath{\mathbb{R}}}
\newcommand{\Rd}{\ensuremath{{\mathbb{R}^d}}}

\newcommand{\ep}{\ensuremath{\varepsilon}}

\newcommand{\bgamma}{\boldsymbol{\gamma}}

\newcommand{\Eep}{\ensuremath{\mathcal{E}_\ep}}

\def\P{{\mathcal P}}

\newcommand{\grad}{\nabla}
\newcommand{\loc}{{\rm loc}}
\newcommand{\KL}{\ensuremath{\text{ KL}}}

\newcommand{\R}{{\mathord{\mathbb R}}}

\newcommand{\N}{{\mathord{\mathbb N}}}
\newcommand{\supp}{{\mathop{\rm supp\ }}}
\newcommand{\id}{\boldsymbol{ \mathop{\rm id}}}

\newcommand{\la}{\left\langle}
\newcommand{\ra}{\right\rangle}
\newcommand{\F}{\mathcal{F}}
\newcommand{\G}{\mathcal{G}}

\newcommand{\V}{\mathcal{V}}

\newcommand{\Ld}{\mathcal{L}^d}

\newcommand{\argmin}{\operatornamewithlimits{argmin}}
\newcommand{\bt}{\mathbf{t}}

\newcommand{\ird}{\int_{\mathord{\mathbb R}^d}}

\newcommand{\E}{\mathcal{E}}
\newcommand{\rS}{{\rm S}}
\newcommand{\J}{{\rm J}}

\newcommand{\teta}{\tilde{\bEta}}

\newcommand{\bEta}{\boldsymbol{\eta}}

\newcommand{\bxi}{\boldsymbol{\xi}}

\renewcommand{\:}{\colon}

\def\P{{\mathcal P}}
\def\S{{\mathcal S}}

\def\epsilon{\varepsilon}
\def\e{\varepsilon}
\def\F{\mathcal{F}}

\newcommand{\be}{\begin{equation}}
\newcommand{\ee}{\end{equation}}
\newcommand{\bes}{\begin{equation*}}
\newcommand{\ees}{\end{equation*}}

\newtheorem{thm}{Theorem}[section]

\newtheorem{cor}[thm]{Corollary}
\newtheorem{prop}[thm]{Proposition}
\newtheorem{lem}[thm]{Lemma}

\theoremstyle{definition}

\newtheorem{defn}[thm]{Definition}

\theoremstyle{remark}
\newtheorem{rem}[thm]{Remark}

\numberwithin{equation}{section}

\begin{document}
\title[A blob method for inhomogeneous diffusion]{A blob method for inhomogeneous diffusion with applications to multi-agent control and sampling}
\date{\today}

\author{Katy Craig}
\address{University of California, Santa Barbara, Department of Mathematics}
\curraddr{}
\email{kcraig@math.ucsb.edu}
\thanks{The work of K. Craig has been supported by NSF DMS grants 1811012 and 2145900, as well as a Hellman Faculty Fellowship. K. Craig and O. Turanova gratefully acknowledge the support from the Simons Center for Theory of Computing, at which part of this work was completed.}

\author{Karthik Elamvazhuthi}
\address{University of California, Riverside, Department of Mechanical Engineering,}
\curraddr{}
\email{kelamvazhuthi@engr.ucr.edu}
\thanks{The work of K. Elamvazhuthi has been supported by AFOSR grants FA9550-18-1-0502 and FA9550-18-1-0502.}

\author{Matt Haberland}
\address{California Polytechnic State University, BioResource and Agricultural Engineering Department}
\curraddr{}
\email{mhaberla@calpoly.edu}

\author{Olga Turanova}
\address{Michigan State University, Department of Mathematics}
\curraddr{}
\email{turanova@msu.edu}
\thanks{\color{black}The work of O. Turanova is supported by NSF DMS grant 1907221 and NSF DMS grant 2204722.}


\date{\today}

\dedicatory{}

\begin{abstract}
As a counterpoint to classical stochastic particle methods for linear diffusion equations, such as Langevin dynamics for the Fokker-Planck equation, we develop a deterministic particle method for the weighted porous medium equation and prove its convergence on bounded time intervals. This generalizes related work on blob methods for unweighted porous medium equations. From a numerical analysis perspective, our method has several advantages: it is meshfree, preserves the gradient flow structure of the underlying PDE,   converges in arbitrary dimension, and captures the correct asymptotic behavior in simulations.

The fact that our method succeeds in capturing the long time behavior of the weighted porous medium equation is significant from the perspective of related problems in quantization. Just as the Fokker-Planck equation provides a way to quantize a probability measure $\bar{\rho}$ by evolving an empirical measure $\rho^N(t) = \frac{1}{N} \sum_{i=1}^N \delta_{X^i(t)}$ according to   stochastic Langevin dynamics so that $\rho^N(t)$ flows toward $\bar{\rho}$, our particle method provides a  way to quantize $\bar{\rho}$ according to   deterministic particle dynamics approximating the weighted porous medium equation. In this way, our method has natural applications to multi-agent coverage algorithms and sampling probability measures.

A specific case of our  method corresponds  to \emph{confined} mean-field dynamics of training a two-layer neural network for a radial basis  activation function. From this perspective, our convergence result shows that, in the overparametrized regime and as the variance of the radial basis functions goes to zero, the continuum limit  is given by the weighted porous medium equation. This generalizes previous results, which considered the case of a uniform data distribution, to the more general inhomogeneous setting.  
As a consequence of our convergence result, we identify conditions on the target function and data distribution for which convexity of the energy landscape emerges in the continuum limit.
  
\end{abstract}

\maketitle

\setcounter{tocdepth}{1}
\tableofcontents

\vspace{-1cm}

\section{Introduction}
Quantization is a fundamental problem throughout the sciences, in which one seeks to approximate a continuum distribution or signal by discrete objects \cite{graf2007foundations}. Mathematically, the quantization problem may be modeled by fixing a \emph{target probability measure} $\bar{\rho}$ on a subset $\Omega$ of $\rr^d$ and seeking locations $\{X^i\}_{i=1}^N$ in $\Omega$ so that the empirical measure $\rho^N = \frac{1}{N} \sum_{i=1}^N \delta_{X^i}$ approximates $\bar{\rho}$ in an appropriate sense. In statistics, this problem arises in the context of \emph{sampling}, since the locations $\{X^i\}_{i=1}^N$  represent approximate samples drawn from $\bar{\rho}$. In control theory, this problem is relevant to \emph{multi-agent coverage algorithms} \cite{cortes2004coverage,bullo2009distributed}, in which one seeks to control a fleet of robots to evolve from their current locations $\{X^i_0\}_{i=1}^N$ to terminal locations $\{X^i\}_{i=1}^N$ distributed according to $\bar{\rho}$.

There is a vast literature on different approaches to   the quantization problem, arising from the many different criteria by which   $\rho^N = \frac{1}{N} \sum_{i=1}^N \delta_{X^i}$ is considered a ``good'' approximation of $\bar{\rho}$. For example, if one seeks $\rho^N$ to approximate $\bar{\rho}$ optimally in the \emph{Wasserstein metric} of \emph{optimal transport}, recent work has shown that this is  closely related to the  well-known Lloyd's algorithm and has a fascinating connection to weighted fast diffusion equations  \cite{iacobelli2018gradient,bourne2021asymptotic,iacobelli2019asymptotic,bourne2015centroidal,caglioti2015gradient,merigot2011multiscale}.  
In the statistics literature, developing efficient sampling methods and quantifying their convergence is an active area of research, from classical methods based on Langevin dynamics to more recent developments, such as Hamiltonian Monte Carlo or Stein Variational Gradient Descent
\cite{eberlelecturenotes, vishnoi2021introduction,liu2016stein}. 
In the control theory literature, recent work has developed multiagent coverage algorithms  based on stochastic and kernelized particle methods for linear diffusions, as well as theoretically explored the potential of nonlinear diffusions for the coverage task, via finite volume and graph-based methods \cite{mesquita2008optimotaxis,elamvazhuthi2016coverage,eren2017velocity,elamvazhuthi2018nonlinear,krishnan2018distributed}. Other authors have  explored the role of different notions of optimality in designing coverage algorithms \cite{MattOlga2018, MattOlga2020}.

In each of these applications,  quantization methods based on partial differential equations play an important role. A classical approach   is given by evolving the locations of the particles by \emph{Langevin dynamics}, 
\begin{align*}
\begin{cases} d X_t^i = \sqrt{2} d B_t^i - \nabla \log \bar{\rho}(X_t^i) dt , \\
X^i(0)= X^i_0 ,
\end{cases}
\end{align*}
which is the stochastic particle discretization of   the  \emph{Fokker-Planck} equation,
 \begin{align} \label{FPeqn}  \tag{${\rm FP}$} 
\begin{cases} \partial_t\rho = \Delta \rho -\nabla \cdot  \left( \rho   \nabla  \log \bar{\rho}  \right)  , \\
    \rho(0)= \rho_0 .
    \end{cases}
   \end{align}

In the present work, we continue in this line of PDE-principled methods for sampling and coverage algorithms.  
We introduce a new  method based on  the   \emph{weighted porous medium equation} (WPME). Given a bounded, convex domain $\Omega \subseteq \Rd$, a   strictly positive target $\bar{\rho} : \Rd \to \R$ that is log-concave on $\Omega$ and satisfies $\int_\Omega \bar{\rho} = 1$,  and a fixed external potential  $V \in C^2(\Omega)$, we consider the equation,
\begin{align} \label{mainpde} \tag{${\rm WPME}$} 
\begin{cases} \partial_t\rho =\nabla \cdot \left( \frac{\bar{\rho}}{2} \nabla \left(\frac{\rho^2}{\bar\rho^2}  \right) \right)  +\nabla \cdot  \left( \nabla V \rho \right)  , \\
    \rho(0)= \rho_0 ,
    \end{cases}
   \end{align}
with   no-flux boundary conditions on $\partial \Omega$. The initial conditions are chosen to satisfy $\rho_0 \geq 0$ and $\int_\Omega \rho_0 =1$. (See  Proposition \ref{prop:PDE ep=0} for the   definition of weak solution.)  

The  dynamics of (\ref{mainpde})  arise in connection to quantization since, for $V=0$,  solutions of (\ref{mainpde}) converge as $t \to +\infty$ to $\bar{\rho}$  on $\Omega$ in the \emph{Wasserstein metric}; see Proposition \ref{longtime}. Consequently, if one can approximate solutions $\rho(t)$ of (\ref{mainpde}) by an empirical measure $\rho^N(t) = \frac{1}{N} \sum_{i=1}^N \delta_{X^i(t)}$, this   naturally leads to a method for flowing the empirical measure toward     $\bar{\rho}$ on $\Omega$ in the long time limit. 

The main goal of the present work is to develop a \emph{deterministic particle method} for (\ref{mainpde}),  constructing an empirical measure $\rho^N(t) = \frac{1}{N}  \sum_{i=1}^N \delta_{X^i(t)} $ and a system of ordinary differential equations to govern the locations of the particles $X^i(t)$ so that   $\rho^N(t)$ indeed converges, as $N \to +\infty$, to a solution $\rho(t)$ of (\ref{mainpde}) on bounded time intervals. In  Sections \ref{asssec}-\ref{mainresultssec} below, we describe the specific assumptions we impose and the precise statements of our results, including  which of our results continue to hold for $\bar{\rho}$ not log-concave, on unbounded domains $\Omega$, and for less regular $V$.

On one hand,  (\ref{mainpde})  is of  interest outside  the context of quantization. Weighted porous medium equations arise throughout the sciences, from models of fluid flow to biological swarming \cite{GMP2013,vazquez2017mathematical}. From this perspective,  Theorem \ref{thm:conv with particle i.d.} of the present work provides a new numerical method for simulating these phenomena. In particular, our work  extends the \emph{blob method} for the porous medium equation ($\bar{\rho} =  1$), which has been  studied by Oelschl\"ager \cite{oelschlager1990large}, Lions and MasGallic \cite{lions2001methode},   Carrillo, Craig, and Patacchini \cite{CarrilloCraigPatacchini}, and  Burger and Esposito \cite{burger2022porous}, to the case of   weighted porous medium  equations. (See below for a more detailed discussion of the relation with these results.)
This provides a provably convergent numerical method for   (\ref{mainpde}) in arbitrary dimensions, contributing to the   substantial   literature on numerical methods for such equations, including classical finite volume, finite element, and discontinuous Galerkin methods \cite{bessemoulin2012finite,burger2010mixed,carrillo2015finite,sun2018discontinuous}, as well as methods based on alternative deterministic particle methods in  one spatial dimension \cite{daneri2022deterministic,carrillo2017numerical,carrillo2016convergence,di2015rigorous,matthes2017convergent},  Lagrangian evolution of the transport map along the flow \cite{evans2005diffeomorphisms,carrillo2010numerical,carrillo2016numerical,matthes2014convergence,westdickenberg2010variational}, and many others \cite{benamou2016discretization,carrillo2021primal,carrillo2019aggregation,gallouet2021convergence}. From a numerical analysis perspective, the key benefits of our approach are that it is meshfree, deterministic, preserves the gradient flow structure and asymptotic behavior, and converges in arbitrary dimension. 

On the other hand, we believe (\ref{mainpde}) is particularly interesting from the perspective of quantization for several reasons. First, as we describe below,    there is a strong analogy between (\ref{mainpde}) and (\ref{FPeqn}), so that a quantization method based on (\ref{mainpde}) provides a counterpoint to classical Langevin dynamics.

A second reason for studying (\ref{mainpde}) in connection with quantization comes from applications in sampling. Over the past five years, Stein Variational Gradient Decent, originally introduced by Liu and Wang \cite{liu2016stein}, has attracted attention in the statistics community as a novel method for sampling a target measure $\bar{\rho}$ via a deterministic interacting particle system, which has a formal Wasserstein gradient flow structure with respect to a \emph{convex} mobility \cite{lu2019scaling,liu2017stein,korba2020non}. Recent work by Chewi et al. \cite{chewi2020svgd} identified that, when $V=0$, Stein Variational Gradient Descent (SVGD) may be interpreted as a kernelized version of (\ref{mainpde}), which has a rigorous Wasserstein gradient flow structure, as we explain below. In this way, understanding properties of (\ref{mainpde}) and its discretizations has the potential shed light on behavior of SVGD more generally.

A third reason for interest in (\ref{mainpde}) from a quantization perspective comes from applications in control theory. This is due to the fact the particle method we succeed in developing for (\ref{mainpde}) is  \emph{deterministic}, an important attribute in the context of coverage algorithms, since the results of the algorithm wouldn't need to be averaged over many runs, and there is hope that future research could  lead to quantitative convergence guarantees. This is in contrast to the case of classical quantization methods based on (\ref{FPeqn}), for which the natural Langevin particle approximation is stochastic.

 A final reason for interest in (\ref{mainpde}) comes from a variant of the quantization problem arising in models of two-layer neural networks. As we will explain below, the particle method we develop to approximate solutions of (\ref{mainpde}) coincides   with \emph{confined} dynamics for training a two-layer neural network with a radial basis function activation function.  In this way, our convergence result sheds light on the continuum limit of two-layer neural networks, showing that they converge to a solution of (\ref{mainpde}) when confined to the domain $\Omega$; see Corollary \ref{2layernncor}. This generalizes the previous convergence result of Javanmard, Mondelli, and Montanari \cite{javanmard2020analysis} to the case of nonuniform data distributions. As a consequence of this result, we are able provide conditions on the target function and data distribution that guarantee that the continuum limit of the training dynamics of two-layer neural networks is the gradient flow of a \emph{convex} energy, where the relevant notion of convexity along Wasserstein gradient flow is \emph{displacement convexity} or \emph{convexity along Wasserstein geodesics}; see Definition \ref{defi:semiconvexity-geod}. This emergence of convexity in the continuum limit is relevant to the behavior of neural networks in practice, where researchers seek to explain why gradient descent dynamics sometimes converge to a global optimum, in spite of the fact that, at the discrete level, the energy landscape is nonconvex \cite{chizat2018global,wojtowytsch2020convergence}.

 The remainder of the introduction proceeds as follows. In Section \ref{Wpmeintro}, we state fundamental properties of  (\ref{mainpde}) and describe the analogy between (\ref{mainpde}) and    (\ref{FPeqn}). In Section \ref{sec1p2}, we  introduce our particle method for approximating   solutions of (\ref{mainpde}). In Section \ref{sec1p3}, we describe the connection with two-layer neural networks. In Sections \ref{asssec} and \ref{mainresultssec}, we state our main assumptions and  results. Finally, in Section \ref{outlinesection}, we outline our approach and describe directions for future work.

\subsection{The weighted porous medium equation}\label{Wpmeintro} 
A key feature of  (\ref{mainpde}), which serves as a guiding principle of the present work, is that is it a  \emph{Wasserstein gradient flow} of the energy, 
\begin{align} \label{Fdef}
 \F: \P(\Rd) \to \R \cup \{+\infty\} , \quad \F(\mu) &= \E(\mu) + \V(\mu) + \V_\Omega(\mu) , \end{align}
where $\P(\Rd)$ denotes the set of Borel probability measures on $\Rd$; and the \emph{internal energy} $\E$, \emph{external potential energy} $\V$, and   \emph{confining potential energy} $\V_\Omega$ are given by,
\begin{align} \label{Edef}
\E(\mu) &=  \begin{cases} \frac12 \int_{\Rd} \frac{|\mu(x)|^2}{\bar\rho(x)}\, dx &\text{ if } \mu \ll \bar{\rho}(x) dx \text{ and } d \mu(x) = \mu(x) dx , \\ +\infty &\text{ otherwise,}\end{cases} \\
 \V(\mu) &= \int_\Rd V(x) d \mu(x) , \\
  \mathcal{V}_\Omega(\mu) &= \begin{cases} 0 &\text{ if } \supp \mu \subseteq \overline{\Omega} , \\
 +\infty&\text{ otherwise.}  \end{cases} \label{Vdef} \end{align}
The internal energy $\E$ induces the nonlinear diffusion term, the external potential   $\V$ induces the convection term, and the confining potential    $\V_\Omega$ restricts the dynamics to $\Omega$, with no-flux boundary conditions on $\partial \Omega$. Our primary interest, and the main mathematical challenge in establishing our results, is in the nonlinear diffusion induced by $\E$ and its approximation by a deterministic particle method.
In Section \ref{preliminariessec}, we provide detailed background on the Wasserstein metric $W_2$ and Wasserstein gradient flows. In  Proposition \ref{prop:PDE ep=0}, we recall the precise statement of the result that solutions of (\ref{mainpde}) are the gradient flow of $\F$ .

The fact that (\ref{mainpde}) has a gradient flow structure  is in close analogy with the (\ref{FPeqn}) equation: in their seminal work \cite{jordan1998variational}, Jordan, Kinderlehrer, and Otto established that (\ref{FPeqn})  is the Wasserstein gradient flow of the \emph{Kullback-Leibler divergence},  
\begin{align*}
\KL(\mu, \bar{\rho}) =  \int_\Omega  \log(\mu / \bar{\rho}) \, d\mu, \quad \text{ for } \mu \ll \bar{\rho}   .
\end{align*}
From this perspective, it is useful to notice that, when $V=0$, (\ref{mainpde}) can also be thought of as the Wasserstein gradient flow of the  $\chi^2$ \emph{divergence} \cite{tsybakov2004introduction},
\[ \chi^2(\mu, \bar{\rho}) = \begin{cases} \frac12 \int_\Omega \frac{|\mu(x) - \bar{\rho}(x)|^2}{\bar{\rho}(x)} dx , & \text{ if } \mu \ll  \mathcal{L}^d \text, \ d \mu(x) = \mu(x) dx , \text{ and } \supp \mu \subseteq \overline{\Omega}, \\ +\infty &\text{ otherwise.}   \end{cases}  \]
This can be seen by noticing   $\int  {|\mu(x) - \bar{\rho}(x)|^2}/{\bar{\rho}(x)} dx =\int  {|\mu(x) |^2}/{\bar{\rho}(x)} dx -1$, so that, when $V=0$, our energy $\F$ agrees with   $\chi^2$, up to a constant that does not affect the dynamics of the gradient flow: $\F(\mu)+1/2 = \chi^2(\mu, \bar \rho)$.
In what follows, we will always suppose that $\bar{\rho}$ is normalized to satisfy  $\int_\Omega \bar{\rho} = 1$, and so that for $\mu \in \P(\Rd)$, the KL divergence and the $\chi^2$ divergence measure the discrepancy between $\mu$ and $\bar{\rho}$ on $\Omega$ and vanish in the case that $\mu = \bar{\rho}$ on $\Omega$.

The gradient flow structures of (\ref{mainpde}) and (\ref{FPeqn}) have important interpretations from the perspective of quantization, since they encode key information about how quickly solutions are flowing toward $\bar{\rho}$. The fact that solutions of the (\ref{FPeqn}) equation are the Wasserstein gradient flow of the KL divergence is equivalent to saying that   they \emph{dissipate the KL divergence as quickly as possible}, with respect to the Wasserstein structure. In the same way, solutions of the (\ref{mainpde}) equation  \emph{dissipate the $\chi^2$ divergence as quickly as possible}, with respect to the Wasserstein structure.

Another important feature of   (\ref{mainpde})  from the perspective of quantization are the available estimates   quantifying its convergence to equilibrium. Chewi  et  al. \cite{chewi2020svgd}, show that, if  $\Omega= \Rd$, $V = 0$, and $\bar{\rho}$ satisfies a Poincar\'e inequality, then,   along smooth solutions, the Kullback-Leibler divergence decreases exponentially:
 \begin{equation} \label{KLdivergenceexpdecay}
\KL(\rho(t), \bar{\rho}) \leq e^{-C_{\bar{\rho}}  \, t} \KL(\rho(0), \bar{\rho}) , \quad \text{ for } C_{\bar{\rho}} >0  .
\end{equation}
If, in addition,  $\bar{\rho}$ is strongly log-concave, then  the $\chi^2$ divergence   decreases exponentially:
\begin{align*}
\chi^2(\rho(t), \bar{\rho}) \leq   e^{-C_{\bar{\rho}} \, t} \chi^2(\rho(0), \bar{\rho}) , \quad \text{ for } C_{\bar{\rho}}>0 , 
\end{align*}
This mirrors the theory for (\ref{FPeqn}), in which a Poincar\'e inquality ensures exponential decay of the $\chi^2$ divergence and log-concavity ensures decay of the KL divergence. (See Matthes, McCann, and Savar\'e's \emph{flow interchange method} for general results of this form \cite{matthes2009family}. In addition, see Grillo, Muratori, and Porzio \cite{GMP2013}, who rigorously proved   exponential convergence to equilibrium of weak solutions in $L^p$ spaces for all $p<+\infty$.) 
Furthermore,  in  the case of the (\ref{mainpde}) equation, if $\bar{\rho}$ merely satisfies a weaker condition, known as an $L^{2/3}$-Poincar\'e inequality, then Dolbeault  et  al. \cite{dolbeault2008lq} showed that the $\chi^2$ divergence decreases polynomially.  This raises the possibility that, for different choices of $\bar{\rho}$ and  initial conditions $\rho_0$, there may exist contexts in which solutions of (\ref{mainpde}) converge to $\bar{\rho}$ with stronger convergence guarantees than solutions of (\ref{FPeqn}). Since developing  general conditions  on the target $\bar{\rho}$ and the initialization $\rho_0$ that distinguish whether (\ref{mainpde}) or (\ref{FPeqn}) equilibrates more quickly remains an active area of research, we do not claim that the dynamics of (\ref{mainpde}) offer superior long time behavior to (\ref{FPeqn}). Instead, we merely observe that, at the continuum level, (\ref{mainpde})  provides competitive dynamics. Understanding when solutions to (\ref{mainpde}) or (\ref{FPeqn}) converge more quickly to equilibrium may, in the future, shed light on which quantization methods are superior in different contexts.

\subsection{Particle approximation of (\ref{mainpde})} \label{sec1p2}
The aim of the present work is to design a deterministic particle method for approximating solutions of (\ref{mainpde}) that preserves its gradient flow structure. Since solutions of (\ref{mainpde}) are  gradient flows of the energy (\ref{Fdef}-\ref{Vdef}), we seek to approximate them by  gradient flows of the \emph{regularized energy}, defined by,
 \begin{align} \label{Fepdef}
 \F_{\ep,k}: \P(\Rd) \to \R \cup \{+\infty\} , \quad \F_{\ep,k}(\mu) &= \E_\ep(\mu) + \V_\ep(\mu) + \V_k(\mu) , \end{align}
for the energies $\E_\ep(\mu)$, $\V_\ep(\mu)$, and $\V_k(\mu)$ given by,
\begin{align} \label{Eepdef}
\Eep(\mu) &= \frac12 \int_{\rr^d} \frac{|\zeta_\epsilon*\mu|^2(x)}{\bar{\rho}(x)}\, dx, \\
\V_\epsilon(\mu) &= \int_{\rr^d} (\zeta_\epsilon*V)(x) d \mu(x),  \\
\V_k(\mu) &= \int_{\rr^d} V_k (x) d \mu(x) .
\end{align}
Here  $\zeta_\ep \in C^\infty(\Rd)$ is a   rapidly decreasing mollifier and $V_k\in C^2(\Rd)$, for $k \in \mathbb{N}$, is a convex function that vanishes on $\Omega$ and approaches $+\infty$ on  $\overline{\Omega}^c$ as $k \to +\infty$.

The energy $\E_\epsilon(\mu)$ is   an approximation, as $\ep \to 0$,  of $\E $. This regularized energy has superior differentiability properties along empirical measures, ensuring that the gradient flow starting at empirical measure initial data leads to a well-posed particle method. It also enjoys the property,
\begin{align} \label{EepstoE}
\Eep(\rho) = \E(\zeta_\ep*\rho),
\end{align}
which is a key element in our proof of an $ H^1$ bound for $\zeta_\epsilon*\rho_\epsilon$ along solutions of the gradient flow; see Theorem \ref{prop:H1 bd}.
The energy $\V_\ep $ is an approximation of $\V$. While many different methods of approximating $\V$ would work well both numerically and theoretically, we focus our attention on $\V_\epsilon$ due to the connection with two-layer neural networks.
Finally, the energy $\V_k $ is an approximation, as $k \to +\infty$, of $\V_\Omega$.

While the main focus of our work is the analysis of how dynamics induced by $\E_\ep$, for general initial data, approximate dynamics induced by $\E$ (indeed, if $\Omega$ is the entire space $\Rd$ and $V$ is taken to be zero, then the energy $\F_{\ep,k}$ is exactly $\E_\ep$), our analysis of how the gradient flow dynamics induced by $\V_k$ converge to those from $\V_\Omega$ as $k \to +\infty$ also generalizes existing results by Alasio, Bruna, and Carrillo to weighted porous medium equations \cite{alasio2020role}.  (See also recent work by Patacchini and Slep\v{c}ev, which uses a similar approach to study well-posedness of aggregation equations on compact manifolds \cite{patacchini2021nonlocal}.) Our   motivations for considering this approximation of the confining potential $\V_\Omega$ are twofold. First, it simplifies the implementation of the particle method, obviating the need to implement reflection boundary conditions. Second, it allows for the most challenging aspect of the analysis --- the relationship between the dynamics induced by $\E_\ep$ and $\E$ --- to be carried out on $\rr^d$, rather than on a domain with boundary.

Wasserstein gradient flows of the regularized energy $\F_{\ep,k}$ are characterized by the equation,
\begin{align} \tag{${\rm WPME}_{\ep,k}$} \label{mainepspde}
\begin{cases}
\partial_t \rho  = \grad \cdot \left( \rho \left(\grad \zeta_\epsilon * \left(  {\zeta_\epsilon* \rho } /  \bar{\rho } \right) + \nabla \zeta_\ep *V + \nabla V_k \right)\right) , \\
\rho(0)  = \rho_0  ,\end{cases}\end{align}
defined on all of $\Rd$ in the duality with $C^\infty_c(\Rd \times (0,+\infty))$; see Proposition \ref{PDEFeps}. If the initial conditions are given by an empirical measure, $\rho_0 = \sum_{i=1}^N \delta_{X^i_0}m^i$, with $\sum_{i=1}^Nm^i=1$, then the solution remains an empirical measure for all time. Concretely, we have $\rho(t) =\sum_{i=1}^N \delta_{X^i(t)}m^i$, and the locations of the particles $\{X^i(t) \}_{i=1}^N$ are characterized as solutions of, 
     \begin{align} \label{mainepsode}
\begin{cases} & \dot{X}^i(t)  = -\sum_{j=1}^N f(X^i,X^j) m^j  - \nabla \zeta_\ep *  V(X^i) - \nabla V_k(X^i) , \\
&X^i(0) = X^i_0 , 
\end{cases}
\end{align}
for,
\begin{align} \label{fdefintro}
f(x,y) := \int_{\mathbb{R}^d} \frac{ \grad \zeta_\epsilon(x - z) \zeta_\epsilon(y - z)}{\bar{\rho}(z)}  \, dz  ;
\end{align}
see Proposition \ref{prop:ODE Eep}. In Section \ref{numericaldetailssec}, we provide sufficient conditions on $\bar{\rho}$ for which the integral in $f(x,y)$ has an analytic formula, in which case it can be precomputed exactly and does not contribute to the computational complexity of our method.

Based on the intuition that $\F_{\ep,k}$ is an approximation of $\F$, it is natural to   hope that gradient flows of $\F_{\ep, k}$   approximate gradient flows of $\F$. Our main result is that this is indeed true. We show that the particle method defined by  (\ref{mainepsode}) converges to a solution of (\ref{mainpde}) on bounded time intervals, provided that the initial conditions $\rho_0$ have bounded entropy,  the number of particles $N$ grows sufficiently quickly, and $\epsilon \to 0$ and $k \to +\infty$ sufficiently rapidly; see Theorem \ref{thm:conv with particle i.d.}. Note that this method formally extends to equations of the form (\ref{mainpde}) with an additional term $-\nabla \cdot (v\rho)$ on the right hand side, for general velocities $v(x,t)$, by adding a term of the form $v(X^i(t), t)$ to the right hand side of (\ref{mainepsode}). 

Our work on the convergence of the $\ep \to 0$, $k \to +\infty$ limit builds on several previous works. All previous works have considered the spatially homogeneous case $\bar{\rho}  \equiv 1$. The first work in this direction was due to Oelschl\"ager \cite{oelschlager1990large}, who considered the case $V=V_k=0$ and proved convergence to classical, strictly positive solutions of (\ref{mainpde}) in arbitrary dimensions and convergence to weak solutions in one dimension. Subsequently, Lions and Mas-Gallic \cite{lions2001methode}, also in the case $V=V_k=0$, proved convergence of (\ref{mainepspde}) as $\ep \to 0$, provided that the initial conditions $\rho_0$ had uniformly bounded entropy, thereby excluding particle initial data required to connect (\ref{mainepspde}) to the system of  ODEs (\ref{mainepsode}). The assumption of bounded entropy played an important role in Lions and Mas-Gallic's proof of a $\dot{H}^1$ bound for regularized solutions to (\ref{mainepspde}). (In fact, the analogous bound also plays an important role in the present work -- see Theorem \ref{prop:H1 bd} for a generalization of this result to the spatially inhomogeneous setting.) Next,   Carrillo, the first author, and Patacchini \cite{CarrilloCraigPatacchini} generalized Lions and Mas-Gallic's approach to porous medium equations of the form, 
\[ \partial_t \rho = \Delta \rho^m + \nabla \cdot (\rho (\nabla V + \nabla W*\rho)). \]
In the case $m =2$, they obtained convergence of the $\epsilon \to 0$ limit under appropriate continuity and semiconvexity assumptions on $V$ and $W$; for $1 \leq m < 2$, they obtained $\Gamma$-convergence of the corresponding energies as $\ep \to 0$; and for $m >2$, they obtained conditional convergence of the $\epsilon \to 0$ limit, as long as certain a priori estimates were preserved along the flow. Again, Carrillo, Craig, and Patacchini's work required the initial data to have bounded entropy, excluding particle solutions. Very recently, Burger and Esposito \cite{burger2022porous} continued the study of the $m=2$ case for more general velocity fields $v(x,t)$,
\[ \partial_t \rho  + \nabla \cdot( \rho v)  = \Delta \rho^m, \]
and weaker regularity on the mollifier $\zeta$.

Our work makes three contributions to this active area of research. We obtain true convergence of the particle method, relaxing the hypothesis that the initial data have bounded entropy by using stability properties of the regularized flow; see Theorem \ref{thm:conv with particle i.d.}. Our result holds for spatially inhomogeneous porous medium equations, allowing general $\bar{\rho} \in C^1(\Rd)$ that are  bounded above and below on $\mathbb{R}^d$ and log-concave on $\Omega \subseteq\Rd$. (See  Section \ref{asssec} for a discussion of where the log-concavity assumption may be weakened.) Finally, by allowing spatially inhomogenous equations, we identify a connection between our particle method and problems in sampling, control theory, and training of two-layer neural networks.

\subsection{Application to two-layer neural networks} \label{sec1p3} 
An additional reason for interest in the convergence of (\ref{mainepsode}-\ref{fdefintro}) to  (\ref{mainpde}), aside from its utility as a particle approximation, is that the dynamics of (\ref{mainepsode}-\ref{fdefintro}) represent a type of \emph{confined} training dynamics for mean field models of two-layer neural networks with a radial basis function activation function. 
  In this context, one is given a data distribution $\nu$, a nonnegative target function  $f_0 \in L^2(\nu)$, and an activation function $\Phi_\ep(x,z) = \zeta_\ep(x-z)$, and one seeks to choose parameters, $\{X^i\}_{i=1}^N$, so that the empirical measure  $\rho^N = \frac{1}{N} \sum_{i=1}^N \delta_{X^i}$  minimizes the following energy, known as the \emph{population risk}:
\begin{align} \label{poprisk}
 \mathcal{R}_\epsilon(\mu) =  \frac12 \int_\Rd \left|\int_\Rd \Phi_\epsilon(x,z)d \mu(x) - f_0(z) \right|^2 d \nu(z) 
 \end{align}
 
 In several recent works, it was discovered that evolving the parameters $X^i(t)$ by gradient descent of the function $(X^1, \dots, X^n) \mapsto \mathcal{R}_\ep(\rho^N) + \V_\Omega(\rho^N)$  is equivalent to evolving the empirical measure $\rho^N$ by the Wasserstein gradient flow of $\mathcal{R}_\ep$ restricted to $\Omega$ \cite{mei2018mean,chizat2018global,sirignano2020mean, rotskoff2018trainability,javanmard2020analysis, weinan2020machine}.  Various methods for treating the boundary conditions are considered, including projection of the gradient descent direction into the convex hull of the domain $\Omega$ \cite{chizat2018global} or projection onto interior approximations of $\Omega$ \cite{javanmard2020analysis}.
 
 To see the connection with (\ref{mainepsode}-\ref{fdefintro}),
 note that using the definition of $\Phi_\ep$, expanding the square, and applying Tonelli's theorem (see also the associativity property of convolution (\ref{eq:convproperty})), we obtain,
\begin{align} \label{ReptoEep} 
 \mathcal{R}_\epsilon(\mu) &
 = \frac12 \int |\zeta_\ep *\mu(z)|^2 d \nu(z) - \int \zeta_\ep*\mu(z) f_0(z) d \nu(z) + \frac12 \int |f_0(z)|^2 d \nu(z)  \\
 &= \E_\ep (\mu) +  \int (\zeta_\ep* V)(x) d \mu(x) + C  = \E_\ep (\mu) +\V_\ep(\mu) + C ,\nonumber
 \end{align} 
 for,
 \begin{align} \label{poprisktoeeps}
 \nu &= 1/\bar{\rho} \ , \quad V = -f_0 \nu \ , \quad C =   \frac12 \int |f_0(z)|^2 d \nu(z) .
 \end{align}
 Moreover,   the confining potential $\V_k$ provides an explicit method for projecting gradient descent dynamics onto an \emph{exterior} approximation of $\Omega$. In this way,  the \emph{confined} training dynamics given by the gradient flow of  $\mathcal{R}_\ep + \V_k = \F_{\ep,k} + C$ for general initial data $\rho_0$ is characterized by (\ref{mainepspde}), and the evolution for particle initial data corresponds to (\ref{mainepsode}-\ref{fdefintro}).  Corollary \ref{2layernncor}, which follows from  our convergence result for the gradient flows of $\F_{\ep,k}$, states that, for well-behaved initial conditions, particle solutions of (\ref{mainepsode}-\ref{fdefintro}) converge to a gradient flow of, 
\begin{align} \label{popriskep0}
 \mathcal{R}(\mu) = \begin{cases} \frac12 \int \left| \mu(z) - f_0(z) \right|^2 d \nu(z) &\text{ if } \mu \ll \mathcal{L}^d |_\Omega , \quad d \mu(z) = \mu(z) dz , \\
+\infty &\text{ otherwise.}
\end{cases}
 \end{align}
 
 This generalizes previous work due to  Javanmard, Mondelli, and Montanari \cite{javanmard2020analysis}, which considered the limit $\ep \to 0$ in the specific case of a uniform data distribution  $\nu = \indi_\Omega/|\Omega|$, smooth target function $f$, bounded convex domain $\Omega$, and compactly supported radial basis function  $\zeta$.   The fact that our result holds for general nonuniform data distributions $\nu$ is significant from the perspective of two-layer neural networks, since, as can be seen in Corollary \ref{2layernncor}, there is an interplay between the data distribution $\nu$ and the target function $f$ to determine when convexity of the energy $\F_{\ep,k}$ emerges in the continuum limit. 
         
\subsection{Assumptions} \label{asssec}
We now describe our assumptions. We consider a domain $\Omega\subseteq \Rd$ satisfying,
\begin{align}
 \text{  $\Omega$ is   nonempty, open, and convex. }\tag{D}
 \label{Omegaas}
\end{align}

We suppose our mollifier satisfies,
\begin{equation} \tag{M}
  \label{mollifieras}
  \begin{split}
&\text{ $\zeta \in C^2(\Rd)$ is even, nonnegative,  $\|\zeta\|_{L^1(\rr^d)}=1$,   $ D^2 \zeta \in L^\infty(\Rd)$, } \\
&\zeta(x) \leq C_\zeta |x|^{-q}  \text{ and }  |\nabla \zeta(x)| \leq C_\zeta |x|^{-q'}, \text{ for } C_\zeta >0 ,  \quad q>d+1  , \quad q'>d .
\end{split}
\end{equation} 
This  assumption is satisfied by both Gaussians and smooth functions with compact support. Note that this assumption ensures that $\zeta$ has finite first moment, $\int_\Rd |x| \zeta(x) dx< +\infty$.

We suppose the external potential $V$ satisfies,
\begin{align} \tag{V}
\text{ $V \in C^2(\Rd) \cap L^1(\Rd) \cap L^\infty(\Rd)$, with $\nabla V \in L^\infty(\Rd)$ and $D^2 V$ uniformly bounded below.}
\label{Vas}
\end{align}
We are optimistic that our results may continue to hold under weaker regularity hypotheses on $V$,  but we leave this   question to future work, since our primary interest is the approximation of the diffusive dynamics arising from $\E$ via the particle method induced by $\E_\ep$.

We suppose that our approximation of the confining potential $V_k$, for $k \in \mathbb{N}$, satisfies,
\begin{align} \tag{C}
 \label{Vkas} 
 &\text{$V_k$ is nonnegative, convex, and twice differentiable with $D^2 V_k \in L^\infty(\Rd)$,} \\
\tag{Ck}
& \text{$V_k = 0$ on $\Omega$ and $
\displaystyle\lim_{k \to \infty} \left( \inf_{x \in B}  V_k(x) \right) = +\infty$ for any ball $B \subset \joinrel \subset \Omega^c$.} \label{Vkinftyas} 
\end{align}
Note that assumption (\ref{Vkas}) ensures $V_k \in L^1(\mu)$ and $\nabla V_k\in L^2(\mu)$  for any $\mu\in \mathcal{P}(\rr^d)$ with $\int |x|^2 d\mu(x) <+\infty$. These assumptions play the following role in our proof: Assumption (\ref{Vkas})   ensures well-posedness of the gradient flows, and  Assumption (\ref{Vkinftyas})  allows us to recover the correct limiting dynamics as $k \to +\infty$.  In particular, note that (\ref{Vkinftyas}) implies that, in the $k\rightarrow \infty$ limit, $V_k$ approximates the hard cutoff potential $V_\Omega$, which is given by,
\begin{equation}
    \label{eq:VOmega def}
    V_\Omega(x) = \begin{cases}
        0 \quad & \text{ for }x\in \overline{\Omega},\\
        +\infty \quad &\text{ otherwise.}
    \end{cases}
\end{equation}

Finally, we suppose that our target $\bar{\rho}$ satisfies the  regularity assumption,
\begin{align}
\text{ $\bar\rho\in C^1(\rr^d)$ and  there exists $C>0$ so that $1/C \leq \bar{\rho}(x) \leq C,$ for all $x \in \mathbb{R}^d$.} \label{targetas}  \tag{T}
\end{align}
Assumption (\ref{targetas}) is sufficient to ensure that the energy $\F_{\ep,k}$ is lower semicontinuous, convex, and subdifferentiable, so that   gradient flows of $\F_{\ep,k}$ are well posed.  
It also allows us to conclude that the energy $\F$ is lower semicontinuous. However, in order to obtain convexity and subdifferentiability of $\F$, hence well-posedness of   gradient flows, we   require $\bar{\rho}$ to be log-concave on $\Omega$; that is,
\begin{align*}
\text{ $x \mapsto \log(\bar{\rho}(x))$ is concave on $\Omega$. }
\end{align*}

It is an   open question whether well-posedness of the gradient flows of  $\F$   could be obtained under weaker assumptions on $\bar{\rho}$. Interestingly, the main estimates in our proof of the convergence of the gradient flows of $\F_{\ep,k}$ as $\epsilon \to 0$ (Theorem  \ref{prop:H1 bd}, Theorem \ref{newenergygammaprop},  and Proposition \ref{prop:epstozerofixedk})  do not require log-concavity of $\bar{\rho}$. Instead, log-concavity comes into play when we seek to identify that the limit as $\ep \to 0$  and $k\rightarrow \infty$  of gradient flows of  $\F_{\ep,k}$ is indeed a gradient flow of $\F$, since log-concavity of $\bar{\rho}$ ensures that the \emph{metric slope} of $\F$ is a \emph{a strong upper gradient}; see Section \ref{OTsec} and \cite[Section 1.2]{ambrosio2008gradient}. For this reason, we are optimistic that, in future work, it may be possible to extend our results to  $\bar{\rho}$ that are not log-concave, once the difficulty of obtaining well-posedness of the gradient flow of $\F$  and characterization of its strong upper gradient are overcome.

\subsection{Main Results} \label{mainresultssec}

To state our main results, first we introduce some notation. Let the \emph{entropy} $\S(\mu)$ and $p$-th \emph{moment} $M_p(\mu)$, where $p\geq 1$, of a measure $\mu\in \mathcal{P}(\rr^d)$ be given by
\begin{align}
\label{def:entropy}
\S(\mu) = \begin{cases} \int_\Rd \mu(x) \log \mu(x) d\mathcal{L}^d(x) & \text{ if }\mu \ll \mathcal{L}^d \text{ and } d \mu(x) = \mu(x) dx , \\ + \infty &\text{ otherwise,} \end{cases}    \qquad M_p(\mu) = \int_\Rd |x|^p d \mu(x).
\end{align}
Recall that a probability measure $\mu \in \P(\Rd)$ lies in the \emph{domain} of an energy $\G: \P(\Rd) \to \R \cup \{+\infty\}$ if $\G(\mu) <+\infty$. We denote this by $\mu \in D(\G)$. We also write 
\begin{equation}
\label{eq:def Pp}
\P_p(\Rd) = \P(\Rd) \cap D(M_p), \text{ for $p\geq 1$.} 
\end{equation} 
Finally, we often use the notion of \emph{narrow} convergence of probability measures; see Definition \ref{def:narrow conv}.

\begin{thm}[convergence of gradient flows as $k \to +\infty$,  $\ep = \epsilon(k) \to 0$] \label{newmaintheorem}
Assume  (\ref{Omegaas}), (\ref{mollifieras}), (\ref{Vas}), (\ref{Vkas}), (\ref{Vkinftyas}), (\ref{targetas}) and that $\bar\rho$ is log-concave on $\Omega$. Fix $T>0$  and  $\rho(0) \in D(\F) \cap D(\S) \cap \P_2(\Rd)$.

For $\epsilon >0$ and $k\in \mathbb{N}$, let $\rho_{\ep,k}\in AC^2([0,T]; \mathcal{P}_2(\rr^d))$ be the gradient flow of $\F_{\ep,k}$ with initial data $\rho(0)$. Then, as $k \to +\infty$, there exists a sequence $\epsilon = \epsilon(k) \to 0$ so that 
\begin{align*}
\lim_{k \to +\infty} W_1(\rho_{\ep,k}(t), \rho(t)) = 0, \text{ uniformly for } t \in [0,T],
\end{align*} 
 where   $\rho \in AC^2([0,T]; \mathcal{P}_2(\rr^d))$ is the gradient flow of $\F$ with initial data $\rho(0)$.

\end{thm}

The preceding theorem  requires that the initial conditions of the gradient flow of $\F_{\ep,k}$ have bounded entropy, which explicitly excludes empirical measure initial data. However, we are able to extend this   result to empirical measure initial data by leveraging stability properties of the gradient flow of $\F_{\ep,k}$. In this way, we obtain the following convergence result for the deterministic particle method to weak solutions of (\ref{mainpde}), provided that the underlying continuum solution has initial data with bounded entropy. In Proposition \ref{prop:PDE ep=0}, we state the precise notion of weak solution of (\ref{mainpde}) that we consider, and in    Lemma \ref{sketchlem2}, we provide an explicit construction of $\rho_{\ep,k}^N(0)$ satisfying condition (\ref{rateofmu0N}).

\begin{thm}[convergence with particle initial data]
\label{thm:conv with particle i.d.}
Assume  (\ref{Omegaas}), (\ref{mollifieras}), (\ref{Vas}), (\ref{Vkas}), (\ref{Vkinftyas}), (\ref{targetas}), and that $\bar\rho$ is log-concave on $\Omega$. 
Fix $T>0$ and   $\rho(0)\in D(\F) \cap D(\S) \cap \P_2(\Rd)$.
For $k,N   \in \mathbb{N}$, $\ep>0$, and $t \in [0,T]$, consider the evolving empirical measure,
\[ \rho^N_{\ep,k}(t)  = \sum_{i=1}^N \delta_{X^{i}_{ \ep,k}}(t) m^i, \quad      m^i \geq 0 ,  \quad  \sum_{i=1}^N m^i = 1 , \] 
where   $X^{i}_{\ep,k} \in C^1([0,T]; \Rd)$   solves,
\begin{align} \label{mainsystemofODEs}
    \begin{cases}
    \dot{X}^{i}_{\ep,k} = -\sum_{j=1}^{N} m^j\int_{\rr^d} \nabla \zeta_\ep(X^{i}_{\ep,k}-z)\zeta_\ep(z-X^{j}_{\ep,k})\frac{1}{\bar\rho(z)}\, dz - \nabla (\zeta_\epsilon*V)(X^{i}_{\ep,k}) - \nabla V_k(X^{i}_{\ep,k}) ,\\
     X^{i}_{\ep,k}(0)=X^{i}_{0,\ep}.
    \end{cases}
\end{align}

Suppose that as  $\ep   \to 0$ there exist $N = N(\ep) \to +\infty$,  so that, for all $k \in \mathbb{N}$,  $\rho^N_{\ep,k}(0) = \sum_{i=1}^N \delta_{X^i_{0,\ep}} m^i$ converges to $\rho(0)$ with the  rate,
\begin{align} \label{rateofmu0N}
\lim_{k \to \infty} e^{- \lambda_\ep T} W_2(\rho^N_{\ep,k}(0), \rho(0)) = 0 , \text{ for } \lambda_\ep = -\ep^{-d-2} ||1/\bar\rho||_{L^\infty(\rr^d)}||D^2\zeta||_{L^\infty(\rr^d)}+ \inf_{\{x, \xi \in \Rd\}} \xi^t D^2 V(x) \xi .
\end{align}
Then, as $k \to +\infty$, there exist $\epsilon = \epsilon(k) \to 0$ and $N = N(\ep) \to +\infty$ for which  $\rho^N_{\ep,k}(t) = \sum_{i=0}^{N}\delta_{X^{i}_{\ep,k}(t)} m^i$ satisfies
\[ \lim_{k \to +\infty} W_1(\rho^N_{\ep,k}(t), \rho(t) ) \to 0 , \text{ uniformly for } t \in [0,T] , \]
 where  $\rho \in AC^2([0,T]; \mathcal{P}_2(\rr^d))$ is the unique weak solution of (\ref{mainpde}) with initial data $\rho(0)$.
\end{thm}

The following corollary ensures that the particle method defined in the previous theorem indeed converges to $\bar{\rho}$  on $\Omega$ in the long time limit, as relevant  for applications in quantization.
\begin{cor}[long time limit] \label{quantcor}
Suppose the assumptions of  Theorem \ref{thm:conv with particle i.d.} hold and again denote $\rho^N_{\ep,k}(t) = \sum_{i=0}^{N}\delta_{X^{i}_{\ep,k}(t)} m^i$. In addition, assume $V=0$, $\Omega$ is bounded, and $  \int_\Omega \bar\rho\, d\mathcal{L}^d  = 1$.
 Then  there exist $k=k(t) \to +\infty$, $\ep=\ep(k) \to 0$, and $N = N(\ep) \to +\infty$ so that  
 \[ \lim_{t \to +\infty} W_1\left(\rho^N_{\ep,k}(\cdot, t),   {\indi_{\overline{\Omega}}\bar{\rho}} \right) = 0 . \]

\end{cor}

The preceding theorems provide sufficient conditions to guarantee convergence of the particle method to (\ref{mainpde}) on bounded time intervals and convergence to the desired target distribution $\bar{\rho}$ on $\Omega$ when $V=0$ and $\Omega$ is bounded. However, these results are purely qualitative, and it remains an   open question to what extent they could be made quantitative in $T$, $k$, $\ep$, and $N$.  For example, an inspection of the construction in Lemma  \ref{sketchlem2} shows that, if the particles are initialized with uniform spacing on a bounded domain $\Omega$, the number of particles is required to grow extremely quickly with respect to $\epsilon$. In particular, it suffices to have
 \[    N(\epsilon,k)^{-1} = o\left( e^{-1/   \epsilon^{d+2}}  \right)   \ \text{ as } \ep \to 0.\]
On the other hand, we observe numerically that $N(\ep) \sim \ep^{-1.01}$ is sufficient for good performance in one dimension. We leave a finer quantitative convergence analysis to future work. For example, it would be interesting to investigate whether  higher regularity of the initial data $ \rho(0)$ could be used to decrease the rate at which $N$ must grow with $\epsilon$ in our rigorous convergence results, as the numerical simulations suggest is possible.

As a second corollary of our main convergence results, we   identify the limit of the confined training dynamics of two-layer neural networks with a radial basis function activation function and quadratic loss, as described in section \ref{sec1p3}. In particular, our result gives sufficient conditions under which the limit of these training dynamics is the gradient flow of a \emph{convex} energy, in the sense that it is convex along Wasserstein geodesics; see Definition \ref{defi:semiconvexity-geod}. 
 \begin{cor}[two-layer neural networks] \label{2layernncor}
Consider a domain $\Omega$ satisfying (\ref{Omegaas}), a radial basis function activation function $\Phi_\ep(x,z) = \zeta_\ep(x-z)$ satisfying (\ref{mollifieras}), a data distribution $\nu = 1/\bar{\rho}$, for $\bar{\rho}$ satisfying (\ref{targetas}) and log-concave on $\Omega$, and a target function $f_0 = -V \bar{\rho}$, for $V$ satisfying  (\ref{Vas}).  For $k \in \mathbb{N}$, consider a confining potential $V_k$ satisfying (\ref{Vkas}) and (\ref{Vkinftyas}). Fix $T>0$. For $\ep >0$, $N   \in \mathbb{N}$, and $t \in [0,T]$, consider the confined  training dynamics of a two-layer neural network corresponding to the   energy $\mathcal{R}_\ep + \V_k$; that is, consider the evolution of the empirical measure of parameters,
\[  \rho_\ep^N(t)  = \sum_{i=1}^N \delta_{X^{i}_{ \ep}}(t) m^i, \quad      m^i \geq 0 ,  \quad  \sum_{i=1}^N m^i = 1 , \] 
where   $X^{i}_{\ep} \in C^1([0,T]; \Rd)$   solves,
\begin{align} 
    \begin{cases}
    \dot{X}^{i}_{\ep} = -\sum_{j=1}^{N} m^j\int_{\rr^d} \nabla \zeta_\ep(X^{i}_{\ep}-z)\zeta_\ep(z-X^{j}_{\ep})\nu(z) \, dz  +\nabla (\zeta_\epsilon*(f_0 \nu))(X^{i}_{\ep})- \nabla V_k(X^{i}_{\ep,k})   ,\\
     X^{i}_{\ep}(0)=X^{i}_{0,\ep}.
    \end{cases}
\end{align}

Suppose   there exists     $\rho(0)\in D(\F) \cap D(\S) \cap \P_2(\Rd)$ so that, for all $\ep >0$, there exists $N = N(\ep)$ so that $\rho^N(0)$ converges to $\rho(0)$ sufficiently quickly, according to  the  rate from equation (\ref{rateofmu0N}). Then, as $k \to +\infty$, there exist $\epsilon = \epsilon(k) \to 0$ and $N = N(\epsilon) \to +\infty$ for which 
 $\rho^N_{\ep}(t) = \sum_{i=1}^{N}\delta_{X^{i}_{\ep}(t)} m^i$ satisfies 
 \[ \lim_{k \to +\infty} W_1(\rho^N_{\ep,k}(t), \rho(t) ) \to 0 , \text{ uniformly for } t \in [0,T] , \]
 where  $\rho \in AC^2([0,T]; \mathcal{P}_2(\rr^d))$ is the unique weak solution of (\ref{mainpde}) with initial data $\rho(0)$.
 
 In particular, whenever $ {\nu}$ is log-convex on $\Omega$ and $f_0 \nu$ is concave on $\Omega$, the limit of the training dynamics is  the gradient flow of the  {convex} energy $\mathcal{R}$.
\end{cor}

 Our last main result concerns the behavior of minimizers of the energies $\F_{\ep,k}$ and $\F$.  Our proof of   Theorem  \ref{newmaintheorem} on the convergence of   gradient flows  as $k \to +\infty$ and $\epsilon = \epsilon(k) \to 0$.  leverages the perspective of Serfaty's general metric space framework for $\Gamma$-convergence of   gradient flows \cite{serfaty2011gamma}, which we recall in Section \ref{gammaconvergencerecall}. As a consequence of this approach, we   easily obtain that, under sufficient compactness assumptions on the approximation of our confining potential $V_k$, minimizers of $\F_{\ep,k}$ converge to a minimizer of $\F$.  Unlike in Theorem \ref{newmaintheorem} on convergence of the gradient flows, the rate at which $\epsilon \to 0$ does not depend on the rate $k \to +\infty$. Likewise, this result does not require $\bar{\rho}$ to be log-concave on $\Omega$.

\begin{thm}[minimizers converge to minimizers]  \label{minimizersthm}
Suppose Assumptions (\ref{targetas}), (\ref{Omegaas}), (\ref{mollifieras}), (\ref{Vas}), (\ref{Vkas}), and (\ref{Vkinftyas}) 
hold. Suppose further that $V_k \geq V_1$ for all $k \in \mathbb{N}$ and the sublevel sets of $V_1$ are compact. Then, if $\rho_{\ep,k} \in\mathcal{P}_2(\rr^d)$ is a minimizer of $\F_{\ep,k}$ for all $\ep>0$, $k \in \mathbb{N}$,  then as $\ep\rightarrow 0$, $k \to +\infty$,   $\rho_{\ep,k}$ narrowly converges to $\rho \in \mathcal{P}(\rr^d)$, where $\rho$ is the unique minimizer of $\F$ .
\end{thm}
 
This theorem has the potential to shed light on the convergence of the gradient flows in the long time limit. In particular, while our main   results on convergence of the gradient flows  only hold on \emph{bounded} time intervals, if one could show that a gradient flow  $\rho_{\ep,k}(t)$ of $\F_{\ep,k}$ indeed converged as $t \to +\infty$ to a minimizer of $\F_{\ep,k}$, uniformly in $\ep>0$ and $k \in \mathbb{N}$, then one could combine the above theorem with the preceding theorems to get convergence of the gradient flows of $\F_{\ep,k}$ to $\F$ globally in time. Proving these estimates remains an open question, closely related to our motivating applications in quantization.

\subsection{Outline of approach and future directions}

\color{black}

\label{outlinesection}  

We now outline our approach to proving these results. We begin, in  Section \ref{preliminariessec}, by  recalling preliminary information on optimal transport, including basic notation in Section \ref{notationsec}, convolution and convergence of measures  in Section \ref{convconvsec},  optimal transport and Wasserstein gradient flows  in Section \ref{OTsec}, and  our variant of  Serfaty's  framework for $\Gamma$-convergence of gradient flows  in Section \ref{gammaconvergencerecall}. 
In Section \ref{energiessec},  we prove several fundamental properties of the energy $\F_{\ep,k}$ and recall known properties of the energy $\F$, including convexity and differentiability  in Section \ref{lscsection}. We give the PDE characterizations of the gradient flows of these energies in Section \ref{GFcharsection} and address the long time behavior of gradient flows of the energy $\F$ in Section \ref{longtimesec}.

With these results in hand, we move on to studying the behavior of gradient flows of $\F_{\ep, k}$ as $\ep \to 0$ and $k \to +\infty$. Section \ref{sec:key estimates} is devoted to proving a key estimate for the analysis of the $\ep \to 0$ limit, which shows that if the initial conditions of the gradient flow of $\F_{\ep,k}$ have bounded entropy, then the mollified gradient flow $\zeta_\ep* \rho_{\ep,k}(t)$ satisfies an $H^1$ bound; see Theorem \ref{prop:H1 bd}. In Section \ref{proofsketchsection}, we sketch our proof of this result, formally integrating by parts, and in Sections \ref{prelimproofsec}-\ref{proofsec}, we prove the result, using the \emph{flow interchange method} developed by Matthes, McCann, and Savar\'e \cite{matthes2009family}.

In Section \ref{sec:Gamma conv}, we use the results of Section \ref{sec:key estimates} to study the $\ep \to 0$ limit. 
In Section \ref{Gammaenergiesep}, we
 obtain $\Gamma$-convergence of the energies $\E_{\ep} + \V_\ep$ as $\ep \to 0$. In Section \ref{eptozeroGFsec}, we move on to considering  convergence of the gradient flows of $\F_{\ep,k}$ as $\ep \to 0$ in Proposition \ref{prop:epstozerofixedk},  under the key hypothesis that the initial conditions of the gradient flow has uniformly bounded entropy.
Ultimately, we prove the gradient flows converge to an ``almost'' curve of maximal slope (see Definition \ref{almostcms}) of the intermediate energy $\F_k$, defined by
 \begin{align} \label{Fkdef}
\F_k(\rho) &= \E(\rho) + \V(\rho) + \V_k(\rho) .
\end{align}
We introduce the notion of an ``almost'' curve of maximal slope as a weakening of the traditional notion of gradient flow.  We need this weakening because we only suppose $\bar{\rho}$ is log-concave on $\Omega$, instead of on all of $\mathbb{R}^d$. Heuristically, this causes the energy $\F_k$ to lack sufficient regularity to define its gradient flow. More precisely, our weak assumptions on $\bar{\rho}$ prevent us from characterizing the strong upper gradient of $\F_k$, which likewise prevent us from defining its curve of maximal slope.

In spite of the fact that our hypotheses on $\bar{\rho}$ are too weak to  identify the $\epsilon \to 0$ limit of gradient flows of the energies $\F_{\ep,k}$ as a true gradient flow, it turns out that our notion of ``almost'' curve of maximal slope is sufficient to identify the behavior as $k \to +\infty$. We consider this limit in
Section \ref{eptozeroGFsecnew}, first obtaining $\Gamma$-convergence of the energies $\F_k$ to $\F$, as $k \to +\infty$, as well as our main theorem on convergence of the minimizers, Theorem \ref{minimizersthm}. We then prove that, as $k \to +\infty$, the ``almost'' curves of maximal slopes of $\F_k$ converge to the unique gradient flow of $\F$. Combining this with the $\ep \to 0$ result from the previous section, we prove our main result, Theorem \ref{newmaintheorem}, showing that gradient flows of $\F_{\ep,k}$  converge to a gradient flow of $\F$ as $k \to +\infty$ and $\ep \to 0$ sufficiently quickly. With these results in hand, we turn in Section \ref{particlesection} to extending the preceding convergence results on the gradient flows as $\ep \to 0$, $k \to +\infty$ to allow for gradient flows with particle initial data, thereby obtaining the proof of Theorem \ref{thm:conv with particle i.d.}. We also prove Corollary \ref{quantcor} on the long time behavior of the particle method and Corollary \ref{2layernncor} on the limit of two-layer neural networks. 
 
 We close in Section \ref{numericssection} with several numerical examples illustrating key properties of our method. We explore the   dynamics and long time behavior of particle solutions,   for targets $\bar{\rho}$ that satisfy the log-concavity assumptions of our main theorems, as well as targets that fail this assumption but satisfy a Poincar\'e inequality. In both cases, we observe that our particle discretization captures the behavior of the continuum PDE when $V=0$ and flows toward $\bar{\rho}$ on $\Omega$ in the long-time limit. We also explore the effect of the confining potential   $V_k$ on the dynamics for various choices of $k$, observing the qualitative   agreement with no-flux boundary conditions on $\Omega$, as well as the quantitative effect on rate of convergence   to (\ref{mainpde}) as $N \to +\infty$, $\ep \to 0$. In the case of strong confinement ($k = 10^9$) and log-concave target $\bar{\rho}$, we observe first order convergence in $N$, with $\epsilon = 4/N^{0.99}$ on $\Omega = (-1,1)$, both for the rate of convergence of the particle method to solutions of   (\ref{mainpde})  and for convergence of the particle method to the target $\bar{\rho}$ on $\Omega$ in the long time limit. Finally, as our scheme   preserves the   gradient flow structure of (\ref{mainpde}), it succeeds in capturing the exponential decay of the KL divergence along particle method solutions (see inequality (\ref{KLdivergenceexpdecay})), up to discretization error  and is energy decreasing for $\F_{\ep,k}$ for all values of $N$, $\ep,$ and $k$. 
 
 There are several directions for future work. Many of our results only lightly use the assumption that $\bar{\rho}$ is log-concave on $\Omega$, and it would be interesting to remove it. A key challenge in this direction is obtaining well-posedness of the gradient flow of $\F$ in the absence of convexity of the energy and proving that the metric slope is a strong upper gradient. A second direction for future work would be to improve methods for computing or approximating $f(x,y)$, as  defined in  (\ref{fdefintro}), which drives the dynamics of our system of ODEs. To compute this exactly involves integrating the reciprocal of the target $\bar{\rho}$ against the mollifiers, which can be done analytically   for a variety of targets $\bar{\rho}$, including  piecewise constant $\bar{\rho}$; see appendix \ref{formulassection}). Better understanding of the minimal information required on $\bar{\rho}$ required to approximate (\ref{fdefintro}) and the effect of this approximation on the dynamics would be important to applying this method in practice, especially when only partial information of $\bar{\rho}$ is known. A third interesting open question  would be to obtain quantitative results on the rate of convergence depending on $N \in \mathbb{N}$, $\ep >0$, and $k \in \mathbb{N}$, particularly if these quantitative estimates could be combined with existing estimates on the long time behavior of (\ref{mainpde}) to provide convergence guarantees regarding the convergence of the particle method to the target $\bar{\rho}$ on $\Omega$.

\section{Preliminaries} \label{preliminariessec}
\subsection{Basic notation} \label{notationsec}

For any $r>0$ and $x \in \Rd$ we use $B_r(x)$ to denote the open ball of center $x$ and radius $r$. We write $\indi_{S}$ for the indicator function of a given subset $S$ of $\rr^d$. i.e., 
\begin{align*}
\indi_S(x) = \begin{cases}1 &\text{ for }x \in S , \\
0 &\text{ otherwise.}\end{cases}\end{align*} 
 We denote the  $d$-dimensional Lebesgue measure by  $\mathcal{L}^d$.  
 
 Given $\mu\in\P(\Rd)$, we write $\mu \ll \mathcal{L}^d$ if $\mu$ is absolutely continuous with respect to $\mathcal{L}^d$, in which case we will denote both the probability measure $\mu$ and its Lebesgue density by the same symbol, e.g. $d \mu(x) = \mu(x) dx$. 
Finally, we let $L^p(\mu;\Omega)$ denote  the Lebesgue space of functions $f$ on $\Omega \subseteq \mathbb{R}^d$ with $|f|^p$ being $\mu$-integrable, and abbreviate $L^p(\Omega) = L^p(\mathcal{L}^d;\Omega)$. (We commit a slight abuse of notation by using  the same notation for the Lebesgue spaces of real-valued and $\rr^d$-valued functions.)

\subsection{Convolution and convergence of measures} \label{convconvsec}
A fundamental aspect of our approach is the regularization of the energy (\ref{Edef}) via convolution with a mollifier. We now recall some elementary results on convolution of probability measures.
For any  $\mu \in\P(\Rd)$ and  $\phi \in L^\infty(\Rd)$, the convolution of $\phi$ with $\mu$ is defined by,
\bes
	\phi*\mu(x) = \int_{\Rd} \phi(x-y) \,d\mu(y) \quad \mbox{for all $x\in\Rd$} .
\ees
Throughout, we use the fact that the definition of convolution allows us to move mollifiers from the measure to the integrand. In particular, for any $f$ bounded below and $\phi \in L^1(\Rd)$  even, we have,
\begin{equation}
    \label{eq:convproperty}
    	\ird f \,d(\phi*\mu) = \ird f*\phi \,d\mu.
\end{equation}
Likewise, we often use the following mollifier exchange lemma, which provides sufficient conditions for  moving functions in and out of  convolutions  within integrals. 
\begin{lem}[{mollifier exchange lemma, \cite[Lemma 2.2]{CarrilloCraigPatacchini}}] \label{move mollifier prop}
Let $f\colon \R^d \to \R$ be Lipschitz continuous with constant $L_f>0$, and let $\sigma$ and $\nu$ be finite, signed Borel measures on $\R^d$. There is $p = p(q,d)>0$ so that, 
\[ 
	\left| \int \zeta_\e *(f\nu) \,d\sigma - \int (\zeta_\e *\nu)f \,d\sigma \right|  \leq \e^{p} L_f \left( \int (\zeta_\e*|\nu|)\,d|\sigma| + C_\zeta |\sigma|(\Rd) |\nu|(\R^d) \right)  \  \text{ for all } \e >0.
\]
\end{lem}

We will often use the following notion of convergence:
\begin{defn}[narrow convergence]
\label{def:narrow conv}
A sequence $\mu_n$ in  $\mathcal{P}(\rr^d)$ is said to \emph{narrowly converge} to $\mu\in \mathcal{P}(\rr^d)$  if $\int f d \mu_n \to \int f d \mu$ for all bounded and continuous functions $f$.
\end{defn}

For fixed $\phi \in C_b(\Rd) \cap L^1(\Rd)$ even and any sequence $\mu_n$ narrowly converging to $\mu$, we immediately obtain from the definition of narrow convergence that, for any $f \in C_b(\Rd)$,
\begin{align} \label{convolutionandnarrow}
\int f (\phi* \mu_n) = \int (f *\phi) d \mu_n  = \int (f *\phi) d \mu = \int f (\phi*\mu),
\end{align}
so $\phi*\mu_n$ narrowly converges to $\phi*\mu$. Moreover, we have:
\begin{lem}[mollifiers and narrow convergence, {\cite[Lemma 2.3]{CarrilloCraigPatacchini}}] \label{weakst convergence mollified sequence}
	Suppose $\zeta_\epsilon$ is a mollifier satisfying Assumption (\ref{mollifieras}), and let $\mu_\e$ be a sequence in $\P(\R^d)$ converging narrowly to $\mu\in\P(\R^d)$. Then $\zeta_\e *\mu_\e$ narrowly converges to $\mu$.
\end{lem}

\subsection{Optimal transport, the Wasserstein metric, and Wasserstein gradient flows} \label{OTsec}

 We now describe basic facts about optimal transport and the Wasserstein metric, which we will use in what follows. For further background, we refer the reader to one of the many excellent textbooks on the subject \cite{ambrosio2008gradient, villani2003topics, santambrogio2015optimal, figalli2021invitation,ambrosio2021lectures}.
 
For a Borel measurable map $\bt \: \R^n \to \R^m$, we say that \emph{$\bt$ transports $\mu \in \P(\R^n)$ to $\nu \in \P(\R^m)$} if $\nu(A) = \mu(\bt^{-1}(A))$ for all measurable sets $A$. We call $\bt$ a \emph{transport map} and denote $\nu$ as $\bt_\# \mu \in \P(\R^m)$,  the \emph{push-forward} of $\mu$ through $\bt$.
For $\mu,\nu\in\P(\R^d)$,  the set of transport plans from $\mu$ to $\nu$ is given by,
\bes
	\Gamma(\mu,\nu) := \{\bgamma \in\P(\R^d\times\R^d) \mid {\pi^1}_\# \bgamma = \mu,\, {\pi^2}_\# \bgamma = \nu\},
\ees
where $\pi^1,\pi^2\colon \R^d \times \R^d \to \R^d$ are the projections of $\R^d\times \R^d$ onto the first and second copy of $\R^d$, respectively. 
 For $p \geq 1$, the \emph{$p$-Wasserstein distance} \cite[Chapter 7]{ambrosio2008gradient} between   $\mu,\nu\in\P_p(\R^d)$ is given by,
\be\label{eq:wass-p}
	W_p(\mu,\nu) = \min_{\bgamma \in \Gamma(\mu,\nu)}  \left( \int_{\R^d\times \R^d} |x-y|^p d \bgamma(x,y) \right)^{1/p},
\ee
where the definitions of the $p$-th moment and the space $\mathcal{P}_p(\rr^d)$ were recalled in (\ref{def:entropy}) and (\ref{eq:def Pp}). 
We say that a transport plan  $\bgamma$ is \emph{optimal} if it attains the minimum in \eqref{eq:wass-p}. We denote the set of optimal transport plans by $\Gamma_0(\mu,\nu)$. 

We  make the following observation: if 
$\bgamma$ is a transport plan from a measure $\mu$ to a Dirac mass $\delta_0$, then  $\bgamma = (\id \times 0)_{\#} \mu$, where $0$ denotes the function $0 : x \mapsto 0$. Using this in the definition (\ref{eq:wass-p}), we obtain,
\begin{equation}
\label{eq:Wp Mp}
W_p^p(\mu,{\delta_0}) = \int_{\Rd \times \Rd} |x - y|^p d \bgamma(x,y) = \int_{\Rd \times \Rd} |x-0|^p d \mu(x) = M_p(\mu).
\end{equation}

Note that 
applying H\"older's inequality in the definition of $W_p(\mu, \nu)$ yields,
\begin{align} \label{W1vsW2}
W_p(\mu,\nu) \leq W_{p'}(\mu,\nu) \ \quad \text{for $1\leq p\leq p'$ and for all } \mu, \nu \in \P_{p'}(\mathbb{R}^d).
\end{align}

Convergence with respect to the $p$-Wasserstein metric is stronger than narrow convergence of probability measures \cite[Remark 7.1.11]{ambrosio2008gradient}. In particular, for $p>1$, if $ \mu_n$ is a sequence in $ \P_p(\R^d)$ and $\mu \in \P_p(\R^d)$, we have,
\begin{align} \label{W2andnarrowconv}
	\mbox{$W_p(\mu_n,\mu) \to 0$ as $n\to\infty$} \iff \left(\mbox{$\mu_n \to \mu$ narrowly and $M_p(\mu_n) \to M_p(\mu)$ as $n\to\infty$}\right).
\end{align}
When $p=1$, the analogous result holds if convergence of first moments is replaced with the requirement that the first moments are \emph{uniformly integrable} \cite[Proposition 7.1.5]{ambrosio2008gradient}.

In order to define Wasserstein gradient flows, we   require the following notion of regularity in time with respect to the Wasserstein metric.

\begin{defn}[absolutely continuous]\label{defi:ac-curve}
We say $\mu:[0,T]\rightarrow \P(\Rd)$ is \emph{2-absolutely continuous} on $[0,T]$, and write $ \mu \in AC^2([0,T];\P_2(\Rd))$, if there exists $f\in L^2([0,T])$ so that, 
\begin{align} \label{eq:ac-p}
	W_2(\mu(t),\mu(s)) \leq \int_s^t f(r)\,d r \quad \mbox{for all $t,s\in (0,T)$ with $s\leq t$.}
\end{align}
\end{defn}
Along such curves, we may define the metric derivative.
\begin{defn}[metric derivative]\label{defi:metric-derivative}
Given $ \mu \in AC^2([0,T];\P_2(\Rd))$, the limit,
\bes
	|\mu'|(t) := \lim_{s\to t} \frac{W_2(\mu(t),\mu(s))}{|t-s|} 
\ees
exists for a.e. $t\in [0,T]$ and is called the\emph{ metric derivative} of $\mu$.
\end{defn}
\noindent In fact, the metric derivative is the minimal square integrable function satisfying ($\ref{eq:ac-p}$): for any $\mu \in AC^2([0,T];\P_2(\Rd))$,  we have  $|\mu'|\in L^2((0,T))$, and for any function $f$ satisfying  (\ref{eq:ac-p}), we have $|\mu'|(t)\leq f(t)$ for a.e. $t\in [0,T]$ (see \cite[Theorem 1.1.2]{ambrosio2008gradient}).

Geodesics form an important class of curves in the Wasserstein metric. Given $\mu_0,\mu_1 \in \P_2(\R^d)$, the geodesics connecting $\mu_0$ to $\mu_1$ are the curves of the form,
\begin{align}\label{eq:geodesic}
	\mu_\alpha = ((1-\alpha) \pi^1 + \alpha \pi^2)_\# \bgamma  \quad \mbox{for $\alpha\in[0,1]$, $\bgamma \in \Gamma_0(\mu,\nu)$}.
\end{align}
More generally, given $\mu_1,\mu_2,\mu_3\in \P_2(\R^d)$, a \emph{generalized} geodesic from $\mu_2$ to $\mu_3$ with base $\mu_1$ is given by,
\begin{align} \label{eq:gen-geodesic}
	\mu_\alpha^{2\to3} = \left((1-\alpha)\pi^2+\alpha\pi^3\right)_\# \bgamma \quad &\mbox{for }\alpha \in[0,1] \text{ and }\bgamma \in \P(\R^d \times \R^d  \times  \R^d) \\ & \text{ such that } {\pi^{1,2}}_\#\bgamma\in\Gamma_0(\mu_1,\mu_2) \text{ and }{\pi^{1,3}}_\#\bgamma\in\Gamma_0(\mu_1,\mu_3), \nonumber
\end{align}
	with $\pi^{1,i}\colon \R^d \times \R^d \times \R^d \to \R^d\times \R^d$ the projection of onto the first and $i$th copies of $\R^d$. Note that when the base $\mu_1$ coincides with one of the endpoints $\mu_2$ or $\mu_3$, a generalized geodesic is a geodesic.

A key property for the uniqueness and stability of Wasserstein gradient flows is  convexity, or more generally semiconvexity, along  generalized geodesics. 
\begin{defn}[semiconvexity]\label{defi:semiconvexity-geod}
	A functional $\G\colon\P_2(\R^d)\to (-\infty,\infty]$ is \emph{semiconvex along generalized geodesics} if there exists $\lambda \in \R$ such that, 
for all $\mu_1,\mu_2,\mu_3 \in \P_2(\R^d)$, there exists a generalized geodesic from $\mu_2$ to $\mu_3$ with base $\mu_1$ for which the following inequality holds:
\begin{align} \label{mainconvexityineq}
	\G(\mu_\alpha^{2\to3}) \leq (1-\alpha)\G(\mu_2) + \alpha\G(\mu_3) - \alpha (1-\alpha) \frac{\lambda}{2} W_{2,\bgamma}^2(\mu_2,\mu_3) \quad \mbox{for all $\alpha\in[0,1]$},
\end{align}
where, 
\bes
	W_{2,\bgamma}^2(\mu_2,\mu_3) := \int_{\Rd\times\Rd\times\Rd} |\pi^2-\pi^3|^2 \,d\bgamma(x,y,z).
\ees
In this case, will sometimes say the functional is \emph{$\lambda$-convex}. 
If a functional is 0-convex, we will say it is \emph{convex}.
\end{defn}
We recall the following sufficient condition for convexity, which is the Wasserstein analogue of the ``above the tangent line'' characterization of convexity from finite dimensional Euclidean space.

\begin{lem}[above the tangent line property {\cite[Proposition 2.8]{craig2017nonconvex}}] \label{abovetanlem}
A functional $\G\colon\P_2(\R^d)\to (-\infty,\infty]$ is \emph{$\lambda$-convex along generalized geodesics} if, for all generalized geodesics $\mu_\alpha^{2 \to 3}$ connecting $\mu_2$ to $\mu_3$ with base $\mu_1$, the map $\alpha \mapsto \G(\mu_\alpha^{2 \to 3})$ is differentiable for all $\alpha \in [0,1]$ and,
\begin{align*}
\G(\mu_3)   - \G(\mu_2) - \left. \frac{d}{d \alpha} \G(\mu_\alpha) \right|_{\alpha = 0} \geq \frac{\lambda}{2} W_{2, \bgamma}^2(\mu_2,\mu_3).
\end{align*}
\end{lem}

For any functional $\G \colon \P_2(\R^d) \to (-\infty , + \infty]$, we denote its \emph{domain} by $D(\G) = \{ \mu \in \P_2(\R^d) \mid \G(\mu) < +\infty\}$, and   say that $\G$ is \emph{proper} if $D(\G) \neq \emptyset$. 
For any measure $\mu$ in the domain of a functional $\G$, we may define the local slope of $\G$ at $\mu$ as follows.
 
 \begin{defn}[local slope] \label{localslopedef}
	Given $\G \colon \P_2(\Rd)  \to (-\infty,\infty]$, for any $\mu \in D(\G)$, the \emph{local slope} is,  
\bes
	|\partial \G|(\mu) = \limsup_{\nu \to \mu} \frac{(\G(\mu) - \G(\nu) )_+}{W_2(\mu,\nu)},
\ees
where  $(s)_+ = \max\{s,0\}$ denotes the positive part of $s$.
\end{defn}

Next, we define the subdifferential of  a functional $\mathcal{G}: \P_2(\Rd) \to (-\infty, +\infty]$ that is lower semicontinuous with respect to Wasserstein convergence and $\lambda$-convex along  generalized geodesics.\footnote{Note that in Ambrosio, Gigli, and Savar\'e \cite[Chapter 10]{ambrosio2008gradient} this is known as the \emph{reduced subdifferential}, which is stronger than their notion of \emph{extended subdifferential}: see Definition 10.3.1 of the extended subdifferential and  equations (10.3.12)-(10.3.13)  for the reduced subdifferential. The reduced subdifferential is sufficient for our purposes, due to the fact that our main $\Gamma$-convergence result considers gradient flow solutions that are absolutely continuous with respect to Lebesgue measure, and we extend the convergence to particle initial data separately.}

\begin{defn}[subdifferential of $\lambda$-convex functional] \label{subdiffdef}
Suppose $\G:\P_2(\Rd) \to (-\infty,+\infty]$  is proper, lower semicontinuous, and $\lambda$-convex along geodesics. Let $\mu \in D(\G)$  and    $\bxi: \Rd \to \Rd$ with $\bxi \in L^2( d \mu)$. We say that $\bxi$ belongs to the \emph{subdifferential} of $\G$ at $\mu$, and write $\bxi \in \partial \G(\mu)$,  if for all $\nu \in \P_2(\Rd)$,
\begin{align} \label{simplersubdiff}
\G(\nu) - \G(\mu) \geq  \int_{\Rd \times \Rd} \la \bxi(x),y-x \ra d\bgamma(x,y) + \frac{\lambda}{2} W_2^2(\mu,\nu) \quad \text{ for all } \bgamma \in \Gamma_0(\mu,\nu).
\end{align}
\end{defn}
\begin{rem}[subdifferential of sum] \label{subdiffsumrem}
Note that if $\G_1$ and $\G_2$ satisfy the hypotheses of Definition \ref{subdiffdef} and $\mu \in D(\G_1) \cap D(\G_2)$, then for any $\bxi_1 \in \partial \G_1(\mu)$ and $\bxi_2 \in \partial \G_2(\mu)$, we have $\bxi_1+\bxi_2 \in \partial (\G_1 + \G_2)(\mu)$.
\end{rem}

The local slope and subdifferential are related by the following proposition, which is a direct adaptation of \cite[Lemma 10.1.5]{ambrosio2008gradient} to the case of functionals which contain measures $\mu$ in their domain that are not necessarily absolutely continuous with respect to Lebesgue measure. We defer the proof to appendix \ref{appendix}.
\begin{prop}[local slope and minimal subdifferential] \label{metricslopesubdifabscts}
Suppose $\G: \P_2(\Rd) \to (-\infty, +\infty]$ is proper, lower semicontinuous, and $\lambda$-convex along generalized geodesics. Then for any   $\mu \in D(|\partial \G|)$, we have,
\begin{align} \label{subdiffmetricineq}
|\partial \G|(\mu) \leq \inf \left\{ \| \bxi\|_{L^2(\mu)} : \bxi \in \partial \G(\mu) \right\} .
\end{align}
If equality holds and $\bxi$ attains the infimum, we will write $\bxi = \partial^\circ \G(\mu)$. In this case, the element of the subdifferential attaining the infimum is unique.

\end{prop}

We now turn to the definition of a gradient flow in the Wasserstein metric (c.f. \cite[Definition 1.1.1, Proposition 8.3.1, Definition 11.1.1, Theorem 11.1.3]{ambrosio2008gradient}). 
\begin{defn}[gradient flow] \label{gradientflowdef}
	Suppose $\G\colon \P_2(\Rd) \to (-\infty,+\infty]$ is proper, lower semicontinous, and $\lambda$-convex along generalized geodesics. A curve $\mu(t) \in AC^2([0,T];\P_2(\Rd))$  is a \emph{gradient flow of $\G$ in the Wasserstein metric} if  $\mu(t)$ is a weak solution of the continuity equation,
\begin{equation}
\label{eq:continuity eqn defn gf}
    \partial_t \mu(t) + \grad \cdot (v(t) \mu(t)) = 0, \text{ in duality with }C_\mathrm{c}^\infty((0,T) \times \Rd)  ,  
\end{equation}
and, 
	\[    v(t) =- \partial^\circ \G(\mu(t)) \text{ for $\mathcal{L}^1$-a.e. } t \in [0,T]. \]
\end{defn}

Next, we recall sufficient conditions for well-posedness of the initial value problem for the gradient flow, when the initial condition $\mu(0)$ is in the closure of the domain of the energy $\overline{D(\G)}$. We also recall equivalent characterizations of the gradient flow as a curve of maximal slope and evolution variational inequality. As the theorem is simply a collection of general results developed by Ambrosio, Gigli, and Savar\'e \cite{ambrosio2008gradient}, we defer its proof to  Appendix \ref{appendix}. Note that our notion of curve of maximal slope differs slightly from Ambrosio, Gigli, and Savar\'e, since we use a version that is integrated in time.

\begin{thm}[well-posedness and characterization of gradient flow] \label{curvemaxslope}
Suppose $\G\colon \P_2(\Rd) \to (-\infty,+\infty]$ is proper, lower semicontinous, and $\lambda$-convex along generalized geodesics and $\mu(0) \in \overline{D(\G)}$. Then, there exists a unique gradient flow 
$\mu(t)$  of $\G$ satisfying $\lim_{t \to 0^+} \mu(t) = \mu(0)$ in the Wasserstein metric.

Furthermore $\mu(t) \in AC^2([0,T];\P_2(\Rd))$  is the gradient flow of $\G$   if and only if $\mu(t)$ satisfies one of the following equivalent conditions:
\begin{enumerate}[(i)]
\item  Curve of Maximal Slope:  \label{item:curve max slope}
\begin{align} \label{curvemaxslopedef}
\frac12 \int_0^t |\mu'|^2(r) dr + \frac12 \int_0^t |\partial \G|^2(\mu(r)) dr \leq \G(\mu(0)) - \G(\mu(t))  , \quad \text{ for all $t \in [0,T]$.} \end{align}
\item  Evolution Variational Inequality: For all $ \nu \in \P_2(\Rd)$ and for $\mathcal{L}^1\text{-a.e. } t \in [0,T]$,
\begin{align*} \frac{1}{2} \frac{d^+}{dt} W_2^2(\mu(t), \nu) + \frac{\lambda}{2} W_2^2(\mu(t), \nu) + \G(\mu(t)) \leq \G(\nu) .
\end{align*}
\end{enumerate}
\end{thm}

\subsection{$\Gamma$-convergence of energies and gradient flows} \label{gammaconvergencerecall}
We now recall the general framework of $\Gamma$-convergence of energies, which is a classical tool in the Calculus of Variations. This provides sufficient conditions that, when combined with some compactness, ensure   minimizers of a sequence of energies converge to a minimizer of a limiting energy. Next, we introduce a variant of Serfaty's scheme of $\Gamma$-convergence of gradient flows  \cite{serfaty2011gamma} that is weak enough to accommodate our assumptions on $\bar{\rho}$.  In particular, it will allow us to study the limiting behavior of both gradient flows of $\F_{\ep,k}$ in Section \ref{sec:Gamma conv}, as well as ``almost'' curves of maximal slope in Section \ref{eptozeroGFsecnew}.

We begin by recalling the notion of $\Gamma$-convergence of energies, focusing in particular on the case of energies  defined on $\P(\Rd)$, with respect to the narrow topology. 
\begin{defn}[$\Gamma$-convergence of energies]
\label{def:gamma conv def}
A sequence of functionals $\G_\alpha: \P(\Rd) \to \R \cup \{ +\infty\}$ is said to $\Gamma$-\emph{converge} to $\G: \P(\Rd) \to \R \cup \{ +\infty\}$ if:
\begin{align}
    \label{eq:gamma conv def}
   &\text{for any  sequence $\rho_\alpha\in \mathcal{P}(\mathbb{R}^d)$ converging narrowly to  $\rho\in \mathcal{P}(\mathbb{R}^d)$, } \liminf_{\alpha \to 0} \G_{\alpha}(\rho_\alpha) \geq \G(\rho) ; \\ \label{newenergygammapart1}
   &\text{for any $\rho \in \P_2(\Rd)$, there exists $\rho_\alpha\in \mathcal{P}(\mathbb{R}^d)$ converging narrowly to  $\rho$  s.t.}
\limsup_{\alpha \to 0}\G_{\alpha}(\rho) \leq \G(\rho) .
\end{align}
 \end{defn}

Next, we prove the following lemma, which provides sufficient conditions for compactness of a sequence of absolutely continuous curves.

\begin{lem}[compactness of absolutely continuous curves]
\label{lem:compactprime} 
Fix $T>0$. Suppose we have a sequence $\{\rho_\alpha\}_{\alpha>0} \subseteq AC^2([0,T];\P_2(\Rd))$ and
\begin{align}
&\sup_{\alpha>0} \int_0^T|\rho_\alpha'|^2(r)\, dr< +\infty \quad \text{ and } \quad \sup_{\alpha>0} M_2(\rho_\alpha(0))<+\infty  .   \label{eq:120602}
\end{align}
Then there exists $\rho\in AC^2([0,T];\P_2(\Rd))$  such that, along a subsequence $\alpha\rightarrow 0$, $W_1(\rho_\alpha(t), \rho(t))\rightarrow 0$ uniformly in $t\in [0,T]$, and  
\begin{equation}
\label{eq:120604}
\liminf_{\alpha\rightarrow 0} \int_0^t|\rho_\alpha'|^2(r)\, dr\geq \int_0^t|\rho'|^2(r)\, dr \text{ for every $t\in [0,T]$.}
\end{equation}
\end{lem}
\begin{proof}
First, we shall  produce $\rho:[0,T]\rightarrow \mathcal{P}_2(\rr^d)$ and a subsequence $\rho_\alpha$ such that $W_1(\rho_\alpha(t), \rho(t))\rightarrow 0$ uniformly in $t\in [0,T]$.  
To this end, we use Proposition \ref{secondmomentbound}, together with hypothesis (\ref{eq:120602})  to find that there exists $C = C(T)>0$,   so that,   for all $t\in [0,T]$ and $\alpha >0$, $\rho_\alpha(t)$ belongs to the set $\{\rho: M_2(\rho)\leq C\}$. This set is is narrowly sequentially compact  \cite[Remark 5.1.5, Lemma 5.1.7]{ambrosio2008gradient} and has uniformly integrable 1st moments  \cite[equation 5.1.20]{ambrosio2008gradient}, so it  is relatively compact in the 1-Wasserstein metric \cite[Proposition 7.1.5]{ambrosio2008gradient}. In particular, $\{\rho_\alpha(t)\}_{\alpha>0}$ is relatively compact with respect to the 1-Wasserstein metric, pointwise in time.

Next, using inequality (\ref{W1vsW2})  and hypothesis (\ref{eq:120602}), we deduce   equicontinuity with respect to the 1-Wasserstein metric: for all $0 \leq s \leq t \leq T$,
\begin{align*}
   \sup_{\alpha>0} W_1(\rho_\alpha(s),\rho_\alpha(t)) &\leq \sup_{\alpha>0} W_2(\rho_\alpha(s),\rho_\alpha(t)) \leq \sup_{\alpha>0} \int_s^t |\rho'_\alpha|(r) dr \leq   \sqrt{t-s} \left( \sup_{\alpha>0} \int_0^T |\rho_\alpha'|^2(r) \right)^{1/2}.
\end{align*}
Therefore, the   Ascoli-Arzel\'a theorem  ensures that there exists $\rho:[0,T] \to \P_2(\Rd)$ so that, up to a subsequence,  $W_1(\rho_\alpha(t), \rho(t))\rightarrow 0$ uniformly in $t\in [0,T]$. 

It remains to show  $\rho\in AC^2([0,T];\P_2(\Rd))$  and (\ref{eq:120604}). To see this, note that hypothesis (\ref{eq:120602}) ensures $\{ |\rho_\alpha'|(r) \}_{\alpha>0}$ is bounded in $L^2([0,T])$. Thus, up to another subsequence, it is weakly convergent to some $\nu(r) \in L^2([0,T])$. Thus, for all $0 \leq s \leq t \leq T$, using the lower semicontinuity of the 2-Wasserstein metric with respect to narrow (hence 1-Wasserstein) convergence,
\begin{align*}
W_2(\rho(s), \rho(t)) \leq \liminf_{\alpha \to 0} W_2(\rho_\alpha(s), \rho_\alpha(t)) \leq \liminf_{\alpha \to 0} \int_s^t |\rho'_\alpha|(r) dr = \int_s^t \nu(r) dr .
\end{align*}
This shows $\rho\in AC^2([0,T];\P_2(\Rd))$. Furthermore, by \cite[Theorem 1.1.2]{ambrosio2008gradient}, it ensures $|\mu'|(r) \leq \nu(r)$ for a.e. $ r \in [0,T]$. Thus, by lower semicontinuity of the $L^2([0,T])$ norm with respect to weak convergence,  (\ref{eq:120604}) holds.   
\end{proof}
 
With the preceding  result in hand, we now introduce our variant of $\Gamma$-convergence of gradient flows. Our conditions strongly mirror Serfaty's framework \cite[Theorem 2]{serfaty2011gamma}, in the context of functionals defined on $(\P_2(\Rd), W_2)$ that are  lower semicontinuous and semiconvex along generalized geodesics. The main difference is that we do not identify either $\| \bEta_\alpha\|_{L^2(\rho_\alpha)}$ or $\|\bEta\|_{L^2(\rho)}$ in the below proposition as a strong upper gradient of $\F_\alpha$ or $\F$. For this reason, we cannot conclude that either $\rho_\alpha$ or $\rho$ is a gradient flow of the respective energy. Our version of Serfaty's framework allows us to accommodate our more general assumptions on $\bar{\rho}$, while still ultimately obtaining our main convergence result, Theorem \ref{newmaintheorem}.

\begin{prop}
\label{prop:conv of almost cms}
Let $\mathcal{F},\mathcal{F}_\alpha:\mathcal{P}_2(\rr^d)\rightarrow \rr$ be functionals that are proper and bounded from below uniformly in $\alpha$, and suppose $\mathcal{F}_\alpha$ $\Gamma$-converges to $\mathcal{F}$ as $\alpha \to 0$. Fix $T>0$. Suppose that, for all $\alpha>0$, there exists $\rho_\alpha\in AC^2([0,T];\P_2(\Rd))$ and, for almost all $r\in [0,T]$, there exists  $\bEta_\alpha(r)\in L^2(\rho_\alpha(r))$, such that
\begin{equation}
\label{hyp:gen prop alpha curve max slope}
\frac{1}{2}\int_0^t|\rho_\alpha'|^2(r)\, dr +\frac{1}{2} \int_0^t\int_{\rr^d}|\bEta_\alpha(r)|^2\, d\rho_\alpha(r) \, dr \leq \mathcal{F}_\alpha(\rho_\alpha (0))- \mathcal{F}_\alpha(\rho_\alpha (t)) \quad  \text{ for all } 0 \leq t \leq T.
\end{equation}
Suppose also that there exists $\rho(0) \in D(\F) \cap \P_2(\Rd)$   so that
\begin{align}
\rho_{\alpha}(0) \xrightarrow[]{\alpha \to 0} \rho(0) \text{ narrowly, } \quad \lim_{\alpha \to 0} \F_{\alpha}(\rho_{\alpha}(0)) = \F(\rho(0)),  \quad  \text{ and } \quad  \sup_{\alpha >0} M_2(\rho_{\alpha}(0))< +\infty  . \label{hyp:gen prop F(0)}
\end{align}   

Then, there exists $\rho \in AC^2([0,T];\P_2(\Rd))$ so that, up to a subsequence in $\alpha$, 
\begin{align} \label{conc:epssubseqconvunif} \lim_{\alpha \to 0} W_1(\rho_{\alpha}(r), \rho(r)) = 0 \text{ uniformly for } r \in [0,T].
\end{align}
Furthermore, we have
\begin{align} \frac12 \int_0^t |\rho'|^2(r) dr +  \frac12 \int_0^t \left( \liminf_{\alpha \to 0} \int_{\mathbb{R}^d} |\bEta_\alpha(r)|^2 d \rho_\alpha(r) \right) dr \leq \F(\rho(0)) - \F(\rho(t))  \quad \text{ for all } 0 \leq t \leq T  . \label{generalconc:almostcurvemaxslope}
\end{align}
\end{prop}

\begin{proof}
    
By (\ref{hyp:gen prop F(0)}), we may assume,
\begin{equation}
    \label{eq:gen prop pf F(0)}
    \sup_\alpha \mathcal{F}_\alpha(\rho_\alpha(0))<+\infty.
\end{equation}
Using this, together with our assumption that   $\F_\alpha$ is bounded from below uniformly in $\alpha$, we see that the right-hand side of (\ref{hyp:gen prop alpha curve max slope}) is bounded from above uniformly in $\alpha$. From this we deduce $\sup_\alpha \int_0^t|\rho_\alpha'|^2(r)\, dr <+\infty$. Therefore, we may apply Lemma \ref{lem:compactprime}. We find that there exists $\rho\in AC^2([0,T];\P_2(\Rd))$ such that (\ref{conc:epssubseqconvunif}) holds along a subsequence $\alpha\rightarrow 0$ and we also have,
\begin{equation}
    \label{eq:gen prop pf met deriv bd}
 \liminf_{\alpha\rightarrow 0}  \int_0^t|\rho_\alpha'|^2(r)\, dr \geq \int_0^t|\rho'|^2(r)\, dr \text{ for all }t\in [0,T].
\end{equation}

Now,  taking $\liminf_{\alpha \rightarrow 0}$ of  (\ref{hyp:gen prop alpha curve max slope}), we find, for all $t \in [0,T]$,
\begin{equation}
\label{eq:upon taking liminf}
\frac{1}{2}\int_0^t|\rho'|^2(r)\, dr +\liminf_{\alpha\rightarrow 0}\frac{1}{2} \int_0^t\int_{\rr^d}|\bEta_\alpha(r)|^2\, d\rho_\alpha(r) \, dr \leq \liminf_{\alpha\rightarrow 0}\left(\F(\rho_\alpha(0)) - \F(\rho_\alpha(t))\right).
\end{equation}
The $\Gamma$-convergence of $\F_\alpha$ to $\F$ and the narrow convergence of $\rho_\alpha(t)$ to $\rho(t)$ (which follows from the convergence in 1-Wasserstein), as well as the hypothesis (\ref{hyp:gen prop F(0)}), imply the following upper bound for the right-hand side of the previous line:
\[
\F(\rho(0))-\F(\rho(t))\geq \limsup_{\alpha\rightarrow 0} \F(\rho_\alpha(0)) - \F(\rho_\alpha(t))\geq \liminf_{\alpha\rightarrow 0} \F(\rho_\alpha(0)) - \F(\rho_\alpha(t)).
\]
Finally, we use Fatou's Lemma to bound the second  term on the left-hand side of (\ref{eq:upon taking liminf}) from below. Together with the previous line, this yields the desired estimate (\ref{generalconc:almostcurvemaxslope}).

\end{proof}

\section{Gradient flows of energies with regularization and confinement}  \label{energiessec}
We now prove several fundamental properties of the internal energy $\E$ and the regularized internal energy $\E_\epsilon$, with the addition of external potential energies, $\V$ and $\V_\epsilon$, as well as the confining energies, $\V_k$ and $\V_\Omega$. In particular, we will characterize their lower semicontinuity,  convexity, and subdifferentiability. Each of these properties provides information about the one-sided regularity of the energy functional, its first derivative, and its second derivative with respect to the Wasserstein metric. Since our study of gradient flows only considers well-posedness of the flow forward in time (which is natural given that our motivating equation is a diffusion equation), these one-sided estimates on the energy functionals' regularity are sufficient for our analysis. We will close the section by applying these properties  to characterize the gradient flows of these energies in terms of partial differential equations. 
\subsection{Fundamental properties of energies} \label{lscsection}

First, we recall that the functionals $\E$ and $\Eep$ are lower semicontinuous with respect to narrow convergence. Since narrow convergence is weaker than Wasserstein convergence, this in turn implies lower semicontinuity with respect to Wasserstein convergence. The proof of this result is standard, and we defer it to appendix \ref{energyappendix}.

\begin{lem}[lower semicontinuity of $\E$ and $\E_\ep$]
\label{lem:semicontinuity} Suppose Assumptions (\ref{targetas}) and (\ref{mollifieras}) are satisfied. Then, for all $\ep >0$, the functionals $\E$ and $\E_\ep$ are lower semicontinuous with respect to  narrow convergence.
\end{lem}The lower semicontinuity of the external potential energies, $\V$ and $\V_\epsilon$, and the confining energies, $\V_k$ and $\V_\Omega$, with respect to narrow convergence is an immediate consequence of the Portmanteau theorem, see e.g. \cite[Lemma 5.1.7]{ambrosio2008gradient}, since they all are obtained by integrating a function that is lower semicontinuous and bounded below against $\rho$. 
\begin{lem}[lower semicontinuity of $\V$, $\V_\ep$, $\V_k$, $\V_\Omega$] \label{Vlsc}
Under Assumptions  (\ref{mollifieras}), (\ref{Omegaas}), (\ref{Vas}), and (\ref{Vkas}), the energies $\V$, $\V_\epsilon$, $\V_k$, and $\V_\Omega$ are lower semicontinuous with respect to narrow convergence.
\end{lem}

The convexity of the energies $\E$, $\V$, $\V_\epsilon$, $\V_k$, and $\V_\Omega$ follows immediately from the theory developed by Ambrosio, Gigli, and Savar\'e \cite{ambrosio2008gradient}. We recall these results in the following proposition. The proof of this proposition is an immediate consequence of existing theory, so we defer it to appendix \ref{energyappendix}.

\begin{prop}[convexity properties of of $\E+ \V_\Omega$, $\V$, $\V_\epsilon$, and $\V_k$]
\label{prop:convexity ep=0}\
\begin{enumerate}[(i)]
\item \label{econvex} Suppose $\bar{\rho}$  is log-concave on $\Omega$, where $\bar{\rho}$ satisfies Assumption (\ref{targetas}) and $\Omega$ satisfies  Assumption (\ref{Omegaas}). Then   $\E + \V_\Omega$ is convex along generalized geodesics.  
\item \label{vconvex} Suppose Assumptions (\ref{mollifieras}), (\ref{Vas}), and (\ref{Vkas}) hold. Then $\V$ and $\V_\epsilon$ are $\lambda$-convex along generalized geodesics, for $\lambda = \inf_{\{x, \xi \in \Rd\}} \xi^t D^2 V(x) \xi$, and $\V_k$ is convex along generalized geodesics. 
\end{enumerate}
\end{prop}

We now aim to show that $\E_\epsilon$ is also semiconvex for all $\epsilon >0$. In order to accomplish this, we begin by characterizing the directional derivative of $\E_\epsilon$. For the reader's convenience, we also recall the directional derivatives of the external potential energies $\V$, $\V_\epsilon$,  and $\V_k$, which have been studied extensively in previous works; see, for example,  \cite[Proposition 10.4.2]{ambrosio2008gradient}.

\begin{prop}[directional derivatives of $\E_\epsilon$, $\V$, $\V_\epsilon$, and $\V_k$]
	\label{prop:Fepder}
	Suppose Assumptions (\ref{targetas}), (\ref{mollifieras}), (\ref{Vas}) and  (\ref{Vkas}) hold.
	Fix $\epsilon >0$,  $\nu_1,\nu_2,\nu_3 \in \mathcal{P}_2(\mathbb{R}^d)$, and $\pmb{\gamma} \in  \mathcal{P}_2(\mathbb{R}^d \times \mathbb{R}^d \times \mathbb{R}^d)$ with $\pi^i_{\#}\pmb{\gamma} =\nu_i$.   Consider the curve, 
	\[ \mu_{\alpha} =   \big ( (1-\alpha)\pi^2 +\alpha \pi^3 \big)_{\#} \pmb{\gamma} \text{ for }\alpha \in [0,1] . \]
	 Then, 
	\begin{align}
	\left. \frac{d}{d\alpha} \E_{\epsilon}(\mu_\alpha) \right|_{\alpha =0} &= \nonumber \frac12 \int \frac{\zeta_\epsilon * \nu_{2  }(x)}{\bar{\rho}(x)}  \int   \la \nabla \zeta_\epsilon\big(x-  y_2\big ) ,  y_3-y_2\ra    d\pmb{\gamma} (y_1,y_2,y_3) \  dx , \\
	\left.	\frac{d}{d \alpha} \V(\mu_\alpha)  \right|_{\alpha = 0} &=  \int   \la\nabla V(y_2), y_3-y_2 \ra  d \bgamma(y_1, y_2, y_3) , \nonumber \\
		\left.	\frac{d}{d \alpha} \V_\epsilon(\mu_\alpha)  \right|_{\alpha = 0} &=  \int \la \nabla (\zeta_\epsilon*V)(y_2) ,y_3-y_2\ra  d \bgamma(y_1, y_2, y_3) , \nonumber \\
\left.	\frac{d}{d \alpha} \V_k(\mu_\alpha)  \right|_{\alpha = 0} &=  \int   \la\nabla V_k(y_2) , y_3-y_2\ra  d \bgamma(y_1, y_2, y_3) . \nonumber 
	\end{align}
 \begin{rem} \label{rem:gamma in Prop Eep der}
 Note that if $\pmb{\gamma} \in  \mathcal{P}_2(\mathbb{R}^d \times \mathbb{R}^d \times \mathbb{R}^d)$  satisfies the hypotheses in the definition of generalized geodesic (\ref{eq:gen-geodesic}), then $\pmb{\gamma}$ satisfies the assumptions of Proposition \ref{prop:Fepder}.
 \end{rem}
 
\end{prop}
\begin{proof}
	We begin with the characterization of the directional derivative    $\left. \frac{d}{d\alpha} \E_{\epsilon}(\mu_\alpha) \right|_{\alpha =0}$. As a first step in this direction, we estimate $\left. \frac{d}{d \alpha} \zeta_\epsilon*\mu_\alpha \right|_{\alpha =0}$. 
For all $x \in \mathbb{R}^d$ and $\alpha \in [0,1]$,  
	\begin{align} \label{eq:pp16ex0}
\frac{1}{\alpha} \left( \zeta_\epsilon *  \mu_{\alpha}(x)-\zeta_\epsilon *  \mu_{0 }(x)  \right)  &=   \int   \frac{1}{\alpha} \left[ \zeta_\epsilon(x -   ((1-\alpha)y_2 + \alpha y_3) ) - \zeta_\epsilon(x -  y_2  )  \right]   \ d \bgamma(y_1,y_2,y_3) .    \end{align}
By the mean value theorem for $\zeta_\epsilon$,  we may bound the integrand by,  
\begin{align} \label{eq:pp16ex02}
\frac{1}{\alpha} \|\nabla \zeta_\epsilon \|_\infty \left|   ((1-\alpha)y_2 + \alpha y_3) -   y_2   \right|    \leq \| \nabla \zeta_\epsilon\|_\infty \left|y_3 - y_2 \right| \in L^1(\bgamma) ,
\end{align}
where the integrability  holds since $M_1(\bgamma) \leq M_2(\bgamma)^{1/2} =\left( M_2(\nu_1)+ M_2(\nu_2) + M_2(\nu_3) \right)^{1/2} <+\infty$.
Thus, by the dominated convergence theorem,
\begin{align} \label{zetaepsmualphader}
\lim_{\alpha \to 0} \frac{1}{\alpha} \left( \zeta_\epsilon *  \mu_{\alpha}(x)-\zeta_\epsilon *  \mu_{0 }(x)  \right) &=  \int  \lim_{\alpha \to 0} \frac{1}{\alpha} \left[ \zeta_\epsilon(x -  ((1-\alpha)y_2 + \alpha y_3) ) - \zeta_\epsilon(x -   y_2  )  \right]   \ d \bgamma(y_1,y_2,y_3)   \nonumber \\
 &= \int   \la\nabla \zeta_\epsilon(x- y_2)  ,y_3-y_2\ra  d \bgamma(y_1,y_2,y_3) .
\end{align}

	Now, we use this to compute 	$ \left. \frac{d}{d\alpha} \E_{\epsilon}(\mu_\alpha) \right|_{\alpha = 0}$.  First, note that we may express the difference quotient as,
			\begin{align} \label{Eepsdiffquot}
\frac{1}{\alpha} \left( \E_{\epsilon}(\mu_\alpha) - \E_{\epsilon}(\mu_0) \right)  &=\frac{1}{2\alpha}\int  \left( (\zeta_\epsilon *\mu_\alpha)^2(x) -  (\zeta_\epsilon *\mu_0)^2(x) \right)  \bar{\rho}(x)^{-1}dx   \\
 &=\int \frac{1}{2\alpha } \left[ (\zeta_\epsilon *\mu_\alpha)(x) +  (\zeta_\epsilon *\mu_0)(x) \right] \left[ (\zeta_\epsilon *\mu_\alpha)(x) -  (\zeta_\epsilon *\mu_0)(x) \right] \bar{\rho}(x)^{-1} dx . \nonumber 
		\end{align}
By equations (\ref{eq:pp16ex0}-\ref{eq:pp16ex02}) 
and the fact that $\bar{\rho}$ is uniformly bounded below, the integrand is dominated by, 
\begin{align*}
g_\alpha(x) := C \left[ (\zeta_\epsilon *\mu_\alpha)(x) +  (\zeta_\epsilon *\mu_0)(x) \right] \quad \text{ for } \quad C = \| \grad \varphi_\epsilon \|_\infty \| \bar{\rho}^{-1}\|_\infty M_1(\bgamma) .
\end{align*}
The narrow convergence of $ \mu_\alpha$ to   $\mu_0$ as $\alpha\rightarrow 0$,  and the fact that $\zeta_\epsilon$ is bounded and continuous ensures that $g_\alpha(x) \to 2C(\zeta_\epsilon*\mu_0)(x)$ pointwise. Furthermore,
\begin{align*}
\lim_{\alpha \to 0} \int g_\alpha(x) d x= \lim_{\alpha \to 0} C \int \zeta_\epsilon*\mu_\alpha(x) dx + C \int \zeta_\epsilon*\mu_0(x) d x = 2C.
\end{align*}
Therefore, by the generalized dominated convergence theorem  \cite[Chapter 4, Theorem 19]{royden1988real} and equations (\ref{zetaepsmualphader}) and (\ref{Eepsdiffquot}),
\begin{align*}
\lim_{\alpha \to 0} \frac{1}{\alpha} \left( \E_{\epsilon}(\mu_\alpha) - \E_{\epsilon}(\mu_0) \right)  &= \int \lim_{\alpha \to 0} \frac{1}{2\alpha } \left[ (\zeta_\epsilon *\mu_\alpha)(x) +  (\zeta_\epsilon *\mu_0)(x) \right] \left[ (\zeta_\epsilon *\mu_\alpha)(x) -  (\zeta_\epsilon *\mu_0)(x) \right] \bar{\rho}(x)^{-1} dx  \\
 &=\int  \frac{(\zeta_\epsilon *\mu_0)(x)}{\bar{\rho}(x)} \int   \la \nabla \zeta_\epsilon(x- y_2) ,y_3-y_2 \ra  d \bgamma(y_1,y_2,y_3) \,    dx.
\end{align*}

Next we consider the directional derivative $\left.	\frac{d}{d \alpha} \V(\mu_\alpha)  \right|_{\alpha = 0}$. By definition of $\V$ and $\mu_\alpha$,
\begin{align*}
\lim_{\alpha \to 0} \frac{1}{\alpha} \left( \V(\mu_\alpha) - \V(\mu_0) \right) = \lim_{\alpha \to 0} \int  \frac{1}{\alpha} \left[ V\left(   (1-\alpha) y_2 + \alpha y_3 \right)  - V(y_2) \right] d \bgamma(y_1,y_2,y_3) .
\end{align*}
By the mean value theorem for $V$, we may bound the integrand by,
\begin{align*}
\frac{1}{\alpha} \| \nabla V\|_\infty |  \left( (1-\alpha) y_2 + \alpha y_3 \right) - y_2| \leq \|\nabla V\|_\infty |y_3-y_2| \in L^1(\bgamma). 
\end{align*}
Thus, by the dominated convergence theorem,
\begin{align*}
\lim_{\alpha \to 0} \frac{1}{\alpha} \left( \V(\mu_\alpha) - \V(\mu_0) \right) &=  \int \lim_{\alpha \to 0} \frac{1}{\alpha} \left[ V\left(   (1-\alpha) y_2 + \alpha y_3  \right) - V(y_2) \right] d \bgamma(y_1,y_2,y_3) \\
&= \int \la \nabla V(y_2) ,y_3-y_2\ra \ d \bgamma(y_1, y_2, y_3),
\end{align*}
which gives the result. The result for  $\V_\epsilon$ follows exactly as above, replacing $V$ with $(\zeta_\epsilon*V)$.

Finally, we consider the directional derivative of $\V_k$. By definition of $\V_k$ and $\mu_\alpha$ and the assumption that $V_k \in C^2$ with $\|D^2 V_k\|_\infty < +\infty$, we may apply the Fundamental Theorem of Calculus to conclude,
\begin{align*}
&\lim_{\alpha \to 0} \frac{1}{\alpha} \left( \V_k(\mu_\alpha) - \V_k(\mu_0) \right) \\
&\quad = \lim_{\alpha \to 0} \frac{1}{\alpha} \int_\Rd   \left[ V_k\left(  (1-\alpha) y_2 + \alpha y_3  \right) - V_k(y_2) \right] d \bgamma(y_1,y_2,y_3) , \\
&\quad = \lim_{\alpha \to 0} \frac{1}{\alpha} \int_\Rd \int_0^\alpha  \int_0^\beta (y_3-y_2)^t D^2 V_k((1-s)y_2+s y_3) (y_3-y_2) \, ds \, d \beta + \alpha \la\nabla V_k(y_2),y_3-y_2\ra d \bgamma(y_1,y_2,y_3) \\
&\quad = \int_\Rd \la\nabla V_k(y_2) ,y_3 - y_2 \ra d \gamma(y_1,y_2,y_3) ,
\end{align*}
where the first term vanishes since $D^2 V_k\in L^\infty(\Rd)$ and $\int |y_3-y_2|^2 d \bgamma  \leq 2(M_2(\nu_1) + M_2(\nu_2))<+\infty$.

	\end{proof}

Using this characterization of the directional derivative of $\E_\ep$, we  now prove that our energy $\E_\epsilon$ is $\lambda_\epsilon$-convex along generalized geodesics, where $\lambda_\epsilon \xrightarrow{\epsilon \to 0} - \infty$.

\begin{prop}[semiconvexity of $\E_\epsilon$] \label{lem:Eepsconvex}
Suppose Assumptions (\ref{targetas}) and (\ref{mollifieras}) hold.
	For all $\epsilon >0$, the functional $\E_{\epsilon}$ is $\lambda_\epsilon$-convex along generalized geodesics,
	where,
	\begin{equation}
	\lambda_\epsilon = - \|1/\bar{\rho}\|_{L^\infty(\mathbb{R}^d)}  \| D^2 \zeta_\epsilon\|_{L^{\infty}(\mathbb{R}^d)} =-\ep^{-d-2}||1/\bar\rho||_{L^\infty(\rr^d)}||D^2\zeta||_{L^\infty(\rr^d)}.
	\end{equation}
\end{prop}
\begin{proof}
	Let $(\mu_\alpha^{2 \rightarrow 3 })_{\alpha \in [0,1]}$ be a generalized geodesic  with base $\mu_1\in \mathcal{P}_2(\rr^d)$ connecting two probability measures $\mu_2,\mu_3 \in \mathcal{P}_2(\mathbb{R}^d)$, and let $\pmb{\gamma} \in \mathcal{P}_2(\mathbb{R}^d \times \mathbb{R}^d \times \mathbb{R}^d)$ be the associated measure as defined in  (\ref{eq:gen-geodesic}). 
	Since $x \mapsto x^2$ is a convex function, using the above the tangent inequality for convex functions yields,
	\begin{align*}
	\E_{\epsilon}(\mu_3)- \E_{\epsilon}(\mu_2) &= \frac{1}{2}\int \frac{ \big (\zeta_\epsilon * \mu_3 (x)\big )^2}{\bar{\rho}(x)}dx -\frac{1}{2}\int \frac{ \big (\zeta_\epsilon * \mu_2 (x)\big )^2}{\bar{\rho}(x)}dx \nonumber  \\
	&\geq   \int \frac{ \zeta_\epsilon * \mu_2 (x)}{\bar{\rho}(x)} \left(\zeta_\epsilon * \mu_3 (x)) - \zeta_\epsilon * \mu_{2} (x) \right) dx \nonumber \\
	&=  \int \frac{ \zeta_\epsilon * \mu_{2} (x) }{\bar{\rho}(x)} \left( \iiint \zeta_\epsilon(x-y_3)- \zeta_\epsilon(x-y_2) \right)  d \pmb{\gamma} (y_1,y_2,y_3) \, dx. \nonumber 
	\end{align*}
	Therefore, by Proposition \ref{prop:Fepder}, 	\begin{align}
	 & \E_{\epsilon}(\mu_3)  - \E_{\epsilon}(\mu_2) -  \left. \frac{d}{d\alpha}\E_{\epsilon}(\mu_\alpha^{2 \rightarrow 3 }) \right|_{\alpha = 0} \nonumber \\
	&\geq  \int \frac{\zeta_\epsilon * \mu_2 (x)}{\bar{\rho}(x)}\bigg( \iiint   \zeta_\epsilon(x-y_3) - \zeta_\epsilon(x-y_2)  - \la \nabla \zeta_\epsilon\big(x-  y_2\big ) , y_2 -y_3 \ra d\pmb{\gamma} (y_1,y_2,y_3)  \bigg)  dx\nonumber \\
	&\geq -\frac{\|D^2 \zeta_\epsilon\|_{L^{\infty}(\mathbb{R}^d)} }{2}\int \frac{\zeta_\epsilon * \mu_2 (x)}{\bar{\rho}(x)} dx \iiint |y_2-y_3|^2d\pmb{\gamma} (y_1,y_2,y_3)\nonumber  dx \nonumber \\
	&\geq-\frac{ \|1/\bar{\rho}\|_{L^\infty(\mathbb{R}^d)}  \| D^2 \zeta_\epsilon\|_{L^{\infty}(\mathbb{R}^d)}}{2} W^2_{2,\pmb{\gamma}}(\mu_2,\mu_3), \nonumber 
	\end{align}
	where we have applied Young's inequality to conclude that $\|\zeta_\epsilon * \mu_2\|_{L^1(\mathbb{R}^d)}=1$. 
	By Lemma \ref{abovetanlem}, this gives the result. 
	
\end{proof}

The preceding results  ensure that our energies $\E +\V_\Omega, \E_\epsilon, \V, \V_\epsilon,$ and $ \V_k$ are proper, lower semicontinuous, and semiconvex along generalized geodesics. Thus, the gradient flows of each of their energies, as well as the sum of any of the energies, is well posed, by Theorem \ref{curvemaxslope}, for any initial conditions in the closure of their domains. However, in order to characterize these gradient flows in terms of partial differential equations and prove our main $\Gamma$-convergence result, we must now characterize the minimal elements of their subdifferentials.
 
 We begin with the following proposition, identifying elements in the subdifferential of $\E_\epsilon$, $\V$, $\V_\epsilon$, and $\V_k$. Note that the subdifferentials of   $\V$, $\V_\epsilon$, and $\V_k$  were characterized in previous work \cite[Proposition 10.4.2]{ambrosio2008gradient}, and we recall key parts these results in item \ref{Vsubdif}   below   for the reader's convenience.

\begin{prop}[subdifferentials of $\E_\epsilon$, $\V$, $\V_\epsilon$, and $\V_k$] \label{subdifflemma} \
\begin{enumerate}[(i)]
\item Suppose Assumptions (\ref{targetas}) and (\ref{mollifieras}) hold. For all $\epsilon >0$ and $\mu \in D(\E_\epsilon)$, we have $\grad \frac{ \delta \E_\ep}{\delta \mu} \in \partial \E_\epsilon(\mu)$, where $\frac{ \delta \E_\ep}{\delta \mu} =  \zeta_\ep*\left(\left(\zeta_\ep*\mu  \right) / \bar{\rho}\right)$. \label{Eepsubdiff}
\item Suppose Assumptions (\ref{mollifieras}), (\ref{Vas}), and (\ref{Vkas}) hold. For all $\mu \in D(\V)$, we have $\nabla V \in \partial \V(\mu)$. Similarly, for all $\mu \in D(\V_\epsilon)$, we have $\nabla (\zeta_\epsilon * V) \in \partial \V_\epsilon(\mu)$, and, for all $\mu \in D(\V_k)$, we have $\nabla V_k \in \partial \V_k(\mu)$. \label{Vsubdif}
\end{enumerate} 
\end{prop}

\begin{proof}
We begin with the proof of \ref{Eepsubdiff}. Fix $\mu, \nu \in \P_2(\Rd)$ and $\bgamma \in \Gamma_0(\mu,\nu)$. Let $\mu_\alpha = ( (1-\alpha) \pi^1 + \alpha \pi^2)_{\#} \bgamma$ be a geodesic from $\mu$ to $\nu$.
By Lemma \ref{lem:Eepsconvex}, $\E_\ep$ is $\lambda_\epsilon$-convex along generalized geodesics, so in particular, it is convex along $\mu_\alpha$, and  Lemma \ref{abovetanlem} ensures,
\begin{align*}
  \E_\ep(\nu) - \E_\ep(\mu) -   \left. \frac{d}{d \alpha} \E_\ep(\mu_{\alpha}) \right|_{\alpha=0}   \geq \frac{ \lambda_\epsilon}{2} W_2^2(\mu,\nu) .
   \end{align*}
Rearranging and applying Proposition \ref{prop:Fepder}, with $\tilde{\bgamma} = (\pi^1, \pi^1, \pi^2)_{\#} \bgamma$, and Fubini's theorem, yields,  
\begin{align*}
\E_\ep(\nu) - \E_\ep(\mu) &\geq \frac12 \int \frac{\zeta_\epsilon * \mu_{  }(x)}{\bar{\rho}(x)}  \int  \la \nabla \zeta_\epsilon\big(x-  y_2\big ), y_2 -y_3 \ra \ d\tilde{\pmb{\gamma}} (y_1,y_2,y_3) \  dx + \frac{ \lambda_\epsilon}{2} W_2^2(\mu,\nu) \\
&= \frac12 \int \frac{\zeta_\epsilon * \mu_{  }(x)}{\bar{\rho}(x)}  \int \la  \nabla \zeta_\epsilon\big(x-  y_1\big ) , y_1 -y_2 \ra \ d \bgamma(y_1,y_2)\  dx + \frac{ \lambda_\epsilon}{2} W_2^2(\mu,\nu) \\
&=   \frac{1}{2} \int \la \grad \zeta_\epsilon*\left(\frac{ \zeta_\epsilon * \mu }{\bar{\rho}} \right) (y_1) , y_2 -y_1 \ra d\pmb{\gamma}  (y_1,y_2)  +\frac{ \lambda_\epsilon}{2} W_2^2(\mu,\nu)  \\
&= \int  \la \grad \frac{ \delta \E_\ep}{\delta \mu} (y_1) , y_2 -y_1 \ra d\pmb{\gamma}  (y_1,y_2)  + \frac{ \lambda_\epsilon}{2} W_2^2(\mu,\nu) .
\end{align*}
This shows $\grad \frac{ \delta \E_\ep}{\delta \mu} \in \partial \E_\ep(\mu)$, by Definition \ref{subdiffdef} of the subdifferential.

For item \ref{Vsubdif}, we will show the result for $\V$, since the result for $\V_\epsilon$  and $\V_k$ follow from the same argument, simply via  replacing $V$ with $\zeta_\epsilon *V$ and $V_k$, respectively. Let $\nu$, $\mu$, $\bgamma$, and $\tilde{\bgamma}$ be as in the proof of item (i). Applying Lemma \ref{abovetanlem}, Proposition \ref{prop:convexity ep=0}, Proposition \ref{prop:Fepder}, and rearranging, again as in the proof of (i), yields,
\begin{align*}
\V(\nu)-\V(\mu) &\geq \int  \la \nabla V(y_2), y_3-y_2 \ra   d \tilde{\bgamma}(y_1, y_2, y_3)+\frac{\lambda}{2} W^2_2(\nu, \mu)\\
&=\int  \la \nabla V(y_1), y_2-y_1 \ra    d \bgamma(y_1, y_2)+\frac{\lambda}{2} W^2_2(\nu, \mu),
\end{align*}
which shows $\grad V \in \partial \V(\mu)$, by Definition \ref{subdiffdef} of the subdifferential. 
\end{proof}

    Next, we characterize the minimal subdifferential of the energy $\F_{\ep, k}=\E_\epsilon +\V_\epsilon + \V_k$ for all $\epsilon >0$, $k \in \mathbb{N}$.  The proof is standard, and we defer it to Appendix \ref{energyappendix}.
\begin{prop}[minimal subdifferential of   $ \F_{\ep, k}$]
\label{prop:subdiff char ep>0}
Suppose Suppose Assumptions (\ref{targetas}), (\ref{mollifieras}), (\ref{Vas}), and (\ref{Vkas}) hold. 
For all $\epsilon >0$ and $k \in \mathbb{N}$, $\mu \in D(\F_{\ep, k})$, 
\begin{align}  \label{subdiffchareepsk}
\nabla \frac{\partial \E_\ep}{\partial \mu} + \nabla (\zeta_\ep*V) +\nabla V_k = \partial^\circ \F_{\ep, k}(\mu).
\end{align}
 
\end{prop}

Finally, we close by recalling Ambrosio, Gigli, and Savar\'e's characterization of the minimal subdifferential of  $\F$ \cite[Theorems 10.4.9-10.4.13]{ambrosio2008gradient}. 
\begin{prop}[minimal subdifferential of  $\F$, {\cite[Theorems 10.4.9-10.4.13]{ambrosio2008gradient}}] 
\label{prop:subdiff char ep=0} Assume (\ref{targetas}), (\ref{Vas}), (\ref{Omegaas}), and that $\bar\rho$ is log-concave on $\Omega$.
 Given $\mu \in D(\F)$, we have   $ |\partial \F|(\mu)< +\infty$ if and only if  $(\mu/\bar{\rho})^2 \in W^{1,1}_\loc(\Omega)$ and there exists $\bxi \in L^2(\mu)$ so that, 
\[   \bxi \mu = \frac{\bar{\rho}}{2} \nabla (\mu/\bar{\rho})^2 + \nabla V \mu  \quad \text{ on } \quad \Omega.  \]
In this case, $\bxi \in \partial^\circ \F(\mu) $.
 \end{prop}

\subsection{Differential equation characterization of gradient flows} \label{GFcharsection}

We close by identifying the   differential equations that characterize gradient flows of $\F_{\epsilon,k}$  and $\F$. These proofs are natural consequences of the properties of the energies proved in the previous section and the definition of gradient flow, so we defer them to Appendix \ref{energyappendix}.

\begin{prop}[PDE characterization of GF of $\F$]
\label{prop:PDE ep=0} Assume  (\ref{Vas}),  (\ref{targetas}), (\ref{Omegaas}), and $\bar\rho$ is log-concave on $\Omega$. For every $\mu_0\in  \overline{D(\F)}$, we have that 
$\mu(t) \in AC^2([0,T];\P_2(\Rd))$ is the unique Wasserstein gradient flow of $\F$  with initial data $\mu_0$  if and only if $\mu(t)$ 
satisfies,
\begin{equation}
\label{eq:PDE E}
\begin{cases}
\partial_t\mu -\nabla \cdot  \left(\frac{\bar{\rho}}{2} \nabla \left(\frac{\mu^2}{\bar\rho^2}  \right)  + \nabla V \mu \right)=0,   \text{ in     duality with } C^\infty_c( \rr^d \times (0,\infty)) ,
  \\
   \lim_{t \to 0^+} \mu(t) =  \mu_0 \quad \text{ in $W_2$},
\end{cases}
\end{equation}
and  satisfies,
\begin{align}
    \label{eq:soln space1}
&\mu(t) \ll \mathcal{L}^d, \mu = 0   \text{ $\mathcal{L}^d$-a.e.  on  }\Rd \setminus \overline{\Omega}, \text{ and }(\mu(t)/ \bar{\rho})^2 \in W^{1,1}_{\loc}(\Omega) \text{ for $\mathcal{L}^1$-a.e. $t>0$,}  \\
&\int_{\rr^d} \left| \bar{\rho} \ \nabla\left( \mu(t)^2/\bar\rho^2 \right)/(2\mu) + \nabla V \right|^2 d \mu     \in L^1_{loc}(0,\infty) . \label{eq:soln space2}
\end{align}
\end{prop}

\begin{rem}[relationship with existing work on nonlinear diffusion equations]  \label{PDEeps0remark}

First, note that if $\overline{\Omega}$ is compact, then the weak formulation of the PDE in equation (\ref{eq:PDE E}) implies that the PDE also holds in the duality with $C^\infty((0, +\infty) \times \Omega)$, which is a weak formulation of the no-flux boundary conditions,
\begin{align} \label{noflux} \frac{\bar{\rho}}{2}  \partial_{\mathbf{n}} \left(\frac{\mu^2}{\bar\rho^2}  \right)  +  \partial_{\mathbf{n}}V \mu = 0 \text{ on } \partial \Omega,
\end{align}
since the test functions are merely required to be compactly supported $ \Rd \times (0, +\infty) $, not $\Omega \times (0,+\infty)$.
In particular, if $\mu$ is a smooth classical solution of (\ref{mainpde}) with no-flux boundary conditions,  it solves (\ref{eq:PDE E}).

In \cite{O}, Otto pioneered the connection between PDEs and Wasserstein gradient flows, characterizing solutions to    homogeneous porous medium equations ($\bar \rho =1$) without boundary $(\Omega = \Rd)$ as gradient flows of the internal energy $\F(\rho) = \frac{1}{2} \int \rho^2$. The notion of solution used in this previous work is stronger than the one in Proposition \ref{prop:PDE ep=0}. In particular, if $\rho$ is a solution to the porous medium equation in this previous sense \cite[Definition 1]{O}, then it is  a solution of (\ref{eq:PDE E}), hence a gradient flow in the sense defined here.

More recently,   Dolbeault, et al. \cite{dolbeault2008lq} and Grillo, Muratori, and Porzio \cite{GMP2013} consider well-posedness of  (\ref{mainpde}). If $u$ is smooth enough, it is a solution to    \cite[equation (1.1)]{GMP2013} (with    $\rho_\nu = \rho_\mu =\bar\rho$, and with $\Omega=\rr^d$) if and only if $\mu:=\bar \rho u$ satisfies (\ref{eq:PDE E}).  More precisely comparing our notion of solution with  \cite[Definition 3.5]{GMP2013},
we observe that our definition requires the same regularity in space, stronger regularity in time, and we employ a smaller class of test functions.

\end{rem}

Next, we provide a PDE characterization of the gradient flow  of $\F_{\epsilon,k}$, the proof of which we again defer to Appendix \ref{energyappendix}.

\begin{prop}[PDE characterization of GF of $\F_{\epsilon,k}$] \label{PDEFeps} Suppose Assumptions (\ref{targetas}), (\ref{mollifieras}), (\ref{Vas}), and (\ref{Vkas}) hold. 
For every $\mu_0\in \overline{D(\F_{\ep,k})}$, we have that $\mu(t) \in AC^2([0,T];\P_2(\Rd))$ is the  unique Wasserstein gradient flow of $\F_{\epsilon,k}$  with initial data $\mu_0$  if and only if $\mu(t)$ satisfies,
\begin{align}
\label{eq:PDE Eepk}
\begin{cases}
    \partial_t\mu -\nabla \cdot  \left(\mu   \left(  \nabla  \zeta_\ep*\left(\frac{\zeta_\ep*\mu}{\bar \rho} \right)+ \nabla (\zeta_\epsilon* V) + \nabla V_k \right) \right)=0,  \text{ in   duality with } C^\infty_c( \rr^d 
    \times (0,\infty)), \\
    \lim_{t \to 0^+} \mu(t) = \mu_0 \text{ in $W_2$} .
    \end{cases}
\end{align} 
\end{prop}

Finally, we characterize the dynamics of  the gradient flow of $\F_{\ep,k}$  when the initial data is given by an empirical measure. We show that it remains an empirical measure for all time, that is, ``particles remain particles'', and we  explicitly state the ODE that characterizes the empirical measure's evolution.  The proof is in Appendix \ref{energyappendix}.

\begin{prop}[particle evolution for $\F_{\epsilon,k}$]
\label{prop:ODE Eep} Suppose Assumptions (\ref{targetas}), (\ref{mollifieras}), (\ref{Vas}), and (\ref{Vkas}) hold. 
Fix $\ep>0$, $N\in \N$, $\{X_0^1, \dots, X_0^N\} \in \Rd$, and $\{m^1, \dots, m^N \} \in \R_+$ satisfying $\sum_{i=1}^N m^i=1$. 
Then, there exists a unique continuously differentiable function $X:[0,\infty)\rightarrow \rr^{Nd}$, with components $(X^1(t),..., X^N(t))$, that satisfies the system,
  \begin{equation}     \label{eq:ODEepk} \begin{cases}
    \dot{X}^{i} = -\sum_{j=1}^N m^j\int_{\rr^d} \nabla \zeta_\ep(X^{i}-z)\zeta_\ep(z-X^{j})\frac{1}{\bar\rho(z)}\, dz - \nabla (\zeta_\epsilon*V)(X^{i}) - \nabla V_k(X^{i}) ,\\
     X^{i}(0)=X^{i}_{0}.
    \end{cases}
    \end{equation}
Moreover, $\mu(t)  := \sum_{i=1}^N \delta_{X^i(t)}m^i $ is the unique Wasserstein gradient flow of $\F_{\epsilon,k}$  with initial conditions $\mu(0)$.
\end{prop}

\subsection{Long-time behavior} \label{longtimesec}
We conclude this section by recalling   known properties of the long time behavior of (\ref{mainpde}) or, equivalently, gradient flows of $\F$, which motivate its connection to quantization.

\begin{prop}[long time behavior, {\cite{ambrosio2008gradient}}] \label{longtime} Assume (\ref{Omegaas}),  (\ref{targetas}), $V=0$, $\Omega$ is bounded, and $\bar\rho$ is log-concave on $\Omega$. Let $\rho_0\in D(\F)$ and let $\rho(t)$ be the gradient flow $\rho$ of $\F$ with initial data $\rho_0$. Then we have,
\[ \lim_{t \to +\infty } W_2\left(\rho(t), \frac{\indi_{\overline{\Omega}}\bar{\rho}}{\int_\Omega \bar\rho\, d\mathcal{L}^d}\right) = 0  . 
\]
\end{prop}
\begin{proof}
This is an immediate consequence of \cite[Corollary 4.0.6]{ambrosio2008gradient}.
\end{proof}

\section{An $H^1$ bound on the mollified gradient flow of $\F_{\epsilon,k}$}\label{sec:key estimates}
A key element in our proof of the convergence of the gradient flows of $\F_{\epsilon,k}$ to a gradient flow of $\F_k$ as $\epsilon \to 0$ is the following $H^1$-type bound on   $\zeta_\epsilon*\rho_\epsilon(t)$ (the  mollified gradient flow of $\F_{\epsilon,k}$)  in terms of the energy, second moment, and entropy of the initial data. We remark that this bound holds without a log-concavity assumption on $\bar\rho$.

\begin{thm}[$H^1$ bound on mollified GF of $\F_{\epsilon,k}$]
\label{prop:H1 bd}
Assume (\ref{targetas}), (\ref{mollifieras}), (\ref{Vas}), and (\ref{Vkas}) hold. There exist positive  constant $C_{\bar\rho}$  and $C_V$, depending on   $\bar\rho$, $V$, and $V_k$, so that, for all  $T>0$, $k\in \mathbb{N}$, and $\epsilon >0$  and for any gradient flow $\rho_\ep\in AC^2([0,T];\P_2(\Rd))$  of  $\F_{\ep,k}$, we have,  
\begin{align}
\label{eq:H1bd thm: conclusion}
\int_0^T|| \nabla\zeta_\ep*\rho_\ep (t) ||^2_{L^2(\rr^d)}\, dt &\leq C_{\bar\rho}\left(\S(\rho_\epsilon(0)) + \sqrt{2 \pi} +(1+T+T e^T)  \left( M_2(\rho_\epsilon(0)) + \F_{\epsilon,k}(\rho_\epsilon(0)) +C_V \right)\right) .
\end{align} 
\end{thm}

\subsection{Proof sketch} \label{proofsketchsection}
First, we describe a formal argument to obtain inequality (\ref{eq:H1bd thm: conclusion}), and then we explain how to make the argument rigorous. 
By Proposition \ref{PDEFeps}, $\rho_\ep(t)$ is a weak solution of the PDE,
\begin{equation}
\label{eq:H1 bd pf PDE}
\partial_t\rho_\ep =  \nabla \cdot \left( \rho_\ep \nabla \zeta_\ep * \left(\frac{\zeta_\ep*\rho_\ep}{\bar\rho} \right) + \rho_\epsilon  \nabla (\zeta_\ep *V)  + \rho_\ep \nabla V_k \right) ,
\end{equation}
in the duality with $C^\infty_c(\Rd \times (0,\infty))$.
Thus, formally evaluating the entropy  $\S(\rho)$ along the gradient flow, differentiating in time, and integrating by parts, we  obtain,
\begin{align} \label{formaltimederivative}
\frac{d}{dt} \S(\rho_\ep)(t) &=  \int_{\Rd} \log(\rho_\ep) \partial_t\rho_\ep \, d\mathcal{L}^d \\
&=  -  \int_{\Rd} \la \nabla \rho_\ep  , \nabla \zeta_\ep * \left(\frac{\zeta_\ep*\rho_\ep}{\bar\rho} \right) \ra+ \la \nabla \rho_\ep , \nabla (\zeta_\ep* V)\ra  + \la\nabla \rho_\ep, \nabla V_k\ra d\mathcal{L}^d \nonumber \\
&=  -  \int_{\rr^d}  \frac{|\nabla\zeta_\ep*\rho_\ep|^2}{\bar\rho} +   \la\nabla (\zeta_\ep* \rho_\ep ),(\zeta_\ep*\rho_\ep) \nabla \left( \frac{1}{\bar\rho} \right)\ra  + \la\nabla (\zeta_\ep*\rho_\ep), \nabla  V\ra - \rho_\ep\Delta V_k \, d\mathcal{L}^d.  \nonumber
\end{align}
Integrating in time and estimating the  terms on the right hand side  then leads to inequality (\ref{eq:H1bd thm: conclusion}). 

The key difficulty in making the above argument rigorous is justifying the time differentiation of the entropy, in the absence of relevant a priori estimates for $\rho_\epsilon$. In order to overcome this difficulty, McCann, Matthes, and Savar\'e introduced the \emph{flow interchange method} \cite{matthes2009family}. Suppose that $\rho_\epsilon(t)$ and $\mu(t)$ are, respectively, the gradient flows of the energy $\F_{\epsilon,k}$ and the entropy $\S$, and we have  $\rho_\epsilon(0) = \mu(0)$. The flow interchange method is based on the following formal observation, with $\nabla_{W_2}$ denoting the Wasserstein gradient and $\la \cdot, \cdot \ra_{W_2, \rho}$ denoting the Wasserstein inner product at $\rho$:
\begin{align*}
\left. \frac{d}{dt} \S(\rho_\epsilon)  \right|_{t=0} &= \left. \la \nabla_{W_2} \S(\rho_\epsilon) , \partial_t \rho_\epsilon \ra_{W_2, \rho_\epsilon} \right|_{t=0} = - \left. \la \nabla_{W_2} \S(\rho_\epsilon),\nabla_{W_2} \F_{\epsilon,k}(\rho_\epsilon)\ra_{W_2, \rho_\epsilon} \right|_{t=0} \\
&= - \left. \la \nabla_{W_2} \S(\mu) ,\nabla_{W_2} \F_{\epsilon,k}(\mu)\ra_{W_2, \mu} \right|_{t=0} 
= \left. \la \partial_t \mu , \nabla_{W_2} \F_{\epsilon,k}(\mu) \ra_{W_2,\mu} \right|_{t=0} = \left. \frac{d}{dt} \F_{\epsilon,k}(\mu)  \right|_{t=0} .
\end{align*}
Consequently, at a fixed time, differentiating $\F_{\epsilon,k}$ along the gradient flow of $\S$ should give the same result as equation (\ref{formaltimederivative}). The former is much easier to justify in practice, since the gradient flow $\mu(t)$ of $\S$ with initial data $\mu(0)$ is precisely the solution of the heat equation on $\rr^d$ with initial data $\mu(0)$  \cite[Examples  11.2.7]{ambrosio2008gradient}, for which we have robust a priori estimates. 

Note that, since the entropy $\S$ is a 0-convex energy  \cite[Proposition 9.3.9]{ambrosio2008gradient}, the evolution variational inequality characterization of gradient flows, recalled in Theorem \ref{curvemaxslope}, ensures that, if $\mu(t)$ is the gradient flow of $\S$, then  for all $ \nu \in \P_2(\Rd)$ and for $\mathcal{L}^1\text{-a.e. } t \geq 0$, 
\begin{align} \label{entropyevi}
\frac{1}{2} \frac{d^+}{dt} W_2^2(\mu(t), \nu)  + \S(\mu(t)) \leq \S(\nu) .
\end{align}

\subsection{Preliminaries for the proof} \label{prelimproofsec}
Now, we introduce the machinery we need for our rigorous argument, following the outline described above. To avoid differentiating $\S(\rho_\epsilon)$ in time, we work  with the discrete time analogue of the gradient flow of $\F_{\epsilon,k}$, given by the minimizing movement scheme  (see Definition \ref{minmovdef}).
\begin{defn}[minimizing movement scheme for $\F_{\ep,k}$]
Given $\mu \in \P_2(\Rd)$, let $\J^n_{\tau, \epsilon} \mu$ denote the $n$th step of the minimizing movement scheme of $\F_{\epsilon,k}$ with time step $\tau$ and initial data $\J^{0}_{\tau, \epsilon} \mu = \mu$.
\end{defn}
Due to the robust a priori estimates available for solutions of the heat equation, we will work with continuous time gradient flow of $\S$.
\begin{defn}[heat flow semigroup]
Gven $\mu \in \P_2(\Rd)$ and $h \geq 0$, we will let $\rS_h \mu$ denote the (continuous time) gradient flow of $\S$ with initial data $\mu$ at time $h$; in other words, $\rS_h$ is the heat flow semigroup operator.
\end{defn}
We will use the fact that, for any $\mu\in \mathcal{P}_2(\rr^d)$, we have,
\begin{equation}
\label{eq:Sh commutes}
\zeta_\ep*(\rS_h(\mu))= \rS_h(\zeta_\ep*\mu).
\end{equation} 

A key step in the proof is computing the derivatives in $h$ of $\E_\ep(S_h(J^n_{\tau, \ep}\mu))$, $\V_\ep(S_h(J^n_{\tau, \ep}\mu))$, and $\V_k(S_h(J^n_{\tau, \ep}\mu))$ at $h=0$. We separate this step into a separate lemma:
\begin{lem}[derivatives along $S_h(\J^n_{\tau, \epsilon} \mu)$]
\label{lem:deriv Eep of Sh}
Assume (\ref{targetas}), (\ref{mollifieras}), (\ref{Vas}), and (\ref{Vkas}) hold. 
Let $\mu\in \P_2(\Rd)$. We have,
\begin{align}
\label{eq:deriv Eep of Sh}
&\limsup_{h \to 0^+} \frac{\E_\ep(\J^n_{\tau, \epsilon} \mu) - \E_\ep(\rS_h(\J^n_{\tau, \epsilon} \mu))}{h}=-
\int_\Rd \frac{1}{ \bar{\rho}}  \Delta (\zeta_\epsilon*\J^n_{\tau, \epsilon} \mu)\left( \zeta_\epsilon*\J^n_{\tau, \epsilon} \mu \right)\, d\mathcal{L}^d,\\
&
\label{eq:deriv Vep of Sh}
\limsup_{h \to 0^+} \frac{\V_\ep(\J^n_{\tau, \epsilon} \mu) - \V_\ep(\rS_h(\J^n_{\tau, \epsilon} \mu))}{h}  = \int_\Rd \la\nabla V , \nabla (\zeta_\epsilon*\J^n_{\tau, \epsilon}  \mu) \ra d\mathcal{L}^d , \text{ and }\\
&\label{eq:deriv Vk of Sh}
\limsup_{h \to 0^+} \frac{\V_k(\J^n_{\tau, \epsilon} \mu) - \V_k(\rS_h(\J^n_{\tau, \epsilon} \mu))}{h}  = -\int_\Rd \Delta V_k \, d\J^n_{\tau, \epsilon}  \mu .
\end{align}
\end{lem}

Our proof of this lemma relies on two key facts, which we now recall.
First, for any $\nu\in \mathcal{P}_2(\rr^d)$,
\begin{equation}
    \label{eq:nu narrow cont}
\text{the map $h\mapsto \rS_h\nu$ is narrowly continuous;}
\end{equation}that is, 
$h\mapsto \int f \, d\rS_h\nu$ is continuous  for any bounded and continuous function $f$.  This holds since $\rS_h\nu$, by virtue of being the gradient flow of $\S$, is in $AC^2([0,T];\P_2(\Rd))$, hence $h \mapsto \rS_h\nu$ is continuous with respect to $W_2$, which implies narrow continuity.

The second fact we will use is that, for any for any $\nu\in \mathcal{P}_2(\rr^d)$ and $\phi\in C^1_c(\rr^d)$, 
\begin{equation}
\label{eq:fact2}
\int_{\rr^d} \phi \,d\rS_h\nu- \int_{\rr^d} \phi  \, d \nu = -\int_0^h\int_{\rr^d} \la\nabla \phi(y) ,\nabla \rS_t\nu(y)\ra dy\, dt.
\end{equation}
Notice that, at a formal level, the integrand on the left-hand side is exactly $\int_0^h \frac{d}{dt}\rS_t\nu dt$, which, upon using the fact that $\rS_t\nu$ satifies the heat equation, and integrating by parts, yields the desired equality. More rigorously, one may obtain (\ref{eq:fact2}) as a consequence of \cite[Lemma 8.1.2]{ambrosio2008gradient}. And, arguing as in  
 \cite[Example 11.1.9]{ambrosio2008gradient}, we have $\rS_h\nu\in W^{1,1}_{\loc}(\rr^d)$ for a.e. $h>0$ and,
\[
\int_0^T\int_{\rr^d}|\nabla \rS_h\nu| 
= \int_0^T\int_{\rr^d}\frac{|\nabla \rS_h\nu|}{ \rS_h\nu} \rS_h\nu
\leq \left(\int_0^T\int_{\rr^d}\frac{|\nabla \rS_h\nu|^2}{\rS_h\nu}\right)^{1/2}\left(\int_0^T\int_{\rr^d}\rS_h\nu\right)^{1/2}
=\sqrt{T} \left(\int_0^T\int_{\rr^d}\frac{|\nabla \rS_h\nu|^2}{S_h \nu}\right)^{1/2} ,
\]
where the quantity on the right-hand side is finite by equation (11.1.38) of \cite{ambrosio2008gradient}.

With these facts in hand, we now turn to the proof of Lemma  \ref{lem:deriv Eep of Sh}.
\begin{proof}[Proof of Lemma \ref{lem:deriv Eep of Sh}] 
We begin by proving equation (\ref{eq:deriv Eep of Sh}).
For all $h>0$, using the definition of $\E_\ep$ and the commutativity relation (\ref{eq:Sh commutes}), we find,
\begin{align}
\frac{\E_\ep(\J^n_{\tau, \epsilon} \mu) - \E_\ep(\rS_h(\J^n_{\tau, \epsilon} \mu))}{h} &= \frac{1}{2h}  \int_\Rd \frac{\left| \zeta_\epsilon*(\J^n_{\tau, \epsilon} \mu)\right|^2}{\bar{\rho}}d\mathcal{L}^d - \frac{1}{2h}  \int_\Rd \frac{\left| \zeta_\epsilon*(\rS_h(\J^n_{\tau, \epsilon} \mu))\right|^2}{\bar{\rho}}d\mathcal{L}^d \nonumber \\
&=   \int_\Rd \frac{1}{2 \bar{\rho}} \left( \frac{\zeta_\epsilon*(\J^n_{\tau, \epsilon} \mu) -  \zeta_\epsilon*(\rS_h(\J^n_{\tau, \epsilon} \mu) )}{h} \right) \left( \zeta_\epsilon*(\J^n_{\tau, \epsilon} \mu) + \zeta_\epsilon*(\rS_h(\J^n_{\tau, \epsilon} \mu) )\right)d\mathcal{L}^d \nonumber\\
&= \int_\Rd \frac{1}{2 \bar{\rho}} \left( \frac{\zeta_\epsilon*(\J^n_{\tau, \epsilon} \mu) -  \rS_h(\zeta_\epsilon*\J^n_{\tau, \epsilon} \mu) }{h} \right) \left( \zeta_\epsilon*(\J^n_{\tau, \epsilon} \mu) + \zeta_\epsilon*(\rS_h(\J^n_{\tau, \epsilon} \mu) )\right)d\mathcal{L}^d. \label{eq:RHS deriv EepSh}
\end{align}

Recalling that $\zeta_\epsilon*(\J^n_{\tau, \epsilon} \mu)$ is a smooth function, and using that   $\rS_h(\zeta_\epsilon*\J^n_{\tau, \epsilon} \mu)$ satisfies the heat equation in the classical sense, we find,
\begin{equation}
    \label{eq:classical soln heat}
\zeta_\epsilon*(\J^n_{\tau, \epsilon} \mu) -  \rS_h(\zeta_\epsilon*\J^n_{\tau, \epsilon} \mu)=-\int_0^h\frac{d}{dt}\rS_t(\zeta_\epsilon*\J^n_{\tau, \epsilon} \mu)\, dt= -\int_0^h \Delta \rS_t (\zeta_\epsilon*\J^n_{\tau, \epsilon} \mu)\, dt.
\end{equation}
Using this in (\ref{eq:RHS deriv EepSh}), we obtain,
\[
\frac{\E_\ep(\J^n_{\tau, \epsilon} \mu) - \E_\ep(\rS_h(\J^n_{\tau, \epsilon} \mu))}{h} 
 =\int_\Rd \frac{1}{2 \bar{\rho}} \left(\frac{1}{h} \int_0^h -\Delta \rS_t(\zeta_\epsilon*\J^n_{\tau, \epsilon} \mu)\, dt\right)\left( \zeta_\epsilon*(\J^n_{\tau, \epsilon} \mu) + \zeta_\epsilon*(\rS_h(\J^n_{\tau, \epsilon} \mu) )\right)d\mathcal{L}^d.
\] 
Classical elliptic regularity implies that $\|\Delta \rS_t(\zeta_\epsilon*\J^n_{\tau, \epsilon} \mu)\|_{L^\infty(\rr^d)}\leq C_{\ep, \tau, n}$ holds for all $t$. Hence,  the integrand on the right-hand side of the previous line is bounded in $L^1(\rr^d)$, independently of $h$. Thus, upon applying  the dominated convergence theorem to take the limit $h\rightarrow 0+$, we find,
\begin{align*}
\limsup_{h\rightarrow 0+}\frac{\E_\ep(\J^n_{\tau, \epsilon} \mu) - \E_\ep(\rS_h(\J^n_{\tau, \epsilon} \mu))}{h} &=-\int_\Rd \frac{1}{ \bar{\rho}}  \Delta (\zeta_\epsilon*\J^n_{\tau, \epsilon} \mu)\left( \zeta_\epsilon*\J^n_{\tau, \epsilon} \mu \right) d\mathcal{L}^d.
\end{align*}
We have again used that $\rS_t(\zeta_\epsilon*\J^n_{\tau, \epsilon})$ satisfies the heat equation in the classical sense, and is therefore continuous in $t$. This completes the proof of equation (\ref{eq:deriv Eep of Sh}).

Next we establish equation (\ref{eq:deriv Vep of Sh}). 
 For all $h>0$, using the definition of $\V_\epsilon$,   followed by (\ref{eq:Sh commutes}), 
 we obtain,
\begin{align*}
\frac{\V_\ep(\J^n_{\tau, \epsilon} \mu) - \V_\ep(\rS_h(\J^n_{\tau, \epsilon} \mu))}{h} &= \frac{1}{h} \left(  \int_\Rd (\zeta_\epsilon*V) \, d\J^n_{\tau, \epsilon} \mu - \int_\Rd (\zeta_\epsilon*V) \, d\rS_h(\J^n_{\tau, \epsilon} \mu) \right)\\
&= \frac{1}{h}  \int_\Rd V \left( \zeta_\epsilon*\J^n_{\tau, \epsilon} \mu -  \zeta_\epsilon*\rS_h(\J^n_{\tau, \epsilon} \mu) \right)\, d\mathcal{L}^d\\
&= \frac{1}{h}   \int_\Rd V \left(\zeta_\epsilon*\J^n_{\tau, \epsilon} \mu -  \rS_h(\zeta_\epsilon*\J^n_{\tau, \epsilon} \mu) \right)\, d\mathcal{L}^d.
\end{align*}
As in the computation for $\E_\ep$, we now use (\ref{eq:classical soln heat})  to find,
\begin{align*}
\frac{\V_\ep(\J^n_{\tau, \epsilon} \mu) - \V_\ep(\rS_h(\J^n_{\tau, \epsilon} \mu))}{h} &=  -  \int_\Rd V \frac{1}{h}\int_0^h \Delta \rS_t (\zeta_\epsilon*\J^n_{\tau, \epsilon} \mu)\, dt\, d\mathcal{L}^d.
\end{align*}
Assumption (\ref{Vas}) implies $V\in L^1(\rr^d)$, so we can pass to the limit in $h$ (again, as above), and find,
\begin{align*}
\limsup_{h\rightarrow 0+} \frac{\V_\ep(\J^n_{\tau, \epsilon} \mu) - \V_\ep(\rS_h(\J^n_{\tau, \epsilon} \mu))}{h}  = - \int_\Rd V \Delta (\zeta_\epsilon*\J^n_{\tau, \epsilon} \mu)\, d\mathcal{L}^d.
\end{align*}
Integrating by parts yields (\ref{eq:deriv Vep of Sh}).

Finally, we establish (\ref{eq:deriv Vk of Sh}). 
 For all $h>0$, using the definition of $\V_k$,   followed by (\ref{eq:fact2}), and an integration by parts, yields,
\begin{align*}
\frac{\V_k(\J^n_{\tau, \epsilon} \mu) - \V_k(\rS_h(\J^n_{\tau, \epsilon} \mu))}{h} &= \frac{1}{h} \left(  \int_\Rd V_k \, d\J^n_{\tau, \epsilon} \mu - \int_\Rd V_k \, d\rS_h(\J^n_{\tau, \epsilon} \mu) \right)\\
&=  \int_\Rd \frac{1}{h}\int_0^h \la\nabla V_k(x) , \nabla\rS_t\J^n_{\tau, \epsilon} \mu (x,t)\ra dt\, dx\\
&= -  \int_\Rd \frac{1}{h}\int_0^h  \Delta V_k(x) \rS_t\J^n_{\tau, \epsilon} \mu(x,t)\, dt\, dx.
\end{align*}
Since $\|\Delta V_k\|_{L^\infty(\rr^d)}$  is bounded,  we use the dominated convergence theorem, as well as  the narrow continuity of $\rS_t\J^n_{\tau, \epsilon} \mu$ in $t$ (see (\ref{eq:nu narrow cont})), to pass to the limit in $h$ and obtain the desired result.

\end{proof}

Before proceeding to the main result of the section, we estimate the right-hand side of (\ref{eq:deriv Eep of Sh}). Notice that the hypotheses on $\phi$ in the statement are satisfied by $\zeta_\ep*J^n_{\tau, \ep} \mu$, since $J^n_{\tau,\ep} \mu \in D(\E_\ep)$.
\begin{lem}
\label{eq:RHS deriv E of Sh}
Let $\phi\in C^{\infty}(\rr^d) \cap L^1(\rr^d)\cap L^2(\rr^d)$. Then we have,
\begin{equation*}
-\int_{\rr^d} \frac{1}{\bar\rho}(\Delta \phi )(\phi)\, d\mathcal{L}\geq C_{\bar\rho}\|\nabla \phi\|_{L^2(\rr^d)}^2 - C'_{\bar\rho}\|\phi\|_{L^2(\rr^d)}^2,
 \end{equation*}
 where $C_{\bar\rho}$ and $C_{\bar\rho}'$ depend only on $\bar\rho$.
\end{lem}
\begin{proof} Integrating by parts,  using the product rule, and the fact that $\bar\rho$ is bounded uniformly away from zero, we find,
 \begin{align*}
    - \int_{\rr^d} \frac{1}{\bar\rho}(\Delta \phi )(\phi)&=\int_{\rr^d} \la\nabla\phi ,\nabla\left(\frac{1}{\bar\rho}\phi\right)\ra=\int_{\rr^d} \frac{|\nabla\phi|^2}{\bar\rho} + \phi\la\nabla \phi,\nabla\left(\frac{1}{\bar\rho}\right)\ra\\
     & \quad\geq C_{\bar\rho}\int_{\rr^d}|\nabla\phi|^2 - C'_{\bar\rho}\int_{\rr^d}|\nabla \phi||\phi|
     \geq \frac{C_{\bar\rho}}{2} \int_{\rr^d} |\nabla\phi|^2 - C'_{\bar\rho}\int|\phi|^2,
 \end{align*}
where the last estimate follows from the Cauchy-Schwartz inequality, and $C'_{\bar\rho}$ changes from line to line (but depends only on $\bar\rho$). 
\end{proof}

\subsection{Proof of $H^1$-type bound} \label{proofsec}

We now apply the previous lemmas to prove the main result of the section.
\begin{proof}[Proof of Theorem \ref{prop:H1 bd}]
By definition of the minimizing movement scheme (see Definition \ref{minmovdef}), for any $\mu \in D(\F_{\epsilon,k})$,
\begin{align*}
\F_{\epsilon,k}(\J^n_{\tau, \epsilon} \mu) - \F_{\epsilon,k}(\rS_h(\J^n_{\tau, \epsilon} \mu)) \leq \frac{1}{2 \tau} \left[ W_2^2(\rS_h(\J^n_{\tau, \epsilon} \mu),\J^{n-1}_{\tau, \epsilon} \mu) - W_2^2(\J^n_{\tau, \epsilon} \mu,\J^{n-1}_{\tau, \epsilon} \mu) \right] .
\end{align*}
Dividing by $h$, taking the limit as $h \to 0$, and applying the evolution variational inequality characterization of the gradient flow of  $\S$, inequality (\ref{entropyevi}), we obtain,
\begin{align} \label{discretetimeflowinterchange}
\limsup_{h \to 0^+} \frac{\F_{\epsilon,k}(\J^n_{\tau, \epsilon} \mu) - \F_{\epsilon,k}(\rS_h(\J^n_{\tau, \epsilon} \mu))}{h} \leq \frac{1}{2 \tau} \left. \frac{d^+}{dh} W_2^2 (\rS_h(\J^n_{\tau, \epsilon} \mu), \J^{n-1}_{\tau, \epsilon} \mu) \right|_{h =0} \leq \frac{\S(\J^{n-1}_{\tau, \epsilon} \mu) - \S(\J^{n}_{\tau, \epsilon} \mu)}{\tau} .
\end{align}
The quantity on the right hand side will play the role of $-\frac{d}{dt} \S(\rho_\epsilon)$ in the $\tau \to 0$ limit. Thus, in order to obtain (\ref{eq:H1bd thm: conclusion}), we aim to bound it from below by estimating    the left hand side of (\ref{discretetimeflowinterchange}).

Recalling that $\F_{\ep,k} = \E_\ep+\V_\ep +\V_k$ and applying Lemma \ref{lem:deriv Eep of Sh}, we find,
 \begin{align*}
& \limsup_{h \to 0^+} \frac{\F_{\epsilon,k}(\J^n_{\tau, \epsilon} \mu) - \F_{\epsilon,k}(\rS_h(\J^n_{\tau, \epsilon} \mu))}{h}  \\
 &\quad = -
\int_\Rd \frac{1}{ \bar{\rho}}  \Delta (\zeta_\epsilon*\J^n_{\tau, \epsilon} \mu)\left( \zeta_\epsilon*\J^n_{\tau, \epsilon} \mu \right)\, d\mathcal{L}^d+\int_\Rd \la\nabla V ,\nabla (\zeta_\epsilon*\J^n_{\tau, \epsilon}  \mu) \ra d\mathcal{L}^d  -\int_\Rd \Delta V_k\, d\J^n_{\tau, \epsilon}  \mu .
\end{align*}
Combining  this with (\ref{discretetimeflowinterchange}), and  and summing over $n$, we obtain,
\begin{align*}
\frac{S(\J^{0}_{\tau, \epsilon} \mu) - S(\J^{n}_{\tau, \epsilon} \mu)}{\tau} &= \sum_{i=1}^n \frac{S(\J^{i-1}_{\tau, \epsilon} \mu) - S(\J^{i}_{\tau, \epsilon} \mu)}{\tau} \\
&\geq  \sum_{i=1}^n  -
\int_\Rd \frac{1}{ \bar{\rho}}  \Delta (\zeta_\epsilon*\J^i_{\tau, \epsilon} \mu)\left( \zeta_\epsilon*\J^i_{\tau, \epsilon} \mu \right)\, d\mathcal{L}^d +\int_\Rd \la \nabla V , \nabla (\zeta_\epsilon*\J^n_{\tau, \epsilon}  \mu) \ra d\mathcal{L}^d -\int_\Rd \Delta  V_k\, d\J^i_{\tau, \epsilon}  \mu .
\end{align*}
Take $\tau = T/n$, and let $\mu_{\tau,\epsilon}(t)$ denote the piecewise constant interpolation of the minimizing movement scheme $\J^n_{\tau, \epsilon} \mu$; see equation (\ref{pwconstinterp}). Then the above line implies,
\begin{align} \label{pretautozeroeqn}
& \S(\mu_{\tau,\epsilon}(0)) - \S(\mu_{\tau,\epsilon}(T)) \\
& \geq  \int_0^{T}  \int_\Rd -\frac{1}{ \bar{\rho}}  \Delta (\zeta_\epsilon*\mu_{\tau,\epsilon}(s))\left( \zeta_\epsilon*\mu_{\tau,\epsilon}(s) \right)+ \la\nabla V, \nabla (\zeta_\epsilon*\mu_{\tau,\epsilon}(s)) \ra d\mathcal{L}^d \, ds 
 - \int_0^{T}\int_\Rd \Delta V_k\,  d\mu_{\tau,\epsilon}(s) \, ds.\nonumber
\end{align}

We  consider the right-hand side. The first term on the right-hand side is the most important one, since this is where the derivative we seek to estimate will come from. First, we note, using  the definition of $\E_\ep$, the properties of $\bar\rho$,  the fact that the energy $\F_{\ep,k}$ decreases along  the minimizing movements scheme (see inequality (\ref{energydecminmov})), and the fact that the minimizing movements scheme is initialized at $\rho_\epsilon(0)$,
\begin{align}
    \label{eq:L2 bd zeta*piecewise}
\|\zeta_\ep*\mu_{\tau,\epsilon}(s)\|^2_{L^2(\rr^d)} &\leq  \|\bar{\rho}^{-1}\|_{L^\infty(\rr^d)} \E_\ep(\mu_{\tau,\epsilon}(s))  \leq  \|\bar{\rho}^{-1}\|_{L^\infty(\Rd)} \left( \F_{\ep,k}(\mu_{\tau,\epsilon}(s))  + \|V\|_{L^\infty(\Rd)} \right)\\
& \leq   \|\bar{\rho}^{-1}\|_{L^\infty(\Rd)} \left(  \F_{\ep,k}(\rho_{\epsilon}(0))+ \|V\|_{L^\infty(\Rd)} \right)<+\infty. \nonumber
\end{align}
Thus, for each fixed $s$, we may apply Lemma \ref{eq:RHS deriv E of Sh} to find,
\[
\int_\Rd -\frac{1}{ \bar{\rho}}  \Delta (\zeta_\epsilon*\mu_{\tau,\epsilon}(s))\left( \zeta_\epsilon*\mu_{\tau,\epsilon}(s) \right)\, d\mathcal{L}^d \geq C_{\bar\rho}\|\nabla \zeta_\ep*\mu_{\tau,\epsilon}(s)\|_{L^2(\rr^d)}^2- C'_{\bar\rho}\| \zeta_\ep*\mu_{\tau,\epsilon}(s)\|_{L^2(\rr^d)}^2.
\]
Using (\ref{eq:L2 bd zeta*piecewise}) to bound the second term on the right-hand side of the previous line from below, and integrating in time, we find,
\[
\int_0^T \int_\Rd -\frac{1}{ \bar{\rho}}  \Delta (\zeta_\epsilon*\mu_{\tau,\epsilon}(s))\left( \zeta_\epsilon*\mu_{\tau,\epsilon}(s) \right)\, d\mathcal{L}^d\, ds \geq  C_{\bar\rho}\int_0^T\|\nabla \zeta_\ep*\mu_{\tau,\epsilon}(s)\|_{L^2(\rr^d)}^2\, ds - TC'_{\bar\rho} \left(\F_{\ep,k}(\rho_{\epsilon}(0)) + \|V\|_{L^\infty(\Rd)} \right) .
\]
(Here $C_{\bar\rho}'$ is allowed to change from line to line, but only depends on $\bar\rho$.) 

Next, we apply the Cauchy-Schwartz inequality to the second term on the right-hand side of (\ref{pretautozeroeqn}) to obtain,
\[
\int_0^{T}\int_\Rd \la \nabla V , \nabla (\zeta_\epsilon*\mu_{\tau,\epsilon}(s)) \ra d\mathcal{L}^d\, ds\geq -\frac{C_{\bar\rho}}{2}\int_0^T\|\nabla \zeta_\ep*\mu_{\tau,\epsilon}(s)\|_{L^2(\rr^d)}^2\, ds - C'_{\bar\rho}T\|\nabla V\|_{L^{2}(\rr^d)}.
\]
Finally, for third term on the right-hand side  of (\ref{pretautozeroeqn}), we bound it from below simply by $T\|\Delta V_k\|_{L^\infty(\rr^d)}$, which is finite by assumption. Using this, along with the two previous estimates, we find,
\begin{align} \label{pretautozeroeqn2}
&\S(\mu_{\tau,\epsilon}(0)) - \S(\mu_{\tau,\epsilon}(T))  \geq  \frac{C_{\bar\rho}}{2}\int_0^T\|\nabla \zeta_\ep*\mu_{\tau,\epsilon}(s)\|_{L^2(\rr^d)}^2\, ds - T C'_{\bar\rho} \left( \F_{\ep,k}(\rho_{\epsilon}(0))+ C_V\right) , \\ 
&\text{where } C_V =  \|V\|_{L^\infty(\Rd)}+  \|\nabla V\|_{L^{2}(\rr^d)}  +\|\Delta V_k\|_{L^\infty(\rr^d)} . \nonumber
\end{align}

We now aim to  send $n \to +\infty$ in inequality (\ref{pretautozeroeqn2}), using the fact that $\mu_{\tau, \epsilon}(t) \to \rho_\epsilon(t)$ narrowly for all $t\geq0$; see Theorem \ref{minmovthm}. Note that, for any $f \in L^2(\Rd)$ and $s \in [0,T]$,
\begin{align*}
  \int_\Rd f  \grad  (\zeta_\epsilon *\mu_{\tau,\epsilon}(s))  
 = -  \int_\Rd (\grad \zeta_\epsilon *  f  )    \mu_{\tau,\epsilon}(s)     
\xrightarrow{n \to +\infty} -  \int_\Rd (\grad \zeta_\epsilon  *  f  )    \rho_\epsilon(s)    =\int_\Rd f  \grad  (\zeta_\epsilon * \rho_\epsilon(s))   . \end{align*}
Thus, $ \grad  (\zeta_\epsilon *\mu_{\tau,\epsilon})(s)   \to \grad  (\zeta_\epsilon *\rho_{\epsilon})(s)  $ weakly in $L^2(\Rd)$ for all $s \in [0,T]$.  By the lower semicontinuity of the $L^2(\Rd)$ norm with respect to weak convergence and Fatou's lemma, sending $n \to +\infty$ in inequality (\ref{pretautozeroeqn2}) yields, 
\begin{align}
    \label{pretautozeroeqn3}
\limsup_{n\rightarrow \infty} \S(\mu_{\tau,\epsilon}(0)) - \S(\mu_{\tau,\epsilon}(T)) \geq \frac{C_{\bar\rho}}{2}&\int_0^T\|\nabla \zeta_\ep*\rho_\ep(s)\|_{L^2(\rr^d)}^2\, ds - T\left(C'_{\bar\rho} \F_{\ep,k}(\rho_{\epsilon}(0))+ C_V \right). 
\end{align}

For the left hand side of (\ref{pretautozeroeqn3}), note that the choice of initial data for the minimizing movement scheme ensures $\S(\mu_{\tau, \epsilon}(0)) = \rho_\epsilon(0)$ for all $\tau >0$ and, by the lower semicontinuity of the entropy with respect to narrow convergence  \cite[Remark 9.3.8]{ambrosio2008gradient}, $\limsup_{n \to \infty} - \S(\mu_{\tau,\epsilon}(T)) \leq - \S(\rho_\epsilon(T))$. Thus, sending $n \to +\infty$ on the left hand side of (\ref{pretautozeroeqn3}), we estimate,
\begin{align} \label{STinequality}
\limsup_{n \to +\infty} \S(\mu_{\tau,\epsilon}(0)) - \S(\mu_{\tau,\epsilon}(T))  \leq \S(\rho_\epsilon(0)) - \S(\rho_\epsilon(T)) .  
\end{align}
Finally, using a Carleman-type estimate \cite[Lemma 4.1]{carrillo2016convergence} to bound the entropy below by a constant plus the second moment and applying Proposition \ref{secondmomentbound} to bound the second moment, we obtain,
\begin{align}  \label{pretauLHS}
\limsup_{n \to +\infty} \S(\mu_{\tau,\epsilon}(0)) - \S(\mu_{\tau,\epsilon}(T))  &\leq  \S(\rho_\epsilon(0)) + (2 \pi)^{d/2} +  M_2(\rho_\epsilon(T)) \nonumber \\
&\leq \S(\rho_\epsilon(0)) + (2 \pi)^{d/2} +(1+T e^T)  \left( M_2(\rho_\epsilon(0)) + \F_{\epsilon,k}(\rho_\epsilon(0)) \right).
\end{align}

Thus, combining inequalities (\ref{pretautozeroeqn3}) and (\ref{pretauLHS}), we obtain
\begin{align*}
 &\S(\rho_\epsilon(0)) + \sqrt{2 \pi} +(1+T e^T)  \left( M_2(\rho_\epsilon(0)) + \F_{\epsilon,k}(\rho_\epsilon(0)) \right)  \geq \frac{C_{\bar\rho}}{2}\int_0^T\|\nabla \zeta_\ep*\rho_\ep(s)\|_{L^2(\rr^d)}^2\, ds - T\left(C'_\rho \F_{\ep,k}(\rho_{\epsilon}(0))+ C_V \right).
\end{align*}
 Rearranging then gives the result.
\end{proof}

\section{Convergence of the energies $\F_{\ep,k}$ and gradient flows as $\ep \to 0$} \label{sec:Gamma conv}
 We now apply the properties of the energy $\F_{\epsilon,k}$ and its gradient flows developed in the previous sections to study the behavior of minimizers and gradient flows as $\epsilon \to 0$ for fixed $k \in \mathbb{N}$. 
In Subsection \ref{Gammaenergiesep}, we begin by proving the $\Gamma$-convergence of the energies $\F_{\ep,k}$  to the energy $\F$.  Next, in Subsection \ref{eptozeroGFsec}, we  analyze the convergence of the gradient flows of $\F_{\epsilon,k}$ as $\epsilon \to 0$ with ``well-prepared'' initial data (bounded entropy and energy). Due to the fact that we only suppose $\bar{\rho}$ is log-concave on $\Omega$, and not on all of $\rr^n$, we are not able to conclude that the limit is a gradient flow of $\F_k$, which we recall is defined by
 \begin{align*}
\F_k(\rho) &= \E(\rho) + \V(\rho) + \V_k(\rho) .
\end{align*}
Instead we merely conclude it is an ``almost'' curve of maximal slope of $\F_k$, see Definition \ref{almostcms}. Nevertheless, this weaker notion is still sufficient for our main convergence result, Theorem \ref{newmaintheorem}, studying the limits as $\ep \to 0$, $k \to +\infty$.

\subsection{$\Gamma$-convergence of the energies and convergence of minimizers} \label{Gammaenergiesep}
 We now prove the $\Gamma$-convergence of the energies $\F_{\ep,k}$  to the energy $\F_k$, in the sense of  Definition \ref{def:gamma conv def}.
\begin{thm}[$\Gamma$-convergence of $\F_{\ep,k}$] 
\label{newenergygammaprop} 
Assume (\ref{targetas}), (\ref{mollifieras}), (\ref{Vas}),  and (\ref{Vkas}) hold. Fix $k\in \mathbb{N}$. Then the energies $\E_\ep+\V_\ep$ $\Gamma$-converge to $\E + \V$ and the energies $\F_{\ep,k}$ $\Gamma$-converge to $\F_k$ as $\ep\rightarrow 0$. In particular,  for any $\mu \in \P_2(\Rd)$, $\lim_{\ep \to 0} \E_{\ep}(\mu) + \V_\ep(\mu) = \E(\mu) + \V(\mu)$.
\end{thm}

\begin{proof}
 We begin with the proof of (\ref{eq:gamma conv def}). We first consider the energies $\E_\ep + \V_\ep$. Let $\rho_\epsilon$  narrowly converge to $\rho$.  Lemma \ref{weakst convergence mollified sequence} implies,
\begin{equation}
\label{eq:zeta* conv}
\zeta_\ep*\rho_\epsilon \text{ narrowly converges to } \rho.
\end{equation}
By definition of $\E_\ep$ and $\E$, we have, as in  (\ref{EepstoE}), $\E_\ep(\rho_\ep) =\E(\zeta_\ep*\rho_\ep)$. 
Taking $\liminf_{\ep\rightarrow 0}$ and using the lower semicontinuity of $\E$ with respect to narrow convergence, as well as (\ref{eq:zeta* conv}), we obtain,
\[
\liminf_{\ep\rightarrow 0}\E_\ep(\rho_\ep) = \liminf_{\ep\rightarrow 0} \E(\zeta_\ep*\rho_\ep) \geq \E(\rho). 
\]
For the $\V_\ep$ term, we first use the properties of convolution, followed by the assumption  $V\in C_b(\rr^d)$ and (\ref{eq:zeta* conv}), to find,
\begin{equation}
\label{eq:gamma conv Vep}
    \int_{\rr^d}(\zeta_\ep*V)\, d\rho_\ep = \int_{\rr^d}V (\zeta_\ep*\rho_\ep)\, d\mathcal{L}^d \rightarrow \int_{\rr^d}V \, d\rho. 
\end{equation}
This concludes the proof of (\ref{eq:gamma conv def}) for $\E_\ep + \V_\ep$. Since $\V_k$ is lower semicontinous, this likewise implies  (\ref{eq:gamma conv def})  holds for $\F_{\ep,k}$

Now we establish (\ref{newenergygammapart1}). Let $\rho\in \mathcal{P}(\rr^d)$. Taking $\rho_\ep = \rho$ for all $\epsilon >0$ in (\ref{eq:gamma conv Vep}),   we find that it suffices to prove $\limsup_{\ep \to 0}\E_\ep(\rho) \leq \E(\rho)$. Without loss of generality, we assume $\rho$ is such that $\E(\rho)< +\infty$, otherwise, the desired inequality is trivially true. Together with the definition of $\E$  and our assumption (\ref{targetas}) that $\bar\rho$ is bounded uniformly above and below, we deduce $\rho\in L^2(\rr^d)$. 
We use the definition of $\E_\ep$ to find,
\begin{align*}
2 \E_\ep(\rho) &= \int_{\rr^d}|\zeta_\ep*\rho|^2(x)\frac{1}{\bar\rho(x)}\, dx =  \int_{\rr^d}\left|\int_{\rr^d}\zeta_\ep(x-y)\rho(y)\, dy\right|^2\frac{1}{\bar\rho(x)}\, dx.
\end{align*}
Next we use Jensen's inequality, followed by Fubini's Theorem,  to obtain,
\begin{align}
\label{eq:Eep upp bd}
2 \E_\ep(\rho)
&\leq \int_{\rr^d}\int_{\rr^d}\zeta_\ep(x-y)\rho(y)^2 \frac{1}{\bar\rho(x)} \, dy \, dx = \int_{\rr^d}\left(\zeta_\ep*\frac{1}{\bar\rho}\right)(y)\rho^2(y)\, dy .
\end{align}
We shall now prove:
\begin{equation}
\label{eq:lim Eep}
\lim_{\ep\rightarrow 0} \left|\int_{\rr^d}\left(\zeta_\ep*\frac{1}{\bar\rho}\right)(y)\rho^2(y) \, dy -2 \E(\rho)\right| =0.
\end{equation} 
Together with (\ref{eq:Eep upp bd}), this  will yield the desired result.

In order to establish (\ref{eq:lim Eep}), we first use the definition of $\E(\rho)$ 
to write,
\begin{align}
\left|\int_{\rr^d}\left(\zeta_\ep*\frac{1}{\bar\rho}\right)\rho^2\, d\Ld -2 \E(\rho)\right|
& = \left|\int_{\rr^d}\left(\zeta_\ep*\frac{1}{\bar\rho}\right)\rho^2\, d\Ld -\int_{\rr^d}\frac{\rho^2}{\bar\rho}\, d\Ld  \right| \leq \int_{\rr^d}\left|\left(\zeta_\ep*\frac{1}{\bar\rho}\right) - \frac{1}{\bar\rho}\right|\rho^2 \, d\Ld. \label{eq:zeta bar rho on all rn}
\end{align}
Fix $\delta>0$ arbitrary.   Since $\rho\in L^2(\rr^d)$, there exists $R>0$   such that  $\int_{B_R^c} \rho^2  \leq \delta$.  Moreover, since $1/\bar\rho$ is uniformly bounded (see Assumption (\ref{targetas})), 
\[
\int_{B_R^c}\left|\left(\zeta_\ep*\frac{1}{\bar\rho}\right)  - \frac{1}{\bar\rho }\right|\rho^2 \, d\Ld \leq C\int_{B_R^c}\rho^2  \leq C\delta,
\] 
where $C$ is independent of $\delta$ and $\ep$. Now, splitting the integral in (\ref{eq:zeta bar rho on all rn}) into integrals over $B_R$ and $B_R^c$,   we find,
\begin{align*}
\left|\int_{\rr^d}\left(\zeta_\ep*\frac{1}{\bar\rho}\right) \rho^2 \, d\Ld  -2\E(\rho)\right|
&\leq \int_{B_R}\left|\left(\zeta_\ep*\frac{1}{\bar\rho}\right)  - \frac{1}{\bar\rho }\right|\rho^2\, d\Ld  + C\delta  \leq  \left\|\left(\zeta_\ep*\frac{1}{\bar\rho}\right) - \frac{1}{\bar\rho}\right\|_{L^\infty(B_R)} \|\rho\|_{L^2(\Rd)} + C\delta.
\end{align*}
Since $1/\bar\rho$ is continuous, $\zeta_\ep*\frac{1}{\bar\rho}$ converges to $1/\bar\rho$  uniformly on compact subsets of $\rr^d$ as $\ep\rightarrow 0$. In particular, we may choose $\ep>0$ small enough so that the value of the right-hand side of the previous line is no larger than $\delta$. 
Since  $\delta>0$ was arbitrary, this completes the proof of estimate (\ref{eq:lim Eep}) and therefore of the theorem.
\end{proof}

 \subsection{Convergence of the gradient flows} \label{eptozeroGFsec}
We seek to identify the limit of gradient flows of $\F_{\ep,k}$ as $\epsilon \to 0$. Heuristically, one may expect that they converge to a weak notion of gradient flow of $\F_k$, but in the absence of a log-concavity assumption on $\bar{\rho}$, the subdifferential of $\F_k$ lacks appropriate regularity for even a weak notion of gradient flow to be well-defined. However, inspired by Serfaty's approach for studying $\Gamma$-convergence of gradient flows, we are still able to identify a limit and show that it nearly satisfies the definition of a curve of maximal slope of $\F_k$. In order to simplify our exposition, we will call the limit an ``almost'' curve of maximal slope of $\F_k$. 

\begin{defn}[``almost'' curve of maximal slope of $\F_k$] \label{almostcms}
A curve $\rho_k \in AC^2([0,T];\P_2(\Rd))$ is an \emph{``almost'' curve of maximal slope} of $\F_k$ if it satisfies,
\begin{align} \frac12 \int_0^t |\rho_k'|^2(r) dr + \frac12 \int_0^t \int_{\mathbb{R}^d} |\bEta_k(r)|^2 d \rho_k(r) dr \leq \F_k(\rho_k(0)) - \F_k(\rho_k(t)) \text{ for all }t\in[0,T],  \label{almostcurvemaxslope1a}
\end{align}
where, for almost every $t \in [0,T]$, 
\begin{align}
\rho_k^2(t)\in  W^{1,1}(\rr^d) \text { and } \bEta_k(t) \in L^2(\rho_k(t)) \text{ satisfies }  \bEta_k \rho_k  = \frac{\bar{\rho}}{2}\nabla \left(\frac{\rho_k^2}{\bar \rho ^2}\right) + \rho_k \nabla (V+ V_k)\label{almostcurvemaxslope1b}.
\end{align}
\end{defn}

We emphasize that,  if $\bar{\rho}$ were log-concave on all of $\mathbb{R}^d$,     $\rho_k \mapsto \int_{\mathbb{R}^d} |\bEta_k |^2 d \rho_k $ would be a strong upper gradient for $\F_k$ and any $\rho_k$ satisfying Definition \ref{almostcms} would   be a true  curve of maximal slope of $\F_k$.

Our approach proceeds as follows. Inspired by Serfaty's framework for $\Gamma$-convergence of gradient flows, in Subsection \ref{sec:conv of met slopes}, we first prove Proposition \ref{prop:estimate for metric slopes}, which gives a weak notion of lower semicontinuity for the metric slopes  along a sequence of gradient flows $\rho_\ep(t)$: we show,
\begin{align} \label{motivate52}
 \liminf_{\ep\rightarrow 0} |\partial \F_{\ep,k}|^2(\rho_\ep(t))\geq \| \bEta_k(t)\|_{L^2(\rho_k(t))}^2,
 \end{align}
 where $\bEta_k$ is as in Definition \ref{almostcms}. 
Next, in Subsection \ref{GFconvsubsection}, we apply this to prove Proposition  \ref{prop:epstozerofixedk} on convergence of the gradient flows for ``well-prepared'' initial data. It is in this result that we employ the key estimate that we established in Theorem \ref{prop:H1 bd}. 

\subsubsection{Limit of metric slopes} \label{sec:conv of met slopes}We begin by identifying sufficient conditions under which the limiting behavior of the metric slopes (\ref{motivate52}) holds.

\begin{prop}[limiting behavior of metric slopes] 
\label{prop:estimate for metric slopes} Assume  (\ref{targetas}), (\ref{mollifieras}), (\ref{Vas}),  and (\ref{Vkas}) hold. 
Fix $k \in \mathbb{N}$. Consider a sequence $\rho_\epsilon$ in $\mathcal{P}_2(\mathbb{R}^d)$ satisfying, 
\begin{align}
\label{eq:new prop assump Eep bd}
& \sup_{\ep>0} \F_{\ep,k}(\rho_\ep) < +\infty, \\
& \liminf_{\ep \to 0}  \|\grad \zeta_\ep *\rho_\ep\|_{L^2(\rr^d )} < +\infty, \text{ and}  \label{eq:H1bound}\\
\label{eq:LHS metric slopes estimate}
&\liminf_{\ep \to 0} \int \left| \nabla \zeta_\ep*\left(\frac{\zeta_\ep*\rho_\ep}{\bar{\rho}} \right)\right|^2\, d\rho_\ep <+\infty.
\end{align} 
In addition, suppose there exists $\rho_k \in \mathcal{P}(\rr^d)$ such that $\rho_\epsilon$ narrowly converges to $\rho_k$. 
Then   $\rho_k^2\in W^{1,1}(\rr^d)$, and there exists  $\bEta_k\in L^{2}(\rho_k)$, with, 
\begin{equation}
\label{eq:to show def eta}
\bEta_k \rho_k  = \frac{\bar{\rho}}{2}\nabla \left(\frac{\rho_k^2}{\bar \rho ^2}\right) + \rho_k \nabla (V+ V_k),
\end{equation}
 and such that,
\begin{equation}
\label{eq:new metric slopes}
\liminf_{\ep \to 0} \int \left| \nabla \zeta_\ep*\left(\frac{\zeta_\ep*\rho_\ep}{\bar{\rho}}\right) + \nabla(\zeta_\ep*V)+\nabla V_k \right|^2\, d\rho_\ep \geq \int |\bEta_k|^2
d \rho_k  .
\end{equation}
\end{prop}

The remainder of this subsection is devoted to the proof of Proposition \ref{prop:estimate for metric slopes}.
We begin with a preliminary lemma, showing that, under the assumptions of Proposition \ref{prop:estimate for metric slopes}, we may upgrade the convergence of  $\zeta_\ep *\rho_\ep$ to $\rho$ from narrow convergence to convergence in $L^2_\loc(\Rd)$.
\begin{lem}[upgraded convergence of $\zeta_\ep*\rho_\ep$]
\label{lem:improve conv}
Assume  (\ref{targetas}), (\ref{mollifieras}), (\ref{Vas}),  and (\ref{Vkas}) hold. Fix $k \in \mathbb{N}$. Consider any sequence $\rho_\epsilon$ in $\mathcal{P}(\mathbb{R}^d)$ and  $\rho_k\in \mathcal{P}(\rr^d)$ such that  $\rho_\epsilon$ narrowly converges to $\rho_k$ and (\ref{eq:new prop assump Eep bd}) and (\ref{eq:H1bound}) are satisfied. Then $\rho_k \in L^2(\rr^d)$,
 and there exists a subsequence (still denoted $\rho_\ep$) along which we have, 
 \begin{align}
\label{eq:zetaep rhoep in H1} &\sup_{\ep >0} ||\zeta_\ep*\rho_\epsilon||_{H^1(\rr^d)} <+\infty  \text{ and,}\\
\label{eq:lem conclude zeta rho ep conv in L2}
&\zeta_\ep*\rho_\ep \text{ converges to $\rho_k$ in $L^2_{loc}(\rr^d)$.}
\end{align}
\end{lem}
\begin{proof}[Proof of Lemma \ref{lem:improve conv}]
By assumption (\ref{eq:new prop assump Eep bd}) and  the definition of $\F_{\ep,k}$, we find,
\begin{equation}
\label{eq:bd on zeta*rho ep L^2}
+\infty > \sup_{\ep>0} \F_{\ep,k}(\rho_\ep)+ \|V\|_{L^\infty(\Rd)}\geq \sup_{\ep>0} \E_\ep(\rho_\ep) = \sup_{\ep>0} \frac{1}{2} \int_{\rr^d}\frac{|\zeta_\ep*\rho|^2 }{\bar{\rho}}  \geq \frac{1}{2 \|\bar{\rho}\|_{L^\infty(\Rd)}} \sup_{\ep>0} \int_{\rr^d} |\zeta_\ep*\rho|^2 .
\end{equation}
Similarly, since Theorem \ref{newenergygammaprop} ensures the $\Gamma$-convergence of $\F_{\ep,k}$ to $\F_k$, statement (\ref{newenergygammapart1}) in Definition \ref{def:gamma conv def} of $\Gamma$-convergence ensures,
\[
+\infty >\sup_{\ep>0} \F_{\ep,k}(\rho_\ep)+\|V\|_{L^\infty(\Rd)} \geq \F_k(\rho_k)+\|V\|_{L^\infty(\Rd)} \geq \E(\rho_k)   \geq \frac{1}{2 \|\bar{\rho}\|_{L^\infty(\Rd)}} \int_{\rr^d} \rho_k^2\, dx,
\] 
so  $\rho_k \in L^2(\rr^d)$.

Combining assumption (\ref{eq:H1bound}) with the estimate (\ref{eq:bd on zeta*rho ep L^2}) we find  that, up to a subsequence, (\ref{eq:zetaep rhoep in H1}) holds.  
Therefore, by the Rellich-Kondrachov embedding theorem, we find that, up to  another subsequence,  $\zeta_\ep*\rho_\ep$ converges in $L^2_{loc}(\rr^d)$. On the other hand, Lemma \ref{weakst convergence mollified sequence},  implies that $\zeta_\ep*\rho_\ep$ narrowly converges to  $\rho_k$. The uniqueness of limits therefore implies  (\ref{eq:lem conclude zeta rho ep conv in L2}).
\end{proof}

A key step in  studying the limiting behavior 
of the metric slopes of $\F_{\epsilon,k}$, as in Proposition \ref{prop:estimate for metric slopes}, is to identify the weak limit of $ \nabla \zeta_\ep*\left(\frac{1}{\bar{\rho}}\left(\zeta_\ep*\rho_\ep\right)\right)$ in $L^1(\rho_\epsilon)$. 
With this weak limit in hand,  the desired result  
will then follow from general results due to Ambrosio, Gigli, and Savar\'e on lower semicontinuity of integral functions with varying measures  \cite[Theorem 5.4.4 (ii)]{ambrosio2008gradient}. In the following lemma, we characterize the weak limit. 
\begin{lem}[weak limit of subdifferentials]
\label{sublem:Jep}
Assume  (\ref{targetas}), (\ref{mollifieras}), (\ref{Vas}),  and (\ref{Vkas}) hold. Fix $k \in \mathbb{N}$. Consider any sequence $\rho_\epsilon$ in $\mathcal{P}(\mathbb{R}^d)$ and  $\rho_k \in \mathcal{P}(\rr^d)$ such that  $\rho_\epsilon$ narrowly converges to $\rho_k$ and (\ref{eq:new prop assump Eep bd}), (\ref{eq:H1bound}), and (\ref{eq:LHS metric slopes estimate}) are satisfied. 
 For all $\ep>0$ and $f\in C^\infty_c(\rr^d)$,  define,
\begin{align}\label{Ldef}
L_\ep(f) = \int_{\rr^d} f\left( \nabla \zeta_\ep*\left(\frac{1}{\bar{\rho}}\left(\zeta_\ep*\rho_\ep\right)\right)\right)\, d\rho_\ep \   \text{ and } \  L(f) = \int_{\rr^d}-\frac{1}{2} \nabla \left(\frac{f}{\bar \rho}\right)\rho_k^2\, dx + \int_{\rr^d} f\rho_k^2\nabla\left(\frac{1}{\bar\rho}\right)\, d x .
\end{align}
There exists a subsequence, still denoted by $\ep$, so that, for any $f\in C^\infty_c(\rr^d)$, we have,  
\begin{equation}
    \label{eq:Lep to L}
\lim_{\ep\rightarrow 0} L_\ep(f) =  L(f).
\end{equation}
Furthermore, $L$ is a bounded linear operator on $L^2(\rho_k)$.
\end{lem}

\begin{proof}
By Lemma \ref{lem:improve conv},  we may choose a subsequence, still denoted $\rho_\ep$, along which (\ref{eq:zetaep rhoep in H1}) and (\ref{eq:lem conclude zeta rho ep conv in L2}) hold.

In order to characterize $\lim_{\ep \to 0} L_\ep(f)$, we begin by breaking up the expression for $L_\ep(f)$ into two terms, which we will estimate separately. Using the definition of $L_\ep$ and properties of convolution, we  find that, for any $f\in C^\infty_c(\rr^d)$,  
\begin{align}
L_\ep(f) &=  \int_{\rr^d} f\left(\zeta_\ep* \left(\frac{1}{\bar{\rho}}\nabla\left(\zeta_\ep*  \rho_\ep\right)\right)\right)\, d\rho_\ep + \int_{\rr^d} f\left(\zeta_\ep* \left(\nabla \left(\frac{1}{\bar{\rho}}\right)\left(\zeta_\ep*\rho_\ep\right)\right)\right)\, d\rho_\ep \nonumber \\
&=  \int_{\rr^d} \left((f \rho_\ep )*\zeta_\ep \right) \left(\frac{1}{\bar{\rho}}\nabla\left(\zeta_\ep*  \rho_\ep\right)\right)\, d \mathcal{L}^d + \int_{\rr^d} \left((f \rho_\ep)*\zeta_\ep\right) \left(\left( \nabla\frac{1}{\bar{\rho}}\right)\left(\zeta_\ep*\rho_\ep\right)\right)\, d \mathcal{L}^d \nonumber \\
&=:I_\ep(f) +J_\ep(f).
\label{eq:Iep Jep def}
\end{align}

We begin by showing,
\begin{equation}
    \label{eq:newlimJep}
    \lim_{\ep\rightarrow 0} J_\ep(f) = \int_{\rr^d} f\rho_k^2\nabla\left(\frac{1}{\bar \rho}\right)\, d \mathcal{L}^d.
\end{equation}
To this end, we apply Lemma \ref{move mollifier prop}, with $\sigma  = \zeta_\ep*\rho_\ep \nabla
\left(\frac{1}{\bar \rho}\right)\,d \mathcal{L}^d$ and $\nu = \rho_\ep$ to find, for  $C_\zeta>0$ as in assumption (\ref{mollifieras}), there exist $p, L_f>0,$ and $C_{\bar{\rho}}>0$ so that,
\begin{align*}
\left|J_\ep(f) - \int_{\rr^d} f (\zeta_\ep*\rho_\ep)^2 \nabla \left(\frac{1}{\bar \rho}\right)\, d \mathcal{L}^d \right|&\leq \ep^p L_f\left(\int_{\rr^d} (\zeta_\ep*\rho_\ep)^2 |\nabla (1/\bar{\rho})| \, d \mathcal{L}^d +C_\zeta \int (\zeta_\ep*\rho_\ep)\left|\nabla \frac{1}{\bar \rho}\right|\,d \mathcal{L}^d \right)\\
& \leq \ep^p L_f C_{\bar{\rho}}\left(\int_{\rr^d} (\zeta_\ep*\rho_\ep)^2\, d \mathcal{L}^d +C_\zeta \right).
\end{align*}
By  (\ref{eq:zetaep rhoep in H1}) of Lemma \ref{lem:improve conv},  the right-hand side converges to 0 as $\ep\rightarrow 0$, which implies that (\ref{eq:newlimJep}) holds.

Next, we consider $\lim_{\ep \to 0} I_\ep(f)$. For any $f\in C^\infty_c(\rr^d)$, define,
\[
\tilde{I}_\ep(f) = \frac{1}{2}\int_{\rr^d} \frac{f}{\bar \rho}\nabla \left(\left(\zeta_\ep *\rho_\ep\right)^2\right)\, d\Ld.
\]
Note that the $L^2$ convergence of $\zeta_\ep * \rho_\ep$ to $\rho_k$ established in  (\ref{eq:lem conclude zeta rho ep conv in L2}) of Lemma \ref{lem:improve conv} ensures that, for any $f\in C^\infty_c(\rr^d)$,
\begin{align*}
 & \lim_{\ep\rightarrow 0} \tilde{I}_\ep(f) = \lim_{\ep\rightarrow 0}  -\frac{1}{2}\int_{\rr^d} \nabla \left(\frac{f}{\bar{\rho}}\right)(\zeta_\ep * \rho_\ep)^2 \, d\Ld = -\frac{1}{2}\int_{\rr^d} \nabla \left(\frac{f}{\bar{\rho}}\right)\rho_k^2\, d\Ld.
\end{align*}
Thus, to complete our proof that $\lim_{\ep \to 0} L_\ep(f) = L(f)$, it suffices to prove that, for any  $f\in C^\infty_c(\rr^d)$,
\begin{equation}
    \label{eq:new suffice}
    \lim_{\ep\rightarrow 0}|I_\ep(f)-\tilde{I}_\ep(f)|=0.
\end{equation}

Using the definitions of $I_\ep(f)$ and $\tilde{I}_\ep(f)$, followed by some rearranging, we find,
\begin{align*}
|I_\ep(f)-\tilde{I}_\ep(f)| & = \left|\int_{\rr^d} \left((f \rho_\ep )*\zeta_\ep \right)(x) \frac{1}{\bar{\rho} (x)}\nabla(\zeta_\ep*  \rho_\ep)(x)\, dx  -  \int_{\rr^d} f(x)(\zeta_\ep*\rho_\ep)(x) \frac{1}{\bar\rho (x)}\nabla(\zeta_\ep*\rho_\ep)(x)\, dx \right|, \\
&= \left| \int_{\rr^d}\left[\left((f \rho_\ep )*\zeta_\ep \right)(x) - f(x)(\zeta_\ep*\rho_\ep)(x) \right] \frac{1}{\bar{\rho} (x)}\nabla(\zeta_\ep*  \rho_\ep)(x)\, dx \right|.
\end{align*}
We have, for all $x\in\rr^d$,
\[
((f \rho_\ep )*\zeta_\ep)(x)  - f(x)(\zeta_\ep*\rho_\ep)(x) = \int_{\rr^d} (f(y)-f(x))\rho_\ep(y)\zeta_\ep(x-y)\, dy.
\] 
Thus, using this and  Fubini's Theorem we find,
\begin{align*}
|I_\ep(f)-\tilde{I}_\ep(f)| & \leq \iint \left|f(y)-f(x)\right|\rho_\ep(y)\zeta_\ep(x-y)\frac{1}{\bar{\rho }(x)}\left|\nabla(\zeta_\ep*  \rho_\ep)(x)\right|\, dy\, dx.
\end{align*}
Since $\bar{\rho}$ is bounded uniformly from below  and $f\in C^\infty_c(\rr^d)$,
\begin{align}
\label{eq:I-L bd}
|I_\ep(f)-\tilde{I}_\ep(f)| & \leq \frac{\|\nabla f\|_\infty}{\inf \bar{\rho}}\iint \left|x-y\right|\rho_\ep(y)\zeta_\ep(x-y)\left|\nabla(\zeta_\ep*  \rho_\ep)(x)\right|\, dy\, dx.
\end{align}

Next, we claim that there exist $C>0$, $\gamma\in (0,1)$ and $\delta>1$, all depending only on $\zeta$, such that, 
\begin{equation}
    \label{eq:zeta ep property}
\zeta_\ep(x-y)|x-y|\leq 
C\ep^\delta \quad \text{ for } |x-y|>\ep^\gamma.
\end{equation}
Indeed, let $q$ be as in Assumption (\ref{mollifieras}), define $\delta'=q-(d+1)>0$ and $\gamma = \delta'/2(d+\delta')$. The definition of $\zeta_\ep$ and assumption (\ref{mollifieras}) imply,
\[
\zeta_\ep(z)|z|= \zeta\left(\frac{z}{\ep}\right)\frac{|z|}{\ep^d}\leq C |z|^{-(d+1+\delta')}\ep^{d+1+\delta'} |z| \ep^{-d} = C|z|^{-d-\delta'}\ep^{1+\delta'}.
\]
Thus, for $|z|>\ep^{\gamma}$ we obtain,  $\zeta_\ep(z)|z|\leq C \ep^{-(d+\delta')\gamma}\ep^{1+\delta'}= C\ep^{1+\delta'/2}$. The inequality (\ref{eq:zeta ep property}) now follows by taking  $\delta = 1+\delta'/2$.

Thus, breaking up the integral on the right-hand side of (\ref{eq:I-L bd}) into two regions and using (\ref{eq:zeta ep property}), we find,
\begin{align*}
&|I_\ep(f)-\tilde{I}_\ep(f)|  \\
&\quad \leq \frac{\|\nabla f\|_\infty}{\inf \bar{\rho}}\left( \ep^\gamma\iint_{|x-y|<\ep^\gamma} \rho_\ep(y)\zeta_\ep(x-y)\left|\nabla(\zeta_\ep*  \rho_\ep)(x)\right|\, dy\, dx 
+ C\ep^\delta \iint_{|x-y|>\ep^\gamma} \rho_\ep(y)\left|\nabla(\zeta_\ep*  \rho_\ep)(x)\right|\, dy\, dx\right)\\
&\quad \leq \frac{\|\nabla f\|_\infty}{\inf \bar{\rho}} \left( \ep^\gamma\int (\rho_\ep*\zeta_\ep)(x)\left|\nabla(\zeta_\ep*  \rho_\ep)(x)\right|\,  dx 
+ C\ep^\delta \int \left|(\nabla\zeta_\ep*  \rho_\ep)(x)\right|\, dx\right).
\end{align*}
Now we use H\"older's inequality for the first term on the right-hand side, and Young's inequality for the second term to obtain,
\begin{align*}
|I_\ep(f)-\tilde{I}_\ep(f)|  &\leq C_f\left( \ep^\gamma ||\zeta_\ep*  \rho_\ep||_{L^2(\rr^d)}||\nabla \zeta_\ep*  \rho_\ep||_{L^2(\rr^d)} 
+ \ep^\delta ||\nabla \zeta_\ep||_{L^1(\rr^d)}\right).
\end{align*}
To bound the first term on the right-hand side we recall that $\zeta_\ep*  \rho_\ep$ is bounded in $H^1(\rr^d)$ uniformly in $\ep$ (see the estimate (\ref{eq:zetaep rhoep in H1}) from Lemma \ref{lem:improve conv}). For the second term, we note $\ep^\delta ||\nabla \zeta_\ep||_{L^1(\rr^d)} = \ep^{\delta-1}  ||\nabla \zeta||_{L^1(\rr^d)}$. Since $\gamma>0$ and $\delta - 1>0$, this ensures $\lim_{\ep \to 0} |I_\ep(f)-\tilde{I}_\ep(f)|  =0$, which completes the proof that $\lim_{\ep \to 0} L_\ep(f) = L(f)$.
 
It remains to show that $L$ is a bounded linear operator on $L^2(\rho_k)$. We will show that, for any $f\in C^{\infty}_c(\rr^d)$, 
\[
|L(f)|\leq C\|f\|_{L^{2}(\rho_k)}.
\]
Indeed, since $\rho_k \in L^2(\Rd)$, $\rho_k$ is a Radon measure, so $C^1_c(\rr^d)$ is dense in $L^2(\rho_k)$    \cite[Corollary 4.2.2]{bogachev2007measure}, and there exists a unique extension of $L$ to $L^2(\rho_k)$ enjoying the same bound.

Fix arbitrary $f\in C^\infty_c(\rr^d)$. By definition of $L_\ep$ in equation (\ref{Ldef}) and H\"older's inequality, 
\[
|L_\ep(f)|\leq \|f\|_{L^2(\rho_\ep)} \left|\left|\nabla \zeta_\ep*\left(\frac{1}{\bar{\rho}}\left(\zeta_\ep*\rho_\ep\right)\right)\right|\right|_{L^2(\rho_\ep)}.
\]
Thus, by assumption (\ref{eq:LHS metric slopes estimate}), there exists $C>0$ so that,
\[
|L(f) | = \liminf_{\ep \to 0} |L_\ep(f)|\leq  C \liminf_{\ep \to 0}  ||f||_{L^2(\rho_\ep)} = C   ||f||_{L^2(\rho_k)} ,
\]
which gives the result.
\end{proof}

We now apply the previous lemmas to prove our result on the limit of the metric slopes.
\begin{proof}[Proof of Proposition \ref{prop:estimate for metric slopes}]
Choose a subsequence, still denoted by $\rho_\ep$, so that, 
\[ \lim_{\ep \to 0} |\partial \F_{\ep, k}|(\rho_\ep) = \liminf_{\ep \to 0} |\partial \F_{\ep, k}|(\rho_\ep) . \]
 It suffices to show  $\rho_k^2\in W^{1,1}(\rr^d)$, there exists ${\boldsymbol{\eta}_k}\in L^2(\rho)$ satisfying (\ref{eq:to show def eta}), and, up to a further subsequence,
\begin{align}
\label{eq:to show met slope est}
&\lim_{\ep\rightarrow 0 }\int_{\rr^d} f\left( \nabla \zeta_\ep*\left(\frac{1}{\bar{\rho}}\left(\zeta_\ep*\rho_\ep\right)\right)+ \nabla(\zeta_\ep*V)+\nabla V_k\right)\, d\rho_\ep = \int f \bEta_k\, d\rho_k \quad \text{ for all }f\in C^\infty_c(\rr^d).
\end{align}
The estimate (\ref{eq:new metric slopes}) then follows by applying \cite[Theorem 5.4.4 (ii)]{ambrosio2008gradient}, completing the proof.

Notice that, for any $f\in C^\infty_c(\rr^d)$, the fact that $\nabla V$ and $\nabla V_k$ are continuous and Lemma \ref{weakst convergence mollified sequence} ensure,
\begin{align}
    \label{eq:Vs in met slopes}
\lim_{\ep\rightarrow 0 }\int_{\rr^d} f\left( \nabla(\zeta_\ep*V)\right)\, d\rho_\ep &=\lim_{\ep\rightarrow 0 }\int_{\rr^d} \nabla V \left( \zeta_\ep*(f\rho_\ep)\right)\, d\mathcal{L}^d = \int f\nabla V \, d\rho_k, \\
\lim_{\ep\rightarrow 0 }\int_{\rr^d} f\left( \nabla V_k\right)\, d\rho_\ep &= \int f\nabla V_k \, d\rho_k.
\end{align}
Next, we use the definitions of $L_\ep(f)$ and $L(f)$, as well as the convergence of $L_\ep(f)$ to $L(f)$ established in (\ref{eq:Lep to L}) of Lemma \ref{sublem:Jep}. Combining these with the Riesz Representation Theorem on $L^2(\rho_k)$ (which we can apply to the operator $L$ due to, again, Lemma \ref{sublem:Jep}), we find  that there exists $\tilde{\bEta} \in L^2(\rho_k)$ such that,
\begin{equation*}
 \lim_{\ep\rightarrow 0 }\int_{\rr^d} f\left( \nabla \zeta_\ep*\left(\frac{1}{\bar{\rho}}\left(\zeta_\ep*\rho_\ep\right)\right)\right)\, d\rho_\ep =\int_{\rr^d}-\frac{1}{2} \nabla \left(\frac{f}{\bar \rho}\right)\rho_k^2\, dx + \int_{\rr^d} f\rho_k^2\nabla\left(\frac{1}{\bar\rho}\right)\, d x = \int f \tilde{\bEta}\, d\rho_k  .
\end{equation*}
Rearranging, we obtain,
\begin{align*}
-\frac{1}{2}\int_{\rr^d} \nabla \left(\frac{f}{\bar \rho}\right)\rho_k^2\, dx &= \int_{\rr^d} f\teta \rho_k - f\rho_k^2\nabla\left(\frac{1}{\bar\rho}\right)\, d x
 =\int_{\rr^d}\frac{f}{\bar\rho}\left(\teta \rho_k \bar \rho - \bar\rho \rho_k^2\nabla\left(\frac{1}{\bar\rho}\right)\right)\, d x.
\end{align*}
Since the previous line holds for all $f\in C^\infty_c(\rr^d)$, we deduce $\rho_k^2\in W^{1,1}(\rr^d)$  and 
\[
\nabla\left(\frac{\rho_k^2}{2}\right) = \teta \rho_k \bar \rho - \bar\rho \rho_k^2\nabla\left(\frac{1}{\bar\rho}\right).
\]
Finally, by the chain rule for $W^{1,1}(\rr^d)$ functions and the previous line, we have,
\begin{align*}
    \nabla\left(\frac{\rho_k^2}{\bar \rho ^2}\right)&= \nabla(\rho_k^2) \frac{1}{\bar \rho ^2} + \rho_k^2\nabla\left(\frac{1}{\bar \rho ^2}\right)= \frac{1}{\bar \rho ^2}  \left( 2\teta \rho_k \bar \rho - 2\bar\rho \rho_k^2\nabla\left(\frac{1}{\bar\rho}\right)\right)+ \rho_k^2\nabla\left(\frac{1}{\bar \rho ^2}\right)= 2\teta \frac{\rho_k}{\bar \rho}.
\end{align*}
Thus, 
\begin{equation}
\label{eq:to show def eta 2}
\tilde{\bEta} \rho_k  = \frac{\bar{\rho}}{2}\nabla \left(\frac{\rho_k^2}{\bar \rho ^2}\right).
\end{equation}
Finally, defining $\bEta_k = \teta+\nabla V+\nabla V_k$, the facts that $\nabla V\in L^\infty(\rr^d)$ and $\nabla V_k\in L^2(\rho_k)$ (see sentence following Assumption (\ref{Vkas})), ensure $\bEta_k \in L^2(\rho_k)$ and (\ref{eq:to show met slope est}) holds.
\end{proof}

 \subsubsection{Convergence of gradient flows} \label{GFconvsubsection}
We now apply the result on the limiting behavior of  the metric slopes, obtained in Proposition \ref{prop:estimate for metric slopes}, as well as the $\Gamma$-convergence of the energies, obtained in Theorem \ref{newenergygammaprop}, to show that  gradient flows of $\F_{\ep,k}$ with ``well-prepared'' initial data converge to an ``almost curve of maximal slope'' of  $\F_k$.  We emphasize that this result does not require a log-concavity assumption on $\bar{\rho}$.

\begin{prop} 
\label{prop:epstozerofixedk} Assume    (\ref{mollifieras}), (\ref{Vas}), (\ref{Vkas}), (\ref{targetas})  hold. Fix $T>0$ and $k \in \mathbb{N}$.
For   $\ep>0$, let $\rho_{\ep,k}\in AC^2([0,T];\P_2(\Rd))$ be a gradient flow of $\F_{\ep,k}$ satisfying,
\begin{align}
&\sup_{\ep>0} \mathcal{S}(\rho_{\ep,k}(0))<\infty \quad \text{ and }  \quad \sup_{\ep>0} M_2(\rho_{\ep,k}(0))<\infty  \label{eq:assum grad flow conv thm entropy} .
\end{align} 
Suppose  there exists $\rho_k(0) \in D(\F_k) \cap \P_2(\Rd)$   
such that,
\begin{align}
\rho_{\ep,k}(0) \xrightarrow[]{\ep \to 0} \rho_k(0) \text{ narrowly } \quad  \text{ and } \quad \lim_{\ep\rightarrow 0} \F_{\ep,k}(\rho_{\ep,k}(0)) = \F_k(\rho_k(0)).\label{eq:assum grad flow conv thm initial time} 
\end{align}
Then, there exists $\rho_k \in AC^2([0,T];\P_2(\Rd))$ that is an ``almost'' curve of maximal slope of $\F_k$, in the sense of Definition \ref{almostcms}, and a subsequence $\epsilon^{(k)}_n$, depending on $k$,  so that 
\begin{align} \label{HYPepssubseqconvunif} \lim_{n \to +\infty} W_1(\rho_{\ep^{(k)}_n,k}(t), \rho_k(t)) = 0 \text{ uniformly for } t \in [0,T].
\end{align}
\end{prop}

\begin{proof}[Proof of Proposition \ref{prop:epstozerofixedk}]
By Theorem \ref{curvemaxslope}, $\rho_{\ep,k}$ is a curve of maximal slope of $\F_{\ep,k}$, so 
 \begin{align} \frac12 \int_0^t |\rho_{\ep,k}'|^2(r) dr + \frac12 \int_0^t |\partial \F_{\ep,k}|^2(\rho_{\ep,k}(r)) dr \leq \F_{\ep,k}(\rho_{\ep,k}(0)) - \F_{\ep,k}(\rho_{\ep,k}(t))  \ , \quad \text{ for all } 0 \leq t \leq T . \label{fepkmaxslope}
 \end{align}
  We seek to apply Proposition \ref{prop:conv of almost cms}. Theorem \ref{newenergygammaprop} ensures that $\F_{\ep,k}$ $\Gamma$-converges to $\F_{k}$.  Next, we note that, the previous line, together with the explicit characterization of $|\partial \F_{\ep,k}|^2$ given in Proposition   \ref{prop:subdiff char ep>0}, yields that the hypothesis (\ref{hyp:gen prop alpha curve max slope}) holds with $\bEta_{\ep,k}(r)\in L^2(\rho_{\ep,k}(r))$ given by,
\[
\bEta_{\ep,k}(r) = \nabla \zeta_{\ep}*\left(\frac{1}{\bar{\rho}}\left(\zeta_{\ep}*\rho_{\ep,k}(r)\right)\right)+\nabla (\zeta_{\ep}*V) +\nabla V_k .
\]
In addition, the hypotheses of the present proposition guarantee that (\ref{hyp:gen prop F(0)}) hold. Thus, Proposition \ref{prop:conv of almost cms}  ensures that  there exists $\rho_k \in AC^2([0,T];\P_2(\Rd))$ and a subsequence $\epsilon^{(k)}_n$, depending on $k$, so that (\ref{HYPepssubseqconvunif}) holds and with,  for all $t\in [0,T]$,  
\begin{align} \label{ep0kfinitealmostcurvemax}
\frac12 \int_0^t |\rho_k'|^2(r) dr + \frac12 \int_0^t \left(\liminf_{n\rightarrow \infty} \int_{\mathbb{R}^d} |\bEta_{\ep_n^{(k)}}(r)|^2 d \rho_{\ep_n^{(k)},k}(r) \right)dr \leq \F_k(\rho_k(0)) - \F_k(\rho_k(t)) .
\end{align}

In order to conclude, it suffices to establish that $\rho_k$ satisfies the conditions of Definition \ref{almostcms}, that is, for almost every $r\in [0,T]$, we have:
\begin{align}
&\rho_k(r)^2\in W^{1,1}(\rr^d),  \label{eq:bd eta ep k want -2}\\ 
&\text{there exists $\bEta_k(r)\in L^2(\rho_k(R))$ satisfying (\ref{almostcurvemaxslope1b}), and}\label{eq:bd eta ep k want -1}\\ 
&\liminf_{n\rightarrow \infty} \int_{\mathbb{R}^d} |\bEta_{\ep_n^{(k)}}(r)|^2 d \rho_{\ep_n^{(k)},k}(r) \geq \int_{\rr^d}|\bEta_k(r)|^2\, d\rho_k(r). \label{eq:bd eta ep k want}
\end{align}

 Note that we may assume 
\begin{align} \label{energyboundtimezero}
\sup_{n \in \mathbb{N}} \F_{\epsilon^{(k)}_n,k} \left(\rho_{\epsilon^{(k)}_n,k}(0) \right)<+\infty.
\end{align}
Combining this with Theorem \ref{prop:H1 bd} and assumption (\ref{eq:assum grad flow conv thm entropy})   of the present theorem, we obtain,
\begin{equation*}
\liminf_{n\to \infty} \int_0^T||  \nabla\zeta_{\ep_n^{(k)}}*\rho_{\ep_n^{(k)},k} (r) ||^2_{L^2(\rr^d)}\, dr <+\infty .\end{equation*} 
Thus, by Fatou's lemma, for almost every $r\in [0,T]$, the above integrand must be finite. Likewise,   inequality (\ref{ep0kfinitealmostcurvemax}) ensures the left-hand side of (\ref{eq:bd eta ep k want}) is finite for a.e. $r\in [0,T]$. 

We seek to apply Proposition \ref{prop:estimate for metric slopes}. 
Fix $r \in [0,T]$ such that
\begin{equation}
 \label{eq:H1bd applying}
\liminf_{n\rightarrow \infty} \int_{\mathbb{R}^d} |\bEta_{\ep_n^{(k)}}(r)|^2 d \rho_{\ep_n^{(k)},k}(r) < +\infty \quad \text{ and } \quad \liminf_{n\rightarrow \infty} ||  \nabla\zeta_{\ep_n^{(k)}}* \rho_{\ep_n^{(k)},k}(r)  ||_{L^2(\rr^d)} <+\infty.
\end{equation} 
 Inequality (\ref{energyboundtimezero}), together with the fact that the energy $\F_{\ep,k}$ decreases in time along the gradient flow $\rho_{\ep, k}$  implies that (\ref{eq:new prop assump Eep bd}) holds at time $r$. Thus, by our standing hypotheses on $V$ and $V_k$, the hypotheses of Proposition \ref{prop:estimate for metric slopes}  hold at time $r$. Consequently, the conclusion of Proposition \ref{prop:estimate for metric slopes}    yields (\ref{eq:bd eta ep k want -2}), (\ref{eq:bd eta ep k want -1}), and (\ref{eq:bd eta ep k want}) at time $r$.

\end{proof}

\section{Convergence of energies $\F_k$ and ``almost'' curves of maximal slope as $k \to +\infty$} \label{eptozeroGFsecnew} 

The present section has two main goals. First, we show that the energies $\F_k$ $\Gamma$-converge to the energy $\F$ and use this to prove  Theorem \ref{minimizersthm}, that minimizers of $\F_{\ep,k}$ converge to the unique minimizer of $\F$ as $\ep \to 0$, $k \to +\infty$. Our second goal is to show that, if $\bar{\rho}$ is log-concave on $\Omega$, then as the  confining potentials $V_k$ approximate $V_\Omega$, the
``almost'' curves of maximal slope of $\mathcal{F}_{k}$ (see   Definition \ref{almostcms}) converge to a gradient flow  of $\mathcal{F}$ as $k \to +\infty$. We then use this to conclude our main result, Theorem \ref{newmaintheorem}, that gradient flows of $\F_{\ep,k}$ converge to a gradient flow of $\F$ as $k \to +\infty$ and $\epsilon = \epsilon(k) \to 0$. Finally, in Section \ref{particlesection}, we extend our result to cover particle initial data (Theorem \ref{thm:conv with particle i.d.}), long time behavior (Corollary \ref{quantcor}). and establish our results concerning two-layer neural networks (Corollary \ref{2layernncor}).

Our result on the $k \to +\infty$ limit generalizes work by  Alasio, Bruna, and Carrillo \cite{alasio2020role} to the case of weighted porous medium equations.  As in Proposition \ref{prop:epstozerofixedk}, which considered the $\epsilon \to 0$ limit, we use an approach based on $\Gamma$-convergence of gradient flows, which is different from the approach used in the aforementioned work \cite{alasio2020role}. We are optimistic this new approach will be more easily generalizable to a range of Wasserstein gradient flows.

We begin by showing $\Gamma$-convergence of the energies $\V_k$ to $\V_\Omega$, in the sense of Definition \ref{def:gamma conv def}.  
\begin{thm}[$\Gamma$ convergence of energies $\V_k$ to $\V_\Omega$]  \label{Vkgammathm} Assume (\ref{targetas}), (\ref{Omegaas}), (\ref{Vas}), (\ref{Vkas}), and (\ref{Vkinftyas}).  Then, the energies $\V_k$ $\Gamma$-converge to $\V_\Omega$ and the energies $\F_{k}$ $\Gamma$-converge to $\F$ as $k\rightarrow \infty$. In particular,  $\lim_{k \to +\infty} \V_k(\mu) = \V_\Omega(\mu)$ for  any $\mu \in \P_2(\Rd)$.
\end{thm}

\begin{proof}
We first establish item (\ref{eq:gamma conv def}) for the energies $\V_k$ and $\V_\Omega$. Without loss of generality, we may assume $\liminf_{k \to +\infty} \V _k(\rho_k)<+\infty$ so, up to a subsequence, \begin{align} \label{Vkbdd}
\sup_{k \in \mathbb{N}} \V_k(\rho_k) <+\infty .
\end{align}

 To show inequality (\ref{eq:gamma conv def}),   
 it suffices to prove that $\supp \rho \subseteq \overline{\Omega}$, since   $\V_k(\rho_k)$ is nonnegative and   $\V_\Omega(\rho)$ would equal zero. Suppose, for the sake of contradiction that $\supp \rho \not \subseteq \overline{\Omega}$, so that there exists $x \in \overline{\Omega}^c$  
  and an open ball $B$ containing $x$ so that $B  \subset \subset \overline{\Omega}^c$ and $\rho(B)>0$. By the Portmanteau theorem, the fact that $\rho_k \to \rho$ narrowly ensures $\liminf_{k \to +\infty} \rho_k(B) \geq \rho(B) >0$. Thus, up to taking another subsequence, we may assume that there exists $\delta >0$ so that $\rho_k(B) \geq \delta$ for all $k \in \mathbb{N}$. By definition of $\V_k$, this implies,
 \begin{align*}
 \liminf_{k \to +\infty} \int_\Rd V_k d \rho_k \geq  \liminf_{k \to +\infty} \int_{B} V_k  d \rho_k \geq  \liminf_{k \to +\infty} \left( \inf_{x \in B} V_k(x) \right) \rho_k(B)  \geq \delta \liminf_{k \to +\infty} \left( \inf_{x \in B} V_k(x) \right) = +\infty ,
 \end{align*}
 where the last inequality follows from Assumption (\ref{Vkinftyas}) on $V_k$. This contradicts (\ref{Vkbdd}). Thus, we must have $\supp \rho \subseteq \overline{\Omega}$, which completes the proof of item (\ref{eq:gamma conv def}). Note that, since $\E$ and $\V$ are lower semicontinuous and bounded below, it follows immediately that inequality (\ref{eq:gamma conv def}) holds for the energies $\F_k$ and $\F$.

 It remains to prove item (\ref{newenergygammapart1}). To this end, we note that we may write $\V_\Omega(\rho) = \int V_\Omega \, d\rho$, where $V_\Omega(x)$ is given by (\ref{eq:VOmega def}). 
 Assumption (\ref{Vkinftyas}) on $V_k$ implies $V_k(x) \leq V_\Omega(x)$ for all $x \in \Rd$. Therefore we find,
\begin{align*}
\limsup_{k \to +\infty} \V_k(\rho)  =  \limsup_{k \to +\infty} \int V_k d \rho \leq \int V_\Omega d \rho =  \V_\Omega(\rho),
\end{align*}
and thus conclude by recalling the definitions of $\V_k$ and $\V$. Likewise, we also obtain (\ref{newenergygammapart1}) for $\F_k$ and $\F$.
\end{proof}

As a corollary of Theorems \ref{newenergygammaprop} and \ref{Vkgammathm}, we obtain the result of Theorem \ref{minimizersthm}: minimizers   of $\F_{\ep,k}$ converge to a minimizer of $\F$. The additional assumptions we add -- that   $V_k $ are all greater than $V_1$ and the sublevel sets of $V_1$ are compact -- are natural in the context of taking the $V_k$'s to be diverging to $+\infty$ off of $\overline{\Omega}$. 

\begin{proof}[Proof of Theorem \ref{minimizersthm}]
First, we show that $\F$ has a unique minimizer. Suppose that $\rho_0$ and $\rho_1$ are both minimizers of $\F$. Since $\F$ is proper, we have $\F(\rho_0), \F(\rho_1) < +\infty$, so $\supp \rho_1, \supp\rho_2 \subseteq \overline\Omega$. Thus,
\begin{align*}
\F((1-\alpha) \rho_0 + \alpha \rho_1)  &= \frac12 \int_\Omega \frac{|(1-\alpha)\rho_0 + \alpha \rho_1|^2}{\bar{\rho}} + \int_\Omega V d ((1-\alpha) \rho_0 + \alpha \rho_1) \\
&= (1-\alpha) \F(\rho_0) + \alpha \F(\rho_1) - \alpha(1-\alpha) \int_\Omega \frac{|\rho_0-\rho_1|^2}{\bar{\rho}} .
\end{align*}
Since $\bar{\rho}$ is uniformly bounded above, we must have $\rho_0=\rho_1$, else $\F((1-\alpha) \rho_0 + \alpha \rho_1) < \F(\rho_0)$ for $\alpha \in (0,1)$, contradicting the choice of $\rho_0$ as a minimizer.

Now, we show minimizers of $\F_{\ep,k}$ converge to the unique minimizer of $\F$. Again, using that  $\F$ is proper,  take $\nu\in \mathcal{P}_2(\rr^d)$ such that $\F(\nu) <+\infty$, so in particular, $\V_\Omega(\nu) = 0 = \V_k(\nu)$. Since $\rho_{\ep,k}$ minimizes $\F_{\ep,k}$, using the fact from Theorem \ref{newenergygammaprop} that $\E_\ep + \V_\ep$ $\Gamma$-converges to $\E + \V$, we have
\begin{align} \label{minconvineq1}
\limsup_{\ep \to 0, k\rightarrow +\infty}\F_{\ep,k}(\rho_{\ep,k}) &\leq \limsup_{\ep \to 0, k\rightarrow +\infty}\F_{\ep,k}(\nu) = \limsup_{\ep \to 0, k \to +\infty} \E_\ep(\nu) + \V_\ep(\nu) + 0 \\
& = \E (\nu) + \V(\nu) + 0 = \F(\nu)< +\infty. \nonumber
\end{align}
Thus,  we may assume that, up to a subsequence, $\F_{\ep,k}(\rho_{\ep,k})$ is uniformly bounded above in $\ep>0$, $k\in \mathbb{N}$. Since $\E_\ep$ and $\V_\ep$ are bounded below uniformly in $\epsilon >0$, $\V_k(\rho_{\ep,k})$ must be   bounded above uniformly in $\ep>0$ and $k \in \mathbb{N}$. Next,   the assumption $V_{k}\geq V_1$ implies,
\[
\sup_{\ep >0, k\in \N}\int V_1\, d\rho_{\ep,k}\leq \sup_{\ep >0, k\in \N}\int V_k\, d\rho_{\ep,k} \leq \sup_{\ep >0, k\in \N}\V_{k}(\rho_{\ep,k})<+\infty.
\]
Together with the fact that the sublevel sets of $V_1$ are compact, this  guarantees that the sequence $\rho_{\ep,k}$ is tight; see \cite[Remark 5.1.5]{ambrosio2008gradient}. Thus, up to another subsequence, there exists $\rho \in \P(\Rd)$ so that $\rho_{\ep,k} \to \rho$. 
By inequality (\ref{minconvineq1}),  Theorem \ref{newenergygammaprop}, and Theorem \ref{Vkgammathm}, we have,  for any $\nu \in D(\F)$,
\begin{align*}
 \F(\nu) &= \E(\nu) + \V(\nu) + \V_\Omega(\nu) = \lim_{\ep \to 0, k\to +\infty} \E_\ep(\nu) + \V_\ep(\nu) + \V_k(\nu)  = \lim_{\ep \to 0, k\to +\infty} \F_{\ep,k}(\nu) \geq  \\
 &\quad \geq \liminf_{\ep \to 0, k\to +\infty} \F_{\ep,k}(\rho_{\ep,k}) \geq  \liminf_{\ep \to 0, k\to +\infty} \Big( \E_{\ep}(\rho_{\ep,k}) +  \V_{\ep}(\rho_{\ep,k}) \Big) +  \liminf_{\ep \to 0, k\to +\infty} \V_{k}(\rho_{\ep,k}) .
 \end{align*}
Next, we choose subsequences that attain the $\liminf$, and then apply Theorem \ref{Vkgammathm}, to find,
\begin{align*}
 \F(\nu)&\geq   \lim_{n \to +\infty} \Big( \E_{\ep_n}(\rho_{\ep_n,k_n}) +  \V_{\ep}(\rho_{\ep_n,k_n}) \Big) +  \lim_{m  \to +\infty} \V_{k_m}(\rho_{\ep_m,k_m}) \\
 &\geq \E(\rho) + \V(\rho) + \V_\Omega(\rho) = \F(\rho)   .
 \end{align*}

Since $\nu$ was an arbitrary measure in the domain of $\F$, this shows $\rho $ is the unique minimizer of $\F$.
Finally, since the above argument shows that every subsequence of $\rho_{\ep, k}$  has a further subsequence that converges to $\rho$,  the original sequence $\rho_{\ep,k}$ must converge to $\rho$.
\end{proof}

We now turn our attention from minimizers to gradient flows and prove that ``almost'' curves of maximal slopes of $\F_k$ converge to a gradient flow of $\F$ as $k \to +\infty$.

\begin{prop} \label{ktoinftyprop}
Assume   (\ref{Omegaas}), (\ref{Vas}), (\ref{Vkas}), (\ref{Vkinftyas}), (\ref{targetas}), and that $\bar{\rho}$ is log-concave on $\Omega$. Fix $T>0$. For $k \in \mathbb{N}$, let $\rho_k \in AC^2([0,T];\P_2(\Rd))$ be an ``almost'' curve of maximal slope of $\F_k$, in the sense of Definition \ref{almostcms}, and suppose  there exists $\rho(0) \in D(\F) \cap \P_2(\Rd)$ such that
\begin{align} \label{kinftyichyp}
\rho_k(0) \xrightarrow{k \to +\infty} \rho(0) \text{ narrowly }, \quad \lim_{k \to+\infty} \F_k(\rho_k(0)) = \F(\rho(0)) \ , \quad \text{ and }\quad \sup_{k \in \mathbb{N}} M_2(\rho_k(0)) < +\infty.
\end{align}
Then
\begin{align} \label{W1convkinfty}
\lim_{k \to +\infty} W_1(\rho_k(t), \rho(t)) = 0, \text{ uniformly for } t \in [0,T],
\end{align}
where $\rho \in AC^2([0,T];\P_2(\Rd))$ is the unique gradient flow of $\F$ with initial conditions $\rho(0)$.
\end{prop}

\begin{proof}
Recall from Theorem \ref{Vkgammathm} that $\F_k$ $\Gamma$-converges to $\F$. Thus, by Proposition \ref{prop:conv of almost cms},  we find that there exists $\rho \in AC^2([0,T];\P_2(\Rd))$ so that, up to a subsequence,  (\ref{W1convkinfty}) holds and
\begin{align} \frac12 \int_0^t |\rho'|^2(r) dr +  \frac12 \int_0^t \left( \liminf_{k\rightarrow \infty} \int_{\mathbb{R}^d} |\bEta_k(r)|^2 d \rho_k(r) \right) dr \leq \F(\rho(0)) - \F(\rho(t))  \quad \text{ for all } 0 \leq t \leq T .   \label{conc:almostcurvemaxslope}
\end{align}
  
Furthermore, by Definition \ref{almostcms} of an ``almost'' curve of maximal slope, we see that $\F_k(\rho_k(t))\leq \F_k(\rho_k(0))$ for all $k \in \mathbb{N}$, $t\in [0,T]$. Combining this with (\ref{kinftyichyp}) yields 
\begin{align} \label{energydecayVk}
  & \sup_{t \in [0,T], k \in \mathbb{N}}  \F_k(\rho_k(t)) \leq  \sup_{k \in \mathbb{N}} \F_k(\rho_k(0)) <+\infty.
  \end{align}
Furthermore, since $\rho_k(t) \xrightarrow{W_1} \rho(t)$, hence narrowly, for all $t \in [0,T]$, Theorem \ref{Vkgammathm}  implies,
\begin{align} \label{Fkboudningamma}
 \sup_{t \in [0,T]} \F(\rho(t)) \leq \sup_{t  \in [0,T]} \liminf_{k \to +\infty} \F_{k}(\rho_k(t))  \leq \sup_{t \in [0,T], k \in \mathbb{N}}\F_k(\rho_k(t)) < +\infty .
\end{align}

Let us use $A_k$ to denote the measure zero subset of $[0,T]$ on which condition (\ref{almostcurvemaxslope1b}) of Definition \ref{almostcms} fails, and note that $\cup_{k=1}^{+\infty} A_k$ is likewise a set of measure zero, so that for almost every $t \in [0,T]$, condition (\ref{almostcurvemaxslope1b}) holds for all $k \in \mathbb{N}$.   The remainder of the proof will be devoted to establishing that, for almost every $t \in [0,T]$, we have 
\begin{align} \label{metricslopektoinfty1}
\liminf_{k \to +\infty} \int_{\mathbb{R}^d} |\bEta_k(t)|^2 d \rho_k(t)  \geq   \int_{\mathbb{R}^d} |\bEta(t)|^2 d \rho(t) ,
\end{align}
for $\rho$ and $\bEta$ satisfying
\begin{align} \label{rhoetaconditions1} & (\rho(t)/\bar{\rho})^2 \in W^{1,1}_\loc(\Omega) \\
 &\bEta(t) \rho(t) = \frac{\bar{\rho}}{2} \nabla (\rho(t)/\bar{\rho})^2 + \nabla V \rho(t) \text{  on } \Omega.  \label{rhoetaconditions2}
\end{align}
Indeed, combining inequalities (\ref{conc:almostcurvemaxslope}) and (\ref{metricslopektoinfty1}) yields
that $\rho$ satisfies
\begin{align*}
\frac12 \int_0^t |\rho'|^2(r) dr + \frac12 \int_0^t \int_{\mathbb{R}^d} |\bEta(r)|^2 d \rho(r) dr \leq \F(\rho(0)) - \F(\rho(t))  \ , \quad \text{ for all } 0 \leq t \leq T.   
\end{align*}
 Proposition \ref{prop:subdiff char ep=0}, which characterizes the metric slope of $\F$, and Theorem \ref{curvemaxslope}, imply that $\rho$ is the unique gradient flow of $\F$ with initial data $\rho(0)$. Finally, we remark that this  argument  shows that every subsequence of $\rho_k$ has a further subsequence that converges to $\rho$ in the  sense (\ref{W1convkinfty}), implying that the original sequence must also converge to $\rho$.

For almost every $t \in [0,T]$,  inequality (\ref{conc:almostcurvemaxslope}) ensures that the left-hand side of (\ref{metricslopektoinfty1}) is finite. Fix such a $t \in [0,T]$.  Since the left-hand side of (\ref{metricslopektoinfty1}) is finite, passing to a subsequence in $k$, we may assume
\begin{equation}
\label{eq:etak L2 bdd}
   \sup_{k\in \mathbb{N}} \| \bEta_k(t)\|_{L^2(\rho_k(t))}  <+\infty .
\end{equation}
To conclude the proof, it remains to show that  (\ref{metricslopektoinfty1}), (\ref{rhoetaconditions1}), and (\ref{rhoetaconditions2}) hold at this time. From now on, we will suppress dependence on $t$, for simplicity of notation.

Since $\V$ and $\V_k$ are bounded below, inequality (\ref{energydecayVk})  implies, 
\begin{align} \label{unifL2bound}
\left( \frac{\inf \bar{\rho}}{2} \right) \sup_{k \in \mathbb{N} } \|\rho_k\|_{2}^2 \leq  \sup_{k \in \mathbb{N} } \frac{1}{2} \int_\Rd \frac{|\rho_k|^2}{\bar{\rho}} = \sup_{k \in \mathbb{N}} \E(\rho_k) < +\infty .
\end{align}
Likewise, inequality (\ref{Fkboudningamma}) and the definition of $\F$   implies   
$\supp \rho(t) \subseteq \overline{\Omega}$. 

Next, we note that, since   Assumption (\ref{Vas}) ensures $\nabla V \in L^\infty$, applying the triangle inequality and  inequality (\ref{eq:etak L2 bdd}) yields,
\begin{align}
\sup_k \int \left| \frac{\bar{\rho}}{2} \nabla (\rho_k/\bar{\rho})^2  +\nabla V_k \rho_k \right| &= \sup_k  \left\| \bEta_k - \nabla V  \right\|_{L^1(\rho_k)} \leq \sup_k  \left\|  \bEta_k - \nabla V  \right\|_{L^2(\rho_k)}   \label{Fkslopederbound}  <+\infty.  
\end{align}

By \cite[Theorem 5.4.4]{ambrosio2008gradient}, provided we have (\ref{rhoetaconditions1}) and (\ref{rhoetaconditions2}),  in order to show (\ref{metricslopektoinfty1}), it suffices to show
\begin{align}
&\liminf_{k \to +\infty } \int_\Rd f \left( \frac{\bar{\rho}}{2} \nabla (\rho_k/\bar{\rho})^2 + \nabla V \rho_k +\nabla V_k \rho_k \right)   =  \int_{\Omega} f \left( \frac{\bar{\rho}}{2} \nabla (\rho/\bar{\rho})^2 + \nabla V \rho \right)   \ \text{ for all }f \in C^\infty_c(\Rd), \label{Fkgammasuffice2}
\end{align}
where we use that $\rho = 0$ a.e. on $\Omega^c$.
By Assumption (\ref{Vas}) on $V$, we have $f \nabla V \in C_b(\Rd)$, so since $\rho_k$ narrowly converges to $\rho$, we find,
\begin{align*}
\liminf_{k \to +\infty } \int_\Rd f \nabla V  \rho_k = \int_{\Omega} f \nabla V \rho \ , \quad \text{ for a.e. } t \in [0,T].
\end{align*} 
Thus, (\ref{Fkgammasuffice2}) is equivalent to the claim that, for all $f \in C^\infty_c(\Rd)$,
\begin{align} \label{Vksloperemains}
\liminf_{k \to +\infty } \int_\Rd f \left( \frac{\bar{\rho}}{2} \nabla (\rho_k/\bar{\rho})^2  +\nabla V_k \rho_k \right) =  \int_{\Omega} f  \frac{\bar{\rho}}{2} \nabla (\rho/\bar{\rho})^2  . 
\end{align}
We will establish (\ref{Vksloperemains}) for test functions $f \in C^\infty_c(\Omega)$. Then, we will  extend to the general case of $f \in C^\infty_c(\Rd)$ via a cutoff function to obtain  (\ref{Vksloperemains}).

First, we consider the region $\Omega$.  By Assumption 
(\ref{Vkinftyas}), which ensures $V_k$ vanishes on $\Omega$ for all $k$, inequality (\ref{Fkslopederbound}) implies,
\begin{align} \label{labelunifW1bd}
\left( \frac{\inf \bar{\rho} }{2}  \right)  \sup_k  \int_{\Omega} \left|  \nabla (\rho_k/\bar{\rho})^2    \right| \leq  \sup_k  \int_{\Omega} \left| \frac{\bar{\rho}}{2} \nabla (\rho_k/\bar{\rho})^2    \right| <+\infty . 
\end{align}
Since $(\rho_k/\bar{\rho})^2 \in W^{1,1}(\rr^d)$, combining (\ref{unifL2bound}) and (\ref{labelunifW1bd}), we obtain that $(\rho_k/\bar{\rho})^2$ is bounded in $W^{1,1}({\Omega})$. Thus, up to a subsequence, $(\rho_k/\bar{\rho})^2$ converges in $L^1({\Omega})$ and almost everywhere to some $g \in L^1({\Omega})$ with $g \geq 0$. 
  Furthermore,
\begin{align*}
 \left\|\frac{\rho_k}{\bar{\rho}} - \sqrt{g} \right\|_{L^1({\Omega})} & \leq \sqrt{|{\Omega}|}  \left\|\frac{\rho_k}{\bar{\rho}} - \sqrt{g} \right\|_{L^2({\Omega})} \leq \sqrt{|{\Omega}|} \left( \int_{\Omega} \left| \frac{\rho_k}{\bar{\rho}} - \sqrt{g} \right| \left(\frac{\rho_k}{\bar{\rho}} + \sqrt{g} \right) \right)^{1/2} \\
 & =   \sqrt{|{\Omega}|} \left( \int_{\Omega} \left|\left(\frac{\rho_k}{\bar{\rho}}\right)^2 - g \right|   \right)^{1/2} ,
\end{align*}
so $\rho_k/\bar{\rho} \to \sqrt{g}$ in $L^1({\Omega})$. Combining this with the fact that $\rho_k \to \rho $ narrowly, we obtain $\sqrt{g} = \rho/\bar{\rho}$ a.e. on ${\Omega}$. Therefore, for all $f\in C^\infty_c(\Omega)$, the fact that $ V_k$ vanishes on $\overline{\Omega}$ ensures,
\begin{align}\label{Fkgammalongeqn}
\liminf_{k \to +\infty } \int_{\Omega} f \left( \frac{\bar{\rho}}{2} \nabla (\rho_k/\bar{\rho})^2  + \nabla V_k \rho_k \right) &= \liminf_{k \to +\infty } \int_{\Omega} f \frac{\bar{\rho}}{2} \nabla (\rho_k/\bar{\rho})^2    = - \liminf_{k \to +\infty } \int_{\Omega} \nabla \left( f  \frac{\bar{\rho}}{2} \right)  (\rho_k/\bar{\rho})^2    \\
&= - \int_{\Omega} \nabla \left( f  \frac{\bar{\rho}}{2} \right) (\rho/\bar{\rho})^2  . \nonumber
\end{align}
By inequality (\ref{Fkslopederbound}), the left hand side of the equation may be bounded above by,
\[ \sup_k \|f\|_\infty \left\|\frac{\bar{\rho}}{2} \nabla (\rho_k/\bar{\rho})^2  + \nabla V_k \rho_k \right\|_{L^1(\Rd)} \leq  \|f\|_\infty (\sup_k \|\bEta_k\|_{L^2(\rho_k)}+ \|\nabla V\|_\infty ) ,\]
which is finite by (\ref{eq:etak L2 bdd}). Thus,  we conclude $(\rho/\bar{\rho})^2 \in BV(\Omega)$. Thus, for all $f \in C^\infty_c(\Rd)$, 
\begin{align} \label{eqnafterlongone}
 - \int_{\Omega} \nabla \left( f  \frac{\bar{\rho}}{2} \right) (\rho/\bar{\rho})^2 =  \int_{\Omega}  f  \frac{\bar{\rho}}{2} \nabla (\rho/\bar{\rho})^2 .
 \end{align}

Next, we seek to apply the Riesz Representation Theorem to the operator,
\[
L(f) = \int_{\Omega}  f  \frac{\bar{\rho}}{2} \nabla (\rho/\bar{\rho})^2 .  
\]
 We first verify the boundedness of this operator on $L^2(\rho;\Omega)$. To this end, we use the definition of $L$ and the equalities (\ref{Fkgammalongeqn}) and (\ref{eqnafterlongone}) to find,
 \[
 L(f) = \liminf_{k \to +\infty } \int_{\Omega} f \left( \frac{\bar{\rho}}{2} \nabla (\rho_k/\bar{\rho})^2  + \nabla V_k \rho_k \right) .
 \]
 Recalling the definition of $\bEta_k$,  then using H\"older's inequality, and  using  the boundedness of $\nabla V$, we obtain,
 \begin{align*}
 \int_{\Omega} \left|f \left( \frac{\bar{\rho}}{2} \nabla (\rho_k/\bar{\rho})^2  + \nabla V_k \rho_k \right) \right| &= \int_{\Omega}|f (\bEta_k\rho_k - \nabla V \rho_k)|\\
 &\leq \|f\|_{L^2(\rho_k;\Omega)}(\sup_k \|\bEta_k\|_{L^2(\rho_k)}+\|\nabla V\|_{L^\infty(\Rd)}) .
 \end{align*}
 Finally, taking the limit in $k$, and using the narrow convergence of $\rho_k$ to $\rho$, we find that the desired bound on $L$ holds:
 \[
 |L(f)| \leq \|f\|_{L^2(\rho;\Omega)} \left(\sup_k \|\bEta_k\|_{L^2(\rho_k)} + \|\nabla V\|_{L^\infty(\rr^d)}\right),
 \]
  where we use  the estimate (\ref{eq:etak L2 bdd}) to see that the first term in the parenthesis is finite.

Thus, by the Riesz Representation theorem, there exists $\mathbf{w} \in L^2(\rho;{\Omega})$ so that, 
\[ \int f   \mathbf{w} \rho = \int f \frac{\bar{\rho}}{2} \nabla (\rho/\bar{\rho})^2  , \quad \text{ for all } f \in C^\infty_c({\Omega} ) .\]
Since $\|\mathbf{w} \rho \|_{L^1(\mathcal{L}^d;{\Omega})} = \|\mathbf{w}\|_{L^1(\rho;{\Omega})} \leq  \|\mathbf{w}\|_{L^2(\rho;{\Omega})} $, this shows that $(\rho/\bar{\rho})^2 \in W^{1,1}(\Omega)$, so (\ref{rhoetaconditions1}) holds. Likewise, $\bEta:= \mathbf{w} + \nabla V \in L^2(\rho)$   satisfies the conditions of  (\ref{rhoetaconditions2}). Finally, integrating by parts on the right hand side of (\ref{Fkgammalongeqn}) gives (\ref{Vksloperemains}) for all $f \in C^\infty_c(\Omega)$. 

It remains to show that (\ref{Vksloperemains}) holds for all $f \in C^\infty_c(\Rd)$. By the fact that we just showed it holds for test functions in  $ C^\infty_c(\Omega)$, for any smooth cutoff function $0 \leq \phi \leq 1$  that is compactly supported in $\Omega$, we have,
\begin{align*}
\liminf_{k \to +\infty } \int_\Rd f \left( \frac{\bar{\rho}}{2} \nabla (\rho_k/\bar{\rho})^2  +\nabla V_k \rho_k \right)&= \liminf_{k \to +\infty } \int_\Rd (f \phi + f (1-\phi)) \left( \frac{\bar{\rho}}{2} \nabla (\rho_k/\bar{\rho})^2  +\nabla V_k \rho_k \right) \\
& =  \underbrace{\int_{\Omega} f \phi \frac{\bar{\rho}}{2} \nabla (\rho/\bar{\rho})^2}_{I_1}  + \underbrace{\liminf_{k \to +\infty } \int_\Rd f (1-\phi) \left( \frac{\bar{\rho}}{2} \nabla (\rho_k/\bar{\rho})^2  +\nabla V_k \rho_k \right)}_{I_2} .
\end{align*}
To estimate $I_1$, note that, 
\begin{align*}
\left|f \phi \frac{\bar{\rho}}{2} \nabla (\rho/\bar{\rho})^2 \right|    \leq \frac{\|f\|_\infty \|\bar{\rho}\|_\infty}{2} | \mathbf{w} \rho | \in L^1(\mathcal{L}^d;\Omega) .
\end{align*}
To estimate   $I_2$, note that,
\begin{align*}
I_2 \leq \liminf_{k \to +\infty} \| f(1-\phi)\|_{L^2(\rho_k)} \left( \| \bEta_k\|_{L^2(\rho_k)} + \|\nabla V\|_\infty  \right)\leq \|f(1-\phi)\|_{L^2(\rho)}  \left(\sup_k \| \bEta_k\|_{L^2(\rho_k)} + \|\nabla V\|_\infty \right) ,
\end{align*}
where, 
\begin{align*}
|f(1-\phi)|^2 \rho \leq |f|^2 \rho \in L^1(\mathcal{L}^d;\Omega) .
\end{align*}
Thus, by the dominated convergence theorem, for all $\delta >0$, choosing $\phi$ sufficiently  close to $1$ pointwise on $\overline{\Omega}$, we obtain,
\begin{align*}
\left| \liminf_{k \to +\infty } \int_\Rd f \left( \frac{\bar{\rho}}{2} \nabla (\rho_k/\bar{\rho})^2  +\nabla V_k \rho_k \right) - \int_{\Omega} f   \frac{\bar{\rho}}{2} \nabla (\rho/\bar{\rho})^2  \right| \leq \left| I_1 - \int_{\Omega} f   \frac{\bar{\rho}}{2} \nabla (\rho/\bar{\rho})^2 \right| + I_2 \leq  \delta . 
\end{align*}
Since $\delta>0$ was arbitrary, this completes the proof of (\ref{Vksloperemains}).

\end{proof}

We conclude with the proof of Theorem \ref{newmaintheorem}.

\begin{proof}[ Proof of Theorem \ref{newmaintheorem}]
As in the statement of the theorem, let $\rho\in AC^2([0,T];\P_2(\Rd))$ be the unique gradient flow of $\F$ with initial condition $\rho(0)$, the existence of which is guaranteed by Proposition \ref{prop:PDE ep=0}.  
By Theorem \ref{newenergygammaprop}, for all $k \in \mathbb{N}$,
\begin{align*}
\lim_{\ep \to 0} \F_{\ep,k}(\rho(0)) = \F_k(\rho(0)),
\end{align*}
so by Proposition \ref{prop:epstozerofixedk}, there exists and ``almost'' curve of maximal slope $\rho_k \in AC^2([0,T];\P_2(\Rd))$ and a  subsequence $\{\epsilon^{(k)}_j\}_{j=1}^\infty$, depending on $k$,  so that  
\begin{align} \label{NEWepssubseqconvunif} \lim_{j \to +\infty} W_1(\rho_{\ep^{(k)}_j,k}(t), \rho_k(t)) = 0 \text{ uniformly for } t \in [0,T].
\end{align}
In particular, for each $k \in \mathbb{N}$, there exists  $\epsilon_k>0$ so that $\lim_{k \to +\infty} \epsilon_k = 0$ and 
\begin{align}  W_1(\rho_{\ep_k,k}(t), \rho_k(t)) < \frac{1}{k} \text{ for all } t \in [0,T].
\end{align}
Furthermore, since Theorem \ref{newenergygammaprop} ensures
\begin{align*}
\lim_{k \to +\infty} \F_k(\rho(0)) = \F(\rho(0)) ,
\end{align*}
Proposition \ref{ktoinftyprop} implies
\begin{align}
\lim_{k \to +\infty} W_1(\rho_k(t), \rho(t)) = 0, \text{ uniformly for } t \in [0,T]. \label{eq:121902}
\end{align}

Fix $\delta >0$. Choose $K_\delta >0$ so that, for all $k \geq K_\delta$, $W_1(\rho_k(t), \rho(t)) < \delta/2$ for all $t \in [0, T]$. Then, for all $k \geq \max\{2/\delta, K_\delta \}$, 
\begin{align*}
W_1(\rho_{\ep_k,k}(t), \rho(t)) &\leq W_1(\rho_k(t), \rho(t)) + W_1(\rho_{\ep_k, k}(t), \rho_k(t)) \leq \frac{\delta}{2} + \frac{1}{k} \leq \delta,
\end{align*}
for all $ t \in [0, T]$. This gives the result.
\end{proof}

\subsection{Extension to particle initial data and application to two-layer neural networks} \label{particlesection}
In the previous sections, we have shown that gradient flows of $\F_{\ep,k}$ with ``well-prepared'' initial data converge to a gradient flow of $\F$, as $k \to +\infty$, $\ep =\ep(k) \to 0$. Unfortunately, our assumption of ``well-preparedness'' requires that the initial data of $\F_{\ep,k}$ have bounded entropy (\ref{eq:assum grad flow conv thm entropy}), which is a crucial assumption for obtaining the $H^1$-type bound on the mollified gradient flow (Theorem \ref{prop:H1 bd}) and the lower semicontinuity of the metric slopes (Proposition \ref{prop:estimate for metric slopes}). This assumption explicitly excludes initial data given by an empirical measure.

We now use stability of the gradient flows of $\F_{\ep,k}$ to extend the convergence result to initial data given by an empirical measure, obtaining the proof of our third major theorem, Theorem \ref{thm:conv with particle i.d.}.  
This is based on the elementary fact that that any measure $\mu\in \mathcal{P}_2(\rr^d)$ can be approximated to arbitrary accuracy by an empirical measure. For lack of a reference, we recall this  in Lemma \ref{sketchlem2}. (In fact, our construction of the empirical measure   in the proof of Lemma \ref{sketchlem2} closely parallels what we employ in our numerical method.)
 It can be seen from the proof of Lemma \ref{sketchlem2}  that, if $\supp \mu \subseteq [-R, R)^d$, then  $N$  can be taken to be the smallest integer larger than $(2\sqrt{d}R/\delta)^d$. More generally, in order for an empirical measure constructed from $N$ i.i.d. samples of  a measure $\mu \ll \mathcal{L}^d$ to converge to $\mu$ in the Wasserstein metric, $N$ must scale like $O(1/\delta^d)$ \cite{dobric1995asymptotics,dudley1969speed}. Our requirement that the initial conditions of $\F_{\ep,k}$ have bounded entropy implies $\mu \ll \mathcal{L}^d$, so this scaling requirement   is sharp in our case.  However, if $\mu$ were permitted to concentrate on lower dimensional sets, recent work by Weed and Bach has shown these requirements can be weakened \cite{weed2019sharp}.

Once we have extended our result to particle initial data, in Theorem \ref{thm:conv with particle i.d.}, we are then able to quickly obtain our two main corollaries. Corollary \ref{quantcor} shows that, on bounded domains $\Omega$ and in the absence of an external potential $V$, the particle solution indeed approximates the target $\bar{\rho}$ in the long time limit. Next, Corollary \ref{2layernncor} shows that the overparametrized limit of two-layer neural networks converges, as the variance of the radial basis function goes to zero, to a solution of (\ref{mainpde}), which is the gradient flow of a convex energy.

We begin with the proof of Theorem \ref{thm:conv with particle i.d.}.

\begin{proof}[Proof of Theorem \ref{thm:conv with particle i.d.}]
First, let $\rho_{\ep,k}(t)$ be the gradient flow of $\F_{\ep,k}$ with initial data $\rho(0)$. By Theorem \ref{newmaintheorem}, as $k \to +\infty$, $\epsilon = \epsilon(k) \to 0$,
\begin{equation}
\label{eq:particle id rhoep wsto rho}
\lim_{  k \to +\infty} W_1(\rho_{\ep,k}(t), \rho(t)) = 0, \text{ uniformly for } t \in [0,T],
\end{equation}
 where $\rho(t)$ is the gradient flow of $\F$ with initial data $\rho(0)$. By Proposition \ref{prop:PDE ep=0}, $\rho$ is the unique weak solution of (\ref{mainpde}).
 Recall from Lemmas \ref{lem:semicontinuity}-\ref{Vlsc} and Propositions \ref{prop:convexity ep=0} and \ref{lem:Eepsconvex} that $\F_{\ep,k}$ is lower semicontinuous and $\lambda_\ep$-convex along generalized geodesics  with, 
\begin{equation}
\label{eq:use lambda ep}
\lambda_\ep  = -\ep^{-d-2} ||1/\bar\rho||_{L^\infty(\rr^d)}||D^2\zeta||_{L^\infty(\rr^d)}+ \inf_{\{x, \xi \in \Rd\}} \xi^t D^2 V(x) \xi,
\end{equation}
and note that $-\infty<\lambda_\ep\leq 0$. 
 
 By Proposition \ref{prop:ODE Eep},  the evolving empirical measure $\rho^N_{\ep,k}(t)$, as defined in the statement of Theorem \ref{thm:conv with particle i.d.},  is the unique gradient flow of $\F_{\ep,k}$ with initial data $\rho^N_{\ep,k}(0)$. 
By  (\ref{eq:particle id rhoep wsto rho}), it suffices to show that, as $k \to +\infty$, $\ep = \ep(k) \to 0$, , $N = N(\ep) \to +\infty$,
\begin{equation}
\label{eq:particle id nuep wsto rhoep}
W_1(\rho^N_{\ep,k}(t),\rho_{\ep,k}(t)) \to 0, \text{ uniformly for  }t\in [0,T]. 
\end{equation}

Since $\rho^N_{\ep,k}(t)$ and $\rho_{\ep,k}(t)$ are both gradient flows of the $\lambda_\epsilon$-convex energy $\F_{\ep,k}$,  classical stability estimates for gradient flows \cite[Theorem 11.2.1]{ambrosio2008gradient} ensure that, for all  $t\in [0,T]$,
\[
W_1(\rho^N_{\ep,k}(t),\rho_{\ep,k}(t)) \leq W_2(\rho^N_{\ep,k}(t), \rho_{\ep,k}(t))\leq e^{-\lambda_\ep t}W_2(\rho^N_{\ep,k}(0), \rho(0))  .\]
By hypothesis (\ref{rateofmu0N}), the right hand side goes to zero uniformly in $t \in [0,T]$, which completes the proof.
\end{proof}

We now apply this to obtain the proof of Corollary \ref{quantcor}.

\begin{proof}[Proof of Corollary \ref{quantcor}]
Let $\rho(t)$ be the solution of (\ref{mainpde}) with initial data $\rho(0)$, as in Theorem \ref{thm:conv with particle i.d.}. By Proposition \ref{prop:PDE ep=0}, $\rho$ is the unique gradient flow of $\F$ with initial data $\rho(0)$, so by Proposition \ref{longtime},
\[ \lim_{t \to +\infty } W_1\left(\rho(t),  {\indi_{\overline{\Omega}}\bar{\rho}} \right) \leq \lim_{t \to +\infty } W_2\left(\rho(t),  {\indi_{\overline{\Omega}}\bar{\rho}} \right) = 0 . \]
Combining this with Theorem \ref{thm:conv with particle i.d.} gives the result.
\end{proof}

We conclude with the proof of Corollary \ref{2layernncor}.

\begin{proof}[Proof of Corollary \ref{2layernncor}]
The evolving empirical measure $\rho^N_{\ep}(t)$, as defined in the statement of Corollary \ref{2layernncor}, coincides with the evolving empirical measure in Theorem \ref{thm:conv with particle i.d.}. Thus, the convergence of $\rho^N_{\ep,k}(t)$ to ${\rho}(t)$ is an immediate consequence of this theorem. 

Furthermore, by Proposition \ref{prop:PDE ep=0}, $\rho(t)$ is the unique gradient flow of $\F$. Expanding the square in the definition of $\mathcal{R}$ and applying Tonelli's theorem, as in equation (\ref{ReptoEep}), we see that $\F$ coincides with   $\mathcal{R}$, up to a constant. By Definitions \ref{subdiffdef} and \ref{gradientflowdef}, the gradient flows of two energies   coincide. Thus, $\rho(t)$ is the gradient flow of $\mathcal{R}$. Similarly, from Definition \ref{defi:semiconvexity-geod}, we see that adding or subtracting a constant from an energy does not affect its convexity properties. Thus, Proposition \ref{prop:convexity ep=0} ensures that $\mathcal{R}$  is convex.

\end{proof}

\section{Numerical Simulation} \label{numericssection}

We now implement the particle method described in Theorem \ref{thm:conv with particle i.d.},   demonstrating how the system of deterministic ordinary differential equations (\ref{mainepsode}-\ref{fdefintro})
can be used to numerically approximate solutions of the diffusive partial differential equation (\ref{mainpde}). Our numerical examples explore long time behavior of solutions, the effect of the confining potential $V_k$ on the dynamics, the decay of the KL divergence along particle method solutions, and the rate of convergence as $N\to +\infty$, $\ep \to 0$, for fixed $k >>1$, both to solutions of (\ref{mainpde}) at intermediate times and to the target $\bar{\rho}$ on $\Omega$ in the long time limit.   Our simulations are conducted in Python using the NumPy, SciPy, CuPy, and Matplotlib libraries \cite{hunter2007matplotlib, CuPy, virtanen2020scipy,van2011numpy}.

\subsection{Details of numerical approach} \label{numericaldetailssec}
We now describe the details of our numerical approach. 
Since the main goal of our study is to illustrate how nonlocal particle dynamics can approximate local diffusion equations, we consider the external potential $V = 0$. We take the dimension $d=1$, a Gaussian mollifier,
\begin{equation} 
\zeta_{\epsilon}(x) =  e^{-x^2/2\epsilon^2} /(\sqrt{2 \pi \epsilon^2}),
\end{equation}
and choose the underlying domain as  $\Omega = (-1,1)$. We approximate no-flux boundary conditions on $\Omega$ via the confining potential,
\begin{equation} \label{confiningpotentialdef}
V_k(x)=\begin{cases}
\frac{k}{2}(x+1)^2 &\text{ if }x < -1,\\ 
\frac{k}{2}(x-1)^2 &\text{ if }x>1,\\
0 &\text{ otherwise,}
\end{cases}
\end{equation}
where $k \in \mathbb{N}$ controls the strength of the confinement.

Unless otherwise specified, we choose, 
\begin{align} \label{epsrelateN}
\epsilon = 4/N^{0.99} .
\end{align} Note that this relationship between $\ep$ and $N$ is better than expected from our rigorous results; see the  discussion after Corollary \ref{quantcor}. As will be seen from our choice of initial conditions $\{X^i_0\}_{i=1}^N$ below, the choice of $\epsilon$ in (\ref{epsrelateN}) ensures that the mollifiers have sufficient overlap and that different particles can ``sense'' each other through the function $f(X^i,X^j)$.

Similarly, unless otherwise specified, we choose 
\begin{align}
k=10^9 .
\end{align}Our choice of $k$, corresponding to strong confinement, is motivated by the desire to more closely approximate the dynamics of   (\ref{mainpde}) on the bounded domain $\Omega$ with no-flux boundary conditions.
We anticipate that different choices of dimension, mollifier, underlying domain, and confining potential may affect the rate of convergence of our method, but, as our main convergence theorems are not quantitative, we leave a detailed numerical analysis of the these  effects to future work.

 The first step in our method is to approximate the initial condition $\rho_0$ in  (\ref{mainpde}) by an empirical measure $  \sum_{i=1}^N \delta_{X^i_0} m^i$ with locations $\{ X^i_0\}_{i=1}^N \subseteq \Rd$ and weights $\{m^i\}_{i=1}^N \subseteq [0,+\infty)$ satisfying   $\sum_{i=1}^N m^i = 1$.  In practice, we do this by dividing the domain $\Omega = (-1,1)$ into $N$ intervals of equal measure. The location $X^i_0$ is chosen to be the center of the $i^{th}$  interval, and the weight  $m^i$ is chosen to approximate the integral of $\rho_0$ over the interval:
 \[   m^i =  h \rho_0(X^i_0) \approx \int_{X_0^i - h/2}^{X_0^i +h/2} \rho_0(x) dx, \quad \text{ for } h = |\Omega|/N .\]
 See Lemma \ref{sketchlem2}.

With the initial conditions in hand, the next step is to solve the system of ODEs (\ref{mainepsode}-\ref{fdefintro}). For general $\bar{\rho}$, this is an integral equation, which would be   expensive to compute. In the present work, we consider $\bar{\rho}$ for which the integral in equation (\ref{fdefintro}) can be pre-computed exactly, yielding to a closed form, analytic expression for $f(X^i,X^j)$ and reducing (\ref{mainepsode}) to a standard system of ODEs. In Appendix \ref{formulassection}, we provide  explicit formulas for $f(X^i, X^j)$ in the case $\bar{\rho}$ is piecewise constant or $\bar{\rho} = C/(1+|x|^2)$, the latter being a prototypical example of a log-concave target. While it will not be  possible to obtain a closed form expression for $f(X^i,X^j)$ for all choices of $\bar{\rho}$, we are optimistic that taking sufficiently accurate piecewise constant approximations would yield good results. We leave a detailed analysis of the convergence of our method under various approximations of the target  $\bar{\rho}$ to future work.
Once a closed form expression for $f(X^i, X^j)$ is obtained, the system of ODEs \eqref{mainepsode} may then be solved using a standard numerical integrator. In the present work, we use the SciPy implementation of the backward differentiation formula (BDF) with a maximum time step of $10^{-5}$.

Finally, we seek to understand qualitative properties of the \emph{particle solution}, that is, the evolving empirical measure,
\begin{align} \label{particlesolution} \rho^N_{\epsilon,k}(t) = \sum_{i=1}^N \delta_{X^i(t)} m^i ,
\end{align}
 as well as its relation to the solution  $\rho(t)$ of   (\ref{mainpde}) and the target  $\bar{\rho}$. To visually depict $\rho^N_{\epsilon,k}(t)$ and compute its difference from $\rho(t)$ and $\bar{\rho}$ with respect to classical $L^p$ norms and statistical divergences, we will often consider the following kernel density estimate, given by convolving $\rho^N_{\ep,k}(t)$ with the mollifier $\zeta_\ep$:
\begin{align} \label{blobrobotrhoeps}
\tilde{\rho}^N_{\epsilon,k}(x,t) = ({\rho}^N_{\epsilon,k}(t)*\zeta_\ep)(x)= \sum_{i=1}^N \zeta_\epsilon(X^i(t) -x) m^i .
\end{align}
 According to Lemma \ref{weakst convergence mollified sequence},   if there exists $\mu \in \P(\Rd)$ so that ${\rho}^N_{\epsilon,k}$ narrowly converges to $\mu$  as $\ep \to 0$, then  the kernel density estimator   $\tilde{\rho}^N_{\epsilon,k}$ also narrowly converges to $\mu$  as $\ep \to 0$. Thus our main results that guarantee convergence of ${\rho}^N_{\epsilon,k}$ also ensure convergence of $\tilde{\rho}^N_{\epsilon,k}$. 

Furthermore, when the target  $\bar{\rho}$ is normalized to satisfy   $\int_\Omega \bar{\rho} = 1$,  solutions of    (\ref{mainpde}) dissipate the  Kullback-Leibler (KL) divergence with respect to  $\bar{\rho} $ on $\Omega$ exponentially quickly in time (see inequality (\ref{KLdivergenceexpdecay})). We will numerically illustrate that this key property is preserved by our approximate solutions $\tilde{\rho}^N_{\epsilon, k}$.  We   compute the KL divergence  on $\bar{\Omega}$ via,
\begin{align} \label{KLdivergencecomputation}
{\rm KL} \left( \frac{\tilde{\rho}^N_{\epsilon,k}(t)}{C_{\epsilon,k,N}(t)}, \bar{\rho} \right) = \int_\Omega \left(\frac{\tilde{\rho}^N_{\epsilon,k}(x,t)}{C_{\epsilon,k,N}(t)} \right) \log \left(\frac{\tilde{\rho}^N_{\epsilon,k}(x,t)}{C_{\epsilon,k,N}(t)\bar{\rho}(x)}\right) dx , \     \text{ for } \    C_{\epsilon,k,N}(t) = \int_{\Omega} \tilde{\rho}^N_{\epsilon,k}(x,t)dx ,
\end{align} 
where the constant $C_{\epsilon,k,N}(t)$ allows us to compensate for the fact that, since $\tilde{\rho}^N_{\epsilon,k}$ is not in general supported on $\overline{\Omega}$,  the restriction of  $\tilde{\rho}^N_{\epsilon,k} $ to $\overline{\Omega}$ is not a probability measure and ${\rm KL}(\tilde{\rho}^N_{\ep,k}(t),\bar{\rho})$ can be negative. On the other hand, $\tilde{\rho}^N_{\epsilon,k}  /C_{\epsilon,k,N}$ is always a probability measure on $\overline{\Omega}$, so that equation (\ref{KLdivergencecomputation}) gives a well-defined, nonnegative statistical divergence.  We compute the integrals in (\ref{KLdivergencecomputation}) numerically, using the SciPy library's \texttt{quad} function.

A final key quantity of our numerical scheme is the value of the energy $\F_{\ep, k}$ along the solution of the gradient flow $\rho^N_{\ep,k}$. At the continuous time level, the gradient flow structure ensures that $\F_{\ep,k}(\rho^N_{\ep,k}(t))$ is always decreasing in time; see Theorem \ref{curvemaxslope} and Proposition \ref{prop:ODE Eep}. To investigate the rate of decrease numerically, we first obtain the following expression for $\F_\ep$ in this setting:   
\begin{align} 
\F_{\ep,k}(\rho^N_{\ep,k}(t)) &= \E_\ep(\rho^N_{\ep,k}(t)) + \V_k(\rho^N_{\ep,k}(t)) = \frac{1}{2}\int_{\Rd} \frac{|\zeta_\ep * \rho^N_{\ep,k}(t)|^2}{\bar{\rho} } \, d {\mathcal{L}}^d +\int_{\Rd} V_k  \, d\rho^N_{\ep,k}(t)  \label{Fepkparticledef} \\
&= \frac{1}{2}\int_{\Rd} \zeta_\ep * \left(\frac{ \zeta_\ep * \rho^N_{\ep,k} }{\bar{\rho}} \right) \hspace{-1mm} (t) \,\, d \rho^N_{\ep,k} (t) +\int_{\Rd} V_k \, d\rho^N_{\ep,k}(t) \nonumber \\
&= \frac{1}{2} \sum_{i=1}^N\sum_{j=1}^N g(X^i(t),X^j(t)) m^im^j  +\sum_{i=1}^N m^iV_k(X^i(t)), \text{ where } \nonumber \\
g(x,y) &:= \int_\mathbb{R} \frac{   \zeta_\epsilon(x - z) \zeta_\epsilon(y - z)}{\bar{\rho}(z)}  \, dz  \label{gxydef}. 
\end{align}
We note that  $g(x,y)$ is related to the function $f(x,y)$ defined in equation (\ref{fdefintro}) by $f = \nabla_x g$, and the integral in the definition of $g$ may be likewise computed explicitly for our choices of $\bar{\rho}$, as we describe in Appendix \ref{formulassection}.

We close this discussion of the details of our numerical method with a few remarks on its efficient implementation in Python. As an interacting particle system, computing the evolution of the particle trajectories   (\ref{mainepsode}-\ref{fdefintro}) is inherently an $O(N^2)$ computation for a strictly positive mollifier $\zeta_\ep$. The expectation is that the computational effort would decrease for a compactly supported mollifier: indeed, if $\supp \zeta_\ep  \subset \subset B_{R\ep}(0)$, then $f(X^i,X^j)$ would vanish for $|X^i-X^j|>2R \ep$.  However, rigorously proving that the computational effort indeed decreases to $O(mN)$, where $m$ represented the average number of particles lying within the radius of a given mollifier, would require careful estimates on the repulsive forces between the particles and is left for future work. Nevertheless, even for a strictly positive mollifier, we are able to achieve good computational speed in practice by using the following techniques. First, we provide an analytical Jacobian to the ODE solver rather than relying on finite difference approximations. Second, we leverage the structure of the integrand to compute these partial derivatives efficiently. Finally, we parallelize these computations using the CuPy library for GPU-accelerated computing \cite{CuPy}.
 These  elements of our implementation  allow us to speed up our calculations by  two orders of magnitude compared to previous work by the first author  \cite{CarrilloCraigPatacchini}. Namely, we performed the same simulations as those used to generate Figure 1 of \cite{CarrilloCraigPatacchini} (the evolution of density over time, starting from Barenblatt initial data), both using the code of \cite{CarrilloCraigPatacchini}, as well as with our implementation. In Figure \ref{compfigure}, we record the resulting improvement in terms of computational time.
\begin{figure}[h!] \footnotesize \label{compfigure}
  \begin{tabular}{|l|l|l|} 
			\hline
			Time & Carrillo, et. al. \cite{CarrilloCraigPatacchini} & Present Work \\
			\hline 
			$N = 100$ \ & 0.04s & 0.05s   \\
			$N = 200$ \ & 0.41s & 0.08s   \\
			$N = 400$ \ & 3.35s & 0.14s   \\
			$N = 800$ \  & 38.96s & 0.35s  \\
			$N = 1600$ & 461.96s & 5.73s \\
			\hline
		\end{tabular}
		\caption{Computational time for  simulation of $\rho^N_{\ep,k}(t)$ using  our numerical method and implementation (right column) and that of  \cite{CarrilloCraigPatacchini} (left column). Here the target is $\bar\rho_{\text{uni}}$, we take  $k=0$, $t=0.15$, and the initial condition is the Barenblatt profile (\ref{eq:barenblatt}.}
\end{figure}

 These simulations were performed on a standard desktop PC (Intel Core i7-10700 CPU @ 2.9 GHz, 16 GB RAM) with a consumer-level GPU (NVIDIA GeForce RTX 2060 Super). This improvement demonstrates how recent advances in open source scientific computing methods, even in high level languages like Python, are making computing interacting particle systems tractable, even for large numbers of particles.

\subsection{Simulation Results}
We now turn to several numerical examples that illustrate key properties of our method.
In the following simulations, we consider three main choices of target: uniform, log-concave, and piecewise-constant, given by,
	\begin{align}
	\label{uniform} \bar{\rho}_{\text{uni}}(x) &= \frac12 ,  \\
 \label{logconcave}	\bar{\rho}_{\text{log-con}}(x) &=
	\frac{2}{\pi (1+|x|^2)}  ,  \\
\bar{\rho}_{\text{pw-const}}(x) &=  \begin{cases} 
1/3 ~ &{\rm for} ~ x \in (-\infty, -0.75) \cup [-0.25,0.25 ) \cup [0.75, +\infty) ,\\
2/3 ~ &{\rm for} ~ x \in [-0.75, -0.25 ) \cup [0.25, 0.75) .
\end{cases} \label{pwconstant} 
	\end{align}

\subsubsection{Evolution of density and particle trajectories}

\begin{figure}[h!]  
 {\centering \hspace{.5cm}  $\bar{\rho}_\text{uni}$ \hspace{4.3cm} $\bar{\rho}_\text{log-con}$ \hspace{4.3cm} $\bar{\rho}_\text{pw-const}$} \\
{\hspace{-.9cm} \includegraphics[height=5.5cm]{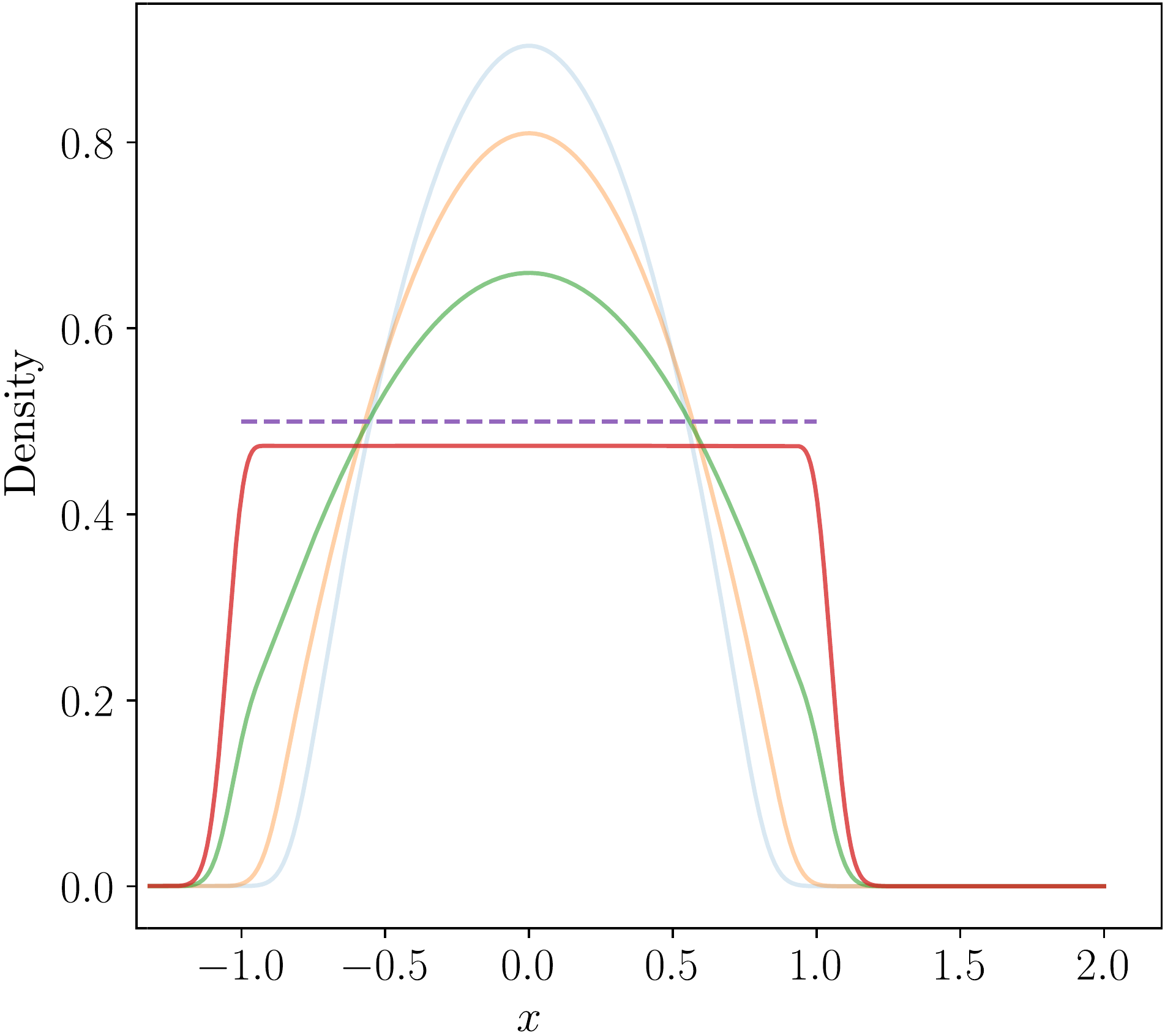}
\includegraphics[height=5.5cm]{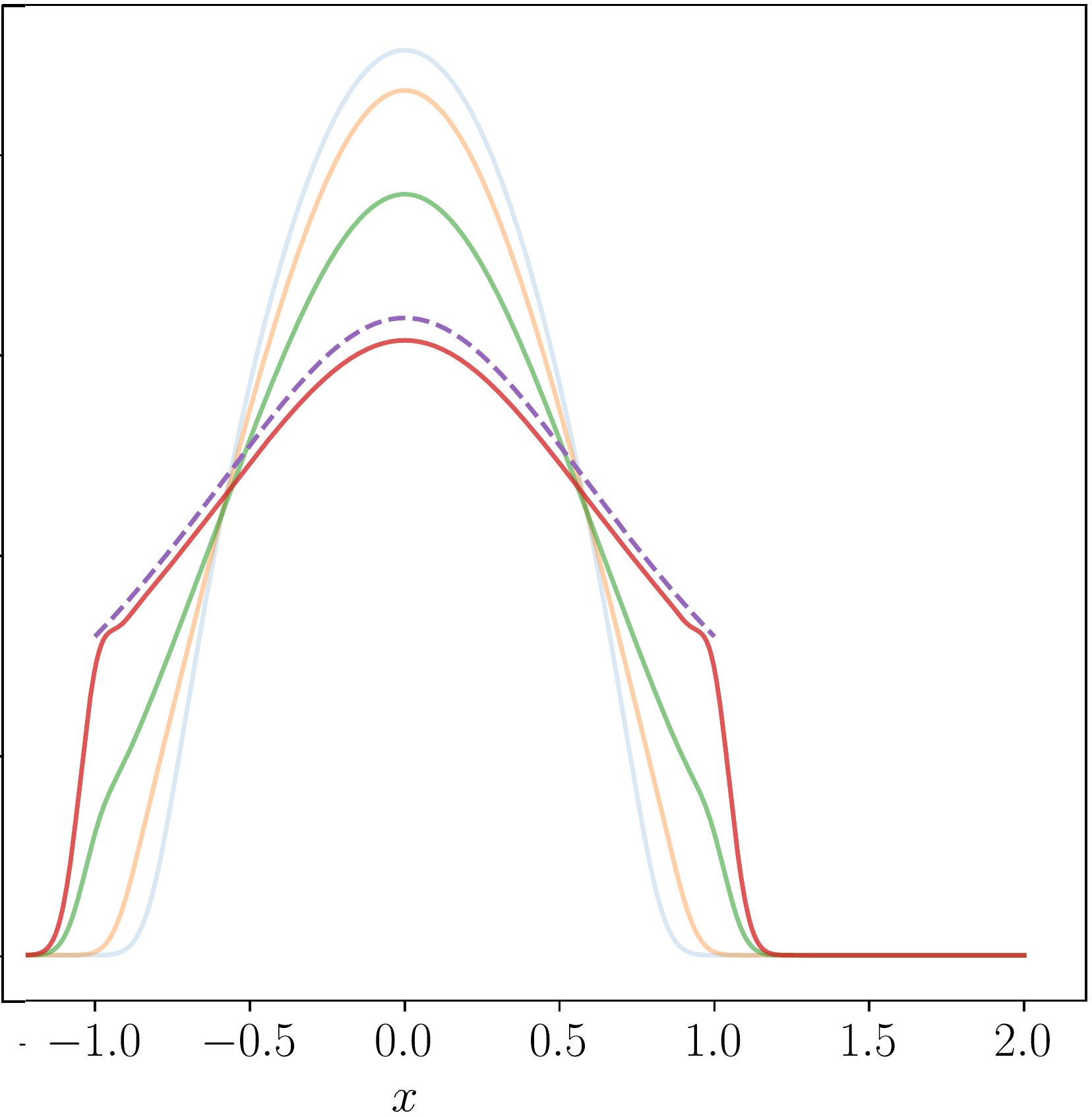}
\includegraphics[height=5.5cm,trim={0cm 0cm .05cm 0cm},clip]{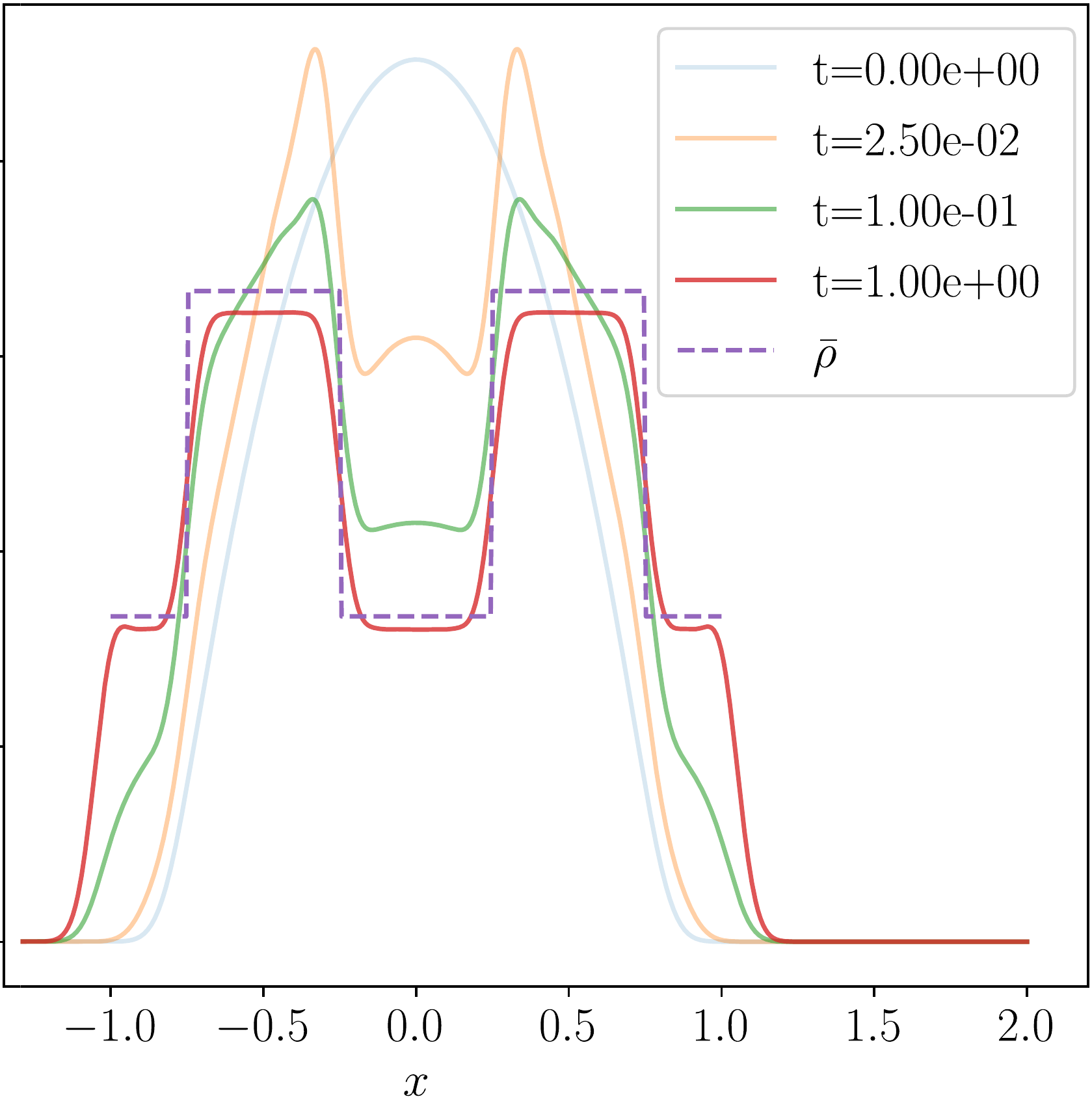}  \hspace{-5cm} }\\
 
\hspace{-3.1cm} 
    \includegraphics[height=5.97cm]{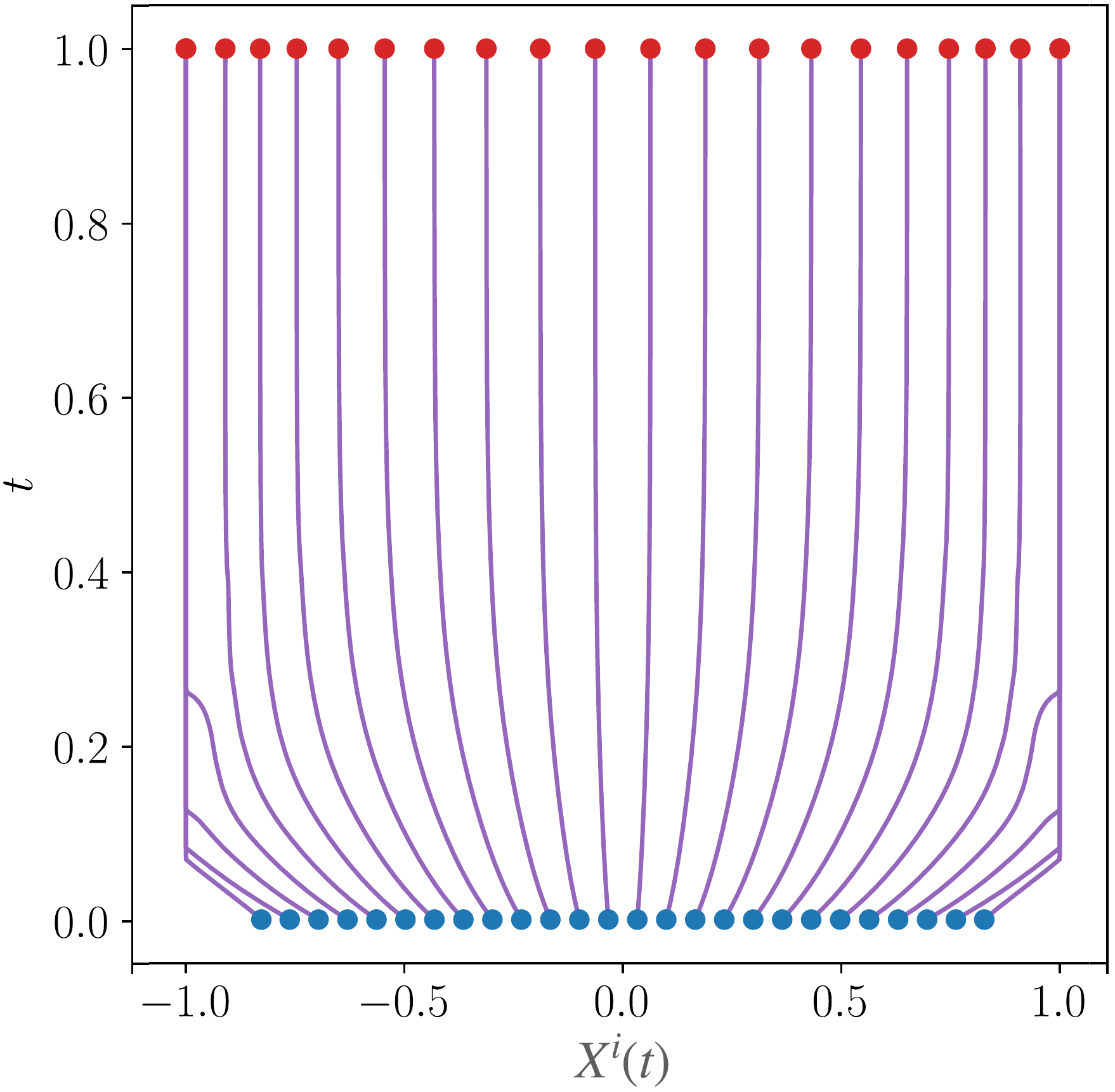}
                      \includegraphics[height=6.03cm]{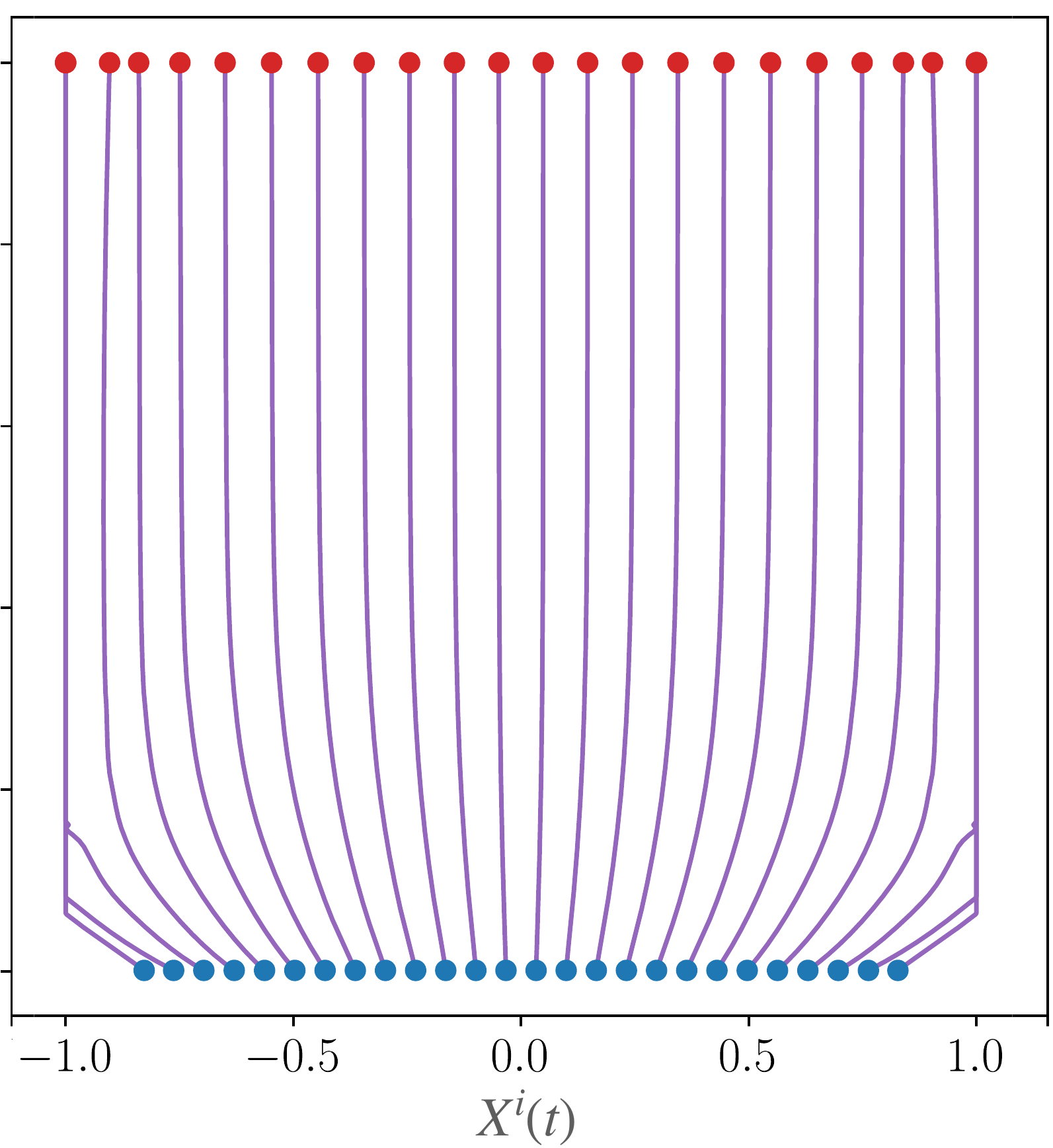}
                      \includegraphics[height=5.97cm,trim={0cm 0cm 0cm 0cm},clip]{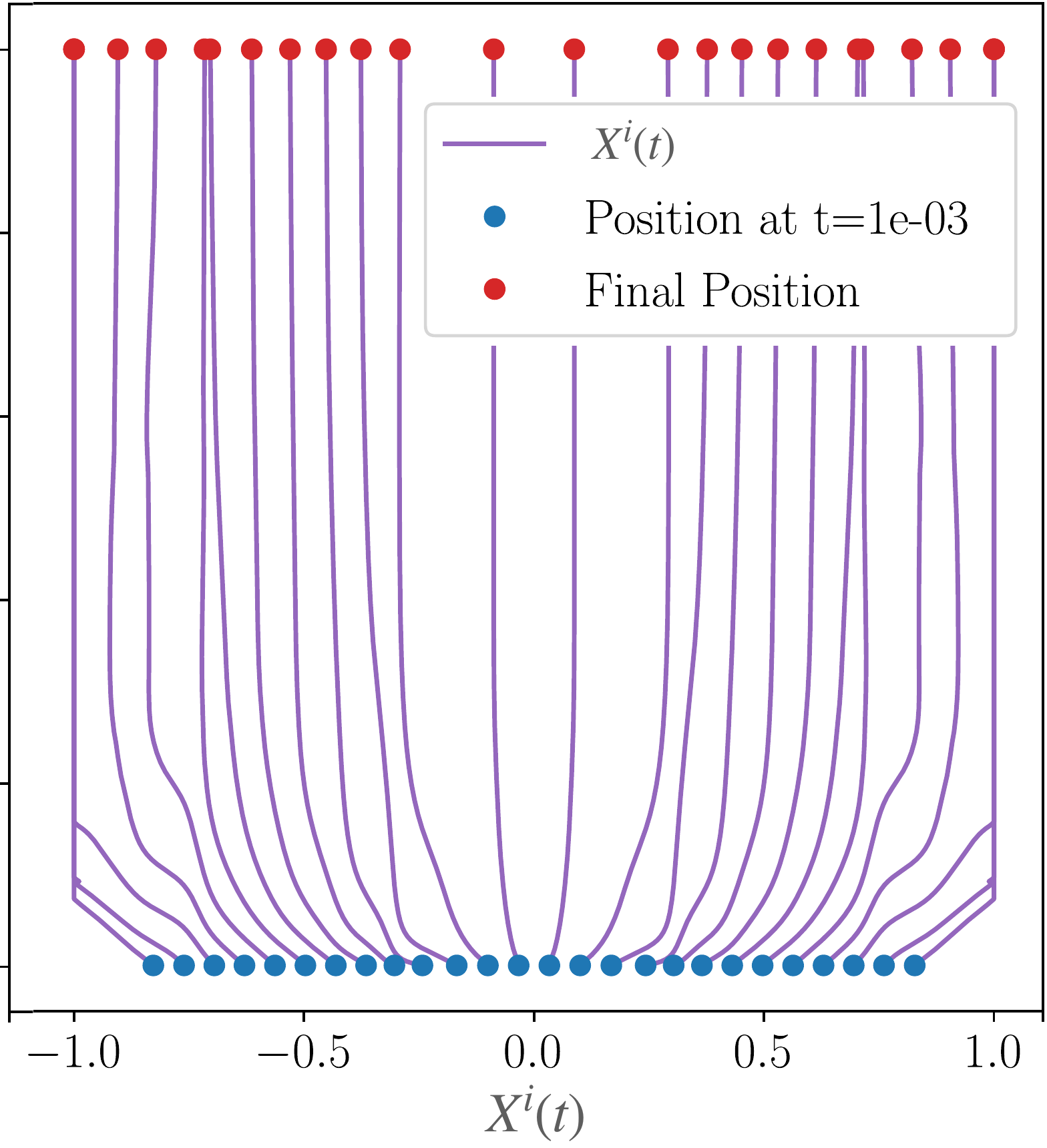}  \hspace{-2.5cm}

\caption{Simulation of the evolution of the density $\tilde{\rho}^N_{\ep,k}(t)$ for the three targets   defined in (\ref{uniform})-(\ref{pwconstant}), with $N=101$, $k=10^9$, and  initial data the Barenblatt profile (\ref{eq:barenblatt}).   Top Row: Snapshots of $\tilde{\rho}^N_{\ep,k}(t)$ for the indicated times $t$.  Bottom Row: 
Evolution of corresponding particle trajectories $X^i(t)$.}
\label{fig:fig1}
\end{figure}

In Figure \ref{fig:fig1}, we illustrate qualitative properties of numerical solutions by plotting   the kernel density estimate $\tilde{\rho}^N_{\epsilon,k}(x,t)$, defined in equation (\ref{blobrobotrhoeps}), in the top row and the trajectories of the particles $X^i(t)$ in the bottom row. We conduct our simulation for $N=101$ particles, of which 20 are plotted in the bottom row.  We consider three choices of target: $\bar{\rho}_{\text{uni}}$ (left), $\bar{\rho}_{\text{log-con}}$ (middle),  and $\bar{\rho}_{\text{pw-const}}$ (right).
 In all cases, our initial condition is  given by a {\it Barenblatt profile} $\psi_\tau(x)$, with $\tau = 0.0625$:
\begin{align}
\psi_\tau(x) &=  \frac{\tau^{-1/3}}{12} \left(3^{4/3} -  \frac{|x|^2}{    \tau^{2/3}} \right)_+ .
\label{eq:barenblatt}
\end{align}

In the top row of Figure \ref{fig:fig1}, we observe that, for all choices of target  $\bar{\rho}$, the kernel density estimate of the solution  $\tilde{\rho}^N_{\epsilon,k}(x,t)$ flows toward $\bar{\rho}$ on $\Omega$. For $\bar{\rho}_{\text{uni}}$ and  $\bar{\rho}_{\text{log-con}}$, this provides numerical verification of Corollary \ref{quantcor}, since   these  targets $\bar{\rho}$ are log-concave. On the other hand, while $\bar{\rho}_{\text{pw-const}}$ is not log-concave, and thus falls outside the scope of our theoretical results,  it does satisfy a Poincar\'e inequality, so previous work on asymptotic behavior on smooth \cite{chewi2020svgd} and weak \cite{dolbeault2008lq, GMP2013} solutions of (\ref{mainpde}) ensure that exact solutions of the continuum PDE converge to the target $\bar{\rho}_\text{pw-const}$ exponentially quickly in time; see, for example, inequality (\ref{KLdivergenceexpdecay}). Consequently, although this case lies outside the realm of our rigorous results, it is not surprising that we observe  convergence of  $\rho^N_{\ep,k}$  to $\bar{\rho}_{\text{pw-const}}$ in the long-time limit numerically.   

In the bottom row of Figure \ref{fig:fig1}, we observe that the particles evolve relatively quickly to their steady state, with most  stopping by time $t=0.3$. This stands in stark contrast to classical stochastic approaches for sampling, such as Langevin dynamics \cite{eberlelecturenotes}, and stochastic methods in the control theory literature \cite{mesquita2008optimotaxis,elamvazhuthi2016coverage}, in which particles remain in perpetual motion, complicating the  choice of an appropriate \emph{stopping time}, beyond which continued evolution doesn't  lead to  improved accuracy.

\begin{figure}[h!]  
 {\centering \hspace{.5cm}  $k=0$ \hspace{3.9cm} $k=100$ \hspace{4.1cm} $k=10^9$} \\

\includegraphics[height=5.42cm]{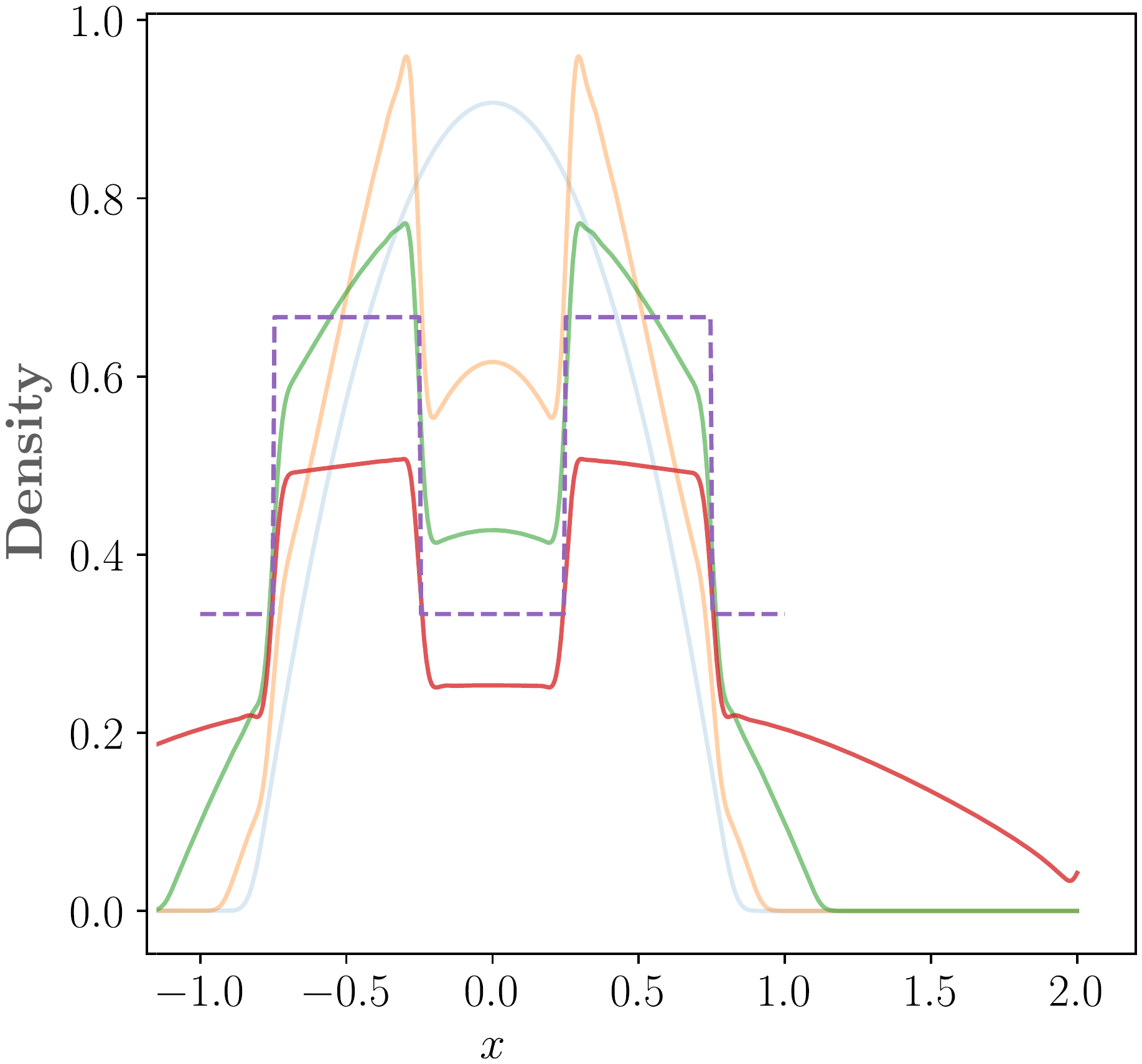}
\includegraphics[height = 5.42 cm,trim={1cm 0cm .05cm 0cm},clip]{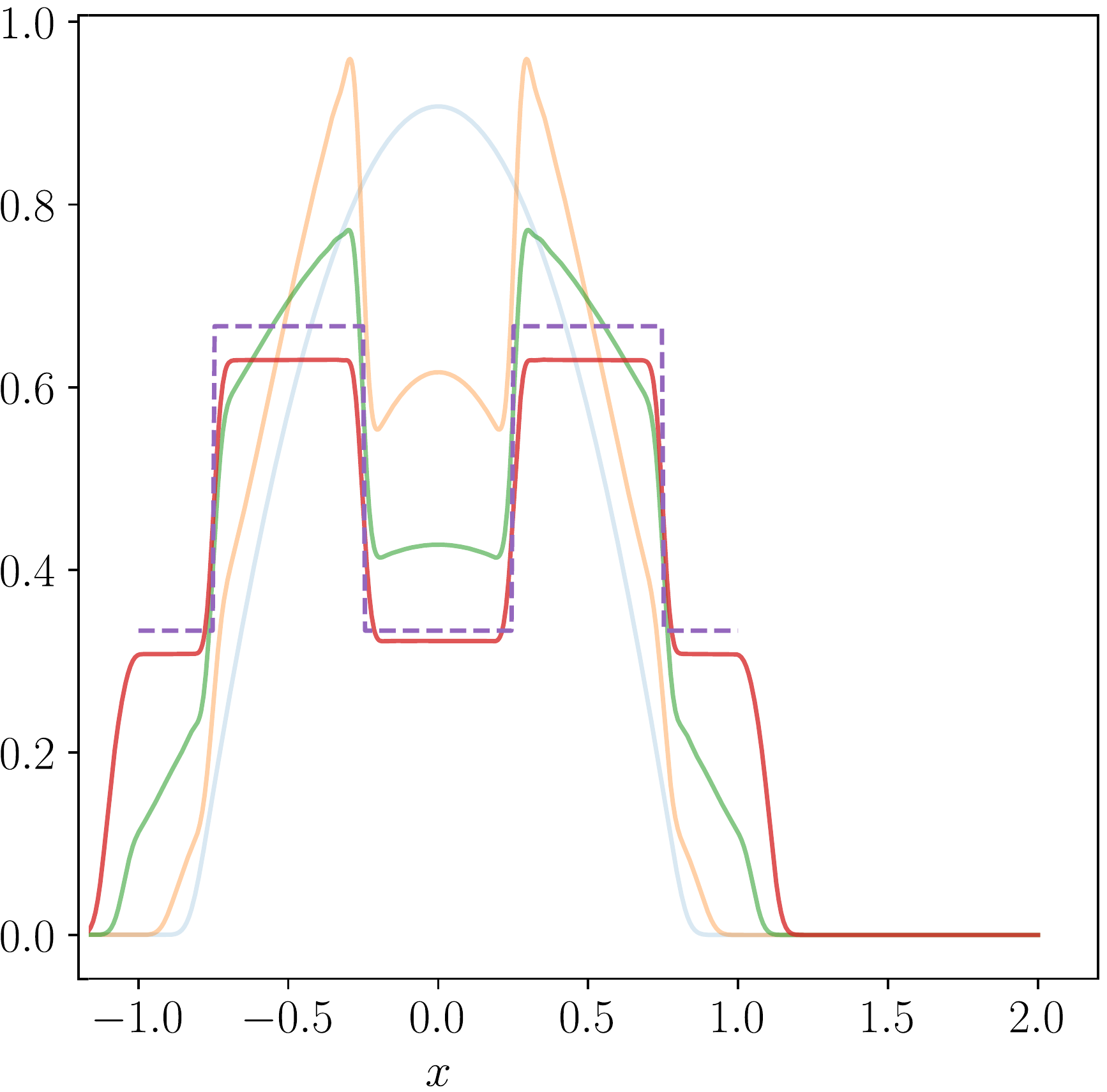}
\includegraphics[height = 5.42cm,trim={1cm 0cm 0cm 0cm},clip]{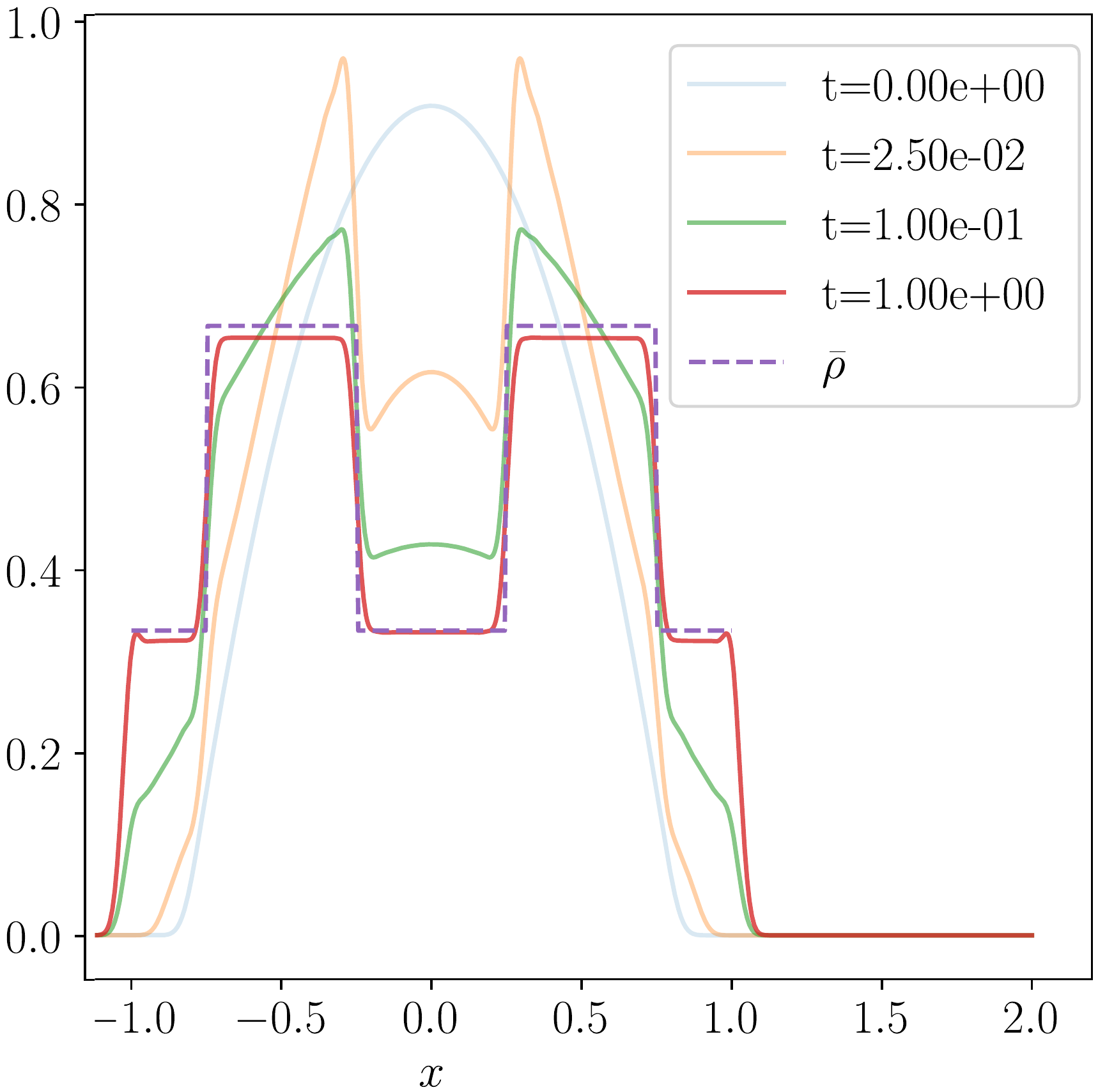} 
\caption{Comparison of how the strength of the confining potential  affects the evolution of the density. Here, $\bar\rho=\bar{\rho}_{\text{pw-const}}$, $N=200$, and the initial data is the Barenblatt profile (\ref{eq:barenblatt}). Left: no confinement ($k=0$). Middle: medium confinement ($k=100)$. Right: strong confinement ($k=10^9$).}
\label{fig:fig3}
\end{figure}

\subsubsection{Effect of confining potential on evolution of density }

In Figure \ref{fig:fig3}, we consider the effect of the confining potential on the dynamics. For a fixed number of particles $N = 200$ and initial conditions given by $\bar{\rho}_\text{pw-const}$, we plot the evolution of the kernel density estimate $\tilde{\rho}^N_{\epsilon,k}(x,t)$ as the strength of the confining potential $V_k$ is increased, from $k=0$ (left, no confinement) to $k=100$ (middle, moderate confinement) and $k =10^9$ (right, strong confinement). All simulations are conducted with Barenblatt initial data, as in equation (\ref{eq:barenblatt}).

In the $k=0$ plot in Figure \ref{fig:fig3}, we observe that the support of $\tilde{\rho}^N_{\ep,k}(x,t)$ quickly spreads outside the closure of the domain $\overline{\Omega} = [-1,1]$. This is due to fact that   $k=0$ implies $V_0 = 0$, by equation (\ref{confiningpotentialdef}), so there is no confining potential, which is equivalent to taking $\Omega = \Rd$. In this case,  Theorem \ref{thm:conv of grad flows}  ensures that, for $\ep>0$ small and $N \in \mathbb{N}$ large, the particle method approximates solutions of  the (\ref{mainpde}) equation on $\Rd$ without boundary. The diffusive effect of this equation causes the particles to spread.

In the $k=100$ plot, we observe that even a weak confining potential causes the support of the kernel density estimate to remain mostly inside of $\overline{\Omega}$, with only a small amount of mass leaking out the sides of the domain. And, in the $k=10^9$ plot, when the confinement effect is very strong,  we observe that the support of the kernel density estimate is even closer to  $\overline{\Omega}$. In general, we expect the support of the kernel density estimate $\tilde{\rho}^N_{\ep,k}(t)$ to \emph{always} be slightly larger than the domain, since even when all particles are confined to $\overline{\Omega}$, the kernel density estimate  will satisfy,
\[ \supp \tilde{\rho}^N_{\ep, k}(t) = \{X^1(t), \dots, X^n(t) \} + \supp \varphi_\epsilon .\]
 However,  in the limit $N \to +\infty$, $\ep \to 0$, and $k\to +\infty$, the support of $\tilde{\rho}^N_{\ep, k}$ will be contained in $\overline{\Omega}$.
Finally, note that, by preventing mass from leaking out of the domain, strong confinement gives the best agreement between the long time behavior ($t = 1$) of the kernel density estimate and   the desired target  $\bar{\rho}_\text{pw-const}$ on $\Omega$, in agreement with Corollary \ref{quantcor}.

\begin{figure}[h!]
 {\centering \hspace{.5cm}  $\bar{\rho}_\text{uni}$ \hspace{4.3cm} $\bar{\rho}_\text{log-con}$ \hspace{4.3cm} $\bar{\rho}_\text{pw-const}$} \\
\hspace{-.5cm} \includegraphics[height=4cm,trim={0cm 0cm .55cm 0cm},clip]{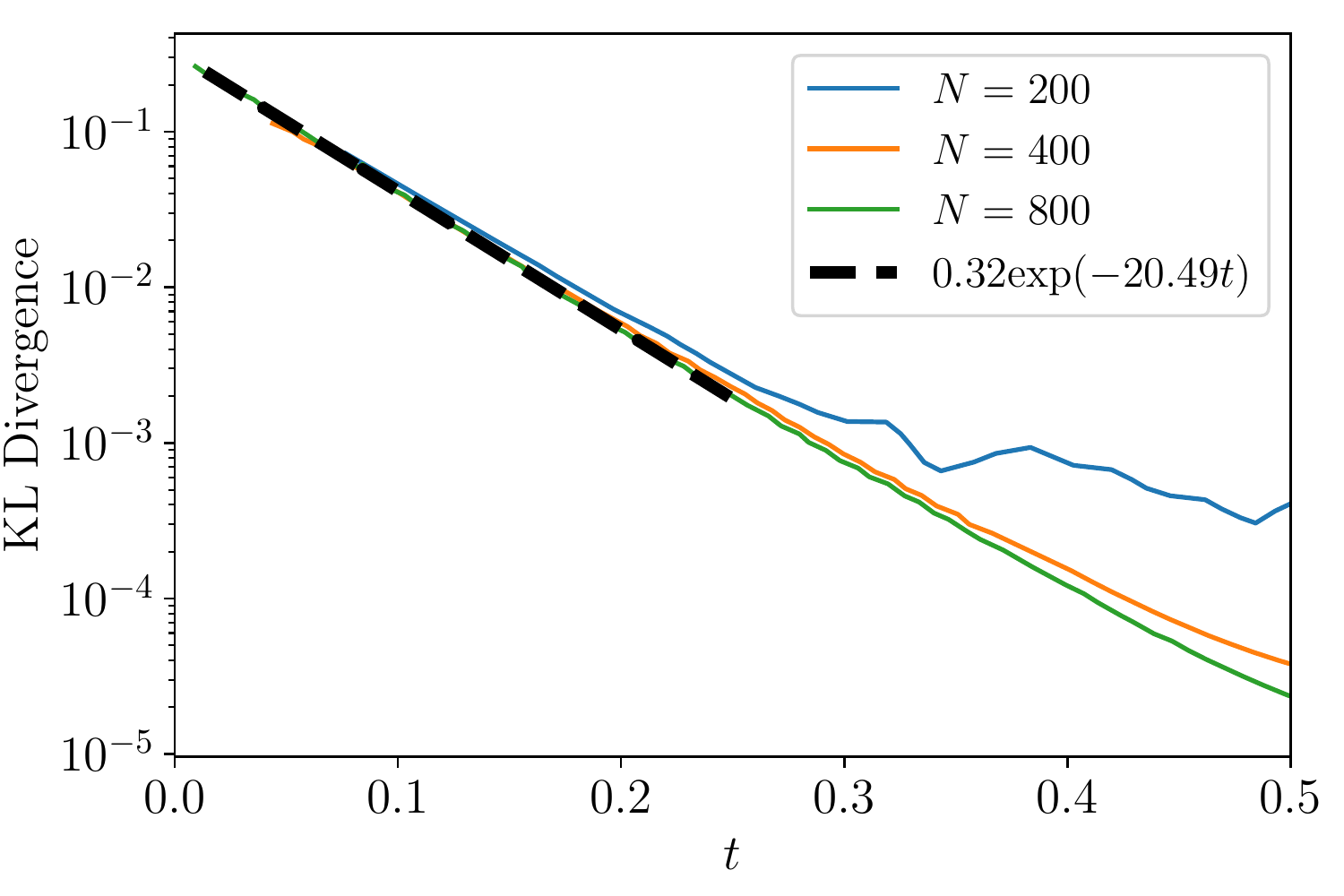}
\includegraphics[height=4cm,trim={.9cm 0cm .7cm 0cm},clip]{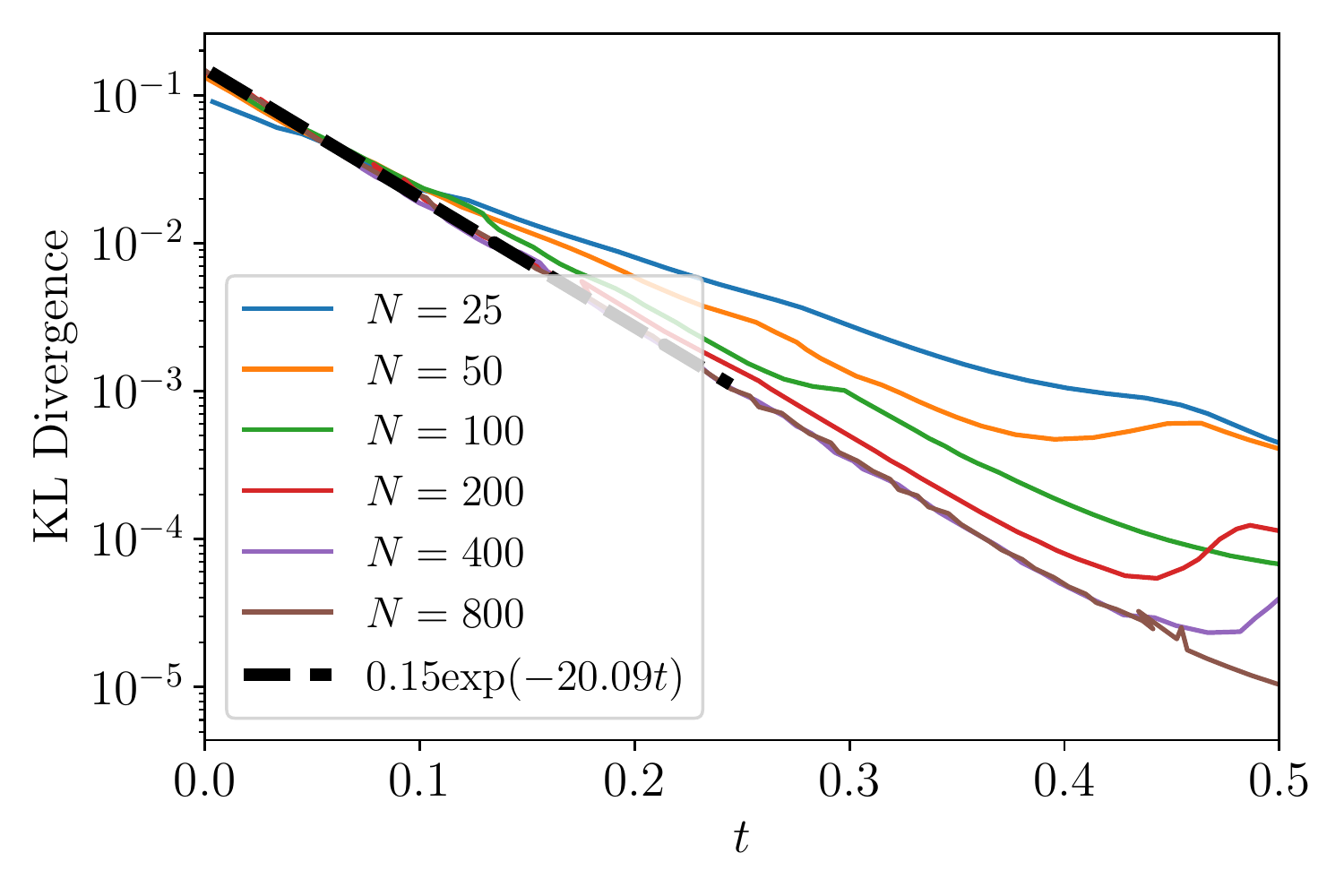}
\includegraphics[height=4cm,trim={.7cm 0cm .55cm 0cm},clip]{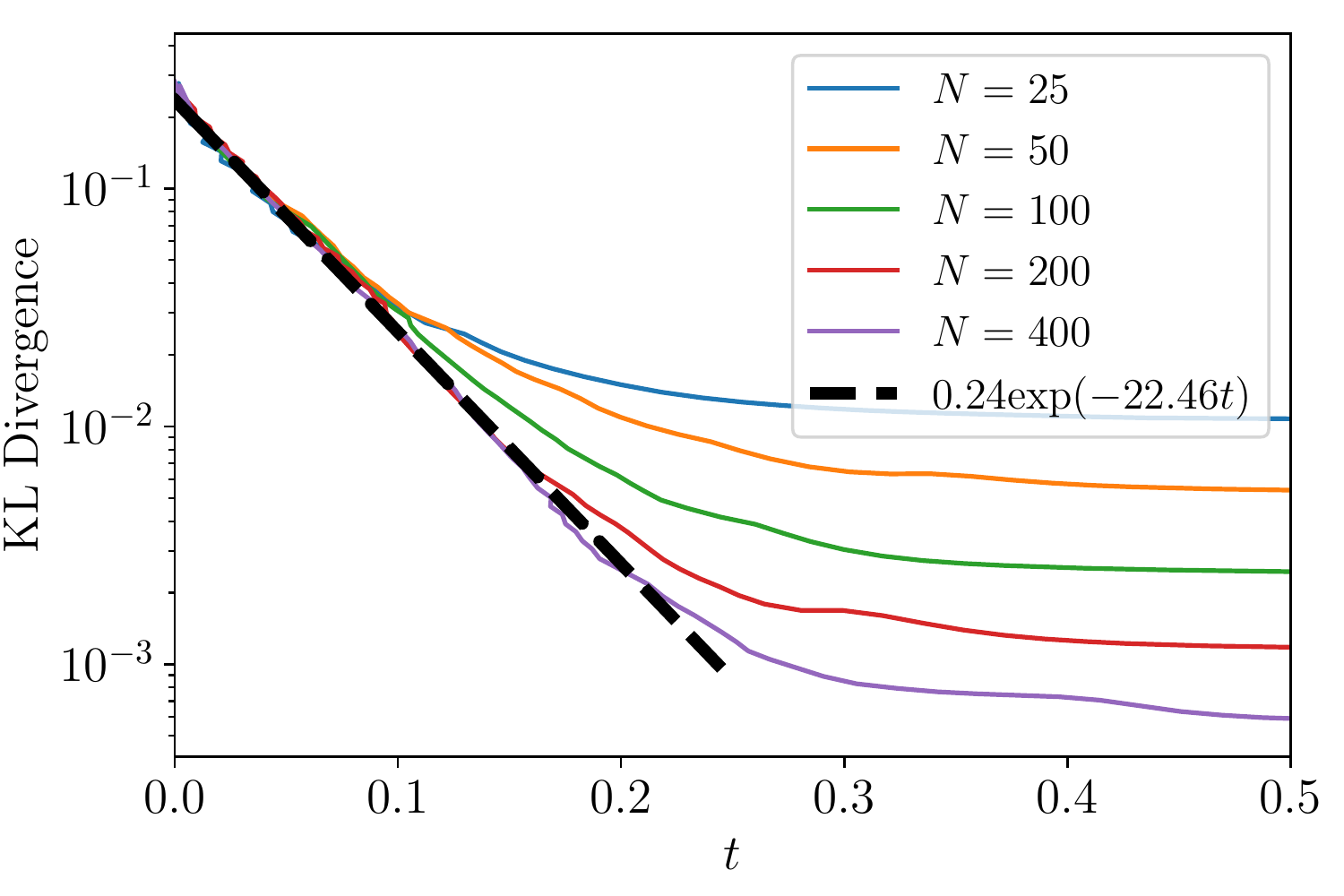}
\caption{Evolution of KL divergence between $\tilde{\rho}^N_{\ep,k}(t)$ and  $\bar{\rho}$ for three choices of target   (\ref{uniform})-(\ref{pwconstant}) and three choices of  $N$  (solid lines). We plot the line of best fit for $t\in [0,0.25]$ (dashed line). We take $k=10^9$, $t=2$, and initial data is the Barenblatt profile (\ref{eq:barenblatt}). 
}
\label{kldecay}
\end{figure}

\subsubsection{Decay of KL divergence}
In Figure \ref{kldecay}, we examine   the decay of KL divergence between the kernel density estimate  $\tilde{\rho}^N_{\ep,k}(t)$ and the  target  $\bar{\rho}$ on $\Omega$, as computed via equation (\ref{KLdivergencecomputation}). We consider three choices of target, $\bar{\rho}_{\text{uni}}$ (left), $\bar{\rho}_{\text{log-con}}$ (middle),  and $\bar{\rho}_{\text{pw-const}}$ (right), and varying numbers of particles $N$.  All simulations are conducted with Barenblatt initial data. Since each of the three targets  $\bar{\rho}$ satisfies a Poincar\'e inequality, the  inequality (\ref{KLdivergenceexpdecay}) implies that the KL divergence between $\bar\rho$ and smooth solutions $\rho(t)$ of the (\ref{mainpde}) equation decays exponentially quickly in time. We seek to observe to what extent this property is preserved by the numerical solution $\tilde{\rho}^N_{\ep,k}(t)$, which approximates $\rho(t)$ in the limit $N \to +\infty$, $\ep   \to 0$, and $k \to +\infty$, as in Theorem \ref{thm:conv with particle i.d.}.

For all three choices of target, we indeed observe an initial regime in which the KL divergence decays exponentially, as indicated by linear decay on the  semilog plots in Figure \ref{kldecay}. We estimate the rate of decay by plotting the line of best fit on the time interval $t \in [0,0.25]$, as shown  by the dashed line. After the initial period of exponential decay, the KL divergence often appears to level off, particularly for smaller numbers of particles. For larger numbers of particles, the period of exponential decay lasts longer. This  indicates that, for smaller numbers of particles, the discretization error in the approximation of (\ref{mainpde}) becomes dominant sooner, slowing the decay of the KL divergence.

The fact that our numerical approximation $\tilde{\rho}^N_{\ep,k}(t)$ preserves, up to discretization error, the key property of exponential decay of the KL divergence testifies to the benefit of  structure-preserving numerical schemes---in our case, designing a numerical scheme that preserves the continuum PDE's gradient flow structure also succeeds in capturing asymptotic behavior at the level of the particle method.
 
\subsubsection{Decay of energy}  
In Figure \ref{egydcay}, we examine   the decay of the energy $\F_{\ep,k}$ along the particle method solution $\rho^N_{\ep,k}(t)$, as computed via equations (\ref{Fepkparticledef}-\ref{gxydef}). We consider three choices of target, $\bar{\rho}_{\text{uni}}$ (left), $\bar{\rho}_{\text{log-con}}$ (middle),  and $\bar{\rho}_{\text{pw-const}}$ (right), and varying numbers of particles $N$.  All simulations are conducted with Barenblatt initial data. 

In all three cases, we observe that the energy decreases along the flow. This is expected since (up to the time discretization error of the ODE solver) our particle method solution $\rho^N_{\ep,k}(t)$ is exactly a gradient flow of the energy $\F_{\ep,k}$.  For both of the log-concave energies, $\bar{\rho}_{\text{uni}}$ and $\bar{\rho}_{\text{log-con}}$, we observe an initial period of exponential decay, for $t \in [0,0.5]$, which we approximate by a line of best fit, shown by the dashed line. We do not observe a corresponding period of exponential decay for the non-log-concave energy $\bar{\rho}_{\text{pw-const}}$.

\begin{figure}[h!]
 {\centering \hspace{.5cm}  $\bar{\rho}_\text{uni}$ \hspace{4.3cm} $\bar{\rho}_\text{log-con}$ \hspace{4.3cm} $\bar{\rho}_\text{pw-const}$} \\
\hspace{-.75cm} \includegraphics[height=4cm,trim={.3cm 0cm 0.40cm 0cm},clip]{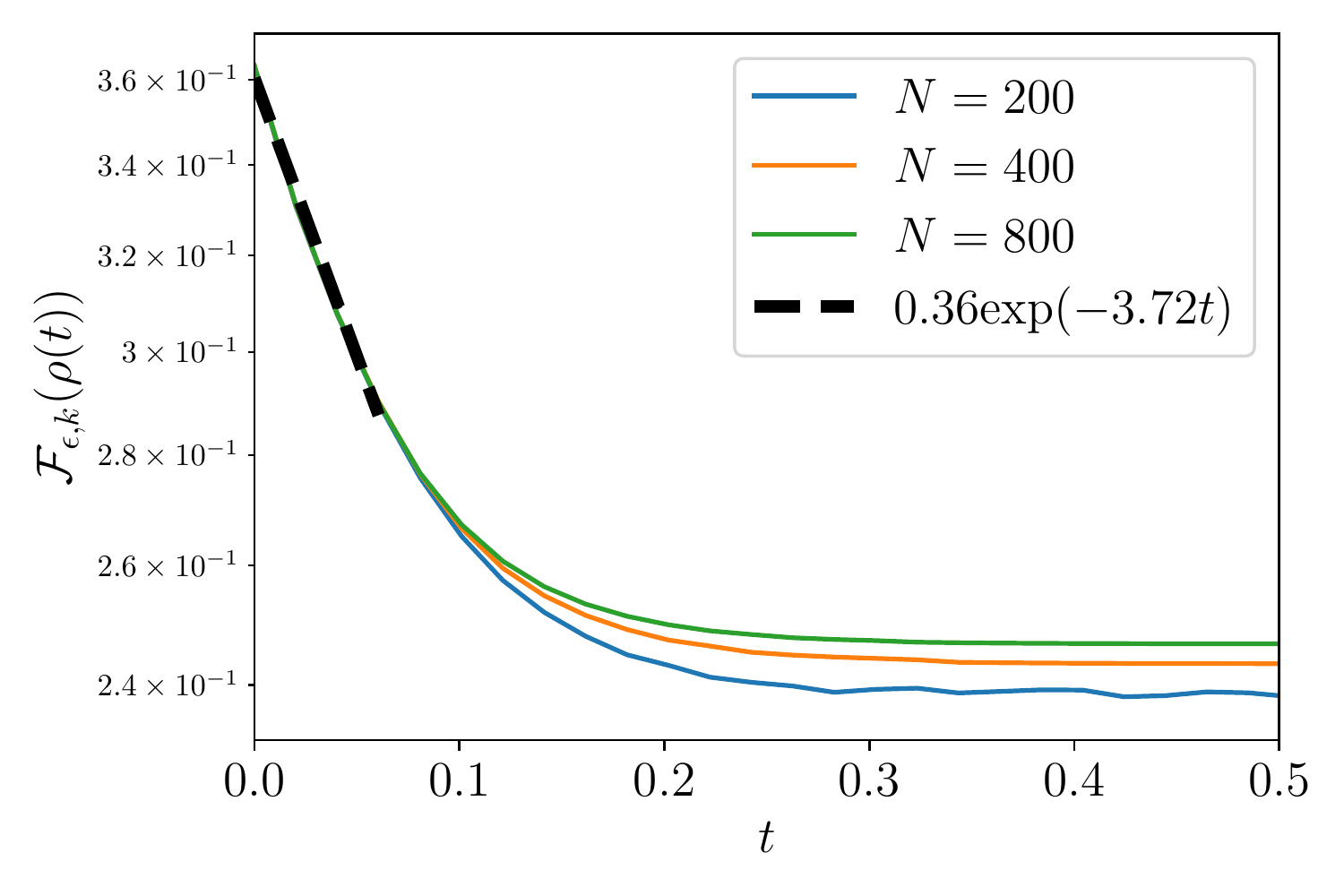}
\includegraphics[height=4cm,trim={1cm 0cm 0cm 0cm},clip]{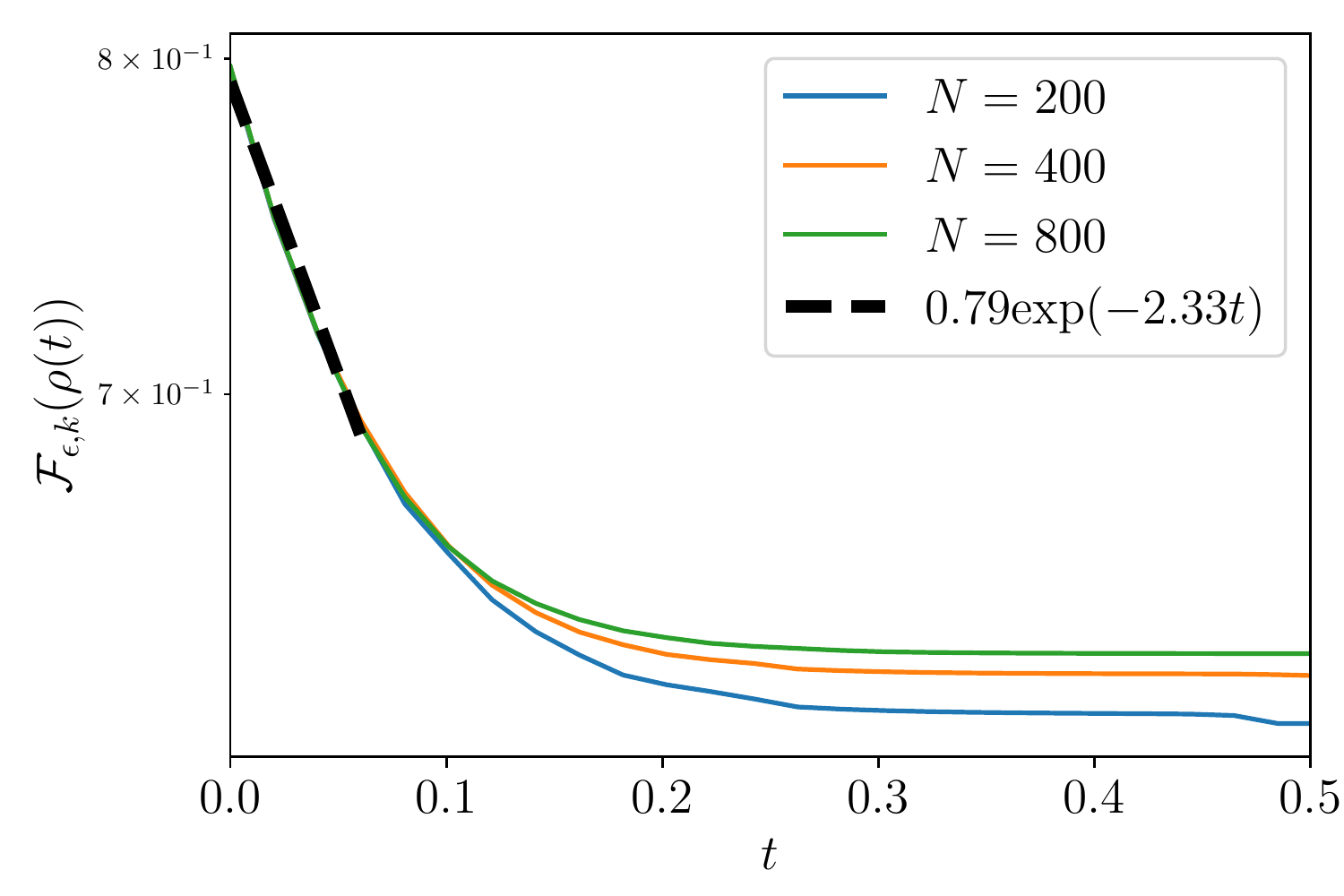}
\includegraphics[height=4cm,trim={1cm 0cm .45cm 0cm},clip]{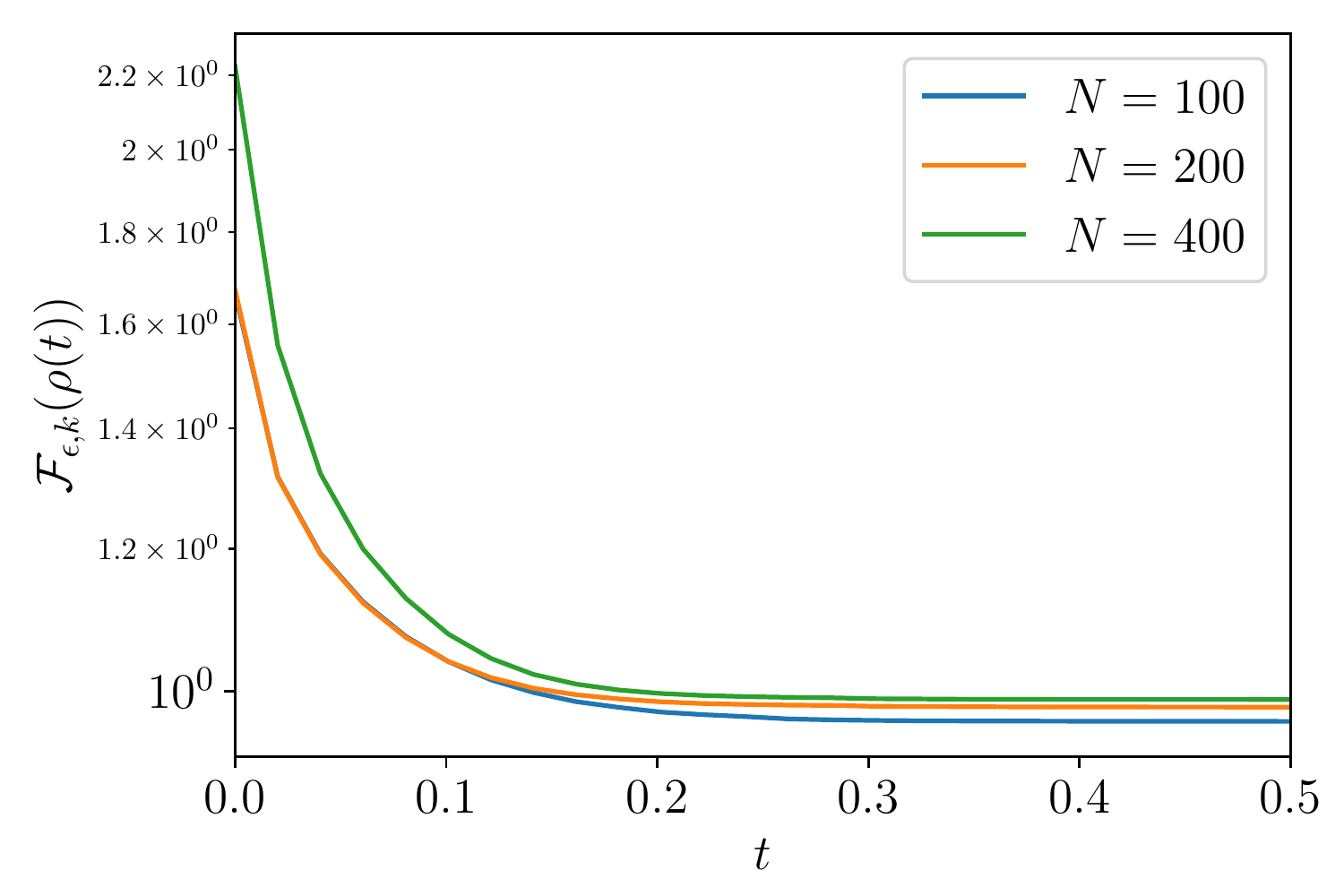}
\caption{Evolution of $\F_{\ep,k}(\rho^N_{\ep,k})$ for three choices of target   (\ref{uniform})-(\ref{pwconstant}) and three choices of $N$  (solid lines). We include the line of best fit for $t\in [0,0.5]$ (dashed line) on the left and middle plots. We take $k=10^9$, and the initial data is the Barenblatt profile (\ref{eq:barenblatt}).  }
\label{egydcay}
\end{figure}

\begin{figure}[h!]
{\centering \hspace{.5cm}  $\bar{\rho}_\text{uni}$ \hspace{6.25cm} $\bar{\rho}_\text{log-con}$ }\\
$k = 0$ \\
\includegraphics[height=5.5cm]{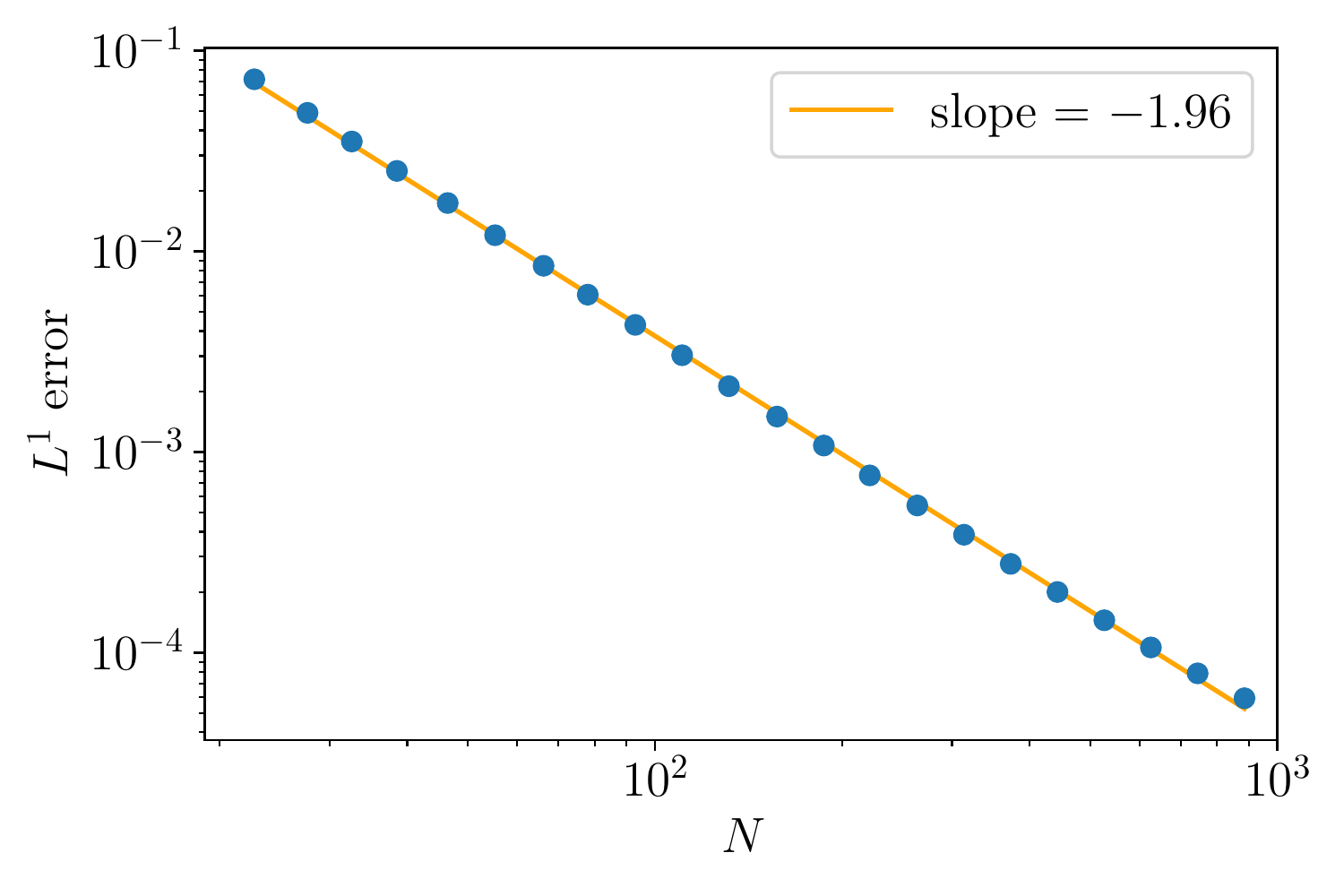}
\includegraphics[height=5.5cm,trim={1cm 0cm .05cm 0cm},clip]{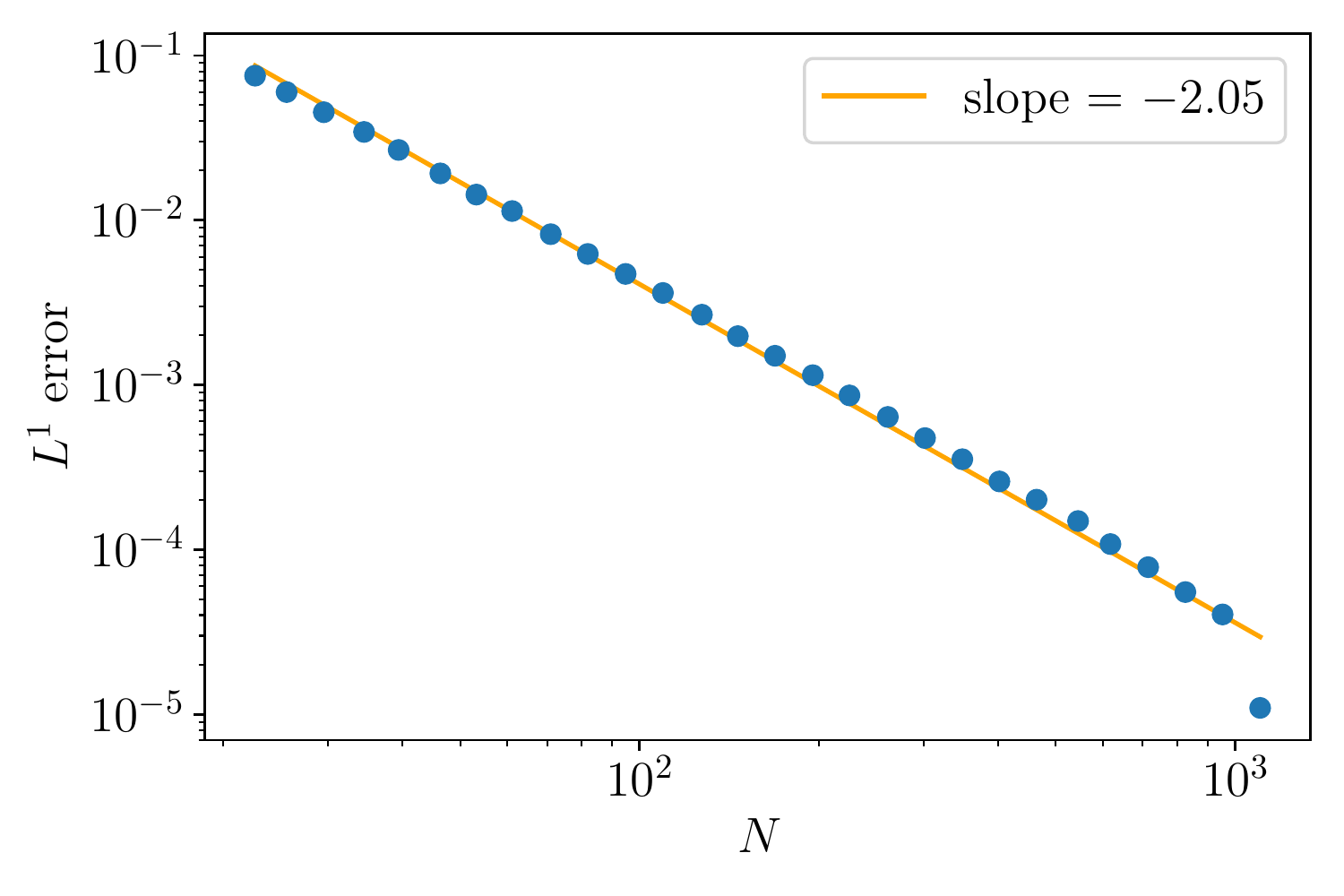} \\
$k=100$\\
\includegraphics[height=5.5cm]{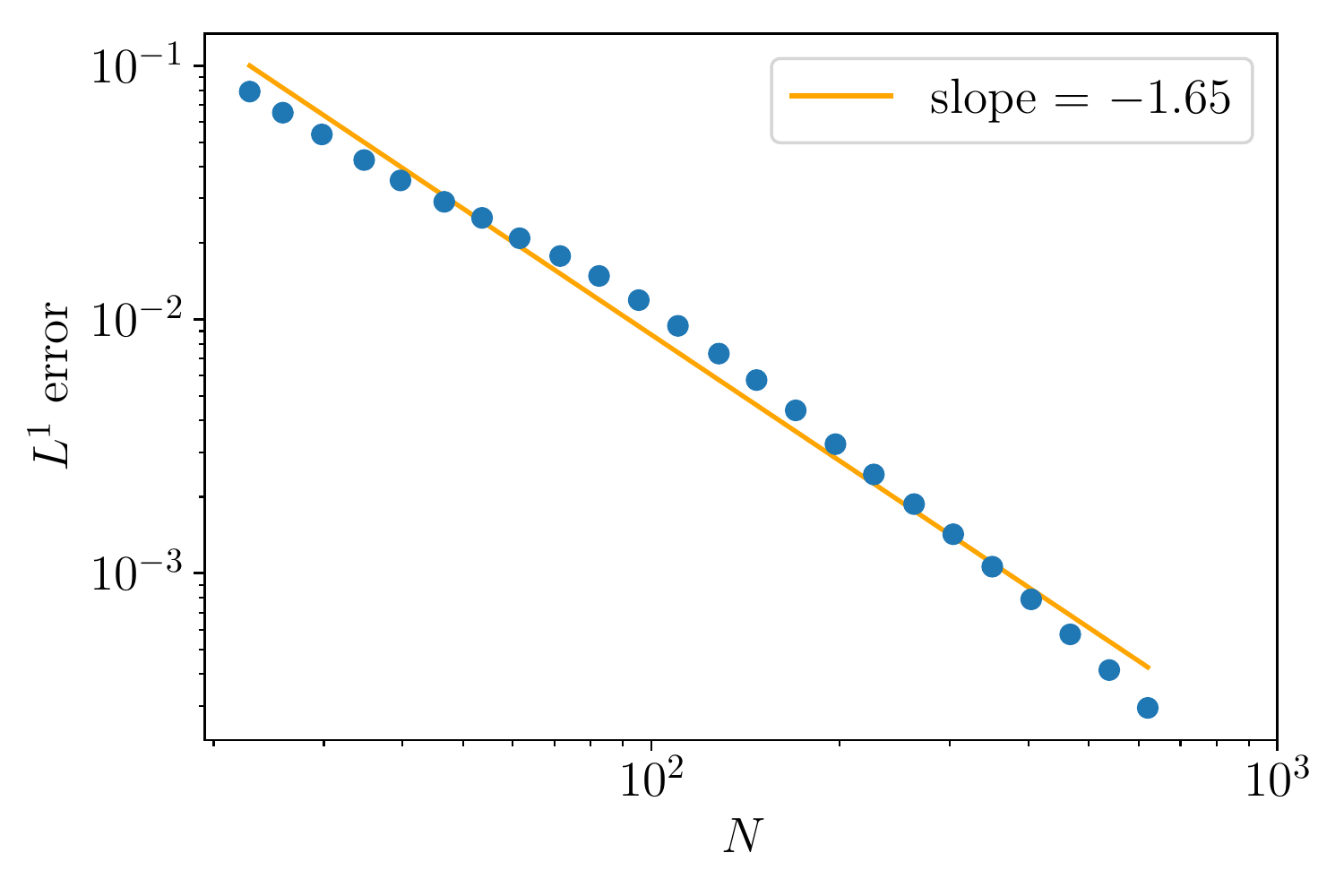}
\includegraphics[height=5.5cm,trim={1cm 0cm 0cm 0cm},clip]{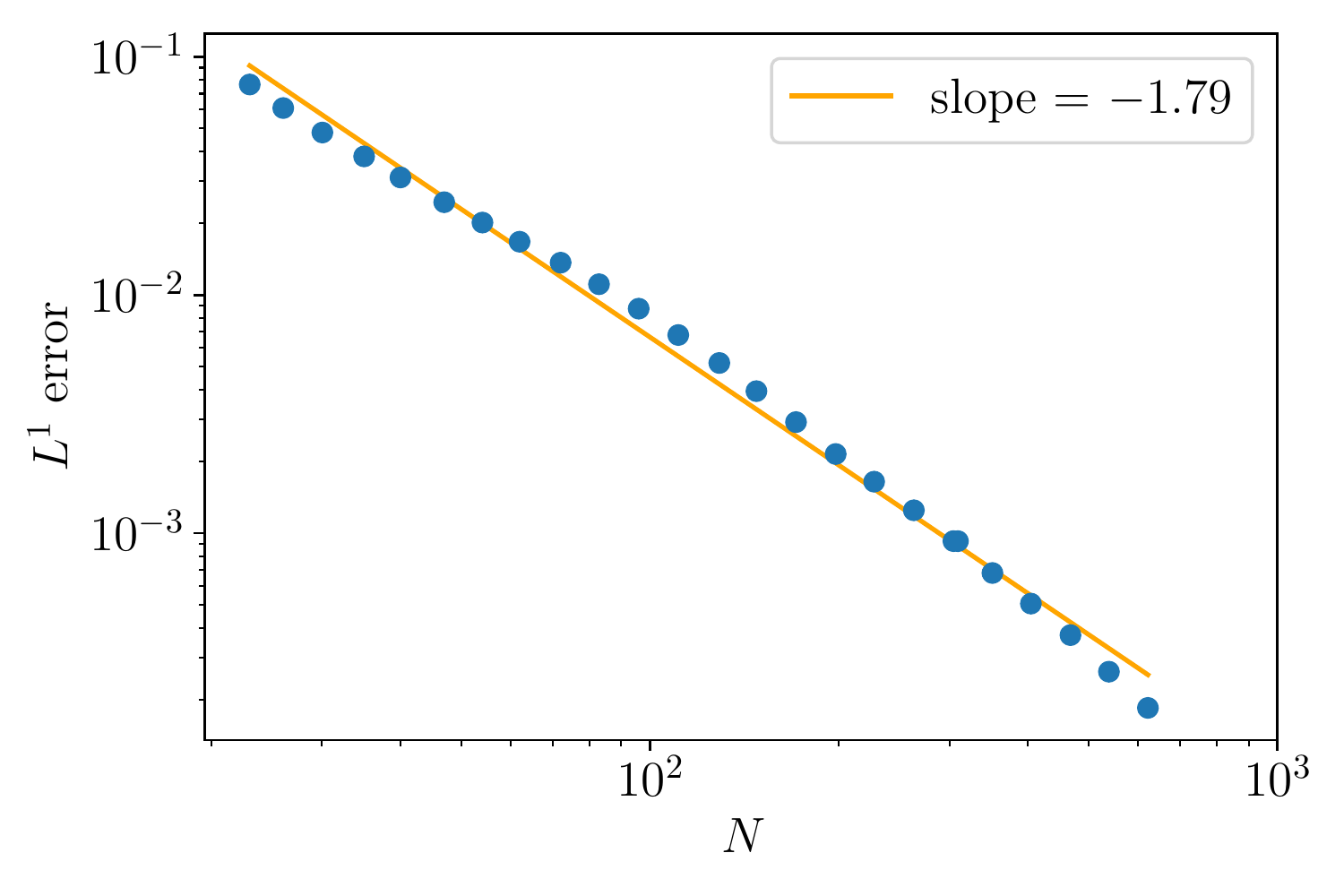} \\
$k= 10^9$ \\
\includegraphics[height=5.5cm]{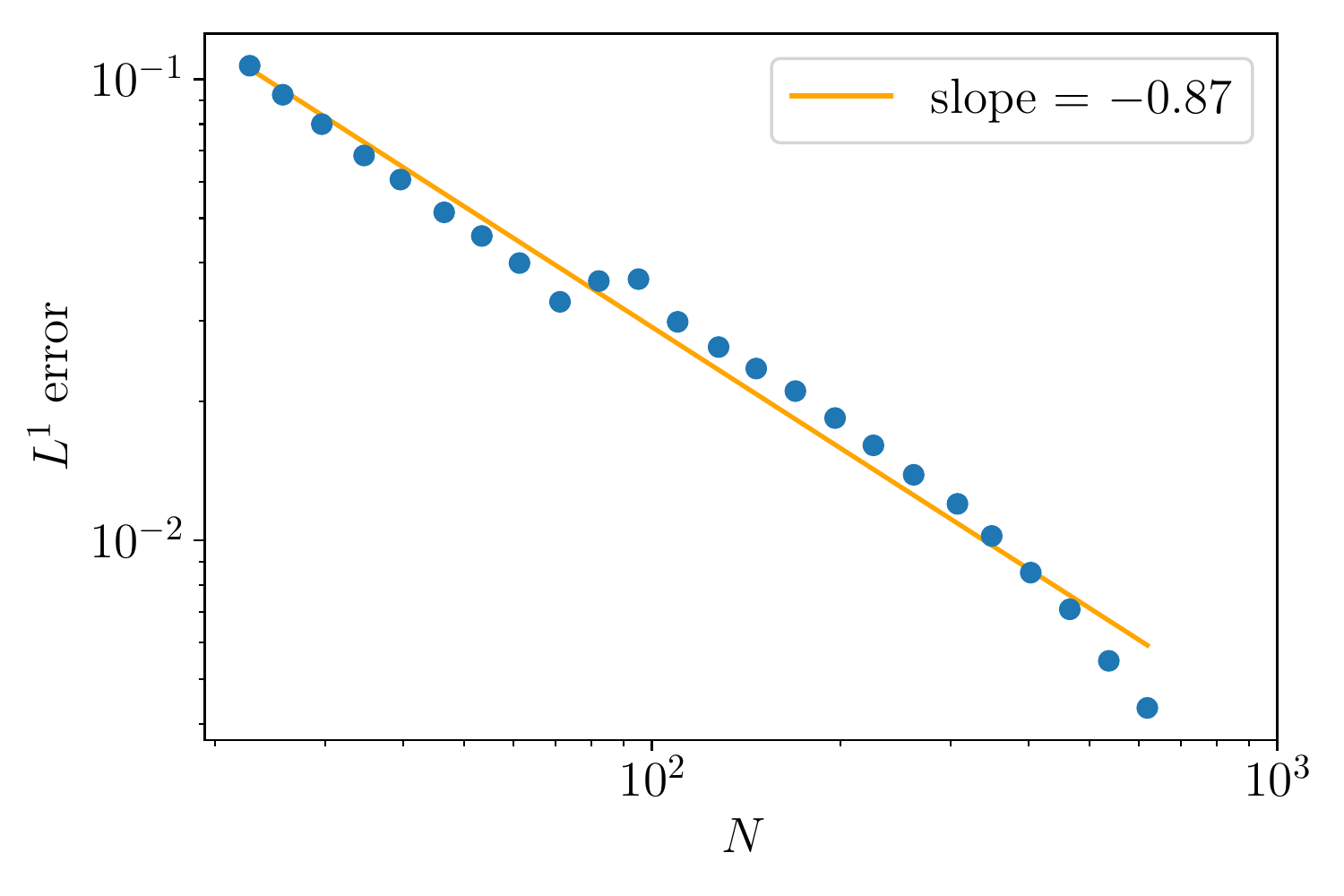}
\includegraphics[height=5.5cm,trim={1cm 0cm 0cm 0cm},clip]{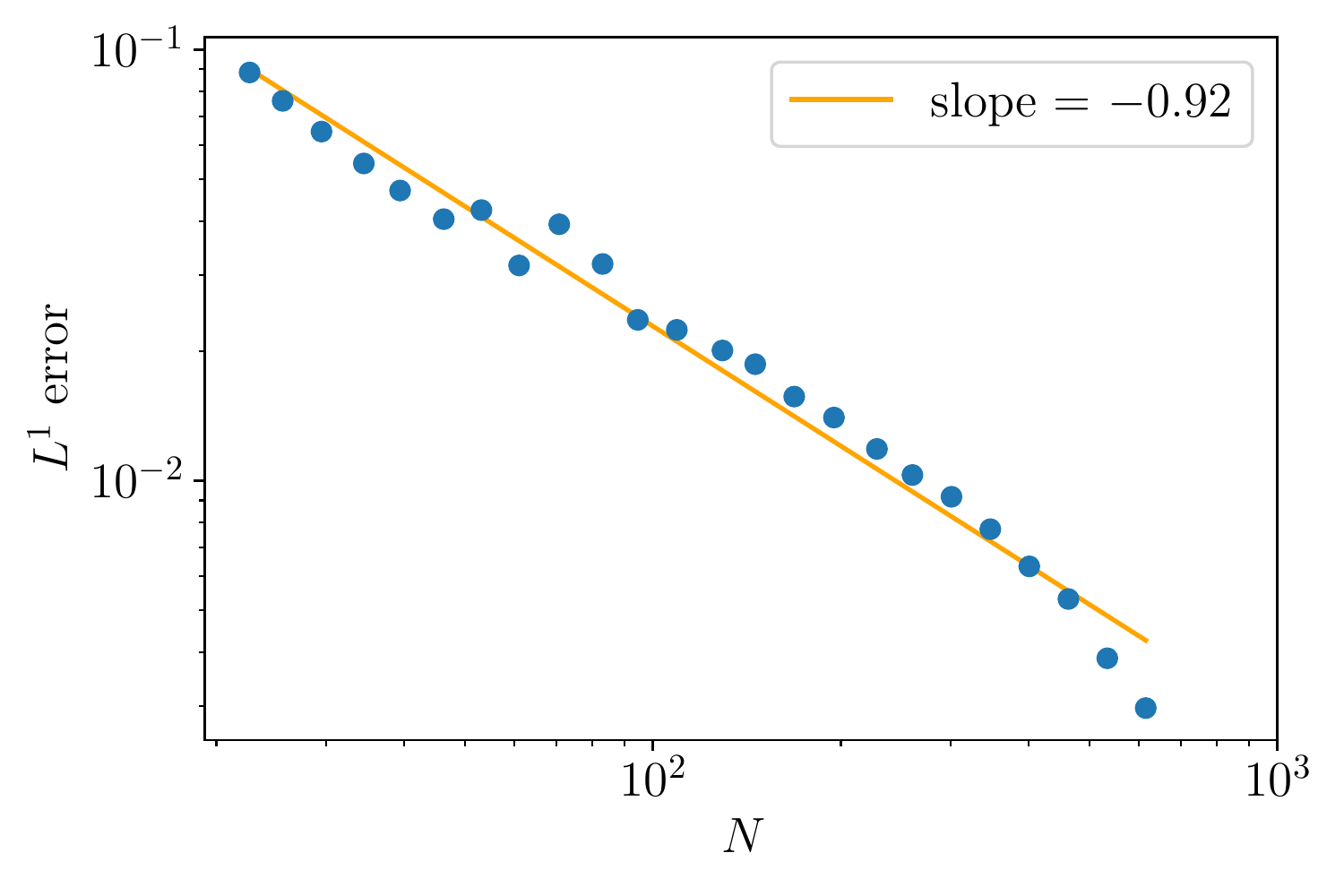}
\caption{The effect of $k$ on the rate of convergence in $N$ of the $L^1$ error (\ref{computeL1norm}) between $\tilde\rho^N_{\ep,k}$ and the numerical solution. In the left-hand column the target is $\bar\rho_{\text{uni}}$ and in the right-hand column the target is $\bar\rho_{\text{log-con}}$. Here $t=0.1$, and the initial condition is the Barenblatt profile (\ref{eq:barenblatt}).}
\label{convergenceWPMEBarenblatt}
\end{figure}

\subsubsection{Convergence to weighted porous medium equation}  

In Figures  \ref{convergenceWPMEBarenblatt} and \ref{convergenceWPMEChar}, we examine the rate of convergence of the kernel density estimate $\tilde{\rho}^N_{\ep, k}(t)$ as $k \to +\infty$, $\ep \to 0$, and $N \to +\infty$. Given that, for general $\bar{\rho}$, we lack an analytic expression for the solution $\rho(t)$ of  (\ref{mainpde}) to which we expect the solutions to converge, we instead compare our numerical solution with $N$ particles at time $t=0.1$ to the numerical solution with $N_{\text{max}} = 1,280$ particles at time $t=0.1$ via, 
\begin{align} \label{computeL1norm} L^1 \text{ error } =  \int_\Omega \left| \tilde{\rho}^N_{\ep(N), k}(x,t) - \tilde{\rho}^{N_{\text{max}}}_{\ep(N_{\text{max}}),k }(x,t) \right| dx , \quad  \quad t = 0.1 ,
\end{align}
where $\ep(N)$ is as in equation (\ref{epsrelateN}) and the integral is evaluated using the SciPy library's \texttt{quad} function. Furthermore, since we only expect good convergence rates when the solution of the underlying weighted porous medium equation is sufficiently regular, we restrict our attention to the smooth targets  $\bar{\rho}_\text{uni}$ and $\bar{\rho}_\text{log-con}$.

 In Figure \ref{convergenceWPMEBarenblatt}, we consider how the presence of a confining potential affects the rate of convergence, for both $\bar{\rho}_\text{uni}$ and $\bar{\rho}_\text{log-con}$. All simulations are conducted with Barenblatt initial data. We choose values of $N$ from $N=20$ to $N=640$, with logarithmic spacing.   In the top row, for no confinement ($k=0$), we observe  second order convergence. In the middle row, for moderate confinement ($k=100$), we observe slightly less than second order convergence. Finally, in the bottom row, for strong confinement ($k = 10^9)$, we observe less than first order convergence. 
 
  This example illustrates that there is a delicate balance underlying the choice of the stegnth of the confining potential. On one hand, the confinement must be selected to be sufficiently strong to prevent mass from leaking out of the domain and to ensure that the long time limit agrees well with the desired target; see Figure \ref{fig:fig3}. On the other hand,  selecting the confinement to be too strong can lead in a degradation of the rate of convergence as $\ep \to 0$, $N \to +\infty$, as more particles would be required for a given degree of accuracy. 
 
 \begin{figure}[h!]
 \hspace{1cm} $\bar{\rho}_\text{uni}$ \hspace{7cm} $\bar{\rho}_\text{log-con}$\\
\includegraphics[height=5.5cm]{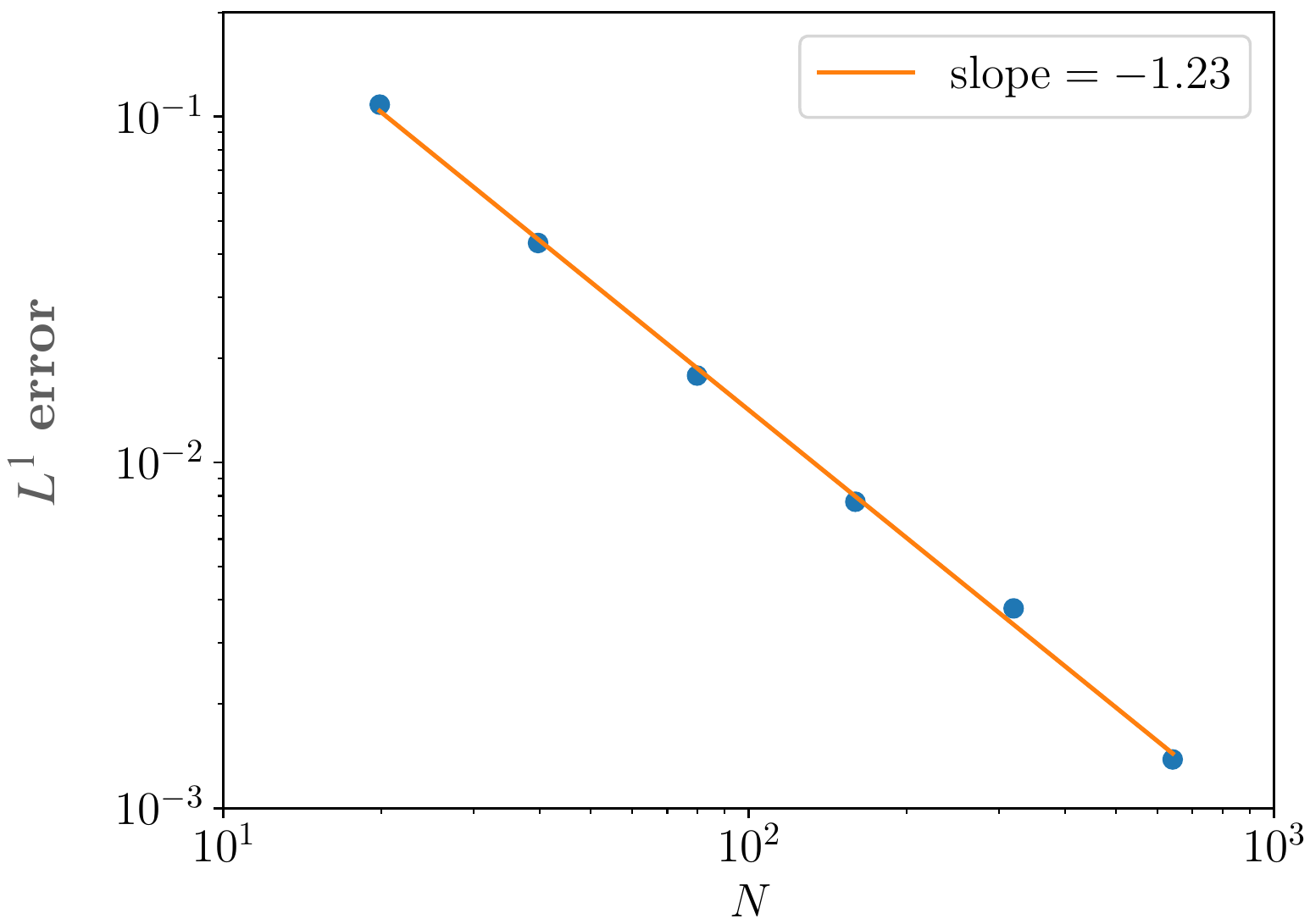}
\includegraphics[height=5.5cm,trim={1cm 0cm 0cm 0cm},clip]{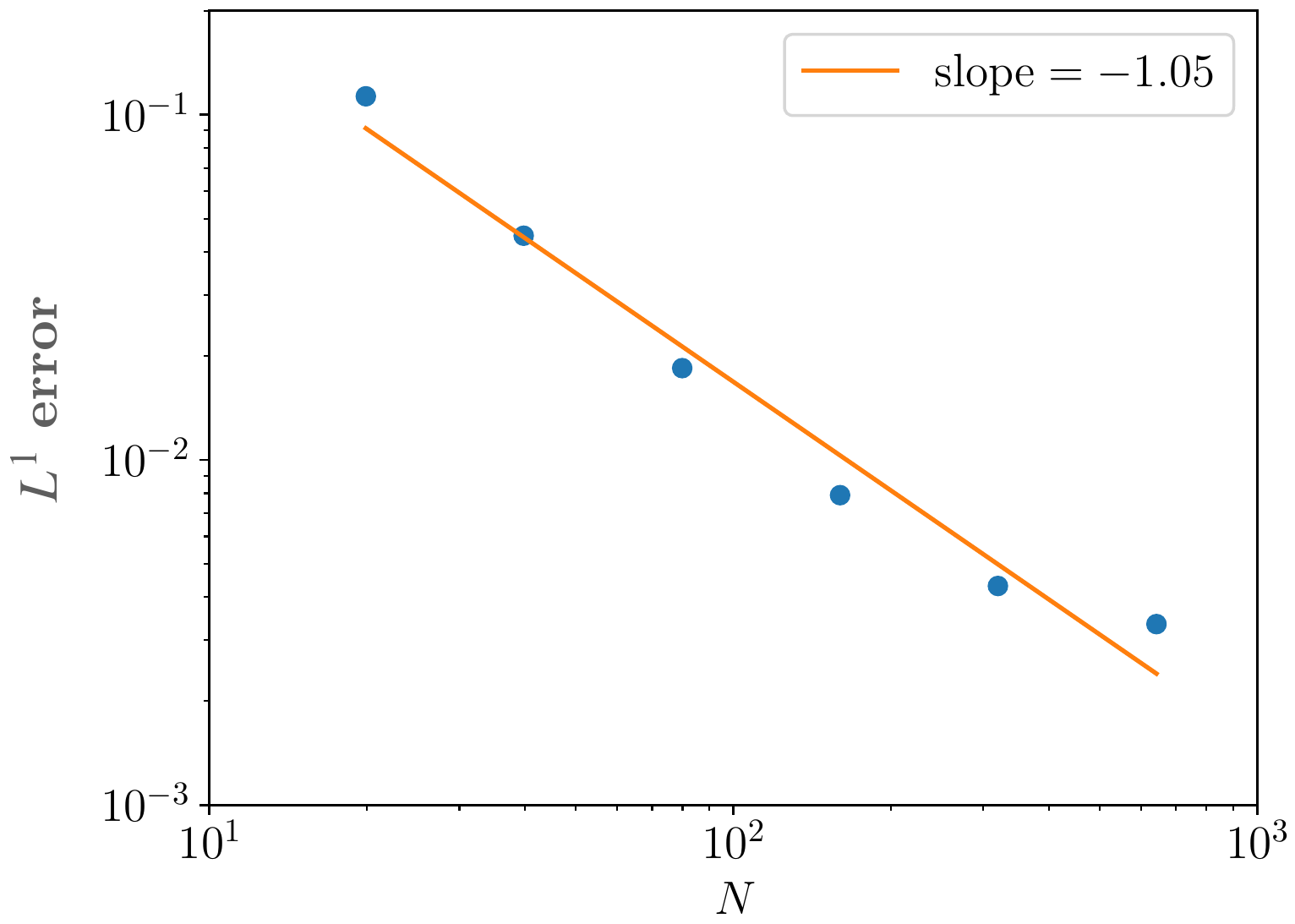} 
\caption{The effect of the initial condition on the rate of convergence in $N$ of the $L^1$ error (\ref{computeL1norm}) between $\tilde\rho^N_{\ep,k}$ and the numerical solution for two choices of target $\bar\rho$.  
Here $k=10^9$, $t=0.1$, and we take the uniform initial condition  (\ref{eq:uniID}).}
\label{convergenceWPMEChar}
\end{figure}

In Figure \ref{convergenceWPMEChar}, we consider the role   the initial conditions play  in determining the rate of convergence of the method. In particular, unlike the previous simulation, which was conducted with Barenblatt initial conditions, we now consider uniform initial conditions,
\begin{equation}
\label{eq:uniID}
    \mu_0(x) = \begin{cases} \frac12 &\text{ if } x \in [-1,1] , \\ 0 &\text{ otherwise.}\end{cases}
\end{equation} 
We consider the case of no confinement, $k = 0$, since the previous figure showed the fastest rate of convergence, of approximately second order, in this case; see Figure \ref{convergenceWPMEBarenblatt}, top row. We compute the $L^1$ error as in equation (\ref{computeL1norm}) with $N_\text{max} = 1,280$ and $N$ from $N =20$ to $N =640$ logarithmically spaced.

Unlike in the previous case, in which we observed near second order convergence in the absence of confinement, in this case we observe closer to first order convergence for both $\bar{\rho}_\text{uni}$ (left) and $\bar{\rho}_\text{log-con}$ (right).   We believe this is due to the fact that the continuum solution $\rho(t)$ of  (\ref{mainpde})  with uniform initial conditions, as above, has worse regularity than the solution for Barenblatt initial conditions. In previous work by the first author and Bertozzi \cite{CrBe14} on a regularized particle method for the related aggregation equation, which also has a gradient flow structure in the Wasserstein metric, it was shown that the rate of convergence of the   particle method  depended strongly on the regularity of the solution of the underlying PDE, in the sense that  lower regularity of the continuum solution led  to a slower rate of convergence of the numerical solution. While the convergence results in the present paper are purely qualitative, it appears that there may a similar dependence on regularity for the rate of convergence of our particle method to  (\ref{mainpde}).

 \begin{figure}[h!]
  \hspace{1cm} $\bar{\rho}_\text{uni}$ \hspace{7cm} $\bar{\rho}_\text{log-con}$\\
 \includegraphics[height=5.5cm]{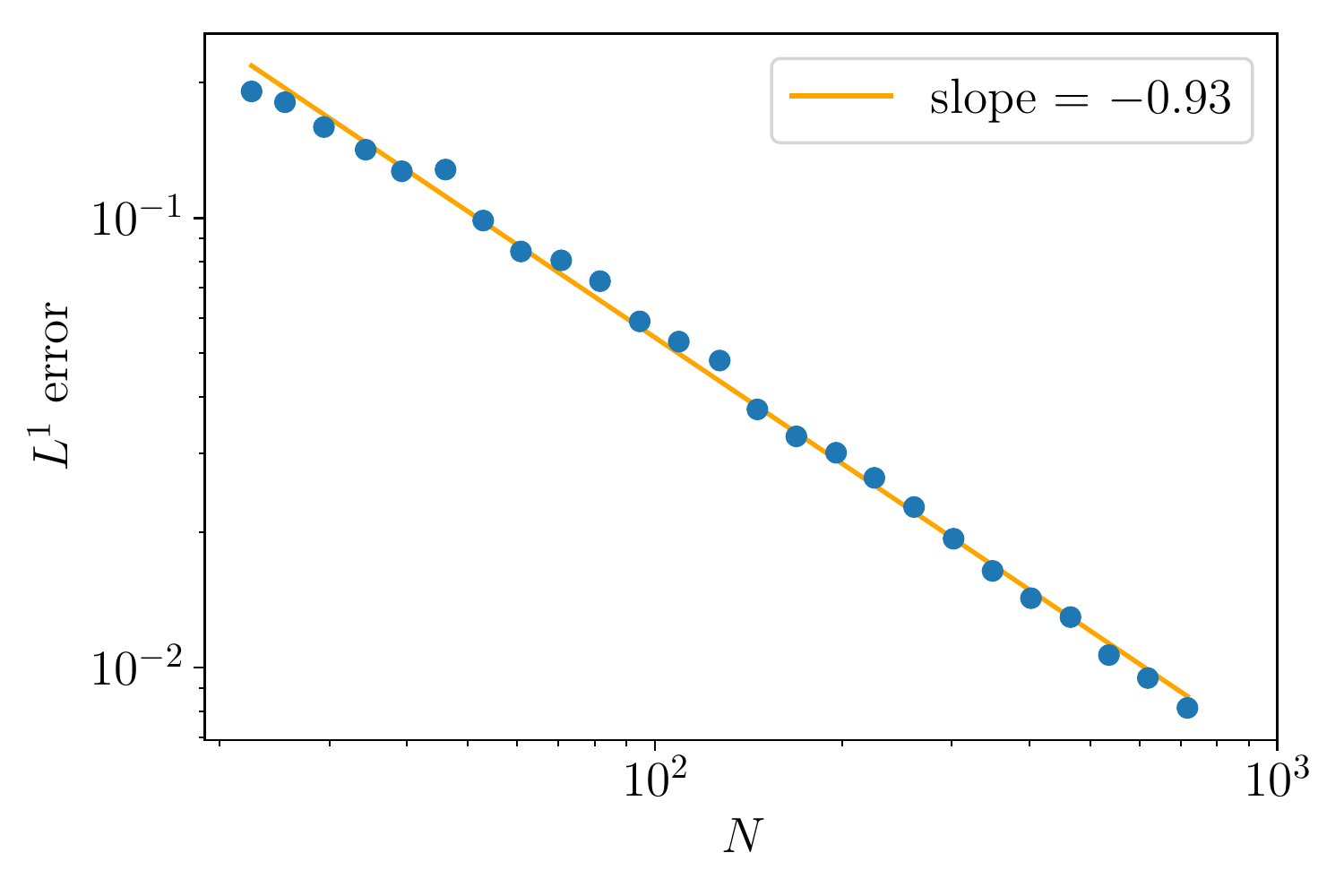}
 \includegraphics[height=5.5cm,trim={1cm 0cm 0cm 0cm},clip]{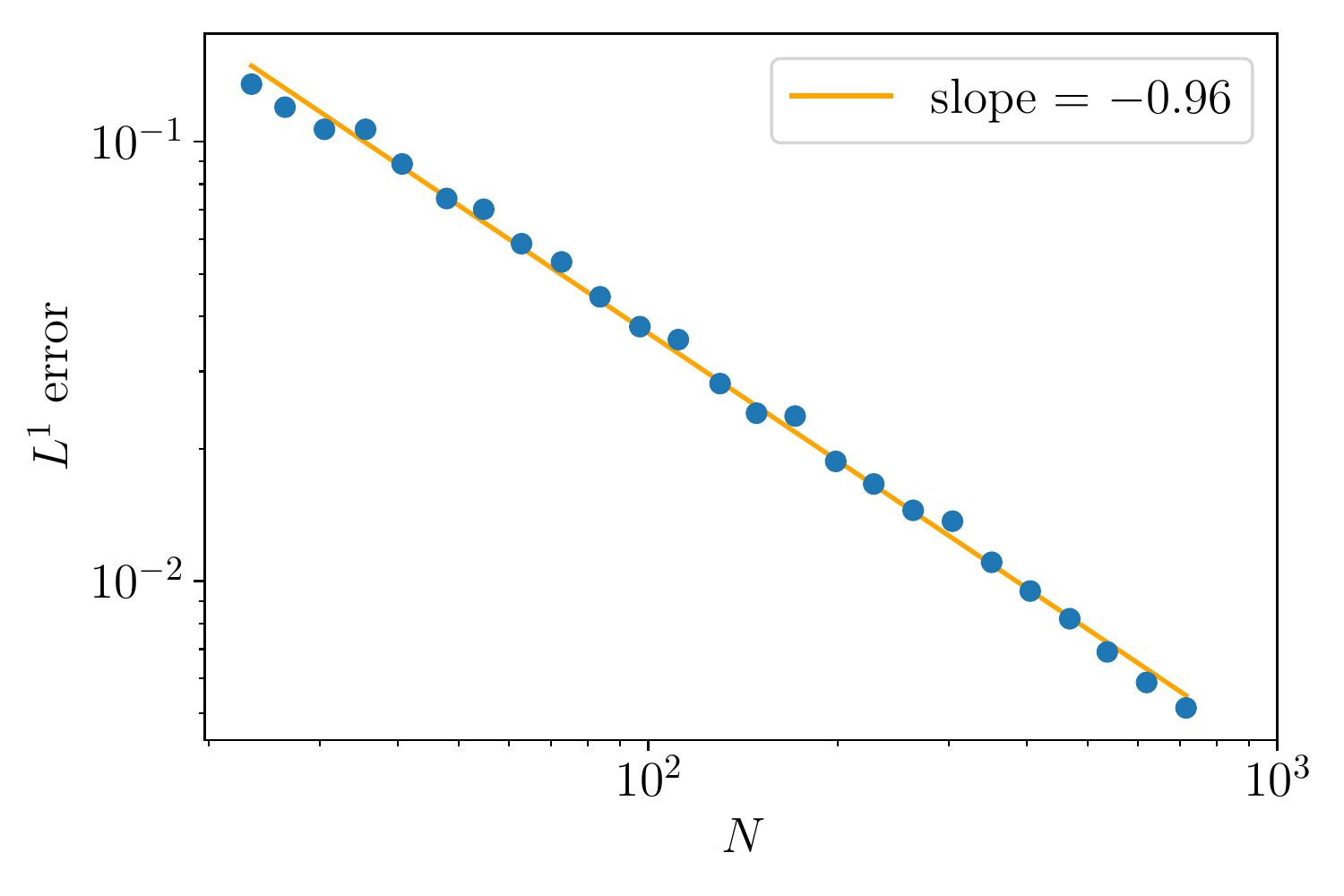}
 \caption{The  rate of convergence in $N$ of the $L^1$ error (\ref{computeL1normlongtime}) between $\tilde\rho^N_{\ep,k}$ and the target $\bar\rho$ for two choices of target. Here $t=2$, $k=10^9$, and we take the uniform initial condition (\ref{eq:uniID}).}
 \label{convergenceSteadyState}
 \end{figure}

 \subsubsection{Convergence to Steady State} 
 In Figure \ref{convergenceSteadyState}, we conclude our study of properties of the numerical method by  examining the rate of convergence  of the kernel density estimate $\tilde{\rho}^N_{ \ep, k}(t)$ to the  target  $\bar{\rho}$ in the long time limit, as the number of particles $N$ increases. As we only expect good convergence rates when the target  is sufficiently regular,  we restrict our attention to the smooth targets $\bar{\rho}_\text{uni}$ and $\bar{\rho}_\text{log-con}$. Furthermore, as illustrated in Figure \ref{fig:fig3}, since strong confinement is necessary to obtain convergence to the target  as $t \to +\infty$, we choose $k = 10^9$. We consider Barenblatt initial conditions and values of $N$ from $N = 20$ to $N = 720$, logarithmically spaced. We compute the $L^1$ error via,
 \begin{align} \label{computeL1normlongtime} L^1 \text{ error } =  \int_\Omega \left| \tilde{\rho}^N_{\ep,  k}(x,T) - \bar{\rho}(x) \right| dx , \quad  \quad T = 2.0 ,
\end{align}
where the integral is evaluated using the SciPy library's \texttt{quad} function.

For both $\bar{\rho}_\text{uni}$ and $\bar{\rho}_\text{log-con}$ we observe nearly first order convergence of our particle approximation to the target $\bar{\rho}$. This provides a quantitative numerical result to complement our qualitative result from Corollary \ref{quantcor}, in which we show that there exist parameters $T \to +\infty$, $k \to +\infty$, $\ep \to 0$, $N \to +\infty$ for which our particle method indeed provides a way to approximate $\bar{\rho}$ on $\Omega$, as relevant for applications in quantization.

\appendix
\addcontentsline{toc}{section}{Appendices}
\addtocontents{toc}{\protect\setcounter{tocdepth}{-1}}

\section{Wasserstein gradient flows} \label{appendix}
 
 We begin with a proof of Proposition \ref{metricslopesubdifabscts}, relating the metric slope and subdifferential.
\begin{proof}[Proof of Proposition \ref{metricslopesubdifabscts}]
By definition of the subdifferential and local slope, for all $\bgamma \in \Gamma_0(\mu,\nu)$,
\begin{align*}
|\partial \G|(\mu) &= \limsup_{\nu \to \mu}\frac{(\G(\mu) - \G(\nu) )_+}{W_2(\mu,\nu)} \leq  \limsup_{\nu \to \mu} \frac{1}{W_2(\mu,\nu)} \left( \int_{\Rd \times \Rd} \la \bxi(x), x-y \ra d \bgamma(x,y) - \frac{\lambda}{2} W_2^2(\mu,\nu) \right)_+  \\
&\leq \limsup_{\nu \to \mu}  \left(\frac{\|\bxi\|_{L^2(\mu)} W_2(\mu,\nu)}{W_2(\mu,\nu) } + \frac{\  \lambda_-}{2}W_2(\mu,\nu)\right) = \|\bxi\|_{L^2(\mu)} ,
\end{align*}
where $\lambda_- = \max \{ -\lambda, 0\}$. This shows inequality (\ref{subdiffmetricineq}). Uniqueness of the minimal subdifferential follows from the strict convexity of $\|\cdot \|_{L^2(\mu)}$.
\end{proof}

We now describe the proof of Theorem \ref{curvemaxslope}, which is a collection of results due to Ambrosio, Gigli, and Savar\'e that ensure well-posedness of Wasserstein gradient flows, as well as their characterization via curves of maximal slope.

\begin{proof}[Proof of Theorem \ref{curvemaxslope}]
Existence and uniqueness of the gradient flow, as well as the fact that the gradient flow is a curve of maximal slope, follows from \cite[Theorem 11.2.1]{ambrosio2008gradient}. 

Conversely, suppose $\mu(t) \in AC^2([0,T];\P_2(\Rd))$ is a curve of maximal slope in the sense of inequality (\ref{curvemaxslopedef}).
By the definition of strong upper gradient \cite[Definition 1.2.1]{ambrosio2008gradient}, Young's inequality, and the fact that, under the assumptions of the theorem, the local slope $|\partial \G|$ is a strong upper gradient \cite[Corollary 2.4.19]{ambrosio2008gradient}, $t \mapsto \G(\mu(t))$ is absolutely continuous, and
\begin{align*}
- \frac{d}{dt} \G(\mu(t)) \leq  \frac{1}{2} |\partial \G|^2(\mu(t)) + \frac{1}{2} |\mu'|^2(t), \text{ for almost every } t \geq 0 .
\end{align*}
Define $f(t) =  \frac{d}{dt} \G(\mu(t)) +  \frac{1}{2} |\partial \G|^2(\mu(t)) + \frac{1}{2} |\mu'|^2(t)$. Then we must have $f(t) \geq 0$ for a.e. $t \geq 0$ and inequality (\ref{curvemaxslopedef}) ensures
\[ \int_0^t f(r) dr \leq 0 \text{ for all } t \geq 0 . \]
Therefore, we must have $f(t) = 0$ for a.e. $t \geq 0$ and, integrating $f$ from $(a,b) \subseteq [0,+\infty)$, we see $\G(\mu(t))$ must be decreasing. This shows $\mu(t)$ is a curve of maximal slope in the pointwise sense of Ambrosio, Gigli, and Savar\'e \cite[Definition 1.3.2]{ambrosio2008gradient}. 
Finally, 
  \cite[Theorem 11.1.3]{ambrosio2008gradient} ensures it is a gradient flow of $\G$. (This theorem applies since functionals that are  $\lambda$-convex are \emph{regular}, in the sense required by the theorem, and functionals that are $\lambda$-convex along generalized geodesics satisfy the required coercivity assumption in \cite[equation 11.1.13b]{ambrosio2008gradient}: see  \cite[Lemma 10.3.8, Definition 10.3.9]{ambrosio2008gradient} for  {regularity} and \cite[Assumption 4.0.1, Lemma 4.1.1]{ambrosio2008gradient} for coercivity.)
  
Finally, the fact that $\mu(t)$ is a gradient flow of $\G$ if and only if it satisfies the Evolution Variational Inequality follows from \cite[Theorem 11.1.4]{ambrosio2008gradient}.
\end{proof}

Next, we define a discrete time approximation of a Wasserstein gradient flow, known as a \emph{minimizing movement scheme}, which was famously  introduced in the Wasserstein context by Jordan, Kinderlehrer, and Otto \cite{jordan1998variational}.
 
\begin{defn}[minimizing movement scheme] \label{minmovdef}
Suppose $\G$ is proper, lower semicontinuous, and $\lambda$-convex along generalized geodesics. Define the \emph{proximal operator} $\J_\tau$ by,
 \begin{align*}
\J_\tau \mu = \argmin_{\nu\in\mathcal{P}_2(\rr^d)} \left\{ \frac{1}{2 \tau} W_2^2(\mu, \nu) + \G(\nu) \right\} ,
\end{align*}
and define the \emph{minimizing movement scheme} $\J^n_{\tau} \mu$ by,
\begin{align*}
 \J^{n}_{\tau } ( \mu) = \underbrace{\J_{\tau} \circ \J_{\tau} \circ \dots \circ \J_{\tau}}_{\text{$n$ times}}(\mu).
\end{align*}

\end{defn}
Note that, by definition, the energy decreases along the minimizing movement scheme:
\begin{align} \label{energydecminmov}
\G(\J^n_\tau \mu) \leq \G(\J^{n-1}_\tau \mu).
\end{align}

We recall the following theorem on the convergence, due to Ambrosio, Gigli, and Savar\'e.

\begin{thm}[convergence of minimizing movement scheme, {\cite[Theorem 4.0.9]{ambrosio2008gradient}}] \label{minmovthm}
Suppose $\G$ is proper, lower semicontinuous, and $\lambda$-convex along generalized geodesics and $\mu \in D(\G)$.  Fix $T >0$, and take a piecewise constant interpolation of the minimizing movement scheme,
\begin{align} \label{pwconstinterp}
 \mu_\tau(s) = \J^n_{\tau} \mu \quad \text{ for } s \in ((n-1)\tau,n \tau] .
 \end{align}
Then for all $t \in [0,T]$, we have $\lim_{n \to +\infty}\mu_\tau(t) = \mu(t)$ narrowly, where $\mu(t)$ is the gradient flow of $\G$ with initial data $\mu$.
\end{thm}
\begin{proof}
This theorem is an immediate consequence of  \cite[Theorem 4.0.9]{ambrosio2008gradient}.
\end{proof}

We continue with an elementary result bounding the Wasserstein distance between a curve of maximal slope and a fixed reference measure.

\begin{prop}[$M_2$ bound for 2-absolutely continuous curves] \label{secondmomentbound}
 Suppose $\rho(t) \in AC^2([0,T];\P_2(\Rd))$.
  Then we have 
  \begin{align} \label{secmomtimebound}
M_2(\rho(t)) \leq  \left(1 + t e^t   \right) \left(M_2(\rho(0)) +  \int_0^T |\rho'|^2(r) dr \right) \quad \text{ for all } t \in [0,T].
\end{align}
\end{prop}
 
 \color{black}
 \begin{proof}
 We shall  show that, for any $\mu \in \P_2(\Rd)$, 
 \begin{align} \label{W22timebound}
 W_2^2(\rho(t),\mu) \leq   \left(1 + t e^t   \right) \left[ W_2^2(\rho(0),\mu) +  \int_0^T |\rho'|^2(r) dr \right] \quad \text{ for all } t \in [0,T].
\end{align}
The desired estimate then follows   from taking $\mu = \delta_0$, as in (\ref{eq:Wp Mp}).

Define $\mathcal{W}(\rho) = - \frac{1}{2} W_2^2(\rho, \mu)$. Since $\mathcal{W}$ is (-1)-convex  and lower semicontinuous \cite[Proposition 9.3.12]{ambrosio2008gradient}, the local slope $|\partial \mathcal{W}|(\rho)$ is a strong upper gradient for $\mathcal{W}$ (see \cite[Definition 1.2.1, Corollary 2.4.10]{ambrosio2008gradient}), which implies,  
\begin{align} \label{strongupperW22}
\left| \mathcal{W}(\rho(t)) - \mathcal{W}(\rho(0)) \right| \leq \int_0^t |\partial \mathcal{W}|(\rho(s)) |\rho'|(s) ds.
\end{align}
Furthermore, using the definition of local slope, rearranging, and applying the triangle inequality, yields,
\begin{align} \label{metricslopeW22}
	|\partial \mathcal{W}|(\rho) &= \limsup_{\nu \to \rho} \frac{  W_2^2(\nu,\mu) -W_2^2(\rho, \mu) }{2 W_2(\rho,\nu)} = \limsup_{\nu \to \rho} \frac{ ( W_2 (\nu,\mu) -W_2 (\rho, \mu)) ( W_2 (\nu,\mu) +W_2 (\rho, \mu)) }{2 W_2(\rho,\nu)} \nonumber \\ 
	&\leq \limsup_{\nu \to \rho} \frac{  W_2(\rho,\nu)(W_2 (\nu,\mu) +W_2 (\rho, \mu) )}{2 W_2(\rho,\nu)} = W_2(\rho,\mu).
\end{align}
Thus, combining (\ref{strongupperW22}) and (\ref{metricslopeW22}), we obtain,
\begin{align*}
\frac{1}{2} \left[ W_2^2(\rho(t),\mu) - W_2^2(\rho(0),\mu) \right] &\leq \left| \mathcal{W}(\rho(t)) - \mathcal{W}(\rho(0)) \right| \leq  \int_0^t W_2(\rho(s),\mu) |\rho'|(s) ds \\
&\leq  \| W_2(\rho(s),\mu) \|_{L^2([0,t])} \| |\rho'|(s) \|_{L^2([0,t])} \leq \frac12 \int_0^t W_2^2(\rho(s),\mu) ds + \frac12 \int_0^T |\rho'|^2(s) ds .
\end{align*}
By Gronwall's inequality, this implies inequality (\ref{W22timebound}).

 \end{proof}

We close this section by providing the construction of an  empirical measure approximating any measure $\mu \in \P_2(\Rd)$.

\begin{lem}[approximation via empirical measures]
\label{sketchlem2}
For all  $\mu\in \mathcal{P}_2(\rr^d)$ and  $\delta>0$,  there exists  $N \in \mathbb{N}$,  $\{X^i\}_{i=1,...N}\subseteq \rr^d$, and $\{m^i\}_{i = 1, \dots, N} \subseteq \rr^+$ with $\sum_{i=1}^N m^i =1$, such that  
$\mu^N = \sum_{i=i}^N \delta_{X^i}m^i$ satisfies
$W_2(\mu, \mu^N)\leq \delta$.
\end{lem}
\begin{proof}[Proof of Lemma \ref{sketchlem2}]

 Throughout this proof, we shall use $Q_r(0)$ to denote a cube in $\rr^d$ centered at $0$ and with side length $r>0$; namely, $Q_r(0) = \left[-\frac{r}{2}, \frac{r}{2}\right)^d$. 
For $x\in \rr^d$, let  $Q_r(x)= Q_r(0)+x$.

Fix $\mu\in \mathcal{P}_2(\rr^d)$ and   $\delta>0$.
 First, we reduce to the case of approximating a compactly supported measure. To this end,   note that since $\mu\in M_2(\rr^d)$,  
there exists   $R>0$ so that $\int_{Q_R(0)^c}|x|^2\, d\mu \leq \left(\frac{\delta}{2}\right)^2$. 
Consider the transport map,
\[
\bt_R(x) = \begin{cases}
x\quad \text{if } x\in Q_R(0),\\
0 \quad \text{otherwise},
\end{cases}
\]
and define $\mu_R$ via $\mu_R = (\bt_R)_{\#}\mu$. Then we find,
\begin{align}
\label{eq:construct emp measure est}
W_2(\mu, \mu_R) & \leq \left(\int|\bt_R(x)-x|^2\, d\mu\right)^{1/2} \leq \left(\int_{Q_R(0)^c}|x|^2\, d\mu\right)^{1/2}\leq \frac{\delta}{2},
\end{align}

We are now ready to define the approximating measure $\mu^N$. Choose $K\in \mathbb{N}$ large enough so that,
\begin{equation}
\label{eq:choiceN}
    K\geq \frac{2\sqrt{d}R}{\delta} ,
\end{equation}
  and consider a grid on $Q_R(0)$ where each cell has side length $R/K$, so that we have 
  $  Q_R(0) = \bigcup_{i=1}^{K^d}Q_{R/K}(X^i)$, 
where the centers $\{X^i\}_{i=1}^{m^d}$ are chosen such that the above union is disjoint. Let $N = K^d$, and define $\mu^N$ to be the sum of Dirac masses at the centers of the cells, with weights given by the mass of $\mu_R$ in each cell:
\[
\mu^N = \sum_{i=1}^{N}\delta_{X^i} m^i, \quad \text{with }m^i =  \mu_R \left(Q_{R/K}(X^i)\right).
\]

To estimate  $W_2(\mu_R, \mu^N)$, we consider the transport map $\bt:\rr^d\rightarrow \rr^d$ which, for $i=1,..., N$, moves all the mass in cell $Q_{R/K}(X^i)$ to $X^i$. Then $\mu^N = \bt_{\#} \mu_R$ and,
\[
W^2_2(\mu_R, \mu^N)\leq  \int |\bt(x) - x|^2\, d\mu_R = \sum_{i=1}^{N}\int_{Q_{R/K}(X^i)} |\bt(x) - x|^2\, d\mu_R \leq  \sum_{i=1}^{N}\int_{Q_{R/K}(X^i)} \left( \frac{\sqrt{d}R}{K} \right)^{2}   \, d\mu_R = \left( \frac{\sqrt{d}R}{K} \right)^2, 
\]
where the second inequality follows from the fact that  mass in the $i$th cell stays in the $i$th cell, so the largest distance mass could be moved is the diagonal length of the cell,   $\frac{\sqrt{d}R}{K}$.  Finally, we conclude by using the definition of $K$ in (\ref{eq:choiceN}), together with the  estimate (\ref{eq:construct emp measure est}), and the triangle inequality:
\[
W_2(\mu, \mu^N)\leq W_2(\mu, \mu_R)+W_2(\mu_R, \mu^N)\leq \frac{\delta}{2} +\frac{\delta}{2}.
\]
\end{proof}

\section{Further properties of energies and gradient flows with regularization and confinement} \label{energyappendix}

We provide the proof of Lemma \ref{lem:semicontinuity}, which ensures that the energies $\E$ and $\E_\ep$ are lower semicontinuous with respect to narrow convergence.

\begin{proof}[Proof of Lemma \ref{lem:semicontinuity}]
First we consider $\E$.  For this energy, lower semicontinuity  follows from the following result of Buttazo \cite[Corollary 3.4.2]{buttazzo1989semicontinuity}: given  $g: \mathbb{R}^d \times \mathbb{R} \rightarrow [0,+\infty]$, consider the functional $\G : \mathcal{P}(\mathbb{R}^d) \rightarrow [0,+\infty]$ defined by, 
\begin{equation}	
\label{eq:intfunc}
\G(\mu) =\begin{cases} \int_{\mathbb{R}^d} g(x,\mu(x))dx ~~ \text{ if } \mu \ll \mathcal{L}^d,\\
	+ \infty ~~ {\rm otherwise.}
	\end{cases}
	\end{equation}
Then if (i) $g$ is  lower semicontinous, (ii) for every $x \in \mathbb{R}^d$, the function $g(x,\cdot)$ is convex on $\mathbb{R}$, and (iii) there exists $\theta : \mathbb{R} \rightarrow \mathbb{R}$ with
 $\lim_{t \rightarrow \infty} \frac{\theta(t)}{t} = \infty ~{\rm and}~ g(x,y)\geq \theta (|y|)$  for every $ x \in \mathbb{R}^d, y \in \mathbb{R}$, then 
 the functional $\G$ is lower semicontinuous with respect to narrow convergence.

We now verify these hypotheses: our energy $\E$  is of the form \eqref{eq:intfunc}, for $g(x,y) =\frac{y^2}{2 \bar{\rho}(x)}$, which satisfies (i) and (ii). Furthermore, by setting $\theta (t) =Ct^2$, where $C =  (\max_{x \in\mathbb{R}^d}2 \bar{\rho}(x))^{-1}$, we see that $g$ satisfies (iii). Thus, $\E$ is lower semicontinuous with respect to narrow convergence.
	
	The lower semicontinuity of $\E_\ep$ follows directly from the definition of $\E_\ep(\mu) = \E(\zeta_\epsilon *\mu)$, Lemma \ref{weakst convergence mollified sequence}, and the lower semicontiuity of $\E$. 
\end{proof}

We now prove Proposition \ref{prop:convexity ep=0} by applying the general results of Ambrosio, Gigli, and Savar\'e \cite{ambrosio2008gradient} to immediately characterize the convexity of $\E$, $\V$, $\V_k$, and $\V_\ep$.

\begin{proof}[Proof of Proposition \ref{prop:convexity ep=0}]
First we show item \ref{econvex}. Define the log-concave extension of $\bar{\rho}$ by
\begin{align*}
\tilde{\rho} = e^{-W} , \quad W(x) = \begin{cases} - \log \bar{ \rho}(x) & \text{ if } x \in \Omega , \\ +\infty &\text{ otherwise.} \end{cases}
\end{align*}
In this way, $\tilde{\rho} = \bar{\rho}$ on $\Omega$, but $\tilde{\rho}$ is log-concave on all of $\mathbb{R}^d$. Furthermore, for all $\rho \in \P_2(\Rd)$,
\[ \E(\rho) + \V_\Omega(\rho ) = \begin{cases} \int_\Rd \frac{|\rho(x)|^2}{\tilde{\rho}(x)} d \mathcal{L}^d(x) & \text{ if }\rho \ll \tilde{ \rho}, \\ + \infty &\text{ otherwise.} \end{cases} \]
Finally, \cite[Theorem 9.4.12]{ambrosio2008gradient} ensures the energy on the right hand side is convex along generalized geodesics.

Item \ref{vconvex} is a consequence of the fact that, for any potential $W:\Rd \to \R\cup \{+\infty\}$ that is proper, lower semicontinuous, bounded below, and $\lambda$-convex, the corresponding energy $\rho \mapsto \int W \rho$ is $\lambda$-convex along generalized geodesics \cite[Proposition 9.3.2]{ambrosio2008gradient}. Next, recall that that $V \in C^2(\Rd)$  with Hessian bounded below implies   $D^2 V \geq \lambda I_{d \times d}$ for $\lambda = \inf_{\{x, \xi \in \Rd\}} \xi^t D^2 V(x) \xi$, hence we also have $D^2 (\zeta_\epsilon * V) \geq \lambda I_{d \times d}$ for all $\epsilon >0$. In particular, both $V$ and $(\zeta_\epsilon*V)$ are $\lambda$-convex, which implies $\V$ and $\V_\epsilon$ are $\lambda$-convex along generalized geodesics. Likewise, since $V_k$ is continuous, bounded below, and convex, $\V_k$ is convex along generalized geodesics. 
\end{proof}

Next we prove Proposition \ref{prop:subdiff char ep>0}, characterizing the minimal element of the subdifferential of $\F_{\ep, k}$. 
  
\begin{proof}[Proof of Proposition \ref{prop:subdiff char ep>0}]
For simplicity of notation, denote,
\begin{align} \label{vprojdef}
 \bxi =  \grad \frac{ \delta \E_\ep}{\delta \mu} + \nabla (\zeta_\epsilon*V)  +\nabla V_k.
\end{align}
Note that Lemma \ref{subdifflemma} and  Remark \ref{subdiffsumrem} on the additivity of the subdifferential ensure that $\bxi \in \partial \F_{\epsilon,k}(\mu)$.
In order to conclude  $\bxi  \in \partial^\circ \F_{\epsilon, k}$, it remains to show that   $\|\bxi  \|_{L^2(\mu)} \leq |\partial \F_{\epsilon,k}|(\mu)$. Proposition \ref{metricslopesubdifabscts}  will then give the result.

 Fix  $\psi \in C^1(\Rd)$ satisfying $\nabla \psi\in L^2(\mu)$, and define $\mu_\alpha = ( \id + \alpha \grad \psi)_{\#}\mu  $. By definition of the Wasserstein distance from $\mu$ to $\mu_\alpha$ in terms of minimizing over all transport plans from $\mu$ to $\mu_\alpha$, equation (\ref{eq:wass-p}), and the fact that $\left(\id \times (\id + \alpha \grad \psi) \right)_{\#} \mu$ is such a plan, 
\begin{align*}
W_2(\mu_\alpha, \mu)  \leq \| ( \id + \alpha \grad \psi) - \id \|_{L^2(\mu)} = \alpha \|\nabla \psi \|_{L^2(\mu)} .
\end{align*}
By definition of the metric slope,
\begin{align} \label{Gmetricslopelowerbound}
|\partial \F_{\epsilon, k}|(\mu) &= \limsup_{\nu \to \mu} \frac{ (\F_{\epsilon, k}(\mu) - \F_{\epsilon, k}(\nu))_+}{W_2(\mu,\nu)} \geq  \limsup_{\alpha \to 0} \frac{ (\F_{\epsilon, k}(\mu) - \F_{\epsilon, k}(\mu_\alpha))_+}{W_2(\mu,\mu_\alpha)} 
\\& \geq \frac{1}{\|\nabla \psi\|_{L^2(\mu)} } \limsup_{\alpha \to 0} \frac{ (\F_{\epsilon, k}(\mu) - \F_{\epsilon, k}(\mu_\alpha))_+}{\alpha} \nonumber .
\end{align}

We now apply inequality (\ref{Gmetricslopelowerbound})   to complete the proof.
  Recall from the sentence following assumption (\ref{Vkas}) that $V_k \in L^1(\nu)$ and $\nabla V_k \in L^2(\nu)$ for all $\nu \in \P_2(\Rd)$. Hence, $\mu_\alpha \in D(\F_{\epsilon,k})$ for all $\alpha \geq 0$. Thus, combining inequality (\ref{Gmetricslopelowerbound}) with Proposition \ref{prop:Fepder}, which  characterizes the directional derivatives of $\E_\epsilon$, $\V_\epsilon$, and $\V_k$, applied with,  
\[ \bgamma = (\id, \id , \id + \nabla \psi)_{\#} \mu , \]
  we obtain,
 \begin{align*}
|\partial \F_{\epsilon,k}|(\mu) \| \nabla \psi \|_{L^2(\mu)} &\geq \lim_{\alpha \to 0} \frac{ \E_\epsilon(\mu ) - \E_\epsilon(\mu_\alpha)}{\alpha} + \frac{\V_\epsilon(\mu) - \V_\epsilon(\mu_\alpha)}{\alpha} +  \frac{\V_k(\mu) - \V_k(\mu_\alpha)}{\alpha} \\
&= -\frac12 \int \frac{\zeta_\epsilon * \mu (x)}{\bar{\rho}(x)}  \int  \la \nabla \zeta_\epsilon\big(x-  y_2\big ) ,y_3-y_2 \ra d\pmb{\gamma} (y_1,y_2,y_3) \\
&\quad + \int \la \nabla (\zeta_\epsilon*V)(y_2) + \nabla V_k(y_2),y_3-y_2\ra dx \\
&=  - \int     \la  \frac12\left( \nabla \zeta_\epsilon*  \left(\frac{\zeta_\epsilon * \mu }{\bar{\rho}}\right) \right)  + \nabla (\zeta_\ep*V) + \nabla V_k  ,   \nabla \psi   \ra \ d \mu . \end{align*}
 Since the above inequality holds for any $\psi \in C^1$ with $\nabla \psi \in L^2(\mu)$, taking, 
 \[ \psi = - \frac12 \left( \zeta_\epsilon*  \left(\frac{\zeta_\epsilon * \mu }{\bar{\rho}}\right) \right)  -   (\zeta_\epsilon*V) - V_k\text{, so that } \nabla \psi = - \bxi ,\]
 we obtain $|\partial \F_{\epsilon,k}|(\mu) \|\nabla \psi\|_{L^2(\mu)} \geq \|\nabla \psi\|_{L^2(\mu)}^2$. Dividing through by $ \|\nabla \psi\|_{L^2(\mu)} = \|\bxi \|_{L^2(\mu)}$   gives the result.
 \end{proof}
 We now turn to a proof of Proposition \ref{prop:PDE ep=0}, which characterizes  the gradient flow of $\F$ in terms of a partial differential equation.
\begin{proof}[Proof of Proposition \ref{prop:PDE ep=0}]
Note that $\mu$ is a gradient flow of $\F$, with initial data $\mu_0 \in \overline{D(\F)}$, then, according to Theorem \ref{curvemaxslope}, $\mu$ is unique and is a curve of maximal slope for $\F$.  
Since $\F \geq -\|V\|_\infty$, this implies that for any $t>0$,
\begin{align}
\int_0^t|\partial \F|^2 (\mu(r))\, dr &\leq \F(\mu_0) +\|V\|_\infty <+ \infty .    \label{eq:slope in L1}
\end{align}
This ensures that $|\partial \F |(\mu(t))<+\infty$ for $\mathcal{L}^1$ almost every $t>0$, and since $D(|\partial \F |) \subseteq D(\F)$, we also have, 
\begin{align} \label{FOmegapdecharenergyfinite}
    \F(\mu(t))< +\infty \quad \text{ for a.e. }t>0.
\end{align}

 By inequality (\ref{FOmegapdecharenergyfinite}), $\mu(t) \ll \mathcal{L}^d$   and $\mu = 0$ a.e. on $\Rd \setminus \overline{\Omega}$ for almost every $t \geq 0$. Furthermore, Proposition \ref{prop:subdiff char ep=0}   implies that, for almost every $t \geq 0$,  $(\mu(t)/\bar\rho)^2\in W^{1,1}_{\loc}(\Omega)$  and that there exists $\bxi(t)\in \partial^\circ \F(\mu)$ with, 
\begin{equation}
    \label{eq:bxi}
 \bxi(t) \mu(t)= \frac{\bar{\rho}}{2}    \nabla (\mu(t)^2/\bar\rho^2)  + \nabla V \mu(t)  \quad  \text{ on } \Omega \quad \text{ and } \quad |\partial \F|(\mu) = \| \bxi(t)\|_{L^2(\mu(t))}.
\end{equation}
By Definition \ref{gradientflowdef} of gradient flow, we obtain that $\mu$ satisfies the continuity equation (\ref{eq:continuity eqn defn gf}) with $v(t)=-\bxi(t)$. Using the expression (\ref{eq:bxi}) for $\bxi$ therefore yields (\ref{eq:PDE E}). Finally, the containment (\ref{eq:soln space2})  follows from inequality (\ref{eq:slope in L1}) and equation (\ref{eq:bxi}).

On the other hand, suppose $\mu$ solves (\ref{eq:PDE E}) and satisfies (\ref{eq:soln space1}-\ref{eq:soln space2}). Then, defining $\bxi$ on the support of $\mu$ via (\ref{eq:bxi}) implies that the hypotheses of Proposition \ref{prop:subdiff char ep=0} are satisfied, so $\bxi \in \partial^\circ \F (\mu)$. From this we find that   (\ref{eq:PDE E}) is exactly the continuity equation in Definition \ref{gradientflowdef}  of the gradient flow, with $v(t)=-\bxi(t)$ satisfying $\|v(t)\|_{L^2(\mu(t)} \in L^1_\loc(0,+\infty)$. Thus, we have that $\mu \in AC^2([0,T];\P_2(\Rd))$ \cite[Theorem 8.3.1]{ambrosio2008gradient}, hence $\mu(t)$ is the unique gradient flow of $\F$ with initial data $\mu_0$, completing the proof of the proposition.
\end{proof}

The next result is a proof of Proposition \ref{PDEFeps}, which characterizes the gradient flow of $\F_{\epsilon,k}$ in terms of a partial differential equation.

\begin{proof}[Proof of Proposition \ref{PDEFeps}]
Suppose that $\mu(t)$ is the gradient flow of $\F_{\epsilon,k}$. Then the fact that $\mu(t)$ satisfies (\ref{eq:PDE Eepk}) follows directly from Definition \ref{gradientflowdef}, Proposition \ref{prop:subdiff char ep>0}, and Theorem \ref{curvemaxslope}. 

Now suppose that $\mu(t)$ satisfies (\ref{eq:PDE Eepk}). Then, the fact that the velocity field in the continuity equations is uniformly bounded ensures, by \cite[Theorem 8.3.1]{ambrosio2008gradient}, that  $\mu \in AC^2([0,T];\P_2(\Rd))$. Thus, the fact that $\mu$ is the gradient flow of $\F_{\epsilon,k}$ is again a consequence of Definition \ref{gradientflowdef}, Proposition \ref{prop:subdiff char ep>0}, and Theorem \ref{curvemaxslope}.
\end{proof}

We now consider the proof of Proposition \ref{prop:ODE Eep}, which shows that the gradient flow of $\F_{\epsilon,k}$ beginning at an empirical measure remains an empirical measure for all time and characterizes the ODE governing the evolution of the locations of the Dirac masses.

\begin{proof}[Proof of Proposition \ref{prop:ODE Eep}]
First note that, for all $\ep>0$ fixed,  the function of $(X^1, ..., X^N)$ that appears on the right-hand side of (\ref{eq:ODEepk}) is Lipschitz continuous, and therefore the ODE system (\ref{eq:ODEepk}) is well-posed.  
Suppose $X^i(t)$ solves (\ref{eq:ODEepk}). We claim that it suffices to show that $\mu(t)  = \sum_{i=1}^N \delta_{X^i(t)}m^i $   solves (\ref{eq:PDE Eepk}). Proposition \ref{PDEFeps} then ensures that $\mu(t)$ is the unique solution of the gradient flow.

The fact that     $\lim_{t \to 0^+} \mu(t) = \mu(0) $ in $W_2$  follows immediately from the definition of $\mu(t)$ and $\mu(0)$. 
Next, note that,  
\begin{align}
-\int_{\rr^d} \nabla \zeta_\ep(X^i(t)-z)\frac{1}{\bar\rho(z)}\sum_{j=1}^{N} m^j\zeta(z-X^j(t)) \, dz   \nonumber &= -\int_{\rr^d} \nabla \zeta_\ep(X^i(t)-z)\frac{1}{\bar\rho(z)} \zeta(z-y) \, d \mu(y)dz    \nonumber  \\
&= -\left(  \nabla\zeta_\ep*\left(\frac{\zeta_\ep*\mu}{\bar \rho}\right) \right)(X^i(t)) . \label{eq:velocity at mu ep}
\end{align}
Now, fix a test function $f\in C^\infty_c(\rr^d\times (0,+\infty))$.  By the Fundamental Theorem of Calculus and equation (\ref{eq:velocity at mu ep}), for each $1\leq i\leq N$,
\begin{align*}
 0 &= \int_0^\infty \frac{d}{dt}f(X^i(t),t)\, dt = \int_0^\infty \la \nabla f(X^i(t), t) , \dot{X}^i(t)\ra + \partial_t f(X^i(t), t)\, dt  \\
 &= \int_0^\infty \la\nabla f(X^i(t), t) , \left(-\left(  \nabla\zeta_\ep*\left(\frac{\zeta_\ep*\mu}{\bar \rho}\right) \right)(X^i(t))-  \nabla (\zeta_\epsilon*V)(X^i(t)) - \nabla V_k(X^i(t))\right) \ra + \partial_t f(X^i(t), t) \, dt.
\end{align*}
Multiplying by $m^i$, summing over $i$, and recalling the definition of $\mu$ yields,
\begin{align*}
0  
&= \int_0^\infty \int_{\rr^d} \la\nabla f(x, t) , \left(-\left(  \nabla\zeta_\ep*\left(\frac{\zeta_\ep*\mu}{\bar \rho}\right) \right)(x)-  \nabla (\zeta_\epsilon*V)(x) - \nabla V_k(x) \right)\ra + \partial_t f(x, t) \, d\mu  (x,t) \, dt .
\end{align*}
Thus,   $\mu $ is a distributional solution of the continuity equation (\ref{eq:PDE Eepk}).
  \end{proof}

\section{Explicit formulas for numerical method} \label{formulassection}
In this section, we collect a few explicit formulas that we use in the implementation of our numerical method. 
For our choices of uniform (\ref{uniform}), log-concave (\ref{logconcave}), and piecewise constant (\ref{pwconstant}) target, we have explicit formulas for the functions $f(x,y)$ and $g(x,y)$ defined in section \ref{numericaldetailssec}: see equations (\ref{fdefintro}) and (\ref{gxydef}). For the log-concave target measure, we obtain,
	\begin{align*}
 	f(x_i,x_j) &= \frac{-2\ep^2 x_i - 6 \ep^2 x_j + x_i^3 + x_i^2x_j - x_i x_j^2 + 4 x_i - x_\J^3 - 4 x_j}{16 \sqrt{\pi} \ep^3} C_{\bar{\rho}}  e^{-(x_i-x_j)^2/(4 \epsilon^2)}  \\
g(x_i,x_j) &=  [\psi_{ij}(+\infty ) - \psi_{ij}(-\infty ) ] \\
	\psi_{ij}(z) &=  \frac{C_{\bar{\rho}}}{8}   \ep e^{\frac{-(x_i^2+x^2_j+2z^2)}{2\ep^2}}  \Big ( -\sqrt{\pi} (2\ep^2+x_i^2+2x_ix_j+z^2+4)  \\  & \quad\quad \quad e^{\frac{(x_i^2+x^2_j+2x_ix_j)}{4\ep^2}} {\rm erf} \big( \frac{x_i+x_j-2z}{2}\big) -2 \ep (x_i+x_j+2z) \ep^{\frac{z(x_i+x_j)}{\ep^2}}\Big ) .
	\end{align*}

	For the uniform and piecewise constant targets, note that both may be expressed as,
	\begin{align*} \bar{\rho}(x) &= \sum_{k =1}^N c_k  \indi_{[b_k,b_{k+1}]}(x), 
	\end{align*}
	where  $\{c_k\}_{k=1}^N$ are positive constants chosen so that $\int_\Omega \bar{\rho} = 1$, $\{b_k\}_{k=1}^{N+1} \subseteq \R$.
	For any  target  of this form, we obtain
\begin{align*} 
	f(x_i,x_j) &=    \sum_{k=1}^N c_k^{-1} \left[ \varphi_{ij}(b_{k+1}) - \varphi_{ij}(b_k) \right]  \\
	\varphi_{ij}(z) &= - \frac{e^{-(x_i^2 + 2 z^2 +   x_j^2)/(2 \epsilon^2)}}{8 \pi \epsilon^3} \left(2 \epsilon e^{z(x_i+x_j)/ \epsilon^2)} - \sqrt{\pi} (x_i - x_j) e^{((x_i+x_j)^2+4z^2)/(4 \epsilon^2)} {\rm erf}\left(\frac{x_i - 2 z + x_j}{2 \epsilon}\right) \right)  \nonumber \\
	&\quad \quad\quad + \frac{e^{-(x_i^2 +  x_\J^2)/(2 \epsilon^2)}}{8 \pi \epsilon^3} \left(2 \epsilon   - \sqrt{\pi} (x_i - x_j) e^{(x_i+x_j)^2 /(4 \epsilon^2)} {\rm erf}\left(\frac{x_i + x_j}{2 \epsilon}\right) \right) \\ 
	g(x_i,x_j) &=  \sum_{k=1}^N [\psi_{ij}(b_{k+1}) - \psi_{ij}(b_k ) ] \\
\psi_{ij}&=	\frac{-1 }{4\sqrt{\pi}}e^{\frac{-(x_i-x_j)^2}{4\ep^2}}{\rm erf}\big(  \frac{x_i+x_j -2z}{2\ep} \big).
\end{align*}

\addtocontents{toc}{\protect\setcounter{tocdepth}{1}}

\bibliographystyle{abbrv}
\bibliography{Blob}
\end{document}